\documentclass[twoside,11pt]{article}
\usepackage{jmlr2e}

\usepackage[dvipsnames]{xcolor}
\usepackage{amsmath, amssymb, amsfonts}
\usepackage{thm-restate}
\usepackage{mathtools, bm, enumitem}
\usepackage{booktabs, makecell}
\usepackage{algorithm, algpseudocode}
\usepackage{graphicx}
\usepackage{color}
\usepackage{upgreek}
\usepackage{minitoc}
\usepackage{lastpage}
\usepackage{hyperref}
\hypersetup{breaklinks, hidelinks}

\newtheorem{assumption}[theorem]{Assumption}

\let\originalparagraph\paragraph
\renewcommand{\paragraph}[2][.]{\originalparagraph{#2#1}}

\let\originalsubparagraph\subparagraph
\renewcommand{\subparagraph}[2][.]{\originalsubparagraph{#2#1}}

\newcommand{\reals}{{\mathbb R}}

\newcommand{\ones}{\operatorname{\mathbf 1}}
\newcommand{\idm}{\operatorname{I}}

\newcommand{\diag}{\operatorname{\bf diag}}

\DeclareMathOperator*{\argmax}{arg\,max}
\DeclareMathOperator*{\argmin}{arg\,min}

\newcommand{\state}{x}
\newcommand{\states}{{\bm \state}}
\newcommand{\ctrl}{u}
\newcommand{\ctrls}{{\bm \ctrl}}

\newcommand{\auxstate}{y}
\newcommand{\auxstates}{{\bm \auxstate}}
\newcommand{\auxxstate}{z}
\newcommand{\auxxstates}{{\bm \auxxstate}}
\newcommand{\auxctrl}{v}
\newcommand{\auxctrls}{{\bm \auxctrl}}
\newcommand{\auxxctrl}{w}
\newcommand{\auxxctrls}{{\bm \auxxctrl}}
\newcommand{\costate}{\lambda}
\newcommand{\costates}{{\bm \costate}}
\newcommand{\dualvar}{\mu}
\newcommand{\dualvars}{{\bm \mu}}

\newcommand{\horizon}{\tau}

\newcommand{\dimstate}{{n_x}}
\newcommand{\dimctrl}{{n_u}}
\newcommand{\microdimctrl}{{m_{\ctrl}}}

\newcommand{\dyn}{f}
\newcommand{\microdyn}{\phi}
\newcommand{\contdyn}{{\psi}}

\newcommand{\ksteps}{k}
\newcommand{\auxdyn}{\varphi}
\newcommand{\diffeostate}{a}
\newcommand{\diffeoctrl}{b}
\newcommand{\traj}{{\dyn^{[\horizon]}}}
\newcommand{\trajin}[1]{{\microdyn^{\{#1\}}}}
\newcommand{\microtraj}{{\microdyn^{[\ksteps]}}}

\newcommand{\augtraj}{g}
\newcommand{\Diffeoctrl}{B}
\newcommand{\Microdyn}{\Phi}
\newcommand{\concatdyn}{F}

\newcommand{\cost}{h}

\newcommand{\obj}{\mathcal{J}}

\newcommand{\initstate}{\bar \state_0}

\newcommand{\costogo}{c}

\renewcommand{\H}{P}
\newcommand{\h}{p}
\newcommand{\G}{Q}
\newcommand{\g}{q}
\newcommand{\A}{A}
\newcommand{\B}{B}

\newcommand{\J}{J}
\renewcommand{\j}{j}

\renewcommand{\P}{P}
\newcommand{\p}{p}
\newcommand{\Q}{Q}
\newcommand{\q}{q}
\newcommand{\R}{R}

\newcommand{\lin}{\ell}
\newcommand{\qua}{q}

\renewcommand{\next}{\operatorname{next}}

\newcommand{\var}{\ctrls}

\newcommand{\grad}{G}
\newcommand{\hess}{H}
\newcommand{\orderlin}{r}
\newcommand{\auxxxstate}{\zeta}

\newcommand{\auxxxctrl}{\omega}
\newcommand{\auxxxctrls}{{\bm \omega}}

\newcommand{\Oracle}{\operatorname{Oracle}}

\newcommand{\LQR}{\operatorname{LQR}}
\newcommand{\DDP}{\operatorname{DDP}}

\newcommand{\reg}{\nu}
\newcommand{\regscaled}{\bar \reg}
\newcommand{\newton}{ {\bm n}}
\newcommand{\lambdaquad}{{\uplambda}}
\newcommand{\deltaquad}{{\updelta}}

\newcommand{\lip}{l}
\newcommand{\smooth}{L}
\newcommand{\smoothess}{M}
\newcommand{\lipdynstate}{{\lip_{\dyn}^\state}}
\newcommand{\lipdynctrl}{{\lip_{\dyn}^\ctrl}}
\newcommand{\smoothdynstate}{{\smooth_{\dyn}^{\state\state}}}
\newcommand{\smoothdynctrl}{{\smooth_{\dyn}^{\ctrl\ctrl}}}
\newcommand{\smoothdynstatectrl}{{\smooth_{\dyn}^{\state\ctrl}}}
\newcommand{\lipdiffeoctrl}{{\lip_{\diffeoctrl}^\auxstate}}

\newcommand{\strgcvx}{\mu}

\newcommand{\simp}{\alpha}
\newcommand{\newsimp}{\beta}
\newcommand{\polysqrt}{\gamma}
\newcommand{\condnb}{\rho}
\newcommand{\localcondnb}{\varrho}
\newcommand{\concord}{\theta_\cost}
\newcommand{\localconcord}{\vartheta_\cost}
\newcommand{\scaling}{{\theta_\augtraj}}
\newcommand{\localscaling}{\vartheta_\augtraj}
\newcommand{\pl}{r}
\newcommand{\plcst}{\mu}
\newcommand{\plcstobj}{m}
\newcommand{\ddpbound}{\eta}
\newcommand{\condnbddp}{\chi}

\newcommand{\linesearch}{\operatorname{LineSearch}}

\newcommand{\cst}{\xi}

\newcommand{\ham}{H}
\newcommand{\trueham}{\bar H}
\newcommand{\intercost}{m}
\newcommand{\finalcost}{\cost_\horizon}

\newcommand{\fullcost}{h}
\newcommand{\fullaugtraj}{g}
\newcommand{\fullobj}{\mathcal{J}}
\newcommand{\polcst}{{\bm k}}
\newcommand{\rollout}{\mathrm{rollout}}

\newcommand{\coderef}{\url{https://github.com/vroulet/ilqc}}

\doparttoc

\jmlrheading{26}{2025}{1-\pageref{LastPage}}{11/22; Revised
2/25}{5/25}{22-1271}{Vincent Roulet, Siddhartha Srinivasa, Maryam Fazel, Zaid Harchaoui}
\ShortHeadings{Iterative Linear Quadratic Algorithms for Nonlinear Control}{Roulet, Srinivasa, Fazel, Harchaoui}
\firstpageno{1}

\begin{document}

\title{
  On Global and Local Convergence \\
  of Iterative Linear Quadratic Optimization Algorithms \\
  for Discrete Time Nonlinear Control
}

\author{\name Vincent Roulet \email vroulet@google.com \\
\addr Google DeepMind \\
Seattle, WA, USA
\AND
\name Siddhartha Srinivasa \email siddh@cs.uw.edu \\
\addr Paul G. Allen School of Computer Science and Engineering\\
University of Washington \\
Seattle, WA, USA
\AND  
\name Maryam Fazel \email mfazel@uw.edu \\
\addr Department of Electrical and Computer Engineering \\
University of Washington \\
Seattle, WA, USA
\AND 
\name Zaid Harchaoui \email zaid@uw.edu \\
\addr Department of Statistics\\
University of Washington\\
Seattle, WA, USA
}

\editor{Martin Jaggi}

\maketitle

\faketableofcontents
\begin{abstract}
A classical approach for solving discrete time nonlinear control on a finite
horizon consists in repeatedly minimizing linear quadratic approximations of the
original problem around current candidate solutions. While widely popular in
many domains, such an approach has mainly been analyzed locally. We
provide detailed convergence guarantees to stationary points as well as local
linear convergence rates for the Iterative Linear Quadratic Regulator (ILQR)
algorithm and its Differential Dynamic Programming (DDP) variant. For problems
without costs on control variables, we observe that global convergence to minima
can be ensured provided that the linearized discrete time dynamics are
surjective, costs on the state variables are gradient dominated. We further
detail quadratic local convergence when the costs are self-concordant. We show
that surjectivity of the linearized dynamics hold for appropriate discretization
schemes given the existence of a feedback linearization scheme. We present
complexity bounds of algorithms based on linear quadratic approximations through
the lens of generalized Gauss-Newton methods. Our analysis uncovers several
convergence phases for regularized generalized Gauss-Newton algorithms. 
\end{abstract}


\section{Introduction}
We consider nonlinear control problems in discrete time of the form 
\begin{align}
	\min_{\substack{\ctrl_0,\ldots, \ctrl_{\horizon-1} \in \reals^\dimctrl\\ \state_0, \ldots, \state_\horizon \in \reals^\dimstate}} \quad 
	& \sum_{t=0}^{\horizon-1} \cost_t(\state_t, \ctrl_t) + \cost_\horizon(\state_\horizon)\label{eq:discrete_pb_ctrl_cost}\\
	\mbox{subject to} \quad & 
	\state_{t+1} = \dyn_t(\state_t, \ctrl_t) \quad \mbox{for} \ t \in \{0, \ldots, \horizon-1\}, 
	\qquad \state_0 = \initstate,\nonumber
\end{align}  
where at the time index $t$, $\state_t$ is the state of the system, $\ctrl_t$ is
the control applied to the system, $\dyn_t$ is the discretized nonlinear
dynamic, $\cost_t$ is the cost applied to the system state and the control
variable,  and $\initstate$ is a given fixed initial state.

Problems of the form~\eqref{eq:discrete_pb_ctrl_cost} have been tackled in
various ways, from direct approaches using nonlinear optimization
\citep{jacobson1970differential, bock1984multiple, de1988differential,
dunn1989efficient, wright1990solution, wright1991partitioned,
rao1998application, betts2010practical} to convex relaxations using semidefinite
optimization~\citep{boyd1997semidefinite}. Numerous packages exist for such
problems such as CasAdi~\citep{andersson2018cassadi},
Pyomo~\citep{bynum2021pyomo}, JumP~\citep{dunning2017jump},
IPOPT~\citep{wachter2006implementation}, or SNOPT~\citep{gill2005snopt},
Crocoddyl~\citep{jallet2023proxddp}, acados~\citep{verschueren2021acados}. A
popular approach proceeds by computing at each iteration the linear quadratic
regulator associated to a linear quadratic approximation of the problem around
the current candidate solutions~\citep{jacobson1970differential,li2004iterative,
sideris2005efficient, tassa2012synthesis}. The resulting feedback policy can
then be applied on the linearized dynamics as in the Iterative Linear Quadratic
Regulator (ILQR) algorithm~\citep[Section 8.8.5]{rawlings2017model},
\citep{li2004iterative, sideris2005efficient}. Alternatively, the feedback
policy can be applied on the original dynamics, as in the iterative Linear
Quadratic Regulator (iLQR) algorithm~\citep[Section
8.8.6]{rawlings2017model},~\citep{tassa2012synthesis}, akin to a Differential
Dynamic Programming (DDP) approach~\citep{mayne1966second}. To avoid confusion,
we name this second approach Iterative Dynamic Differentiable Programming
(IDDP).

\paragraph{Motivation}
Empirically, these approaches often exhibit fast convergence to efficient or
optimal controllers which explain their popularity in applied
control~\citep{tassa2012synthesis, giftthaler2018family}, and the renewed
interest for linear quadratic control in neuro-dynamic programming and
reinforcement learning~\citep{fazel2018global, recht2019tour,
kakade2020information, simchowitz2020naive, westenbroek2021stability}. The
empirical performance of the ILQR and IDDP algorithms are illustrated in
Figure~\ref{fig:conv}. The first problem considered in Figure~\ref{fig:conv}
consists in swinging up a pendulum to a vertical position in finite time, the
second problem consists in controlling a simple model of a car to be at
predefined positions at given times. The detailed experimental setting is
presented in Section~\ref{sec:exp}. Most importantly, the costs consists
in quadratic state costs  bounded below by 0, i.e., of the form
$\cost_t(\state_t, \ctrl_t) = (\state_t - \hat \state_t)^\top Q_t (\state_t -
\hat \state_t)$ for $Q_t$ positive definite and $\hat x_t$ a reference
state. In Figure~\ref{fig:conv}, we plot $c^{(k)}/c^{(0)}$ in log-scale, where
$c^{(k)} \geq 0$ denotes the total cost at iteration $k$ computed by means of a
gradient descent, an ILQR algorithm or an IDDP algorithm, and $c^{(0)}$ denotes
an initial cost given by initializing the control variables at $0$. We observe
that both the ILQR and the IDDP algorithms converge to an optimal cost, i.e.,
$c^{(k)} \rightarrow 0$. Moreover, both algorithms outperform a simple gradient
descent and appear to exhibit a fast convergence after some iterations. The
empirical behavior illustrated in Figure~\ref{fig:conv} does not hold for any
nonlinear control problem as illustrated in Appendix~\ref{app:exp_sup} with a
more realistic model of a car taken from~\citet{liniger2015optimization}. Yet,
the examples presented in Figure~\ref{fig:conv} are surprising from an
optimization viewpoint as the problems considered escape the usual paradigm of
convex or linear optimization. 

\begin{figure}
	\begin{center}
        $\vcenter{\hbox{\includegraphics[width=0.15\linewidth]{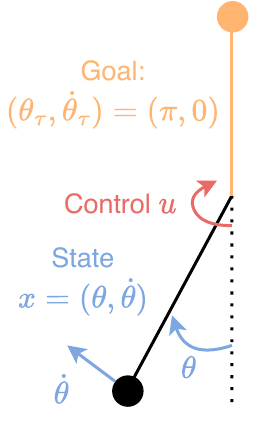}}}$\hspace*{0pt}
	$\vcenter{\hbox{\includegraphics[width=0.65\linewidth]{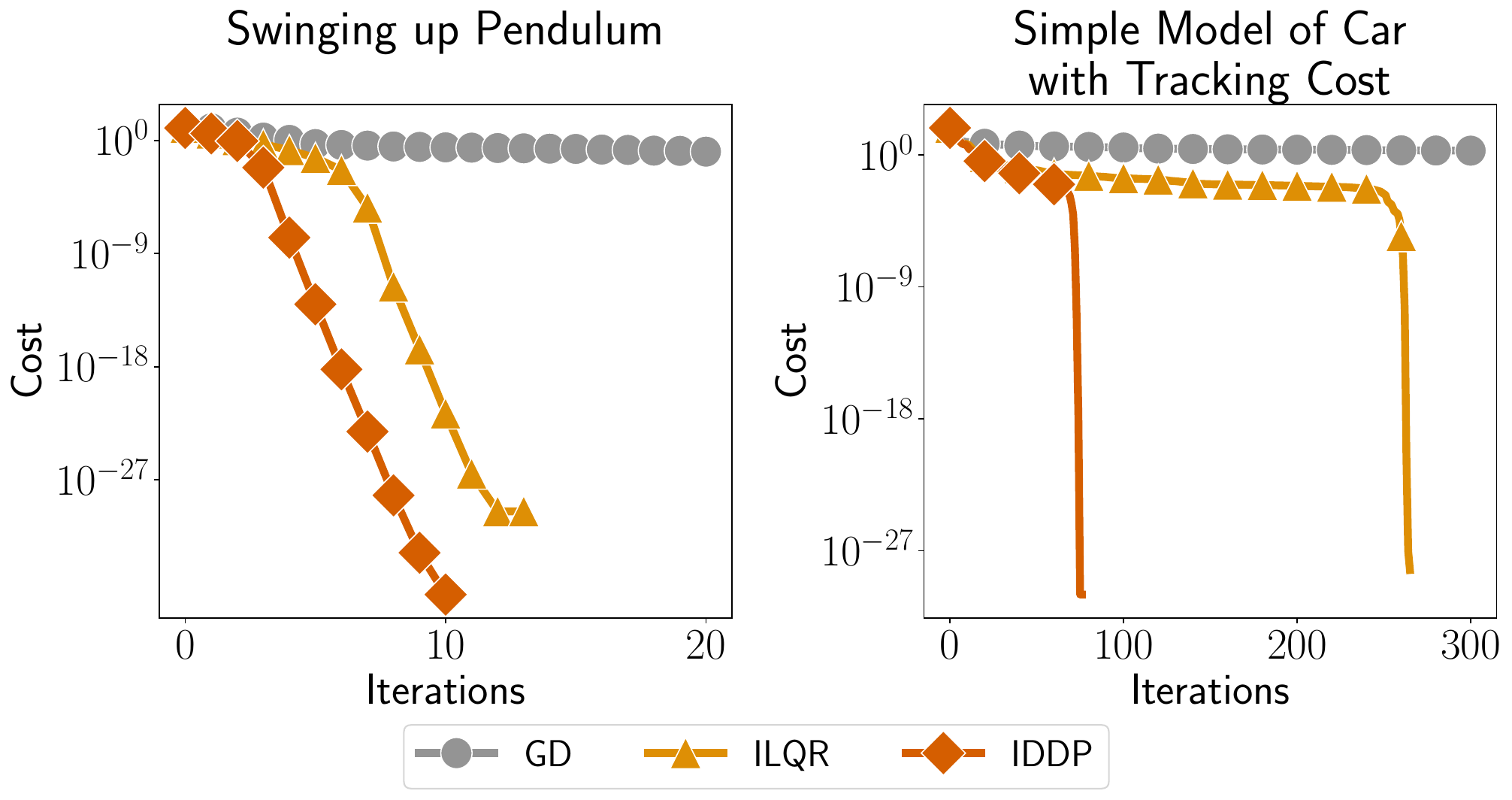}}}$\hspace*{0pt}
	$\vcenter{\hbox{\includegraphics[width=0.2\linewidth]{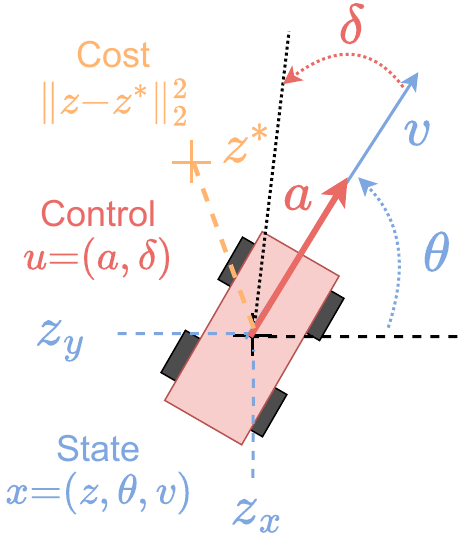}}}$
	\caption{Cost along iterations of ILQR, IDDP and Gradient Descent (GD) on two
	discrete time nonlinear control problems detailed in
	Section~\ref{sec:exp}.\label{fig:conv}}
	\end{center}
\end{figure}

The empirical efficiency of ILQR and IDDP on some nonlinear control problems, as
the ones illustrated in Figure~\ref{fig:conv}, motivates then the following
questions.
\begin{enumerate}[itemsep=0ex]
  \item \emph{What conditions on a discrete time nonlinear control problem ensure algorithms such as ILQR and IDDP converge to a globally optimal solution?}
  \item \emph{What convergence behaviors can we expect from these algorithms?}
\end{enumerate} 

We first present generic convergence results for the ILQR and IDDP algorithms on
problems~\eqref{eq:discrete_pb_ctrl_cost}. These results ensure global
convergence to stationary points and local convergence to minima in
Theorems~\ref{thm:gen_stat_cvg} and~\ref{thm:gen_local_cvg}. However, the
aforementioned convergence results do not explain the convergence to
\emph{global minima} observed in Figure~\ref{fig:conv}. 

We then turn our attention to nonlinear control problems without control costs and
with time-invariant dynamics $\dyn$, i.e., problems of the form
\begin{align}
	\min_{\substack{\ctrl_0,\ldots, \ctrl_{\horizon-1} \in \reals^\dimctrl\\ \state_0, \ldots, \state_\horizon \in \reals^\dimstate}} \quad & \sum_{t=1}^{\horizon} \cost_t(\state_t) \label{eq:discrete_pb} \\
	\mbox{subject to} \quad & \state_{t+1} = \dyn(\state_t, \ctrl_t) \quad \mbox{for} \ t \in \{0, \ldots, \horizon-1\}, \qquad \state_0 = \initstate.\nonumber
\end{align}  
Considering time-invariant dynamics make clearer the relationship with the
underlying continuous dynamical system. Generalizations to time-variant systems
are pointed out when applicable. However, problems of the
form~\eqref{eq:discrete_pb} conserve the main challenge  of generic discrete
time control problems~\eqref{eq:discrete_pb_ctrl_cost}, that is, the
nonlinearity of the dynamics, which prevent us from using classical results from
convex analysis even if the state costs $h_t$ are convex. The nonlinearity of the
dynamics distinguish problems~\eqref{eq:discrete_pb}  from the linear quadratic
settings studied by, e.g.,~\citet{fazel2018global, zhang2019policy,
zhang2020policy, sun2021learning, lin2021perturbation}, for which convergence to
global minima of policy methods have been shown by means of algebraic
considerations. The absence of costs on controls variables restrict the problem
class compared to problems of the form~\eqref{eq:discrete_pb_ctrl_cost}.
However, this also allows focusing on the properties of the dynamics to
understand the properties of non-convex problems~\eqref{eq:discrete_pb} and paves
the way to analyze generic problems of the
form~\eqref{eq:discrete_pb_ctrl_cost}.

\paragraph{Approach}
Our analysis stems from observing that, for strongly convex costs, convergence
to global minima of the ILQR or IDDP algorithms is ensured if the linearized
dynamics, i.e., the mappings $\auxctrl \mapsto \nabla_{\ctrl} \dyn(\state,
\ctrl)^\top \auxctrl$, are surjective, where $ \nabla_{\ctrl} \dyn(\state,
\ctrl)^\top \in \reals^{\dimstate \times \dimctrl}$ is the Jacobian of the
dynamic with respect to the control variable on a state $\state$ for a given
control $\ctrl$. To quantify the convergence of the ILQR and IDDP algorithms, we
consider the existence of a parameter $\sigma$ such that 
\begin{equation}\label{eq:global_conv_cond}
	\forall \state, \ctrl \in \reals^{\dimstate}\times  \reals^{\dimctrl},  \quad \sigma_{\min} (\nabla_\ctrl \dyn(\state, \ctrl)) \geq \sigma >0,
\end{equation}
where $\sigma_{\min}(\nabla_\ctrl \dyn(\state, \ctrl)) = \inf_{\lambda \in
\reals^\dimstate} \| \nabla_\ctrl \dyn(\state, \ctrl)\lambda\|_2 /\|\lambda\|_2$
is the minimal singular value of the transpose Jacobian of the discrete time dynamics
$\dyn$ w.r.t. the control variable.  Eq.~\eqref{eq:global_conv_cond} ensures the
injectivity of $\costate \mapsto \nabla_\ctrl \dyn(\state, \ctrl)\costate$ which
is equivalent to the surjectivity of $\auxctrl \mapsto\nabla_\ctrl \dyn(\state,
\ctrl)^\top\auxctrl$. Our main theorem is then stated below for strongly convex
costs provided adequate smoothness assumptions on the costs and the dynamics. 
\begin{theorem}\label{thm:intro} 
In problem~\eqref{eq:discrete_pb}, consider  costs $\cost_t$  that are strongly
convex with Lipschitz-continuous gradients and Lipschitz-continuous Hessians
and a dynamic $\dyn$ that is Lipschitz-continuous with Lipschitz-continuous
Jacobians. If the linearized dynamics are surjective, i.e., $\dyn$
satisfies~\eqref{eq:global_conv_cond}, then a regularized ILQR or IDDP
algorithm converges to a \emph{global minimum} with a \emph{local quadratic 
convergence rate}. 
\end{theorem}

Our analysis is based on  decomposing the problem at several scales. At the
scale of the trajectory, the objective can be seen as the composition of a total
cost function and a  function, which, given a sequence of controls, outputs the
corresponding trajectory. From an optimization viewpoint, the ILQR or the IDDP
algorithms, which use linear quadratic approximations of the objective,  amount
then to generalized Gauss-Newton algorithms~\citep{sideris2005efficient,
diehl2019local, messerer2021survey}. One contribution of this work is then to
detail the convergence rates of regularized generalized Gauss-Newton algorithms
for the composition of an outer strongly convex function and an inner function
with non-singular transpose Jacobians. 

Both algorithms take advantage of the dynamical structure of the problem to
implement a step of a Gauss-Newton algorithm. Similarly, the convergence
guarantees for the ILQR or IDDP algorithms can be detailed using the properties
of the problem at the scale of a single time step. In particular,
condition~\eqref{eq:global_conv_cond} entails a simple condition on the dynamic
to ensure convergence to global minima.

Finally, the dynamic itself can further be decomposed at the scale of the
discretization method used to define the discrete time control problem.
Condition~\eqref{eq:global_conv_cond} may then be ensured by considering a
multi-rate sampling method, i.e., sampling the control variables at a higher
rate than the sampling of the costs on the state variables. By combining all
aforementioned scales, we obtain worst-case convergence guarantees to global
optima for the ILQR and IDDP algorithms.

\paragraph{Outline} 
We start by presenting classical nonlinear control algorithms for
problem~\eqref{eq:discrete_pb}, i.e., the Iterative Linear Quadratic Regulator
(ILQR) and its variant IDDP, a.k.a. iLQR, in Section~\ref{ssec:ilqr} and
Section~\ref{ssec:ddp}, and cast them as closed-box oracles. We provide
convergence guarantees to stationary points of both algorithms for generic
problems of the form~\eqref{eq:discrete_pb_ctrl_cost}, as well as linear local
convergence guarantees, in Section~\ref{ssec:gen_cvg}. We analyze the properties
of problem~\eqref{eq:discrete_pb} with respect to the dynamics $\dyn$ in terms
of smoothness and surjectivity of the linearized dynamics in
Section~\ref{sec:prop}.  We further decompose the properties of the dynamic
$\dyn$ with respect to the underlying discretization scheme in
Section~\ref{sec:suff_cond}. We analyze the convergence of the ILQR and IDDP
algorithms in, respectively, Section~\ref{ssec:conv_ilqr},
Section~\ref{ssec:ddp_conv}. In particular, in Section~\ref{ssec:global_conv},
we demonstrate the convergence to global optima of the ILQR algorithm provided
that the costs are gradient dominated, the dynamics have surjective
linearizations~\eqref{eq:global_conv_cond} and both costs and dynamics are
smooth. We show the \emph{local quadratic convergence} of the ILQR algorithm
provided that the costs are self-concordant, the dynamics have surjective
linearizations~\eqref{eq:global_conv_cond} and both costs and dynamics satisfy
appropriate smoothness conditions in Section~\ref{ssec:local_conv}.
Theorem~\ref{thm:intro} is detailed for the ILQR algorithm in
Section~\ref{ssec:overall_conv} and convergence of the IDDP algorithm is
analyzed in Section~\ref{ssec:ddp_conv}. Numerical experiments are presented in
Section~\ref{sec:exp} to assess the theoretical findings. We discuss related
work in Section~\ref{sec:related_work}.

Additional numerical illustrations of the ILQR and IDDP algorithms can be found
in the companion paper~\citep{roulet2021techreport} and reproduced or further
explored by using the companion toolbox {\small
\url{https://github.com/vroulet/ilqc}}.

\paragraph{Summary of contributions}
For problems of the form~\eqref{eq:discrete_pb_ctrl_cost}, we demonstrate global
convergence to stationary points and local linear convergence to minima of both
ILQR and IDDP algorithms under usual regularity assumptions
(Theorems~\ref{thm:gen_stat_cvg},~\ref{thm:gen_local_cvg}). For problems of the
form~\eqref{eq:discrete_pb}, we make the following contributions.
\begin{enumerate}[nosep]
  \item We present sufficient conditions for convergence to a global minimum of
  the problem through the lens of a \emph{gradient-dominating property} of the objective.
  Namely, we show that a gradient-dominating property of the objective can be
  decomposed into the properties of the discrete time dynamic and ensured for
  appropriate discretization schemes (Lemma~\ref{lem:injectivity},
  Theorem~\ref{thm:suff_cond}).
	\item We prove that the ILQR algorithm \emph{converges to a global minimum}
	if the cost is smooth, gradient dominated, and if the dynamic is smooth with
	non-singular transpose Jacobians~\eqref{eq:global_conv_cond}
  (Theorem~\ref{thm:global_conv}).
	\item We prove that the ILQR algorithm \emph{converges locally with a
	quadratic rate} if the cost is smooth and self-concordant, and if the dynamic
	is smooth with non-singular transpose Jacobians~\eqref{eq:global_conv_cond}
  (Theorem~\ref{thm:conv_ggn}).
	\item We show and detail the \emph{global and local convergence} to minima of
	both ILQR and IDDP algorithms for smooth and strongly convex costs and smooth
	dynamic with non-singular transpose Jacobians~\eqref{eq:global_conv_cond}
	(Theorems~\ref{thm:global_local_conv} and~\ref{thm:ddp_conv}).
  \item Inspired from the theoretical findings, we also present a line-search
  variant of the ILQR algorithm that keep the same global and local convergence
  guarantees to minima, while not requiring any knowledge of problems
  constants (Corollary~\ref{cor:line_search}).
\end{enumerate}

\vspace{1em}
\paragraph{Notations}
For a sequence of vectors $\state_1, \ldots, \state_{\horizon} \in
\reals^\dimstate$, we denote by semicolons their concatenation s.t. $\states =
(\state_1;\ldots;\state_\horizon) \in \reals^{\horizon \dimstate}$. For a
multivariate function $f:\reals^d\rightarrow \reals^n$, we denote  
$\nabla f(x) = (\partial_{x_i} f_j(x))_{\substack{i\in \{1,\ldots, d\} j \in
\{1,\ldots, n\}}} \in \reals^{d \times n}$ the
transpose of the Jacobian of $f$ on $x$. For  $f:\reals^d \times \reals^p
\rightarrow  \reals^n$,  $x\in \reals^d$, $y\in \reals^p$, we denote $\nabla_x
f(x, y)= (\partial_{x_i} f_j(x, y) )_{\substack{i\in \{1,\ldots, d\} j \in
\{1,\ldots, n\}}} \in \reals^{d \times n}$ the partial transpose Jacobian of $f$ w.r.t.
$x$ on $(x, y)$. For  $f:\reals^d \rightarrow \reals$, we denote $f^* = \min_{x
\in \reals^d} f(x)$. For  $f:\reals^d \rightarrow \reals^n$,
$h:\reals^n\rightarrow \reals$, and  $x \in \reals^d$, we denote the linear
expansion of $f$ around $x$ and the quadratic expansion of $h$ around $x$ as,
respectively,
\[
\lin_f^x: y \rightarrow\nabla f(x)^\top y, \quad  \qua_h^x: y \rightarrow\nabla h(x)^\top y + 
\frac{1}{2} y^\top \nabla^2 h(x)y.
\]
For  $f:\reals^d\rightarrow \reals^n$, we denote 
$
\lip_f {=} \sup_{x, y\in \reals^d, x\neq y} \|f(x){-}f(y)\|_2 /\|x{-}y\|_2
$
the Lipschitz-continuity constant of $f$.
For  a matrix $A \in \reals^{d\times n}$, we denote by $\|A\|_2
=\sigma_{\max}(A) = \sup_{\lambda \in \reals^n} \|A\lambda\|_2 /\|\lambda\|_2 $
and  $\sigma_{\min}(A) = \inf_{\lambda \in \reals^n} \|A\lambda\|_2
/\|\lambda\|_2$  the largest and smallest singular values of $A$ respectively.

\section{Nonlinear Control Algorithms}\label{sec:algos}
The objective in~\eqref{eq:discrete_pb_ctrl_cost} only depends on the control variables
$\ctrls = (\ctrl_0;\ldots;\ctrl_{\horizon-1}) \in \reals^{\horizon\dimctrl}$ and
can be written as
\begin{align}
	\label{eq:obj_full}
	\obj(\ctrls) = & \quad \sum_{t=0}^{\horizon-1} \cost_t(\state_t, \ctrl_t) + \cost_\horizon(\state_\horizon)\\ 
	&  \mbox{s.t.} \ \ \state_{t+1} = \dyn_t(\state_t, \ctrl_t), \quad \mbox{for} \ t \in \{0, \ldots, \horizon-1\}, \qquad \state_0 = \initstate. \nonumber
\end{align}

Problem~\eqref{eq:discrete_pb_ctrl_cost} consists then in minimizing $\obj$. In
the following, we always \emph{assume that $\obj$ has at least one minimizer
$\ctrls^*$}. The classical ILQR \citep{li2004iterative, sideris2005efficient},
and IDDP algorithms~\citep{tassa2012synthesis} compute the next iterate as  
$
	\ctrls_{\next} = \ctrls + \Oracle_{\reg}(\obj)(\ctrls),
$
for given control variables $\ctrls$. Here, $\Oracle_{\reg}(\obj)$ is an oracle,
which, given a regularization parameter $\reg$ and control variables $\ctrls$,
outputs a  direction  $\Oracle_{\reg}(\obj)(\ctrls)$. The original ILQR or IDDP
algorithms did not incorporate an additional
regularization~\citep{li2004iterative, tassa2012synthesis}. Our implementation
is a variant that leads to non-asymptotic convergence guarantees of these
algorithms~\citep{roulet2019iterative}.

\subsection{Iterative Linear Quadratic Regulator}\label{ssec:ilqr}
Given control variables $\ctrls = (\ctrl_0;\ldots; \ctrl_{\horizon-1})$ with associated
trajectory $\state_1,\ldots,\state_\horizon$, and a regularization $\reg>0$, an
Iterative Linear Quadratic Regulator (ILQR) algorithm  computes the next command
by computing the Linear Quadratic Regulator (LQR)  associated with a  quadratic
approximation of the costs and a linear approximation of the dynamics around the
current trajectory. 

Formally, the  next iterate is computed as $\ctrls_{\next} =
\ctrls + \LQR_\reg(\obj)(\ctrls) $, where
\begin{align}
	\nonumber\LQR_\reg(\obj)(\ctrls) & =   
	\argmin_{\auxctrl_0,\ldots,\auxctrl_{\horizon-1} \in \reals^\dimctrl} 
	\sum_{t=0}^{\horizon-1}
  \begin{pmatrix}
    p_t \\
    q_t
  \end{pmatrix}^\top \begin{pmatrix}
    \auxstate_t \\
    \auxctrl_t 
  \end{pmatrix}
  + \frac{1}{2} \begin{pmatrix}
    \auxstate_t \\
    \auxctrl_t 
  \end{pmatrix}^\top 
  \begin{pmatrix}
    P_t & R_t \\
    R_t^\top & Q_t 
  \end{pmatrix}
  \begin{pmatrix}
    \auxstate_t \\
    \auxctrl_t 
  \end{pmatrix} + \frac{\nu}{2}\|\auxctrl_t\|_2^2\\
  & \hspace*{75pt}
  + p_\horizon^\top \auxstate_\horizon
  + \frac{1}{2} \auxstate_\horizon^\top P_\horizon \auxstate_\horizon \label{eq:lqr_pb} \\
	& \hspace{50pt} \mbox{s.t.}\quad \auxstate_{t+1} = \A_t \auxstate_t + \B_t \auxctrl_t , \ \mbox{for} \ t\in \{0, \ldots, \horizon-1\},  \auxstate_0 = 0, \nonumber \\
	\mbox{with} \quad 
  P_\horizon & = \nabla^2_{\state_\horizon\state_\horizon} \cost_\horizon(\state_\horizon), \ p_\horizon = \nabla_{\state_\horizon} \cost_\horizon(\state_\horizon), \nonumber\\
  \P_t & = \nabla_{\state_t\state_t}^2 \cost_t(\state_t, \ctrl_t), \ \p_t = \nabla_{\state_t} \cost_t(\state_t, \ctrl_t), \hspace{62pt} \mbox{for} \ t\in \{0, \ldots, \horizon-1\},\nonumber\\
  \Q_t & = \nabla_{\ctrl_t\ctrl_t}^2 \cost_t(\state_t, \ctrl_t), \ \q_t = \nabla_{\ctrl_t} \cost_t(\state_t, \ctrl_t), \hspace{62.5pt} \mbox{for} \ t\in \{0, \ldots, \horizon-1\},\nonumber\\
  R_t & = \nabla_{\ctrl_t\state_t}^2 \cost_t(\state_t, \ctrl_t) \hspace{160pt} \mbox{for} \ t\in \{0, \ldots, \horizon-1\},\nonumber\\
 \A_t & = \nabla_{\state_t} \dyn_t(\state_t, \ctrl_t)^\top, \  \B_t = \nabla_{\ctrl_t} \dyn_t(\state_t, \ctrl_t)^\top, \hspace{56pt} \mbox{for} \ t \in \{0, \ldots, \horizon-1\}.\nonumber
\end{align}
The minimum above is well-defined as long as either the costs are convex, or
the regularization $\reg$ is large enough. The implementation of the ILQR oracle
is presented in Algorithm~\ref{algo:lqr_ddp}. Its computational scheme is
illustrated in Figure~\ref{fig:ilqr_comput_scheme}. 

Problem~\eqref{eq:lqr_pb} is first instantiated in a \emph{forward pass} by
collecting all first order or second order information on the dynamics and the
costs necessary to pose problem~\eqref{eq:lqr_pb}.

Problem~\eqref{eq:lqr_pb} is then solved by dynamic
programming~\citep{bertsekas2017dynamic}. Namely, the cost-to-go
$\costogo_t(\auxstate_t)$ from a state $\auxstate_t$ at time $t$  is computed
recursively in a \emph{backward pass} as, starting from
$\costogo_\horizon(\auxstate_\horizon) = \frac{1}{2} \auxstate_\horizon^\top
\P_\horizon\auxstate_\horizon + \p_\horizon^\top \auxstate_\horizon$, 
\begin{align}\nonumber
\costogo_t( \auxstate_t) &= \min_{\auxctrl_t, \ldots \auxctrl_{\horizon-1} \in \reals^\dimctrl}	\sum_{s=t}^{\horizon-1}
\begin{pmatrix}
  p_s \\
  q_s
\end{pmatrix}^\top \begin{pmatrix}
  \auxstate_s \\
  \auxctrl_s 
\end{pmatrix}
+ \frac{1}{2} \begin{pmatrix}
  \auxstate_s \\
  \auxctrl_s
\end{pmatrix}^\top 
\begin{pmatrix}
  P_s & R_s \\
  R_s^\top & Q_s 
\end{pmatrix}
\begin{pmatrix}
  \auxstate_s \\
  \auxctrl_s 
\end{pmatrix} + \frac{\nu}{2}\|\auxctrl_s\|_2^2\\
& \hspace*{75pt}
+ p_\horizon^\top \auxstate_\horizon
+ \frac{1}{2} \auxstate_\horizon^\top P_\horizon \auxstate_\horizon  \nonumber\\
& \qquad \mbox{s.t.}\quad \auxstate_{s+1} = \A_s \auxstate_s + \B_s \auxctrl_s , \ \mbox{for} \ s\in \{t, \ldots, \horizon-1\}, \nonumber\\
& = \min_{\auxctrl_t \in \reals^\dimctrl}\left\{
  \begin{pmatrix}
    p_t \\
    q_t
  \end{pmatrix}^\top \begin{pmatrix}
    \auxstate_t \\
    \auxctrl_t 
  \end{pmatrix}
  + \frac{1}{2} \begin{pmatrix}
    \auxstate_t \\
    \auxctrl_t
  \end{pmatrix}^\top 
  \begin{pmatrix}
    P_t & R_t \\
    R_t^\top & Q_t
  \end{pmatrix}
  \begin{pmatrix}
    \auxstate_t \\
    \auxctrl_t 
  \end{pmatrix}
  + \frac{\nu}{2}\|\auxctrl_t\|_2^2
  + \costogo_{t+1}(\A_t \auxstate_t {+} \B_t \auxctrl_t)
  \right\}  \label{eq:costogo_subpb}\\
& = \frac{1}{2}\auxstate_t^\top J_t \auxstate_t + \auxstate_t^\top \j_t, \label{eq:costogo}
\end{align}
where $\J_t$, $\j_t$ are computed recursively in line~\ref{line:costogo} of
Algorithm~\ref{algo:lqr_ddp}. The optimal control at time $t$ from state
$\auxstate_t$ is then given by an affine policy
\begin{align}\label{eq:pol}
\pi_t(\auxstate_t) & = \argmin_{\auxctrl_t \in \reals^\dimctrl}\left\{
  \begin{pmatrix}
    p_t \\
    q_t
  \end{pmatrix}^\top \begin{pmatrix}
    \auxstate_t \\
    \auxctrl_t 
  \end{pmatrix}
  + \frac{1}{2} \begin{pmatrix}
    \auxstate_t \\
    \auxctrl_t
  \end{pmatrix}^\top 
  \begin{pmatrix}
    P_t & R_t \\
    R_t^\top & Q_t
  \end{pmatrix}
  \begin{pmatrix}
    \auxstate_t \\
    \auxctrl_t 
  \end{pmatrix}
  + \frac{\nu}{2}\|\auxctrl_t\|_2^2
  + \costogo_{t+1}(\A_t \auxstate_t {+} \B_t \auxctrl_t)
  \right\} \nonumber\\ 
  & = K_t \auxstate_t + k_t, 
\end{align}
where $K_t, k_t$ are computed in line~\ref{line:policy} of
Algorithm~\ref{algo:lqr_ddp}. The cost-to-go functions and policies are
well-defined as long as all costs $\cost_t$ are convex or if the regularization
$\reg$ is large enough (see e.g.~\citep{roulet2021techreport}).

The solution of the LQR problem~\eqref{eq:lqr_pb}, is given by
\emph{rolling-out} the policies along the linear trajectories
of~\eqref{eq:lqr_pb}. The oracle is $\LQR_\reg(\obj)(\ctrls) = (\auxctrl_0;
\ldots;\auxctrl_{\horizon-1})$, where, starting from $\auxstate_0 = 0$,
\[
\auxctrl_t = \pi_t(\auxstate_t), \quad \auxstate_{t+1} = \A_t \auxstate_t + \B_t \auxctrl_t \quad \mbox{for} \ t \in \{0, \ldots, \horizon-1\}.
\]

Solving~\eqref{eq:lqr_pb} by dynamic
programming comes at a linear cost with respect to the length of the trajectory.
Namely, in terms of elementary computations, the ILQR oracle  has a computational
cost 
\begin{equation}\label{eq:time_cost}
\mathcal{C}(\dimstate, \dimctrl, \horizon) = O(\horizon(\dimstate + \dimctrl )^3).
\end{equation} 
Note that, in nonlinear control problems, the state and control dimensions are
generally small. On the other hand, the horizon $\horizon$ may be large if, for
example, for a fixed continuous time horizon, a small discretization stepsize
was used to define~\eqref{eq:discrete_pb}. The ILQR algorithm keeps a
linear complexity with respect to the leading dimension $\horizon$ of the
problem. The linear quadratic problem~\eqref{eq:lqr_pb} can also be solved by
alternative linear algebra subroutines ranging from matrix-free solvers that
take advantage of differentiable programming framework, or by introducing
Lagrange multipliers. We refer to~\citet{wright1991partitioned}, for more
details.

Overall an ILQR algorithm  computes a sequence of iterates as
\begin{equation}\label{eq:ilqr_algo}\tag{ILQR}
	\ctrls^{(k+1)} = \ctrls^{(k)} + \LQR_{\reg_k}(\obj)(\ctrls^{(k)}),
\end{equation}
starting from control variables $\ctrls^{(0)}$, where $\reg_k$ are regularization
parameters that may depend on the current iterate and $\LQR_\reg$ is implemented
by Algorithm~\ref{algo:lqr_ddp}.

\begin{figure}
  \centering
  \includegraphics[width=\linewidth]{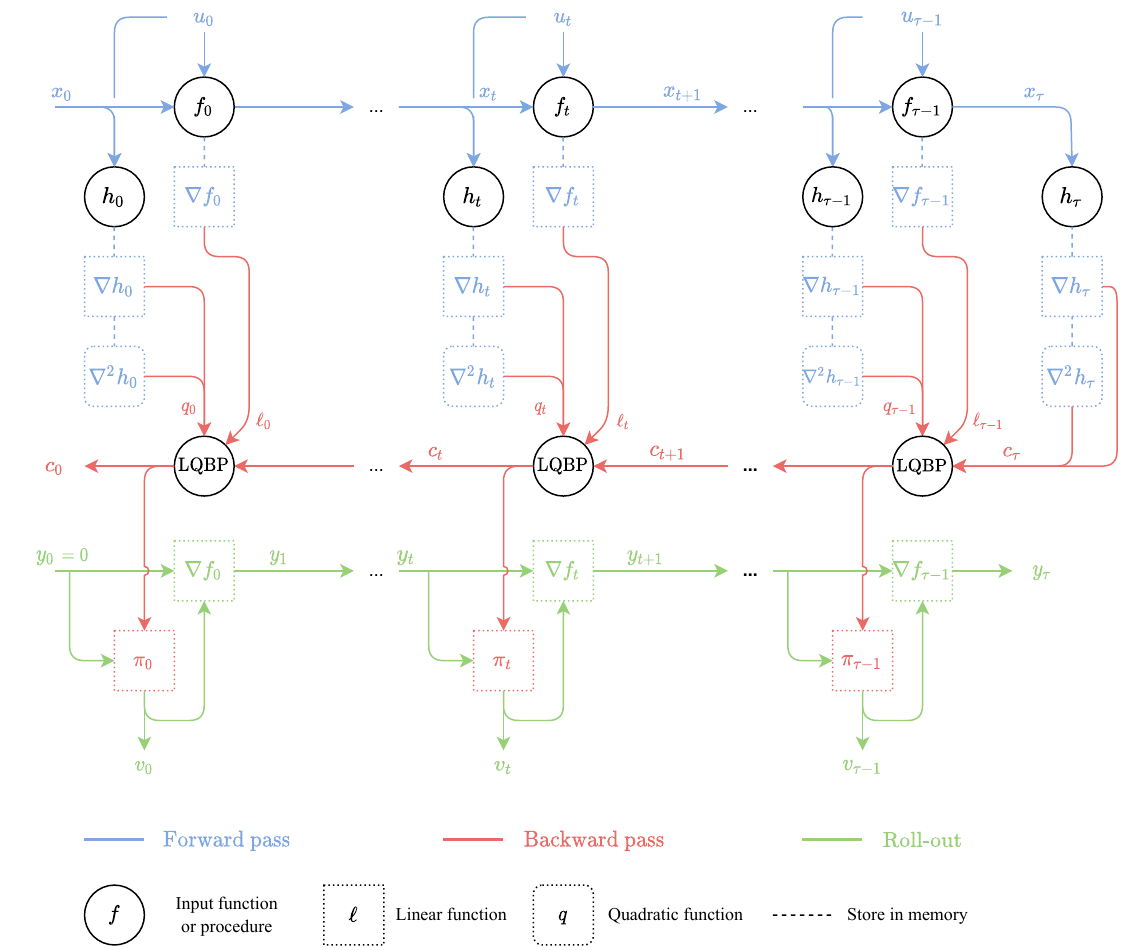}
  \caption{
  Computational scheme of the ILQR algorithm. The algorithm proceeds in three phases. 
  In the \emph{forward pass}, the first derivatives of the dynamics as well as the first and second
  derivatives of the costs are stored in memory (or the inputs are checkpointed to 
  access these derivatives). 
  During the \emph{backward pass} the cost-to-go functions are back-propagated at each 
  time step through matrix products and inversions, 
  denoted simply LQBP for linear-quadratic backpropagation.
  The policies computed in the backward pass are used in a final \emph{roll-out} phase through 
  the linearized dynamics to output a candidate sequence of control inputs.}
  \label{fig:ilqr_comput_scheme}
\end{figure}

\subsection{Iterative Differential Dynamic Programming}\label{ssec:ddp} The IDDP
algorithm is an instance of a Differential Dynamic Programming (DDP) approach. A
DDP approach considers computing approximate solutions of~\eqref{eq:discrete_pb}
around the current iterate by dynamic programming using approximations of the
dynamics and the costs. We refer the reader to,
e.g.,~\cite{jacobson1970differential,tassa2012synthesis, roulet2021techreport}
for a detailed presentation. The original DDP approach uses quadratic
approximations of the dynamics~\citep{jacobson1970differential}. Here, we focus
on the implementation using linear approximations of the dynamics and quadratic
approximations of the costs as used by, e.g.,~\cite{tassa2012synthesis}. In this
case,  a DDP approach amounts to computing the same policies $\pi_t$ as an ILQR
algorithm but rolling-out the policies along the original dynamics rather than
the linearized ones. 
 
Namely, the oracle output by IDDP is given as 
\begin{align}
	\DDP_\reg(\obj)(\ctrls) & 
 = ( \auxctrl_0; \ldots; \auxctrl_{\horizon-1}) \nonumber\\
\mbox{where}\  \auxctrl_t & = \pi_t(\auxstate_t), \ \auxstate_{t+1} 
= \dyn_t(\state_t +\auxstate_t, \ctrl_t +  \auxctrl_t) - \dyn_t(\state_t, \ctrl_t) \ \mbox{for} \ t \in \{0, \ldots, \horizon-1\}, 
\label{eq:ddp_oracle}
\end{align}
as presented in Algorithm~\ref{algo:lqr_ddp}.  The computational complexity of
this approach is the same as the one of the ILQR approach. By iterating the
above steps, starting from initial control variables $\ctrls^{(0)}$, we obtain
the iterative Linear Quadratic Regulator (IDDP) algorithm, which computes
iterates of the form
\begin{equation}\label{eq:iddp_algo}\tag{IDDP}
	\ctrls^{(k+1)} = \ctrls^{(k)} + \DDP_{\reg_k}(\obj)(\ctrls^{(k)}),
\end{equation}
where  the regularization parameters $\reg_k$ may depend on the current iterate
and $\DDP_\reg$ is implemented by Algorithm~\ref{algo:lqr_ddp}. \\

\begin{algorithm}[t]\caption{ILQR and IDDP steps  for
problem~\eqref{eq:discrete_pb}\label{algo:lqr_ddp}}
\begin{algorithmic}[1]
    \State {\bf Inputs:} Controls $\ctrls  = (\ctrl_0; \ldots;
		\ctrl_{\horizon-1})\in \reals^{\horizon\dimctrl}$, regularization $\reg>0$,
		initial state $\initstate \in \reals^{\dimstate}$, horizon $\horizon$,
		dynamic $\dyn:\reals^\dimstate \times \reals^\dimctrl \rightarrow
		\reals^\dimstate$, costs $(\cost_t)_{t=1}^{\horizon}$, oracle type $\Oracle
		\in \{\LQR, \DDP\}$. 
    \Statex \underline{Forward pass}  
    \Comment{ {\it \small instantiate problem~\eqref{eq:lqr_pb} for the given
		control variables}} 
    \State Initialize $\state_0 = \initstate$ 
    \For{$t=0, \ldots, \horizon-1$} 
    \State Compute $\state_{t+1} = \dyn(\state_t, \ctrl_t)$ and
		$\cost_t(\state_t, \ctrl_t)$
    \State Compute and store 
    \begin{gather*}
      A_t = \nabla_{\state_t} \dyn(\state_t, \ctrl_t)^\top, \
      B_t = \nabla_{\ctrl_t} \dyn(\state_t, \ctrl_t)^\top, \\ 
      p_t = \nabla_{\state_t} \cost_t(\state_t, \ctrl_t), \ 
      q_t = \nabla_{\ctrl_t} \cost_t(\state_t, \ctrl_t), \\
      P_t = \nabla^2_{\state_t\state_t} \cost_t(\state_t, \ctrl_t), \
      Q_t = \nabla^2_{\ctrl_t \ctrl_t} \cost_t(\state_t, \ctrl_t), \
      R_t = \nabla^2_{\state_t \ctrl_t} \cost_t(\state_t, \ctrl_t)   
    \end{gather*}
    \EndFor 
    \State Compute $h_\horizon(\state_\horizon)$,
    $p_\horizon = \nabla_{\state_\horizon} \cost_\horizon(\state_\horizon)$
    $P_\horizon = \nabla^2_{\state_\horizon\state_\horizon} \cost_\horizon(\state_\horizon)$
    \Statex \underline{Backward pass}  
    \Comment{ {\it \small compute optimal policies for
		problem~\eqref{eq:lqr_pb}}} 
    \State Initialize $J_\horizon = P_\horizon$, $j_\horizon= p_\horizon$
    \For{$t=\horizon -1,\ldots 0$}
    \State Compute the cost-to-go functions $\costogo_t: \auxstate_t \rightarrow
		\frac{1}{2} \auxstate_t^\top \J_t \auxstate_t + \j_t^\top \auxstate_t$
		defined in~\eqref{eq:costogo} as \label{line:costogo}
    \begin{align*}
      \J_t & = 
      \H_t
      + \A_t^\top  \J_{t+1}\A_t 
      - (\R_t + \A_t^\top  \J_{t+1}\B_t) (\G_t + \reg \idm + \B_t^\top \J_{t+1}\B_t)^{-1} (\R_t^\top + \B_t^\top  \J_{t+1} \A_t)\nonumber \\
      \j_t & = \h_t 
      + \A_t^\top \j_{t+1} - (\R _t+ \A_t^\top  \J_{t+1}\B_t) (\G_t + \reg \idm + \B_t^\top  \J_{t+1}\B_t)^{-1} (\g_t+ \B_t^\top  \j_{t+1})
    \end{align*}
    \State Store the policies $\pi_t: \auxstate_t \rightarrow K_t \auxstate_t +
		k_t$ defined in~\eqref{eq:pol} as \label{line:policy}
    \begin{align*}
      K_t  & = -(\G_t + \reg \idm + \B_t^\top\J_{t+1}\B_t)^{-1} (\R_t^\top + \B_t^\top \J_{t+1} \A_t),  \\
      k_t  & = -(\G_t + \reg \idm + \B_t^\top\J_{t+1}\B_t)^{-1} (\g_t + \B_t^\top \j_{t+1}) 
    \end{align*}
    \EndFor 
    \Statex \underline{Roll-out pass} 
    \Comment{{\it \small apply the computed policies along the linearized or the
		exact dynamics}} 
    \State Initialize $\auxstate_0 = 0$ \For{$t=0, \ldots, \horizon-1$} 
    \If{$\Oracle$ is $\LQR$} 
    \State Compute 
    $
    \auxctrl_t = \pi_t(\auxstate_t), \ \auxstate_{t+1} = \A_t \auxstate_t + \B_t \auxctrl_t
    $
    \ElsIf{$\Oracle$ is $\DDP$} 
    \State Compute 
    $ 
    \auxctrl_t = \pi_t(\auxstate_t), \ \auxstate_{t+1} 
    =\dyn(\state_t +\auxstate_t, \ctrl_t +  \auxctrl_t) - \dyn(\state_t, \ctrl_t)
    $ 
    \EndIf
    \EndFor 
    \State{\bf Output:} Control directions $\auxctrls
		=(\auxctrl_0;\ldots;\auxctrl_{\horizon-1})$
\end{algorithmic}
\end{algorithm}

\subsection{Generic Convergence Guarantees}\label{ssec:gen_cvg}
We start by presenting convergence guarantees of the ILQR and IDDP algorithm for
generic problems of the form~\eqref{eq:discrete_pb_ctrl_cost}. First, with an
appropriate choice of regularization both algorithms can converge globally to a
stationary point at a polynomial rate (Theorem~\ref{thm:gen_stat_cvg}). Such a
stationary point of $\obj$ satisfies naturally necessary optimality conditions
for problem~\eqref{eq:discrete_pb_ctrl_cost} as recalled
in Appendix~\ref{app:opt_cond}. Note that necessary optimality conditions in discrete
time control problem differ from their continuous time counterpart as discussed 
in detail in Appendix~\ref{app:opt_cond}.
\begin{theorem}\label{thm:gen_stat_cvg}
  For problem~\eqref{eq:discrete_pb_ctrl_cost}, assume that the dynamics $\dyn_t$ are
  Lipschitz continuous with Lipschitz continuous Jacobians and that the costs $\cost_t$ 
  are Lipschitz continuous with Lipschitz continuous gradients and Lipschitz continuous 
  Hessians. Then, provided that the regularization $\reg$ is larger than some $c_1>0$, 
  the iterates of the \ref{eq:ilqr_algo} or the \ref{eq:iddp_algo} algorithms satisfy
  \[
    \min_{k\in \{0, \ldots, K\}} \|\nabla \fullobj(\ctrls^{(k)})\|_2 
    \leq \sqrt{\frac{2(c_2 + \nu)
    \left(\fullobj(\ctrls^{(0)}) - \min_{\ctrls \in \reals^{\horizon \dimctrl}} \fullobj(\ctrls)\right)}{K+1}},
  \]
  for $c_1, c_2$ depending on the smoothness properties of the dynamics and the
  costs.  
\end{theorem}
\begin{proof}
  Detailed statements and proofs are presented in Lemma~\ref{lem:gen_stat_cvg_rggn} 
  and Lemma~\ref{lem:iddp_stat_cvg} for the \ref{eq:ilqr_algo} and the \ref{eq:iddp_algo} algorithms, respectively.
\end{proof}
We can also demonstrate local linear convergence of both algorithms towards 
a minimum under regular assumptions.
\begin{theorem}\label{thm:gen_local_cvg} For
  problem~\eqref{eq:discrete_pb_ctrl_cost}, assume that the dynamics $\dyn_t$
  and the costs $\cost_t$ are Lipschitz continuous with Lipschitz continuous
  Jacobians and Lipschitz continuous Hessians. Let $\ctrls^{(k)}$ denote the
  $k$\textsuperscript{th} iterate of the \ref{eq:ilqr_algo} or the
  \ref{eq:iddp_algo} algorithms. Assume $\ctrls^{(k)}$ to be close to a minimum
  $\ctrls^*$ of $\fullobj$ with positive definite Hessian. If the regularization $\reg$
  is larger than some $c_1>0$, then the iterations of the \ref{eq:ilqr_algo} or the
  \ref{eq:iddp_algo} algorithm converge linearly to $\ctrls^*$ as 
  \[
    \|\ctrls^{(k+1)} -\ctrls^*\|_2
    \leq  \left(1 - \frac{c_2}{\reg}\right)\|\ctrls^{(k)}-\ctrls^*\|_2,
  \]
  for $c_1, c_2$ depending on the smoothness properties of the dynamics and the
  costs.
\end{theorem}
\begin{proof}
  Detailed statements and proofs are presented in 
  Lemma~\ref{lem:local_gen_cvg_rggn} and Lemma~\ref{lem:iddp_local_gen_cvg} 
  for the \ref{eq:ilqr_algo} and the \ref{eq:iddp_algo} algorithms, respectively.
\end{proof}
\begin{remark}
    Compared to a Newton method that can converge locally at a quadratic
    rate on problems of the form~\eqref{eq:discrete_pb_ctrl_cost}
    ~\citep{nocedal2006numerical, de1988differential, dunn1989efficient},
    the ILQR and IDDP algorithms converge locally only at linear rate a 
    priori (see also~\citet{baumgartner2023local}).
    Similarly, the original Differential Dynamic Programming (DDP) approach of 
    \citet{jacobson1970differential} can converge locally at a quadratic
    rate \citep{murray1984differential, liao1991convergence, di2019newton}.
    However, the local linear convergence
    rates presented in Theorem~\ref{thm:gen_local_cvg} do not match the
    superlinear rates observed in practice in Figure~\ref{fig:conv}
    (see also \citet{roulet2021techreport}). Hence, we consider in the following
    additional properties of the problem that can uncover both the global convergence
    of the ILQR and IDDP algorithms as well as their fast local convergence.
\end{remark}

\section{Conditioning Analysis}\label{sec:cond}

To understand the convergence behavior of the \ref{eq:ilqr_algo} and \ref{eq:iddp_algo} algorithms displayed in Figure~\ref{fig:conv}, we consider a restricted class of control problems without control costs of the form~\eqref{eq:discrete_pb}.
Namely, from now on, we consider objectives of the form
\begin{align}
	\label{eq:obj}
	\obj(\ctrls) = & \quad \sum_{t=1}^\horizon \cost_t(\state_t) \\ 
	&  \mbox{s.t.} \ \ \state_{t+1} = \dyn(\state_t, \ctrl_t), \quad \mbox{for} \ t \in \{0, \ldots, \horizon-1\}, \qquad \state_0 = \initstate, \nonumber
\end{align}
for $\ctrls = (\ctrl_0;\ldots;\ctrl_{\horizon-1}) \in
\reals^{\horizon\dimctrl}$. Such objectives keep the main difficulty of
nonlinear control problems: for nonlinear dynamics $\dyn$, the overall objective
$\obj$ is non-convex such that convergence to global minima is a priori not
guaranteed by even a simple gradient descent. Nevertheless, by decomposing the
objective at the scale of the dynamics, and further decomposing the dynamics by
an appropriate discretization scheme, we can identify sufficient conditions for
convergence to global minima linked to usual notions in nonlinear control. We
can then further show the convergence of the \ref{eq:ilqr_algo} and
\ref{eq:iddp_algo} algorithms to a global minimum, and detail the several phases
of convergence
(Sections~\ref{ssec:intuition},~\ref{ssec:conv_ilqr},~\ref{ssec:ddp_conv}).

\subsection{Objective Decomposition}\label{sec:prop} 

The objective $\obj$, defined in~\eqref{eq:obj}, can be decomposed into (i) the
costs associated to a given trajectory, and (ii) the function that, given an
input command, outputs the corresponding trajectory, defined below. 
\begin{definition}\label{def:traj_func} We define the control \emph{of}
	$\horizon$ steps of a discrete time dynamic $\dyn:\reals^\dimstate \times
	\reals^\dimctrl \rightarrow \reals^\dimstate$ as the function $\traj:
	\reals^\dimstate \times \reals^{\horizon\dimctrl} \rightarrow \reals^{\horizon
	\dimstate}$, which, given an initial point $\state_0 \in \reals^{\dimstate}$
	and a command $\ctrls = (\ctrl_0;\ldots;\ctrl_{\horizon-1})\in
	\reals^{\horizon\dimctrl}$, outputs the corresponding trajectory $\state_1,
	\ldots, \state_\horizon$, i.e., 
	\begin{align}
		\label{eq:traj}
		\traj(\state_0, \ctrls) & = ( \state_1;\ldots;\state_\horizon) \\
		\mbox{s.t.} \quad \state_{t+1} &= \dyn(\state_t, \ctrl_t) \hspace{20pt} \mbox{for} \ t \in \{0,\ldots, \horizon-1\}. \nonumber
	\end{align}
\end{definition}
By defining the cost $\cost(\states)$ of a trajectory $\states=(\state_1,
\ldots, \state_\horizon)$ as the sum of the cost of the states,
problem~\eqref{eq:discrete_pb} amounts to solving
\begin{equation}\label{eq:comp_pb_ctrl}
\min_{\ctrls\in \reals^{\horizon\dimctrl}} \left\{\obj(\ctrls)=\cost (\traj(\initstate, \ctrls))\right\}, \mbox{for} \	\traj(\state_0, \ctrls) \ \mbox{given in~\eqref{eq:traj}}, \ \cost(\states) = \sum_{t=1}^\horizon \cost_t(\state_t).
\end{equation}
For convex costs $\cost$, if the dynamic $\dyn$ is linear, then  the function
$\traj$ is also  linear and the overall problem~\eqref{eq:comp_pb_ctrl} is then
convex, hence easily solvable from an optimization viewpoint using, e.g., a
gradient descent.  

For nonlinear dynamics, the problem is a priori not convex regardless of the
convexity of the costs. Yet, convergence guarantees to global minima of, e.g., first order
methods, may still be obtained by considering whether the objective satisfies a
gradient dominating property~\citep{polyak1964some, lojasiewicz1963topological},
i.e., whether there exists, for example $\plcstobj >0$, such that for any
$\ctrls \in \reals^{\horizon\dimctrl}$, $\|\nabla \obj(\ctrls)\|_2^2 \geq
\plcstobj \left(\obj(\ctrls)- \obj^*\right).
$
To focus on the properties on the nonlinear dynamic, we consider costs that are gradient dominated, e.g., such that for any $\states \in \reals^{\horizon\dimstate}$, we have $\|\nabla \cost(\states)\|_2^2 \geq \plcst(\cost(\states) - \cost^*)$ for some $\plcst>0$. 
In that case, a sufficient condition for the objective to satisfy a gradient
dominating property is that the control of $\horizon$ steps of the dynamic
satisfies $\sigma_{\min}(\nabla_\ctrls \traj(\initstate, \ctrls)) \geq \sigma
>0$ for any $\ctrls\in \reals^{\horizon\dimctrl}$, since then we have, for
$\states = \traj(\initstate, \ctrls)$,
\begin{equation}\label{eq:grad_dom_obj}
\|\nabla \obj(\ctrls)\|_2^2 = \|\nabla_\ctrls \traj(\initstate, \ctrls) \nabla \cost(\states)\|_2^2 \geq \sigma^2 \|\nabla \cost(\states)\|_2^2 \geq \sigma^2\plcst(\cost(\states) - \cost^*),
\end{equation}
where $\cost^* = \min_{\states \in \reals^{\horizon\dimctrl}} \cost(\states)$.
Since the set $\{\ctrls \in \reals^{\horizon\dimctrl}: \nabla \obj(\ctrls)=0\}$
is not  empty as we assumed that the problem has a minimizer, the above equation
implies that $\cost^* = \obj^*$ and so that the overall objective satisfies  a
gradient dominating property. We investigate then whether the control of
$\horizon$ steps of a dynamic $\dyn$ can satisfy the aforementioned condition by
considering the properties of the dynamic $\dyn$. 

The condition $\sigma_{\min}(\nabla_\ctrls \traj(\initstate, \ctrls)) >0$ can be
interpreted as the surjectivity of the linearized control of $\horizon$ steps,
i.e.,  
the mapping $\auxctrls = (\auxctrl_0;\ldots;\auxctrl_{\horizon-1}) \rightarrow
\nabla_\ctrls \traj(\state_0, \ctrls)^\top \auxctrls =
(\auxstate_1;\ldots;\auxstate_\horizon) $ which can  be decomposed as
\begin{align*}
	\auxstate_{t+1} = \nabla_{\state_t} \dyn(\state_t, \ctrl_t)^\top \auxstate_t + \nabla_{\ctrl_t} \dyn(\state_t, \ctrl_t)^\top \auxctrl_t  \quad\mbox{for} \ t \in \{0, \ldots, \horizon-1\}, \quad \auxstate_0 = 0.
\end{align*}
We recognize here the linearized trajectories that are at the heart of the ILQR
and IDDP algorithms. Our analysis stems from understanding that the surjectivity
of the linearization of the control of $\horizon$ steps, i.e,  $\auxctrls
\rightarrow  \nabla_\ctrls \traj(\state_0, \ctrls)^\top \auxctrls$, is inherited
from the surjectivity of the linearization of a single step of the discrete
dynamic, i.e., $\auxctrl \rightarrow  \nabla_{\ctrl} \dyn(\state, \ctrl)^\top
\auxctrl$ as formally stated in the following lemma. Note that
Lemma~\ref{lem:injectivity} and the subsequent analysis of the algorithms
presented in Section~\ref{sec:algos} can be extended to time-varying discrete time
dynamics as presented in Lemma~\ref{lem:injectivity_time_varying}.
\begin{lemma}\label{lem:injectivity} If the linearized dynamics, $\auxctrl
	\rightarrow \nabla_\ctrl \dyn(\state, \ctrl)^\top \auxctrl$, of a Lipschitz
	continuous discrete time dynamic $\dyn$ are surjective in the sense that there
	exists $\sigma_\dyn>0$ s.t. 
	\begin{equation}\label{eq:single_step_injectivity}
		\forall \state, \ctrl \in \reals^\dimstate \times \reals^\dimctrl, \quad 	\sigma_{\min} (\nabla_\ctrl \dyn(\state, \ctrl))\geq \sigma_\dyn >0,
	\end{equation}
	then the linearizations, $\auxctrls \rightarrow  \nabla_\ctrls \traj(\state_0,
	\ctrls)^\top \auxctrls$, of the control of $\horizon$ steps of the dynamic
	$\dyn$ is also surjective, namely,  
	\begin{equation}\label{eq:inj_traj_from_dyn}
			\forall \state_0, \ctrls \in \reals^\dimstate \times \reals^{\horizon\dimctrl}, \quad 	\sigma_{\min} (\nabla_\ctrl \traj(\state_0, \ctrls))\geq
			\sigma_\traj := \frac{\sigma_\dyn}{1+\lipdynstate} >0,
	\end{equation}
where $\lipdynstate = \sup_{\ctrl \in \reals^\dimctrl} \lip_{\dyn(\cdot,
\ctrl)}$ is the maximal Lipschitz-continuity constant of the functions
$\dyn(\cdot, \ctrl)$ for any $\ctrl \in \reals^{\dimctrl}$. 
\end{lemma}
\begin{proof}Fix $\state_0 \in \reals^\dimstate$. Given a sequence of controls
$\ctrls = (\ctrl_0;\ldots;\ctrl_{\horizon-1})\in \reals^{\horizon\dimctrl}$ with
corresponding trajectory  $\states=(\state_1;\ldots;\state_{\horizon})=
\traj(\state_0, \ctrls) \in \reals^{\horizon\dimstate}$, and $\dualvars =
(\dualvar_1;\ldots;\dualvar_\horizon) \in \reals^{\horizon\dimstate}$, the
Jacobian transpose vector product $\nabla_\ctrls \traj (\state_0, \ctrls)  \dualvars$  is
written
\begin{align*}
\nabla_\ctrls \traj (\state_0, \ctrls)  \dualvars & = (\nabla_{\ctrl_0}\dyn(\state_0, \ctrl_0) \lambda_1; \ldots; \nabla_{\ctrl_{\horizon-1}} \dyn(\state_{\horizon-1}, \ctrl_{\horizon-1})\lambda_\horizon) \\
\mbox{s.t.} \quad \lambda_t & = \nabla_{\state_t} \dyn(\state_t, \ctrl_t)  \lambda_{t+1} + \dualvar_t \quad \mbox{for} \ t\in \{1, \ldots, \horizon-1\}, \quad \lambda_\horizon = \dualvar_\horizon.
\end{align*}
For  $\states = (\state_1;\ldots;\state_{\horizon})$,  $\ctrls =
(\ctrl_0;\ldots;\ctrl_{\horizon-1})$, define  
$
\concatdyn(\states, \ctrls) = (\dyn(\state_0, \ctrl_0) ; \ldots; \dyn(\state_{\horizon-1}, \ctrl_{\horizon-1}))$. 
By using the upper block diagonal structure of $\nabla_\states
\concatdyn(\states, \ctrls)$, we have
\begin{align*}
(\idm - \nabla_\states \concatdyn(\states, \ctrls))^{-1}\dualvars & = (\lambda_1;\ldots;\lambda_\horizon)\\
\mbox{s.t.} \quad \lambda_t & = \nabla \dyn_{\state_t}(\state_t, \ctrl_t)  \lambda_{t+1} + \dualvar_t \quad \mbox{for} \ t\in \{1, \ldots, \horizon-1\}, \quad \lambda_\horizon = \dualvar_\horizon.
\end{align*}
The Jacobian transpose vector product can then be written compactly as
\begin{equation}
	\nonumber
	\nabla_\ctrls \traj (\state_0, \ctrls)  \dualvars = \nabla_{\ctrls} \concatdyn(\states, \ctrls) (\idm - \nabla_\states \concatdyn(\states, \ctrls))^{-1}\dualvars.
\end{equation}
Hence,  for  any command $\ctrls\in \reals^{\horizon\dimctrl}$ and any $\state_0
\in \reals^\dimstate$,
\begin{align*}
		\sigma_{\min} (\nabla_\ctrls \traj(\state_0, \ctrls)) \geq \frac{\sigma_{\min}(\nabla_\ctrls \concatdyn(\states, \ctrls))}{\sigma_{\max}( \idm - \nabla_\states \concatdyn(\states, \ctrls))} \geq \frac{	\sigma_{\dyn} }{ 1+\lipdynstate} .
\end{align*}
\end{proof}

Similarly, the smoothness properties of the control $\traj$ corresponding to
dynamics $\dyn$ can be expressed in terms of the smoothness properties of the
dynamics $\dyn$ as shown in the following lemma. 
\begin{lemma}
  If  $\dyn$ is
Lipschitz continuous with Lipschitz continuous Jacobians, then the function
$\ctrls\rightarrow \traj(\state_0, \ctrls)$ is $\lip_\traj$-Lipschitz-continuous
and has $\smooth_\traj$-Lipschitz-continuous Jacobians with 
\begin{equation}\label{eq:smooth_traj_from_dyn}
	\lip_\traj \leq \lipdynctrl S, \qquad 
	\smooth_\traj \leq S(\smoothdynstate \lip_\traj^2 + 2 \smoothdynstatectrl\lip_\traj + \smoothdynctrl)  \qquad 
	S =\sum_{t=0}^{\horizon-1} (\lipdynstate)^t,
\end{equation}
where the constants $\lipdynctrl = \sup_{\state \in \reals^\dimstate}
\lip_{\dyn(\state, \cdot)}$, $\smoothdynstate = \sup_{\ctrl\in \reals^\dimctrl}
\lip_{\nabla_\state \dyn(\cdot, \ctrl)} $, $\smoothdynctrl = \sup_{\state\in
\reals^\dimstate} \lip_{\nabla_\ctrl \dyn(\state, \cdot)} $,
$\smoothdynstatectrl = \sup_{\state\in \reals^\dimstate} \lip_{\nabla_\ctrl
\dyn(\cdot, \ctrl)}$ are maximal Lipschitz continuity constants of partial
functions or Jacobians of the dynamics. 
\end{lemma}
\begin{proof}
  This is a direct corollary of the time-varying version presented in Lemma~\ref{lem:smooth_traj_from_dyn_time_varying}.
\end{proof}

At first glance the Lipschitz continuity constant of the function $\traj$ and
its Jacobians appear to depend exponentially on the horizon $\tau$ through the
constant $S$ defined above. However, recall that problems of the
form~\eqref{eq:discrete_pb} stem from the discretization of a continuous problem
on a finite time interval $[0, T]$. The Lipschitz continuity constant of the
discretized dynamics depend then on the discretization step $\Delta$, which
depends itself on the discrete time horizon $\tau$ as $\Delta= T/\horizon$.
Hence, the dependency of the smoothness constants of the problem may not depend
exponentially on $\horizon$. 

For example, if the continuous time dynamics of the problem are given by a
function $\mathrm{f}$ and an Euler discretization scheme is used, then the
discretized dynamic take the form $\dyn(\state_t, \ctrl_t) = \state_t + \Delta
\mathrm{f}(\state_t, \ctrl_t)$ with $\Delta = T/\horizon$ and the Lipschitz
continuity parameter of the discretized dynamics is then $\lipdynstate \leq 1 +
\Delta \lip_{\mathrm{f}}^\state $. Hence, the constant $S$ defined above can be
upper bounded as $S \leq \sum_{t=0}^{\horizon-1} (1+ \lip_\mathrm{f}^\state
T/\horizon)^t \leq  (\exp(T\lip_\mathrm{f}^\state ) -1)
\horizon/(T\lip_\mathrm{f}^\state )$ and since $\lipdynctrl, \smoothdynstate,
\smoothdynstatectrl, \smoothdynctrl$ are all proportional to $\Delta = T/\tau$,
the smoothness constants derived in~\eqref{eq:smooth_traj_from_dyn} are
independent of $\horizon$ in this case and only depends on the length $T$ of the
continuous time problem.  

\subsection{Dynamic Decomposition}\label{sec:suff_cond} We have isolated
condition~\eqref{eq:global_conv_cond} as a sufficient condition to ensure
convergence of, e.g., a gradient descent, to global minima. It remains to consider whether this
assumption can be satisfied on concrete examples. Note that
assumption~\eqref{eq:global_conv_cond} requires $\dimctrl \geq \dimstate$. While
the underlying continuous control problem may have less control variables than
state variables, by considering multiple steps of a simple Euler discretization
method, we may still ensure the validity of~\eqref{eq:global_conv_cond} as
illustrated in Example~\ref{exm:suff_cond_pendulum}.

\begin{example}\label{exm:suff_cond_pendulum}
Consider the continuous time evolution of a pendulum
\begin{align*}
    \dot \theta(t) & = \omega(t), \\
    ml^2 \dot \omega(t) & = p(\theta(t), \omega(t), \ctrl(t)) 
    \coloneqq -m l g \sin \theta(t) - \mu \omega(t) + \ctrl(t),
\end{align*}
where $\theta$ is the angle with the vertical axis, $\omega$ 
is the angular speed, $\ctrl$ is a torque 
applied to the pendulum which defines the control we have on the system, and
$m$, $l$, $\mu$, $g$ are physical constants of the problem described in
Section~\ref{sec:exp}. The state is defined by
$\state = (\theta, \omega)$. Using a simple Euler scheme, the 
discretized dynamics cannot satisfy \eqref{eq:single_step_injectivity}, since
we would have only one variable $u_t$ to control two elements, 
$\theta_t, \omega_t$, at each time step $t$.

On the other hand, one can consider a two-step discretization scheme such that
the controls are divided in two variables $u_t = (v_t, v_{t+1/2})$. The 
dynamics read then
\begin{align*}
    \theta_{t+1/2} & = \theta_t + \Delta \omega_t, \hspace*{73.5pt}
    \theta_{t+1} = \theta_{t+1/2} + \Delta \omega_{t+1/2},
    \\
    ml^2 \omega_{t+1/2} & = 
    \omega_t + \Delta p(\theta_t, \omega_t, v_t), \quad
    ml^2 \omega_{t+1} = 
    \omega_{t+1/2} + \Delta p(\theta_{t+1/2}, \omega_{t+1/2}, v_{t+1/2}),
\end{align*}
where $\Delta$ is some discretization step. Intuitively, the variable $v_{t+1/2}$
fully controls $\omega_{t+1}$, while the variable $v_t$ fully controls
$\theta_{t+1}$. One can verify that the Jacobian of the discretized dynamics 
$x_{t+1} = f(x_t, u_t)$ for $u_t = (v_t, v_{t+1/2})$ are then surjective which then
ensure the surjectivity of the overall control of the pendulum in $\tau$ steps and
the efficiency of the ILQR and IDDP algorithms as observed in Figure~\ref{fig:conv}
and further justify in Section~\ref{sec:cvg}.
\end{example}

Formally, in this section, we assume that the discrete time dynamic $\dyn$ can
be further decomposed as the control \emph{in} $\ksteps$ steps of some
elementary discrete time dynamic $\microdyn$ as defined below. Concretely,
$\microdyn$ may correspond to a single Euler discretization step of some
continuous time dynamic. The discrete time dynamic $\dyn$ amounts then to
$\ksteps$ steps of such Euler discretization scheme and can be formulated as
$\dyn(\state_t, \ctrl_t)= \trajin{\ksteps}(\state_t, \ctrl_t)$, for some $k\geq
0$. On the other hand, we consider the costs to be computed only at the scale of
the dynamic $\dyn$, i.e., the sampling of the costs and the sampling of the
dynamics differ, hence the terminology multi-rate sampling. 
\begin{definition}\label{def:traj_func_ksteps} We define the control \emph{in}
	$\ksteps$ steps of a discrete time dynamic $\microdyn:\reals^\dimstate \times
	\reals^\microdimctrl \rightarrow \reals^\dimstate$ as the function
	$\trajin{\ksteps}: \reals^\dimstate \times \reals^{\ksteps\microdimctrl}
	\rightarrow \reals^{\dimstate}$, which, given a  state $\auxstate_0$ and a
	sequence of controls $\auxctrls = (\auxctrl_0;\ldots;\auxctrl_{\ksteps-1})$,
	outputs the state computed after $k$ steps, i.e.,
	\begin{align}
		\trajin{\ksteps}(\auxstate_0, \auxctrls) & =\auxstate_\ksteps \\
		\mbox{s.t.} \quad \auxstate_{s+1} &= \microdyn(\auxstate_s, \auxctrl_s) \quad \mbox{for} \ s \in \{0,\ldots, \ksteps-1\}. \nonumber
	\end{align}
\end{definition}
Our overall approach is illustrated in Figure~\ref{fig:zooming}. 
\begin{figure}[!t]
  \centering
  \includegraphics[width=0.9\linewidth]{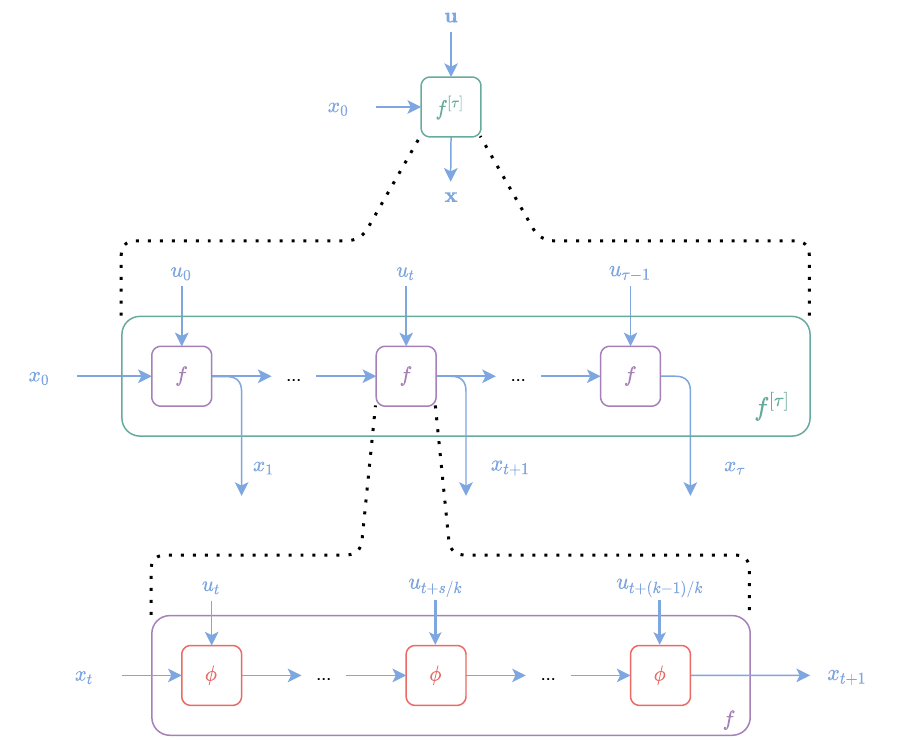}
  \caption{
    Zooming into the properties of the dynamics.
    The overall objective can be split at the scale of each step of the dynamics. Each step can be further decomposed at the scale of the discretization scheme to link properties of the underlying dynamic to global properties of the objective.\label{fig:zooming}
  }
\end{figure}
Our goal is then to know whether, by considering enough steps of $\microdyn$, we
can ensure the surjectivity of the linearized dynamic $\auxxctrls \mapsto
\nabla_\auxctrls \trajin{\ksteps}(\auxstate_0, \auxctrls)^\top \auxxctrls$. To
build some intuition, consider a system driven by its acceleration such that the
state of the system is determined by the position and the velocity
($\dimstate=2$) and the control is a scalar force ($\microdimctrl=1$)
determining the acceleration, hence controlling effectively the speed at each
time-step. For such system, the state of the system cannot be fully determined
in one step of an Euler discretization scheme, as only the velocity is affected
by the control. However, in two steps we can control both the position and the
velocity, hence we may satisfy assumption~\eqref{eq:global_conv_cond} as illustrated
in Example~\ref{exm:suff_cond_pendulum}. To
formalize and generalize this intuition, we  
consider the availability of a feedback linearization scheme as defined below
(adapted from~\citet{aranda1996linearization}). A brief exposition of static
feedback linearization schemes in continuous time and the associated Brunovsky's
form are presented in Appendix~\ref{app:static_lin}.
\begin{definition}\label{def:lin_feedback} A discrete time system defined by
	$\auxstate_{t+1} = \microdyn(\auxstate_t, \auxctrl_t)$ with $\auxstate_t \in
	\reals^\dimstate, \auxctrl_t \in \reals^\microdimctrl$ can be linearized by
	static feedback if there exist some diffeomorphisms
	$\diffeostate:\reals^{\dimstate}\rightarrow \reals^\dimstate$ and
	$\diffeoctrl(\auxstate, \cdot) :\reals^\microdimctrl \rightarrow
	\reals^\microdimctrl$ such that the reparameterization of the system as
	$\auxxstate_t = \diffeostate(\auxstate_t)$, $\auxxctrl_t =
	\diffeoctrl(\auxctrl_t, \auxstate_t)$ is linear. Namely,  there exists
	$\microdimctrl$ indexes $\orderlin_1, \ldots, \orderlin_{\microdimctrl}$ with
	$\sum_{j=1}^{\microdimctrl} \orderlin_j = \dimstate$ such that $\auxxstate$
	can be decomposed as $\auxxstate_t = (\auxxxstate_{t,
	1};\ldots;\auxxxstate_{t, \microdimctrl})$ with $\auxxxstate_{t, j} \in
	\reals^{\orderlin_j}$ decomposed as
	\begin{align}\nonumber
		\auxxxstate_{t+1, j}^{(i)} = \auxxxstate_{t, j}^{(i+1)} \  \mbox{for all}\   i  \in \{1, \ldots, \orderlin_j-1\}, \quad
		\auxxxstate_{t+1}^{(\orderlin_j)} = \auxxctrl_t^{(j)}, \ \mbox{for all} \ j \in \{1, \ldots, \microdimctrl\},
	\end{align}
where $\auxxxstate^{(i)}$ denotes the $i$\textsuperscript{th} coordinate of
$\auxxxstate$. 
\end{definition} 
For single-input system ($\microdimctrl=1$, $\orderlin = \dimstate$), the
reparameterized system takes the canonical Brunovsky
form~\citep{brunovsky1970classification}
\begin{align}\label{eq:brunovsky}
	\auxxstate_{t+1}^{(i)} = \auxxstate_t^{(i+1)} \  \mbox{for all}  \ i  \in \{1, \ldots, \dimstate-1\}, \quad
	\auxxstate_{t+1}^{(\dimstate)} = \auxxctrl_t,
\end{align}
i.e., $\auxxstate_{t+1} = D \auxxstate_t + \auxxctrl_t e$, where  $D =
\sum_{i=1}^{\dimstate-1} e_ie_{i+1}^\top$ is the upper shift matrix in
$\reals^\dimstate$ with $e_i$ the $i$\textsuperscript{th} canonical vector in
$\reals^\dimstate$, such that $(D\auxxstate)^{(i)} = \auxxstate^{(i+1)}$, and
$e= e_\dimstate$.

As a concrete example, consider the Euler discretization with stepsize
$\Delta>0$ of a single input continuous time system driven by its
$\dimstate$\textsuperscript{th} derivative through a  differentiable function
$\contdyn$, that is, 
\begin{equation}\label{eq:ode_single_input}
	\auxstate_{t+1}^{(i)} = \auxstate_t^{(i)} + \Delta \auxstate_t^{(i+1)}, \ \mbox{for all }  i  \in \{1, \ldots, \dimstate-1\}, \quad
	\auxstate_{t+1}^{(\dimstate)} = \auxstate_{t}^{(\dimstate)} + \Delta\contdyn(\auxstate_t,  \auxctrl_t).
\end{equation}
As shown in Lemma~\ref{lem:ode} in Appendix~\ref{app:static_lin}, such system can
easily be reparameterized in Brunovsky's form~\eqref{eq:brunovsky} and
provided that $|\partial_\auxctrl \contdyn(\auxstate,  \auxctrl)|>0$ for all
$\auxstate \in \reals^\dimstate, \auxctrl \in \reals$, we can have access to a
feedback linearization scheme, i.e., we can reparameterize the system in a
linear form using diffeomorphisms. 

The canonical representation~\eqref{eq:brunovsky} clarifies why the surjectivity
of the linearized dynamics may hold by taking enough steps as it is clear that,
in the representation~\eqref{eq:brunovsky}, by controlling the system in
$\dimstate$ steps we directly control the output. Namely, we have that
$\auxxstate_{\dimstate}^{(i)} = \auxxctrl_{i-1}$ for all $i \in \{1, \ldots,
\dimstate\}$.  
So for the system~\eqref{eq:brunovsky}, considering $\dimstate$ steps  ensures
condition~\eqref{eq:global_conv_cond}. The following theorem shows that this
property is kept when considering the original system.

\begin{theorem}\label{thm:suff_cond} If a  discrete time system $\auxstate_{t+1}
	= \microdyn(\auxstate_t, \auxctrl_t)$  is linearizable by static feedback as
	defined in Def.~\ref{def:lin_feedback}, then $ \trajin{\ksteps}$, the control
	in $\ksteps \geq \orderlin= \max\{\orderlin_1, \ldots,
	\orderlin_{\microdimctrl}\}$ steps of $\microdyn$,  has surjective
	linearizations, i.e., it satisfies $\sigma_{\min}( \nabla_\auxctrls
	\trajin{\ksteps}(\auxstate_0, \auxctrls) )>0 $ for any $ \auxstate_0 \in
	\reals^\dimstate, \auxctrls \in \reals^{\ksteps\microdimctrl}$.
	
	Quantitatively, if  the system defined by $\auxstate_{t+1} =
	\microdyn(\auxstate_t, \auxctrl_t)$  is linearizable by static feedback with
	transformations $\diffeostate$  and $\diffeoctrl$ that are Lipschitz
	continuous and  such that
	\[
	\inf_{\auxstate \in \reals^\dimstate} \sigma_{\min} (\nabla \diffeostate(\auxstate)) \geq \sigma_\diffeostate >0, \quad 
	\inf_{\auxstate \in \reals^\dimstate,\auxctrl \in \reals^\microdimctrl} \sigma_{\min}(\nabla_{\auxctrl} \diffeoctrl(\auxstate, \auxctrl)) \geq \sigma_{\diffeoctrl} >0,
	\]
	then  the control in $\ksteps \geq \orderlin$ steps of the dynamic $\microdyn$
	satisfies, for  $\lipdiffeoctrl = \sup_{\auxctrl \in \reals^\microdimctrl}
	\lip_{\diffeoctrl(\cdot, \auxctrl)}$,
	\[
	\inf_{\auxstate_0 \in \reals^\dimstate,\auxctrls \in \reals^{\ksteps\microdimctrl}} \sigma_{\min}( \nabla_\auxctrls \trajin{\ksteps}(\auxstate_0, \auxctrls) ) \geq \frac{\sigma_\diffeoctrl}{\lip_\diffeostate} \frac{1}{1+ (\orderlin-1) \lipdiffeoctrl/\sigma_\diffeostate } >0.
	\]
\end{theorem}
\begin{proof}
	We present the main steps of the proof, additional technical details are
	provided in the Appendix~\ref{app:comput}. We detail first the {single-input
	case} described in~\eqref{eq:brunovsky}, i.e., $\microdimctrl = 1$ and
	$\orderlin =\dimstate$. Moreover, we consider first $\ksteps=\dimstate$. Let
	$\auxctrls = (\auxctrl_0; \ldots; \auxctrl_{\ksteps-1}) \in \reals^\ksteps$
	and denote $\auxstate_{\ksteps} = \trajin{\ksteps}(\auxstate_0, \auxctrls)$.
	In the reparameterization of the system in the form~\eqref{eq:brunovsky}, we
	have that $\auxxstate_{\ksteps}^{(i)} = \auxxctrl_{i -1}$ for all $i \in \{1,
	\ldots, \dimstate\}$. By defining, for  $\auxstate_0$ fixed, $\auxstates =
	(\auxstate_1;\ldots;\auxstate_{\ksteps})$ and  $\auxctrls =
	(\auxctrl_0;\ldots;\auxctrl_{\ksteps-1})$, the function
	$
	\Diffeoctrl(\auxstates, \auxctrls) = 
	(\diffeoctrl(\auxstate_0, \auxctrl_0) ; \ldots; \diffeoctrl(\auxstate_{\ksteps-1}, \auxctrl_{\ksteps-1})) \in \reals^\dimstate$, 
	we have  
	$\trajin{\ksteps}(\auxstate_0, \auxctrls) = \diffeostate^{-1}
	(\Diffeoctrl(\microtraj(\auxstate_0, \auxctrls), \auxctrls))$, where
	$\microtraj(\auxstate_0, \auxctrls)$ denotes the control \emph{of} $\ksteps$
	steps of  $\microdyn$. Hence, denoting $\auxstates =
	(\auxstate_1;\ldots;\auxstate_\ksteps) = \microtraj(\auxstate_0, \auxctrls)$,
	we have
	\[
	\nabla_\auxctrls \trajin{\ksteps}(\auxstate_0, \auxctrls) = 
	\left(
	\nabla_\auxctrls \Diffeoctrl(\auxstates, \auxctrls) 
	+  \nabla_\auxctrls \microtraj(\auxstate_0, \auxctrls) 
	\nabla_\auxstates \Diffeoctrl(\auxstates, \auxctrls) 
	\right)
	\nabla \diffeostate(\auxstate_k)^{-1}.
	\]
	Since $\nabla_\auxstates \Diffeoctrl(\auxstates, \auxctrls)$ is strictly upper
	block triangular, $\nabla_\auxctrls \microtraj(\auxstate_0, \auxctrls)$
	is upper  block triangular, $\nabla_\auxctrls \Diffeoctrl(\auxstates,
	\auxctrls)$ is diagonal with non-zero  entries, we  have that
	$\nabla_\auxctrls \trajin{\ksteps}(\auxstate_0, \auxctrls)$ is invertible
	
	Now, consider $k>\dimstate$ and $s=\ksteps-\dimstate>0$. Denote
	$\auxctrl_{a:b} = (\auxctrl_a;\ldots;\auxctrl_b)$ for $a<b$. Let $\auxstate_0
	\in \reals^\dimstate$ and $\auxctrls =
	(\auxctrl_0;\ldots;\auxctrl_{\ksteps-1})\in \reals^{\ksteps}$. We have
	$\trajin{\ksteps}(\auxstate_0, \auxctrls) = \trajin{\dimstate}(
	\trajin{s}(\auxstate_0, \auxctrl_{0:s-1}), \auxctrl_{s:\ksteps-1})$. Hence,
	denoting $\auxstate_s =  \trajin{s}(\auxstate_0, \auxctrl_{0:s-1})$, we have 
	\begin{equation}\label{eq:decomp_grad_ksteps}
		\nabla_{\auxctrls} \trajin{\ksteps}(\auxstate_0, \auxctrls) 
		= \left(\begin{matrix}
			\nabla_{\auxctrl_{0:s-1}} \trajin{s}(\auxstate_0, \auxctrl_{0:s-1})
			\nabla_{\auxstate_s} \trajin{\dimstate}(\auxstate_s,  \auxctrl_{s:\ksteps-1}) \\
			\nabla_{\auxctrl_{s:\ksteps-1}}\trajin{\dimstate}(\auxstate_s, \auxctrl_{s:\ksteps-1})
		\end{matrix}\right).
	\end{equation}
	The function $\auxstate_s,  \auxctrl_{s:\ksteps-1}\rightarrow
	\trajin{\dimstate}(\auxstate_s, \auxctrl_{s:\ksteps-1})$ amounts to the
	control of $\microdyn$ in $\ksteps-s=\dimstate$ steps. Hence,  the matrix
	$\nabla_{\auxctrl_{s:\ksteps-1}}\trajin{\dimstate}(\auxstate_s,
	\auxctrl_{s:\ksteps-1})$ is invertible, so $\nabla_{\auxctrls}
	\trajin{\ksteps}(\auxstate_0, \auxctrls) $ has full column rank. Overall, we
	showed the first part of the claim, i.e., that $\sigma_{\min}(
	\nabla_\auxctrls \trajin{\ksteps}(\auxstate_0, \auxctrls) )>0 $ for any $
	\auxstate_0 \in \reals^\dimstate, \auxctrls \in
	\reals^{\ksteps\microdimctrl}$, provided that $\ksteps \geq \dimstate$. 
	
	We consider now deriving quantitative bounds. We  focus on the single-input
	case and start with $\ksteps=\dimstate$. Define, for  $\auxstate_0$ fixed,
	$\auxstates = (\auxstate_1;\ldots;\auxstate_{\ksteps})$, $\auxctrls =
	(\auxctrl_0;\ldots;\auxctrl_{\ksteps-1})$, the function $\Microdyn(\auxstates,
	\auxctrls) = (\microdyn(\auxstate_0, \auxctrl_0) ; \ldots;
	\microdyn(\auxstate_{\ksteps-1}, \auxctrl_{\ksteps-1}))$. By decomposing
	$\nabla_\auxctrls \microtraj(\auxstate_0, \auxctrls)$ as in
	Lemma~\ref{lem:injectivity}, we get 
	\begin{align*}
		\nabla_\auxctrls \trajin{\ksteps}(\auxstate_0, \auxctrls) & = 
		\left(\nabla_\auxctrls \Diffeoctrl(\auxstates, \auxctrls) 
		+  \nabla_\auxctrls \Microdyn(\auxstates, \auxctrls)
		(\idm - \nabla_\auxstates \Microdyn(\auxstates, \auxctrls))^{-1} 
		\nabla_\auxstates \Diffeoctrl(\auxstates, \auxctrls)\right) 
		\nabla \diffeostate(\auxstate_k)^{-1}.
	\end{align*}
	Given the feedback linearization scheme, the discrete time dynamic $\microdyn$
	can be rewritten as $\auxstate_{t+1} = \microdyn(\auxstate_t, \auxctrl_t) =
	\diffeostate^{-1}(D \diffeostate(\auxstate_t) + \diffeoctrl(\auxstate_t,
	\auxctrl_t) e)$, where $D$ is the upper shift matrix in $\reals^\dimstate$ and
	$e=e_\dimstate$ is the $\dimstate$\textsuperscript{th} canonical vector in
	$\reals^\dimstate$. Hence, we have for $t \in \{0, \ldots, \ksteps-1\}$, 
	\begin{align*}
		\nabla_{\auxctrl_t} \microdyn(\auxstate_t, \auxctrl_t) 
		& = \partial_{\auxctrl_t} \diffeoctrl(\auxstate_t, \auxctrl_t) e^\top \nabla \diffeostate(\auxstate_{t+1})^{-1} \\
		\nabla_{\auxstate_t} \microdyn(\auxstate_t, \auxctrl_t) & = \left(\nabla\diffeostate(\auxstate_t) D^\top  + \nabla_{\auxstate_t} \diffeoctrl(\auxstate_t, \auxctrl_t) e^\top\right)\nabla \diffeostate(\auxstate_{t+1})^{-1}.
	\end{align*}
In the sequel, we denote  the Kronecker product by $\otimes$ and for  $R_1,
	\ldots, R_n \in \reals^{p\times q}$ we denote by $\diag((R_i)_{i=1}^n)  =
	\sum_{i=1}^{n} e_ie_i^\top {\otimes} R_i \in \reals^{np \times nq}$ the block
	diagonal matrix with blocks $R_1, \ldots, R_n$, for $e_i $ the
	$i$\textsuperscript{th} canonical vector in $\reals^n$. Since
	$\nabla_{\auxctrls} \Microdyn(\auxstates, \auxctrls) =
	\diag((\nabla_{\auxctrl_t} \microdyn(\auxstate_t,
	\auxctrl_t))_{t=0}^{\ksteps-1} )$, $\nabla_\auxstates  \Microdyn(\auxstates,
	\auxctrls) = (D \otimes \idm ) \diag(\nabla_{\auxstate_t}
	\microdyn(\auxstate_t, \auxctrl_t)_{t=0}^{\ksteps-1})$, we have that (see
	Appendix~\ref{app:comput} for more details)
	\begin{align}
		\nabla_{\auxctrls} \Microdyn(\auxstates, \auxctrls) &= \diag((\partial_{\auxctrl_t} \diffeoctrl(\auxstate_t, \auxctrl_t) )_{t=0}^{\ksteps-1})
		(\idm \otimes e^\top )
		\diag((\nabla \diffeostate(\auxstate_{t+1})^{-1})_{t=0}^{\ksteps-1})\nonumber \\
		\nabla_\auxstates  \Microdyn(\auxstates, \auxctrls) & =
		(D \otimes \idm )
		\diag((\nabla \diffeostate(\auxstate_t))_{t=0}^{\ksteps-1})
		(\idm \otimes D^\top)
		\diag((\nabla \diffeostate(\auxstate_{t+1})^{-1})_{t=0}^{\ksteps-1}) \nonumber\\
		& \quad + (D \otimes \idm )
		\diag((\nabla_{\auxstate_t} \diffeoctrl(\auxstate_t, \auxctrl_t))_{t=0}^{\ksteps-1})
		(\idm \otimes e^\top)
		\diag((\nabla \diffeostate(\auxstate_{t+1})^{-1})_{t=0}^{\ksteps-1}),\label{eq:decomp_grad_suff_cond}
	\end{align}
	and,  similarly, $\nabla_\auxctrls \Diffeoctrl(\auxstates, \auxctrls) =
	\diag((\partial_{\auxctrl_t} \diffeoctrl(\auxstate_t, \auxctrl_t)
	)_{t=0}^{\ksteps-1})$, \hspace{-1pt} $\nabla_\auxstates
	\Diffeoctrl(\auxstates, \auxctrls) = (D \otimes \idm )
	\diag((\nabla_{\auxstate_t} \diffeoctrl(\auxstate_t,
	\auxctrl_t))_{t=0}^{\ksteps-1})$. Denoting $A = \diag((\nabla
	\diffeostate(\auxstate_{t}))_{t=0}^{\ksteps-1})$,  $C = 	\diag((\nabla
	\diffeostate(\auxstate_{t+1}))_{t=0}^{\ksteps-1})$,  $V =
	\diag((\partial_{\auxctrl_t} \diffeoctrl(\auxstate_t, \auxctrl_t)
	)_{t=0}^{\ksteps-1})$, $Y =  \diag((\nabla_{\auxstate_t}
	\diffeoctrl(\auxstate_t, \auxctrl_t))_{t=0}^{\ksteps-1})$, $E = 	\idm \otimes
	e^\top$, $F=D \otimes \idm $ and $G = \idm \otimes D^\top$, we get that 
	\begin{align}
		\nabla_\auxctrls \trajin{\ksteps}(\auxstate_0, \auxctrls) \nabla \diffeostate(\auxstate_k) & = V(\idm + EC^{-1}(\idm  - FAGC^{-1}  - FYEC^{-1})^{-1}FY) \nonumber\\
		& \stackrel{(i)}= V(\idm -EC^{-1}(\idm  - FAGC^{-1})^{-1}FY )^{-1} , \nonumber\\
		& \stackrel{(ii)}= V(\idm -E(\idm  - FG)^{-1}F A^{-1}Y )^{-1}, \nonumber\\ 
		& \stackrel{(iii)}= V\left(\idm -E\left(\sum_{i=1}^{\ksteps-1} D^i\otimes(D^\top)^{i-1}\right)A^{-1} Y \right)^{-1}. \label{eq:final_decomp_suff_cond}
	\end{align}
	Above, in $(i)$ we used the Sherman-Morrison-Woodbury identity, in $(ii)$ we
	used that  $FA =CF$ and  $FA^{-1} = C^{-1}F$ (see Appendix~\ref{app:comput}),
	in $(iii)$ we used that $FG = D\otimes D^\top$ is nilpotent of order
	$\ksteps=\dimstate$ since $D^\ksteps= 0$ (see Appendix~\ref{app:comput}). The
	result follows for $\ksteps =\dimstate$ from the assumptions of Lipschitz
	continuity and non-singularity of the transpose Jacobians of the diffeomorphisms, and
	from the fact that $\|E\|_2\leq 1$ and $\|D\|_2\leq 1$. For $\ksteps
	>\dimstate$, we have from~\eqref{eq:decomp_grad_ksteps}, that for any $\lambda
	\in \reals^\dimstate$,  $\|\nabla_{\auxctrl} \trajin{\ksteps}(\auxstate_0,
	\auxctrls) \lambda\|_2 \geq
	\|\nabla_{\auxctrl_{s:\ksteps-1}}\trajin{\dimstate}(\auxstate_s,
	\auxctrl_{s:\ksteps-1})\lambda\|_2$, hence the result follows. 
	
	For multi-input systems, let $\orderlin= \max\{\orderlin_1,\ldots,
	\orderlin_{\microdimctrl}\}$. One easily verifies that for any $\ksteps \geq
	\orderlin$, the system in its linear representation can be written as
	$\auxxstate_\ksteps = M \auxxctrls$ for $\auxxctrls =
	(\auxxctrl_0;\ldots;\auxxctrl_{\ksteps-1})$ with $\sigma_{\min}(M^\top) =1$.
	The first part of the claim follows then as in single input case. For the
	second part, the system can be decomposed by blocks and treated as in the
	single-input case, see Appendix~\ref{app:comput} for more details. 
\end{proof}
Overall, Theorem~\ref{thm:suff_cond} shows that for, e.g., a dynamical system
driven by its $k$\textsuperscript{th} derivative as
in~\eqref{eq:ode_single_input}, by considering a dynamic $\dyn$ defined by $k$
steps of an Euler discretization of the system,
condition~\eqref{eq:global_conv_cond} can be ensured, which in turns can ensure
a gradient dominating property for the objective. The ILQR and IDDP algorithms
are not just gradient descent algorithms. It remains now to exploit
assumption~\eqref{eq:global_conv_cond} to uncover the efficiency of the ILQR or
IDDP algorithms. 

\section{Convergence Analysis}\label{sec:cvg}
To analyze the convergence of the ILQR and the IDDP algorithms, we consider
problem~\eqref{eq:discrete_pb} at the scale of the whole trajectory and analyze
problem~\eqref{eq:discrete_pb} as a compositional problem of the form
\begin{equation}\label{eq:total_cost}
\min_{\ctrls\in \reals^{\horizon\dimctrl}} \left\{\obj(\ctrls)=\cost (\augtraj(\ctrls))\right\}, \ \mbox{where} \ \augtraj(\ctrls)=\traj(\initstate, \ctrls) \ \mbox{and} \ \cost(\states) = \sum_{t=1}^\horizon \cost_t(\state_t).
\end{equation}
Note however that the dynamical structure of the problem revealed at the state
scale is essential to the implementation of the ILQR and IDDP algorithms. We
state our assumptions for convergence at the state scale and translate
them at the trajectory scale. A table of all constants introduced for the
convergence analysis with their respective units is provided in
Appendix~\ref{app:index} for ease of reference.
\begin{assumption}\label{asm:conv} We consider convex costs $\cost_t$ that have
	$\smooth_\cost$-Lipschitz-continuous gradients and
	$\smoothess_\cost$-Lipschitz-continuous Hessians for all $t \in \{1, \ldots,
	\horizon\}$. We consider the dynamics to be Lipschitz-continuous with
	Lipschitz continuous Jacobians and
	satisfying~\eqref{eq:single_step_injectivity}.
	
	In consequence, the total cost~$\cost$ defined in~\eqref{eq:total_cost} is
	convex,  has $\smooth_\cost$-Lipschitz-continuous gradients and
	$\smoothess_\cost$-Lipschitz-continuous Hessians. The function $\augtraj$
	defined in~\eqref{eq:total_cost} is $\lip_{\augtraj}$-Lipschitz-continuous
	with $\smooth_\augtraj$-Lipschitz-continuous Jacobians satisfying
	\begin{equation}\label{eq:inj_traj}
		\forall \ctrls \in\reals^{\horizon \dimctrl}, \quad 	\sigma_{\min} (\nabla \augtraj(\ctrls))\geq \sigma_\augtraj >0,
	\end{equation}
	where $\lip_\augtraj = \lip_\traj$, $\smooth_\augtraj = \smooth_\traj$ are
	given in ~\eqref{eq:smooth_traj_from_dyn} and $\sigma_\augtraj = \sigma_\traj$
	is given in~\eqref{eq:inj_traj_from_dyn}.
\end{assumption}

\subsection{Convergence Proof Sketches}\label{ssec:intuition}
\paragraph{The ILQR algorithm is a generalized Gauss-Newton algorithm}
From a high-level perspective, the  ILQR algorithm consists in linearizing the
function $\augtraj:\ctrls\rightarrow \traj(\initstate, \ctrls)$ that
encapsulates the dynamics, taking a quadratic approximation of the costs~$\cost$
around the current trajectory $\states =\augtraj(\ctrls) $ and minimizing the
resulting approximation with an additional regularization. Formally, as
previously observed by~\cite{sideris2005efficient, roulet2019iterative}, the
ILQR algorithm is then computing
\begin{align}
\LQR_\reg(\obj)(\ctrls) & = \argmin_{\auxctrls \in \reals^{\horizon\dimctrl}} \qua_{\cost}^{\augtraj(\ctrls)}(\lin_{\augtraj}^{\ctrls}(\auxctrls)) + \frac{\reg}{2} \|\auxctrls\|_2^2\nonumber \\
& = -(\nabla \augtraj(\ctrls)\nabla^2 \cost(\augtraj(\ctrls)) \nabla \augtraj(\ctrls)^\top +\reg \idm)^{-1}  \nabla \augtraj(\ctrls)\nabla \cost(\augtraj(\ctrls)),\label{eq:rggn_oracle}
\end{align}
where $\lin_\augtraj^{\ctrls}$ and $\qua_\cost^{\augtraj(\ctrls)}$ are the
linear and quadratic approximations of, respectively, the control in $\horizon$
steps around $\ctrls$ and the total costs around $\augtraj(\ctrls)$ as defined
in the notations. Equation~\ref{eq:rggn_oracle} clearly reveals that the ILQR
algorithm amounts to a regularized generalized Gauss-Newton
algorithm~\citep{diehl2019local} implemented by a dynamic programming procedure
exploiting the structure of the problem. 

\paragraph{Proof sketch of global convergence}
By choosing a large enough regularization, the updates of the \ref{eq:ilqr_algo}
algorithm approach the ones of a gradient descent as we have from the expression
of $\LQR_{\reg}$ in~\eqref{eq:rggn_oracle} that for $\reg \gg 1$, $\var +
\LQR_{\reg}(\obj(\var)) \approx \var - \reg^{-1} \nabla \obj(\var)$. This
suggests that the~\ref{eq:ilqr_algo} algorithm can converge globally to a global
minimum, just as a gradient descent given a gradient dominating
property~\eqref{eq:grad_dom_obj}~\citep{polyak1964some, bolte2017error}. 

Formally, to ensure global convergence, we consider taking a regularization
$\reg$ that may depend on the current command $\ctrls \in
\reals^{\horizon\dimctrl}$, s.t. for $\auxctrls = \LQR_{\reg}(\obj)(\ctrls)$, 
\begin{equation}
		\obj\left(\ctrls +\auxctrls\right) \leq	\cost\circ\augtraj(\ctrls) + \qua_{\cost}^{\augtraj(\ctrls)}\circ \lin_\augtraj^\ctrls(\auxctrls) 
	+ \frac{\reg}{2} \|\auxctrls\|_2^2 = \obj(\ctrls) +\frac{1}{2}\nabla \obj(\ctrls)^\top\auxctrls. \label{eq:suff_cond_descent}
\end{equation}
Given the analytic form of $\auxctrls = \LQR_{\reg}(\obj)(\ctrls)$
in~\eqref{eq:rggn_oracle}, the above condition ensures that $\obj\left(\ctrls
+\auxctrls\right) -\obj(\ctrls)  \leq  -\alpha \|\nabla
\cost(\augtraj(\ctrls))\|_2^2$, for some constant $\alpha$ that depends on the
regularization $\reg$ and the properties of the objective. Hence, if $\cost$
satisfies a gradient dominating property, i.e., there exists $\plcst>0, \pl \in
[1/2,1)$ s.t. $\|\nabla \cost(\states)\|_2^2 \geq \plcst^\pl (\cost(\states)
-\cost^* )^{\pl}$ for any $\states\in \reals^{\horizon \dimstate}$, global
convergence to a global minimum can be ensured given a constant regularization.
For example, if $\pl=1/2$, by taking a constant regularization
ensuring~\eqref{eq:suff_cond_descent}, we get a global linear convergence rate. 

We further show that a regularization ensuring~\eqref{eq:suff_cond_descent} can
be chosen to scale as a function of $\|\nabla \cost(\augtraj(\ctrls))\|_2$,
which helps decompose the computational complexity in (i) the complexity of
solving $\min_{\states\in \reals^{\horizon \dimstate}}\cost(\states)$ given an
assumption on its gradient dominance and the smoothness properties of the costs,
(ii) a term that depends on the initial gap and condition numbers associated to
the approximation of a gradient descent by a Gauss-Newton method though the
smoothness properties of the cost, the dynamics and the surjectivity of the
dynamics.

\paragraph{Proof sketch of local convergence}
The rate of convergence sketched above can be refined by analyzing the
local behavior of the algorithm around a solution. Namely, if $\augtraj$
satisfies~\eqref{eq:inj_traj}, then the matrix $\nabla \augtraj(\ctrls)^\top
\nabla \augtraj(\ctrls)$ is invertible. Denoting  $\states = \augtraj(\ctrls)$,
$\grad =  \nabla \augtraj(\ctrls) $ and $\hess = \nabla^2 \cost(\states)$, we
then have by standard linear algebra manipulations, that the oracle returned by
the ILQR algorithm satisfies
\begin{align*}
\LQR_\reg(\obj)(\ctrls)  & = - (GHG^\top + \reg\idm)^{-1} G  \nabla \cost(\states) \\
& = -\grad (\hess\grad^\top \grad+\reg \idm)^{-1}  \nabla \cost(\states) \tag{\textit{Push-through identity}}\\
& =  -\grad  (\grad^\top \grad)^{-1}(\hess +\reg  (\grad^\top \grad)^{-1})^{-1}  \nabla \cost(\states). \tag{$G^\top G$ \textit{ invertible}}
\end{align*}
Consider then the trajectory associated to a single step of ILQR, i.e., for
$\auxctrls =\LQR_\reg(\obj)(\ctrls)$,
\begin{align*}
\auxstates = \augtraj(\ctrls + \auxctrls )  & \approx \augtraj(\ctrls) + \nabla \augtraj(\ctrls)^\top  \auxctrls  = \states  -( \nabla^2 \cost(\states) +\reg (\nabla \augtraj(\ctrls)^\top \nabla \augtraj(\ctrls))^{-1})^{-1}  \nabla \cost(\states).
\end{align*}
For $\reg \ll 1$, the difference of the trajectories $\auxstates-\states$  is
close to a Newton direction on the total costs~$\cost$. In other words, the ILQR
algorithm may be analyzed as an approximate Newton method on the total costs. In
particular, this suggests  that the algorithm can have a local quadratic
convergence rate if (i) the costs satisfy the assumptions required for a Newton
method to converge locally quadratically, such as self-concordance, (ii) the
regularization decreases fast enough. 

\paragraph{Proof sketch of total complexity}
If the costs are strongly convex then they satisfy a gradient dominating property
and are self-concordant.  To satisfy condition~\eqref{eq:suff_cond_descent}, the
regularization can then be chosen to be proportional to the norm of the gradient
of the costs at the current iterate, i.e., $\nu_k = \bar \nu_k \|\nabla
\cost(\augtraj(\var^{(k)}))\|_2$ for $\bar \nu_k$ bounded above by a constant
which ensures that $\nu_k$ tends to 0 with the iterations $k$. By satisfying
condition~\eqref{eq:suff_cond_descent}, we can ensure global convergence, while
by having $\nu_k \rightarrow 0$, we can ensure local quadratic convergence. 

\paragraph{Proof sketch of convergence of the IDDP algorithm}
The \ref{eq:iddp_algo} algorithm cannot be simply analyzed as an instance of a
classical optimization algorithm. However, a careful analysis of the difference
in the updates of the \ref{eq:ilqr_algo} and~\ref{eq:iddp_algo} algorithms for
strongly convex costs reveal that the difference in their oracles can be bounded
as  $\|\DDP_\reg(\obj)(\ctrls) - \LQR_{\reg}(\obj)(\ctrls)\|_2 \leq
\ddpbound\|\LQR_\reg(\obj)(\ctrls)\|_2^2$ for some constant $\ddpbound$
independent of $\ctrls$ and $\reg$. This observation enables us to derive an
appropriate rule for selecting the regularization for the \ref{eq:iddp_algo}
algorithm and to ensure that the quadratic local convergence is maintained since
the approximation error of $\LQR$ by $\DDP$ is quadratic. 

\begin{remark}
  If the function $g$ is surjective and satisfies
  Assumption~\eqref{eq:inj_traj}, then local quadratic convergence of e.g. a
  Gauss-Newton method or a Levenberg-Marquardt method (for $h$ quadratic) is
  known, see \citet[Chapter 9.2.2]{bjorck2024numerical},
  \citet{bergou2020convergence}. For $h$ non-quadratic, local quadratic
  convergence of generalized Gauss-Newton methods has also been shown in some
  special cases by~\citet[Section 2.2]{messerer2021survey}. Compared to these
  results, we consider deriving a convergence rate decomposed into a first slow
  convergence phase and a fast local convergence phase. Moreover, we consider
  quantitative bounds involving the constants in~\eqref{eq:inj_traj} and
  additional assumptions (self-concordance or gradient dominant assumptions).
\end{remark}

\subsection{Convergence Analysis of ILQR}\label{ssec:conv_ilqr}
\subsubsection{Global Convergence Rate to Global Minima}\label{ssec:global_conv}
We start by analyzing the ILQR algorithm provided that the costs satisfy a
sufficient condition for convergence to global minima, namely gradient
dominance, a.k.a. a Polyak-{\L}ojasiewicz
inequality~\citep{lojasiewicz1963topological, polyak1964some, bolte2017error}.
We consider convergence in objective values $\obj(\ctrls)$ for
problem~\eqref{eq:discrete_pb}, and analyze the number of iterations $k$ to
reach an accuracy $\varepsilon$, that is, such that $\obj(\ctrls^{(k)}) -
\min_{\ctrls \in \reals^{\horizon\dimctrl}} \obj(\ctrls) \leq \varepsilon$.

\begin{theorem}\label{thm:global_conv}
  
  Given Assumption~\ref{asm:conv}, the sufficient decrease
	condition~\eqref{eq:suff_cond_descent} is satisfied for a regularization
	\[
	\reg(\ctrls) =\frac{\smooth_\augtraj \|\nabla \cost(\augtraj(\ctrls))\|_2}{2} \polysqrt\left(\frac{ \smooth_\augtraj\|\nabla \cost(\augtraj(\ctrls))\|_2}{4 \lip_\augtraj^2 \smooth_\cost(\newsimp +1)}\right),
	\]
	where $\polysqrt(x) = 1 + \sqrt{1 + 1/x}$ and $\newsimp = \smoothess_\cost
	\lip_\augtraj^2/(3 \smooth_\augtraj \smooth_\cost)$. In addition to
	Assumption~\ref{asm:conv}, consider that the costs are dominated by their
	gradients, i.e., there exists $\pl \in [1/2, 1)$ and $ \plcst>0$ such that
  $
      \|\nabla \cost_t(\state) \|_2 \geq \plcst^\pl (\cost_t(\state) - \cost_t^*)^\pl
  $
  for all $\state \in \reals^{\dimstate}, t\in \{1, \ldots, \horizon\}$. The total
	 cost satisfies then, for  $\plcst_\cost = \plcst/\horizon^{(2\pl-1)/2\pl},$
		\begin{equation}\label{eq:pl}
	\forall \states \in \reals^{\horizon \dimstate}, \quad 	\|\nabla \cost(\states) \|_2 \geq \plcst_\cost^\pl (\cost(\states) - \cost^*)^\pl. 
	\end{equation}
	If $\pl=1/2$, given regularizations $\reg_k = \reg(\ctrls^{(k)})$, the number
	of iterations of the \ref{eq:ilqr_algo} algorithm to converge to an accuracy
	$\varepsilon$ in objective values for problem~\eqref{eq:discrete_pb}, is
	at most
	\begin{align*}
		k & \leq 4 \scaling \sqrt{\delta_0}\polysqrt\left(\frac{\scaling\sqrt{\delta_0}}{\simp}\right) + 2\condnb_\cost \ln\left(\frac{\delta_0}{\varepsilon}\right),
	\end{align*}
  and, if $1/2<\pl<1$, the number of iterations to converge to an accuracy
	$\varepsilon$, is at most
	\begin{align*}
		k \leq
		\frac{2}{2\pl -1} \frac{\condnb_\cost}{\varepsilon^{2\pl-1}}  + \frac{2}{1-\pl}\scaling\delta_0^{1-\pl} 
		+\sqrt{2\scaling\simp}\frac{1}{1-3\pl/2}
		\left(\varepsilon^{1-3\pl/2} - \left(\frac{\simp}{\scaling}\right)^{1/\pl-3/2}\right),
	\end{align*}
	with $\condnb_\cost = \smooth_\cost/\plcst_\cost^{2\pl}$, $\condnb_\augtraj =
	\lip_{\augtraj}/\sigma_{\augtraj}$, $\concord =
	\smoothess_\cost/(2\plcst_\cost^{3\pl})$, $\scaling =
	\smooth_\augtraj/(\sigma_\augtraj^2\plcst_\cost^\pl)$, $\simp =
	4\condnb_\augtraj^2 \condnb_\cost ( \newsimp + 1)$, $\delta_0 =
	\obj(\ctrls^{(0)}) - \obj^*$ and the case $\pl=2/3$ is to be understood
	limit-wise.
\end{theorem}
Before presenting the proof, a few remarks are in order. 
\begin{remark}
	Consider the case $\pl=1/2$ in Theorem~\ref{thm:global_conv}. The constants
	appearing in the bound are (i) the condition number $\condnb_\cost =
	\smooth_\cost /\plcst_\cost$ of the total cost~$\cost$, (ii) the condition
	number $\condnb_\augtraj = \lip_\augtraj/\sigma_\augtraj$ of the Jacobian of
	$\augtraj$, $\nabla \augtraj(\ctrls)$, (iii) a constant $\concord
	=\smoothess_\cost/(2\plcst_\cost^{3/2})$ that can be interpreted as a bound on
	the self-concordance parameter of the cost~$\cost$ if the total costs are
	strongly convex, (iv) a constant $\scaling =
	\smooth_\augtraj/(\sigma_\augtraj^2\sqrt{\plcst_\cost})$ whose dimension is
	the same as $\concord$, i.e., the inverse of the squared root of the
	objective. Finally, the terms $\newsimp$ and $\simp$ are additional dimension
	independent constants that act as additional condition numbers.
 \end{remark}
\begin{remark}\label{rmk:slow_cvg_phase}
	The rate of convergence in Theorem~\ref{thm:global_conv} for $r=1/2$ is
	composed of $(i)$ a term $\condnb_\cost
	\ln\left({\delta_0}/{\varepsilon}\right)$ that is the linear complexity
	associated to the computation of $\min_{\states \in \reals^{\horizon\dimctrl}
	} \cost(\states)$ by a gradient descent on  a function $\cost$ that has
	Lipschitz-continuous gradients with a gradient dominance property and (ii) a
	term $4 \scaling
	\sqrt{\delta_0}\polysqrt\left({\scaling\sqrt{\delta_0}}/{\simp}\right)$ that
	depends on the initial gap and appropriate condition numbers on the problem.
	To understand the effect of this second term, consider computing the value of
	the gap $\delta_j$ after $j$ iterations such that the complexity of reducing
	the gap further by a factor $1/e \approx1/2$ is dominated by the logarithmic
	term such that we enter a linear phase of convergence. Formally, after $j$
	iterations of the algorithm, the remaining number of iterations to reduce the
	gap further by a factor $1/e$, i.e., reach an accuracy
	$\varepsilon=\delta_j/e$, is $ 4\scaling
	\sqrt{\delta_j}\polysqrt\left({\scaling\sqrt{\delta_j}}/{\simp}\right)+2
	\condnb_\cost $.  To neglect the first term in favor of the second term we
	need $\tilde \polysqrt(\scaling \sqrt{\delta_j}/\simp) \leq
	\condnb_\cost/(2\simp)\leq 1$ for $\tilde{\polysqrt}(x)= x + \sqrt{x^2+x}$,
	which is satisfied for $\delta_j \leq c^2/\scaling^2$ with
	$c=\condnb_\cost/(16\condnb_\augtraj^2(1+\beta))$. So up to a multiplicative
	factor $c$, the parameter $1/\scaling^2$ plays the role of a gap determining a
	linear convergence phase. 
\end{remark}
\begin{remark}
	For $\smooth_\augtraj=0$, the terms depending on $\delta_0$ uniquely vanish
	since $\scaling = 0$ in this case. We then get  the classical rates when
	minimizing a function $\cost$ that satisfy~\eqref{eq:pl} with a first-order
	method. The rates can be improved by analyzing the local behavior of the
	algorithm to take advantage of the quadratic approximations of the total
	costs~$\cost$ as shown in Section~\ref{ssec:local_conv}.
\end{remark}
\begin{proof}[Proof of Theorem~\ref{thm:global_conv}]
	The validity of the gradient dominating property for the total costs is
	presented in Lemma~\ref{lem:plcst_decomp} in Appendix~\ref{app:lemmas}. Note
	that if $\cost$ satisfies~\eqref{eq:pl} and $\augtraj$
	satisfies~\eqref{eq:inj_traj}, then for any $\ctrls \in
	\reals^{\horizon\dimctrl}$, we have $\|\nabla (\cost\circ \augtraj)
	(\ctrls)\|_2 \geq \sigma_\augtraj \plcst_\cost^\pl (\cost(\augtraj(\ctrls)) -
	\cost^*)^\pl$. Hence, for $\ctrls^* \in \argmin_{\ctrls \in
	\reals^{\horizon\dimctrl}} \obj(\ctrls)$ with $\obj=\cost \circ \augtraj$, we
	get $0 = \|\nabla \obj(\ctrls^*) \|_2 \geq  \sigma_\augtraj \plcst_\cost^\pl
	(\cost(\augtraj(\ctrls^*)) - \cost^*)^\pl \geq 0$, such that we have $\obj^* =
	\cost^*$.
	
	We have from Lemma~\ref{lem:bound_approx_self_concord} that for any $\ctrls,
	\auxctrls \in \reals^{\horizon\dimctrl}$, denoting $a_0 = \smoothess_\cost
	\lip_{\augtraj}^3/3 {+} \smooth_\augtraj\smooth_\cost \lip_\augtraj$, 
	\[
	|(\cost \circ \augtraj)(\ctrls {+} \auxctrls) 
	{-} ( \cost \circ \augtraj)(\ctrls) 
	{-} \qua_{\cost}^{\augtraj(\ctrls)}\circ \lin_\augtraj^\ctrls(\auxctrls) | 
	\leq\frac{\smooth_\augtraj \|\nabla \cost(\augtraj(\ctrls))\|_2 {+}  a_0\|\auxctrls\|_2}{2}\|\auxctrls\|_2^2.
	\]
	Since  $\|\LQR_\reg(\obj)(\ctrls) \|_2 \leq \lip_{\augtraj} \|\nabla
	\cost(\augtraj(\ctrls))\|_2/\reg$, condition~\eqref{eq:suff_cond_descent} is
	satisfied for $\reg>0$ s.t.
	$
	a_1 + a_2/\reg \leq \reg,
	$
	where $a_1= \smooth_\augtraj \|\nabla \cost(\augtraj(\ctrls))\|_2 $ $a_2 = a_0
	\lip_\augtraj\|\nabla \cost(\augtraj(\ctrls))\|_2$. Therefore, denoting
	$\polysqrt(x) = 1+ \sqrt{1+1/x}$, condition~\eqref{eq:suff_cond_descent} is
	satisfied for any 
	\begin{align*}
	\reg\geq \reg(\ctrls) = \frac{a_1{+} \sqrt{a_1^2{+ }4a_2}}{2} & =
\frac{\smooth_\augtraj \|\nabla\cost(\augtraj(\ctrls))\|_2}{2}\polysqrt\left(\frac{\smooth_\augtraj^2\|\nabla \cost(\augtraj(\ctrls))\|_2}{4a_0 \lip_\augtraj}\right),
	\end{align*}
with $a_0 = \lip_{\augtraj}\smooth_\augtraj\smooth_\cost(\newsimp+1)$ for
	$\newsimp =  \smoothess_\cost \lip_\augtraj^2/(3 \smooth_\augtraj
	\smooth_\cost)$. We have then for $\auxctrls =
	\LQR_{\reg(\ctrls)}(\obj)(\ctrls)$, $\grad = \nabla \augtraj(\ctrls)$, $\hess
	= \nabla^2\cost(\augtraj(\ctrls))$, since
	condition~\eqref{eq:suff_cond_descent} is satisfied,
	\begin{align}
		\obj(\ctrls + \auxctrls) - \obj(\ctrls) & \leq -\frac{1}{2} \nabla \cost(\augtraj(\ctrls))^\top \grad^\top (\grad\hess\grad^\top + \reg(\ctrls) \idm)^{-1} \grad \nabla \cost(\augtraj(\ctrls)) \nonumber \\
		& = -\frac{1}{2} \nabla \cost(\augtraj(\ctrls))^\top  (\hess+ \reg(\ctrls)( \grad^\top \grad)^{-1})^{-1}  \nabla \cost(\augtraj(\ctrls)) \nonumber\\
		& \leq -\frac{1}{2} \frac{\sigma_\augtraj^2}{\sigma_\augtraj^2 \smooth_\cost + \reg(\ctrls)} \|\nabla \cost(\augtraj(\ctrls))\|_2^2 \leq - \frac{b_1x^2}{\sqrt{b_2 x^2 +b_3x}  + b_4x +b_5}, \label{eq:decrease_iter}
	\end{align}
	where $x = \|\nabla\cost(\augtraj(\ctrls))\|_2$, $b_1 = \sigma_\augtraj^2$,
	$b_2 = \smooth_\augtraj^2$, $b_3 = 4a_0 \lip_\augtraj$, $b_4 =
	\smooth_\augtraj$, $b_5 = 2 \sigma_{\augtraj}^2 \smooth_\cost$.
	
	The function $f_1: x \rightarrow b_1x^2/(\sqrt{b_2 x^2 +b_3x}  + b_4x +b_5)$
	is increasing for $x \geq 0$. Hence, denoting $\delta =
	\cost(\augtraj(\ctrls)) - \cost^* = \obj(\ctrls) - \obj^* $, we have
	$f_1(\|\nabla \cost(\augtraj(\ctrls))\|_2) \geq f_1((\plcst_\cost\delta)^\pl)$
	by  assumption~\eqref{eq:pl}. Denoting $\delta_k = \obj(\ctrls^{(k)}) -
	\obj^*$ for $k$ the iteration of the ILQR algorithm, we then have 
	$
	f_2'(\delta_k)(\delta_{k+1} - \delta_k) \leq - 1,
	$
	with  
	\begin{align*}
		f'_2(\delta)  = \frac{1}{f_1((\plcst_\cost \delta)^\pl)}
		& = \frac{2\condnb_\cost}{\delta^{2\pl}} + \frac{\scaling}{ \delta^\pl} +  \frac{\scaling\sqrt{\delta^{2\pl} + \simp\delta^\pl/\scaling}}{\delta^{2\pl}} 
		=  \frac{2\condnb_\cost}{\delta^{2\pl}}  + \frac{\scaling\polysqrt(\scaling\delta^r/\simp)}{\delta^r},
	\end{align*}
	with $\condnb_\cost = \smooth_\cost/\plcst_\cost^{2\pl}$, $\condnb_\augtraj =
	\lip_{\augtraj}/\sigma_{\augtraj}$, $\concord =
	\smoothess_\cost/(2\plcst_\cost^{3\pl})$, $\scaling = \smooth_\augtraj
	/(\sigma_\augtraj^2\plcst_\cost^\pl)$, $\simp = 4\condnb_\augtraj^2
	\condnb_\cost ( \newsimp + 1)$, $\newsimp =  \smoothess_\cost
	\lip_\augtraj^2/(3 \smooth_\augtraj \smooth_\cost).$

	Since $f_2$ is concave on $\reals^+$, we deduce that 
	$
	f_2(\delta_{k+1}) - f_2(\delta_k) \leq -1
	$ and so $f_2(\delta_k) \leq -k + f_2(\delta_0)$. Note that $f_2$ is strictly decreasing,  so we get that, for the algorithm to reach an accuracy $\varepsilon$, we need at most $k \leq f_2(\delta_0) -f_2(\varepsilon)$ iterations.  
	
	If $\pl=1/2$, one can verify that $\delta \rightarrow
	a\ln(2a\sqrt{\delta}\polysqrt(\sqrt{\delta}/a)+a^2)+2\sqrt{\delta}\polysqrt(\sqrt{\delta}/a)$
	is an antiderivative of $\delta\rightarrow
	\polysqrt(\sqrt{\delta}/a)/\sqrt{\delta}$ for any $a>0$. Hence, for $\pl=1/2$,
	the number of iterations to converge to an accuracy $\varepsilon$ is at most
	\begin{align*}
		k & \leq 
		2\condnb_\cost \ln\left(\frac{\delta_0}{\varepsilon}\right) {+ }
		2\scaling\left(\sqrt{\delta_0}\polysqrt\left(\frac{\scaling\sqrt{\delta_0}}{\simp}\right) {-} \sqrt{\varepsilon}\polysqrt\left(\frac{\scaling\sqrt{\varepsilon}}{\simp}\right)\right)
\\
	& \quad 	 {+} 
		\simp
		\ln\left(\frac{2\scaling\sqrt{\delta_0}\polysqrt(\scaling\sqrt{\delta_0}/\simp) {+} \simp}{2\scaling\sqrt{\varepsilon}\polysqrt(\scaling\sqrt{\varepsilon}/\simp) {+} \simp}\right)
		\\
		& \leq 2\condnb_\cost \ln\left(\frac{\delta_0}{\varepsilon}\right) + 2 \scaling \sqrt{\delta_0}\polysqrt\left(\frac{\scaling\sqrt{\delta_0}}{\simp}\right)
		+ \simp 
		\ln\left(1{+}  2\frac{\scaling\sqrt{\delta_0}}{\simp}\polysqrt\left(\frac{\scaling\sqrt{\delta_0}}{\simp}\right)\right).
	\end{align*}
 By using that $\ln(1+x) \leq x$ for $x > -1$, we get the claimed bound in this
 case. 
	
	If $1/2<\pl<1$, by integrating $f_2$,  the number of iterations  to converge
	to an accuracy $\varepsilon$ is at most
	\begin{align*}
		k \leq 
		\frac{2\condnb_\cost}{2\pl -1} \left(\frac{1}{\varepsilon^{2\pl-1}} - \frac{1}{\delta_0^{2\pl -1}}\right) 
		+ \frac{\scaling}{(1-\pl)}\left(\delta_0^{1-\pl} - \varepsilon^{1-\pl}\right) 
		+ \int_{\varepsilon}^{\delta_0} \frac{\scaling\sqrt{x^{2\pl} + \simp x^\pl/\scaling}}{ x^{2\pl}}dx.
	\end{align*}
	The bound follows in this case by using that, for $1/2<\pl<1$, and $a>0$,
	\begin{align*}
	\int_{\varepsilon}^{\delta_0} \frac{\sqrt{x^{2\pl} + a x^\pl}}{ x^{2\pl}}dx 
	& \leq \int_\varepsilon^{a^{1/\pl}} \frac{\sqrt{2a}}{x^{3\pl/2}}dx + \int_{a^{1/\pl}}^{\delta_0} \frac{1}{x^\pl} dx.
	\end{align*}
\end{proof}

\subsubsection{Local Convergence Rate to Minima}\label{ssec:local_conv} As we analyze the
ILQR algorithm locally as an approximate Newton method on the costs, we use the
notations and assumptions used to analyze  a Newton method. Namely, we assume
the costs $\cost_t$ to be strictly convex, and we define the norm induced by the
Hessian at a point $\states\in \reals^{\horizon\dimstate}$ and its dual norm as,
respectively, for $\auxstates \in \reals^{\horizon\dimstate}$, 
\[
\|\auxstates\|_\states = \sqrt{\auxstates^\top \nabla ^2\cost(\states) \auxstates}, \quad 
\|\auxstates\|_\states^* = \sqrt{\auxstates^\top \nabla ^2\cost(\states)^{-1} \auxstates}.
\]
For a matrix $A \in \reals^{\horizon\dimstate \times \horizon\dimctrl}$, we
denote $\|A\|_\states = \|\nabla ^2\cost(\states)^{1/2} A\|_2$ the norm induced
by the local geometry of $\cost$ w.r.t. the Euclidean norm. Finally, we denote
the Newton decrement of the cost function, as, for $\states \in
\reals^{\horizon\dimctrl}$, 
\[
\lambda_\cost(\states) = \sqrt{\nabla \cost(\states)^\top \nabla^2 \cost(\states)^{-1} \nabla \cost(\states)}.
\] 
To analyze the local convergence of the ILQR algorithm we consider the costs to
be self-concordant~\citep[Definition 5.1.1]{nesterov2018lectures}. In addition,
we consider smoothness properties of the function $\augtraj$ with respect to the
geometry induced by the Hessian of the costs as presented in the assumptions
below.  
\begin{assumption}\label{asm:self_concord} We consider that the costs $\cost_t$
	and so the total cost $\cost$ are  strictly convex and  the following
	constants, defined for $\augtraj: \ctrls\rightarrow \traj(\initstate, \ctrls)$
	with $\traj$ the control in $\horizon$ steps of the dynamic $\dyn$ defined
	in~\eqref{eq:traj}, are finite
	\begin{align*}
	\lip & =\sup_{\substack{\ctrls, \auxctrls \in \reals^{\horizon\dimctrl}\\ \auxctrls \neq 0}}
	\frac{\|\augtraj(\ctrls+\auxctrls) -\augtraj(\ctrls)\|_{\augtraj(\ctrls)}}{\|\auxctrls\|_2}, \quad   
	\smooth = \sup_{\substack{\ctrls, \auxctrls \in \reals^{\horizon\dimctrl} \\\auxctrls \neq 0}}
	\frac{\|\nabla \augtraj(\ctrls+\auxctrls)^\top -\nabla \augtraj(\ctrls)^\top\|_{\augtraj(\ctrls)}}{\|\auxctrls\|_2} \\
	\localconcord & = \sup_{\substack{\states, \auxstates_1, \auxstates_2, \auxstates_3 \in \reals^{\horizon\dimstate}
	\\\auxstates_1\neq 0, \auxstates_2 \neq 0, \auxstates_3 \neq 0}}
	\frac{|\nabla^3 \cost(\states)[\auxstates_1, \auxstates_2, \auxstates_3]|}
	{2\|\auxstates_1\|_\states \|\auxstates_2\|_\states\|\auxstates_3\|_\states}, \quad 	
	\sigma  = \inf_{\substack{\ctrls\in \reals^{\horizon\dimctrl}, \dualvars \in \reals^{\horizon\dimstate}\\ \dualvars\neq 0}} 
	\frac{\|\nabla \augtraj(\ctrls) \dualvars\|_2}{\|\dualvars\|_{\augtraj(\ctrls)}^*}.
	\end{align*}
In consequence, $\cost$ is $\localconcord$-self concordant~\cite[Definition
5.1.1, Lemma 5.1.2]{nesterov2018lectures}, and we have that $\sigma \leq
\sigma_{\min}(\nabla \augtraj(\ctrls) \nabla^2\cost(\augtraj(\ctrls))^{1/2})$, $
\sigma_{\max}(\nabla \augtraj(\ctrls) \nabla^2\cost(\augtraj(\ctrls))^{1/2})
\leq \lip$,  for any $\ctrls \in \reals^{\horizon\dimctrl}$.
\end{assumption}
In terms of the dynamic and the individual costs,
Assumption~\ref{asm:self_concord} is satisfied if $\cost_t$ is strongly convex
for all $t$ such that the total costs $\cost$ are strongly convex and if
Assumption~\ref{asm:conv} is also satisfied. In that case, we have
	\begin{align}\label{eq:constants_strgly_cvx}
		\lip \leq \sqrt{\smooth_\cost} \lip_\augtraj, \quad 
		\smooth \leq \sqrt{\smooth_\cost} \smooth_\augtraj, \quad 
		2\localconcord \leq {\smoothess_\cost}/{\strgcvx_\cost^{3/2}}, \quad
		\sigma \geq \sqrt{\strgcvx_\cost} \sigma_\augtraj.
	\end{align}
Given Assumption~\ref{asm:self_concord} and equipped with a stepsize
proportional to the Newton decrement, we can show a local quadratic convergence
rate of the ILQR algorithm.

\begin{theorem}\label{thm:conv_ggn} Given Assumption~\ref{asm:self_concord},
	consider the \ref{eq:ilqr_algo} algorithm for  problem~\eqref{eq:discrete_pb}
	with regularizations of the form $\reg_k = \regscaled
	\lambda_\cost(\augtraj(\ctrls^{(k)}))$ for some $\regscaled \geq 0$. For $
	k\geq 0$ such that
	\begin{equation}\label{eq:cond_reg_quad_conv}
		\lambda_\cost(\augtraj(\ctrls^{(k)}))	 
		<\lambdaquad =\frac{1}{\max\{4\localconcord + 3\localscaling + 2\regscaled /\sigma^{2}, 2\localcondnb \localconcord\}},
	\end{equation}
	where $\localcondnb  = \lip/\sigma$ and $\localscaling = \smooth/\sigma^2$, we
	have 
$
	\lambda_\cost(\augtraj(\ctrls^{(k+1)})) 
	\leq  \lambdaquad^{-1} \lambda_\cost(\augtraj(\ctrls^{(k)}))^2,
$
	and the ILQR algorithm converges quadratically to the minimum value of
	problem~\eqref{eq:discrete_pb}. 
\end{theorem}

\begin{remark}
If $\cost$ is a quadratic, such that the algorithm reduces to a Gauss-Newton
algorithm and $\localconcord= 0$, the radius of quadratic convergence reduces to
$
\lambdaquad =  1/(3\localscaling + 2\regscaled).
$
If in addition, no regularization is in effect, the radius of quadratic
convergence reduces to $\lambdaquad =1/{3\localscaling}$, which can be expressed
as $1/(3\scaling\sqrt{\condnb_\cost})$ if the total cost is  $\strgcvx_\cost$
strongly convex with $\scaling, \condnb_\cost$ defined as in
Theorem~\ref{thm:global_conv} and $\sigma, \smooth$ expressed using~\eqref{eq:constants_strgly_cvx}. So up to 3$\sqrt{\condnb_\cost}$, the
parameter $1/\scaling$ acts again as a radius of fast convergence as in
Theorem~\ref{thm:global_conv}.
\end{remark}
\begin{remark}
For better readability, we simplified the expression of the radius of
convergence.  A closer look at the proof shows that a non-zero regularization
may lead to a larger radius of convergence than no regularization. 
\end{remark}
\begin{proof}[Proof of Theorem~\ref{thm:conv_ggn}]	
	Let $\ctrls \in \reals^{\horizon\dimctrl}$ $\grad = \nabla \augtraj(\ctrls)$,
	$\hess = \nabla^2 \cost(\augtraj(\ctrls))$,  
	$\auxctrls = \LQR_\reg(\obj)(\ctrls)$ with $\reg = \regscaled
	\lambda_\cost(\augtraj(\ctrls))$. Assume that  
	\[\lambda_\cost(\augtraj(\ctrls))  \leq 1/\max\{
	 \sqrt{2\localconcord\localscaling}c_1  , 2\localconcord\localcondnb c_2,
	 2\localconcord c_2 \},\] where   
	$c_1 {=} \max\{1{-} \regscaled /(\sqrt{2\localconcord\smooth}\lip), 0\}$, $c_2
	{= }\max\{1{-}\regscaled/(2\lip^2\localconcord), 0\}$, $\localcondnb {=
	}\lip/\sigma$, $\localscaling {=} \smooth/\sigma^2$. We have 
	\begin{align}
		\lambda_\cost(\augtraj(\ctrls {+}\auxctrls))
		& {\leq} \underbrace{\|\nabla \cost(\augtraj(\ctrls{+}\auxctrls)) {-} \nabla \cost(\augtraj(\ctrls) {+} \grad^\top \auxctrls ) \|_{\augtraj(\ctrls{+}\auxctrls)}^*}_A 
		+ \underbrace{\| \nabla \cost(\augtraj(\ctrls) {+} \grad^\top \auxctrls ) \|_{\augtraj(\ctrls{+}\auxctrls)}^*}_B.  \label{eq:initial_split}
	\end{align}
	\paragraph{Bounding $A$ in~\eqref{eq:initial_split}}
	
	By definition of $\lip$ in Assumption~\ref{asm:self_concord} and
	Lemma~\ref{lem:bound_ggn_oracle_self_concord}, we have
	\begin{equation}\label{eq:bound_v}
	\|\augtraj(\ctrls+\auxctrls)- \augtraj(\ctrls)\|_{\augtraj(\ctrls)}\leq \lip \|\auxctrls\|_2, \qquad   \|\auxctrls\|_2 \leq \frac{\lip\lambda_\cost(\augtraj(\ctrls))}{\lip \sigma  +\regscaled \lambda_\cost(\augtraj(\ctrls))} .
	\end{equation}
	One easily verifies that $x/(1+ax) \leq c$ if $0\leq x \leq c/\max\{1-ca, 0\}$
	for any $a, c>0$. So for $\lambda_\cost(\augtraj(\ctrls))) \leq
	1/(2\localconcord\localcondnb c_2)$, we have $	\|\augtraj(\ctrls+\auxctrls)-
	\augtraj(\ctrls)\|_{\augtraj(\ctrls)}\leq 1/(2\localconcord)$. Hence, using
	that $\cost$ is $\localconcord$-self-concordant, Theorem 5.1.7 of
	\citet{nesterov2018lectures} applies and by using the definition of $\smooth$
	in Assumption~\ref{asm:self_concord}, we have 
	\begin{align*}
		\|\augtraj(\ctrls{+}\auxctrls) {-} \augtraj(\ctrls) {-} \grad^\top \auxctrls \|_{\augtraj(\ctrls{+}\auxctrls)} & 
		\leq \frac{1}{1{-}\localconcord 	\|\augtraj(\ctrls{+}\auxctrls){-} \augtraj(\ctrls)\|_{\augtraj(\ctrls)}}\|\augtraj(\ctrls{+}\auxctrls) {-} \augtraj(\ctrls) {-} \grad^\top \auxctrls \|_{\augtraj(\ctrls)} \\ 
		&\leq 2  \left\|\int_{0}^{1} \nabla \augtraj(\ctrls{+}t\auxctrls)^\top \auxctrls dt - \nabla\augtraj(\ctrls)^\top \auxctrls\right\|_{\augtraj(\ctrls)} = \smooth \|\auxctrls\|_2^2.
	\end{align*}
	Using~\eqref{eq:bound_v}, for $	\lambda_\cost(\augtraj(\ctrls)) \leq 1/(
	\sqrt{2\localconcord\localscaling}c_1)$, we get
	$
	{\|\augtraj(\ctrls{+}\auxctrls) {-} \augtraj(\ctrls) {-} \grad^\top \auxctrls \|_{\augtraj(\ctrls{+}\auxctrls)} {\leq} 1/(2\localconcord)}.
	$
	Since the total cost $\cost$ is $\localconcord$-self-concordant, we can then
	use Lemma~\ref{lem:smooth_self_concord}  to obtain
	\begin{align}
		A & \leq \frac{1}{1- \localconcord\|\augtraj(\ctrls+\auxctrls) - \augtraj(\ctrls) - \grad^\top \auxctrls \|_{\augtraj(\ctrls+\auxctrls)} }\|\augtraj(\ctrls+\auxctrls) - \augtraj(\ctrls) - \grad^\top \auxctrls \|_{\augtraj(\ctrls+\auxctrls)}  \nonumber
		\\
		& \leq \frac{2\smooth \lip^2 \lambda_\cost(\augtraj(\ctrls))^2 }{(\lip\sigma + \regscaled  \lambda_\cost(\augtraj(\ctrls)))^2}. \label{eq:bound_A}
	\end{align}
	
	\paragraph{Bounding B in~\eqref{eq:initial_split}} Recall that for
	$\lambda_\cost(\augtraj(\ctrls))) \leq  1/(2\localconcord\localcondnb c_2)$,
	we have $	\|\augtraj(\ctrls+\auxctrls)-
	\augtraj(\ctrls)\|_{\augtraj(\ctrls)}\leq 1/(2\localconcord)$. Since $\cost$
	is $\localconcord$-self-concordant, we have then~\cite[Theorem
	5.1.7]{nesterov2018lectures},
	\begin{align}
		B & \leq \frac{1}{1{-}\localconcord\|\augtraj(\ctrls{+}\auxctrls){-}\augtraj(\ctrls)\|_{\augtraj(\ctrls)}} \| \nabla \cost(\augtraj(\ctrls) {+} \grad^\top v ) \|_{\augtraj(\ctrls)}^* \leq 2\| \nabla \cost(\augtraj(\ctrls) {+} \grad^\top v ) \|_{\augtraj(\ctrls)}^*. \label{eq:bound_B_init}
	\end{align}
	Denote $\reg = \regscaled\lambda_\cost(\augtraj(\ctrls))$ and define $\newton
	=-(\hess + \reg( \grad^\top \grad)^{-1})^{-1} \nabla \cost(\augtraj(\ctrls))
	$. Using that
	\[
\auxctrls =-\grad(\grad^\top \grad)^{-1} (\hess +\reg( \grad^\top \grad)^{-1})^{-1} \nabla \cost(\augtraj(\ctrls)),
	\]
	and denoting $\states = \augtraj(\ctrls)$, we have then
	\begin{align}
		\| \nabla \cost(\augtraj(\ctrls) + \grad^\top \auxctrls ) \|_{\augtraj(\ctrls)}^* & = \|\nabla \cost(\states+\newton) - \nabla \cost(\states) - (\hess+ \reg( \grad^\top \grad)^{-1})\newton\|_{\states}^* \nonumber \\
		& \leq \|\nabla \cost(\states+\newton) - \nabla \cost(\states) - \hess \newton \|_{\states}^* + \reg\|(\grad^\top \grad)^{-1} \newton \|_{\states}^*.  \label{eq:B_split}
	\end{align}
	The first term can be bounded as in the proof of local convergence of a Newton
method~\cite[Theorem 5.2.2]{nesterov2018lectures}. Namely, we have
		\begin{align*}
			\|\nabla \cost(\states+\newton) - \nabla \cost(\states) - \hess \newton \|_{\states}^* = \|\int_{0}^{1}(\nabla^2 \cost(\states+t\newton)-\nabla^2\cost(\states))\newton dt\|_{\states}^*.
		\end{align*}
	Since $\sigma_{\max}(\nabla \augtraj(\ctrls)
	\nabla^2\cost(\augtraj(\ctrls))^{1/2})  \leq \lip $, we have
		\[
		\|\newton\|_{\states} = \|(\idm {+} \reg \hess^{-1/2}( \grad^\top \grad)^{-1} \hess^{-1/2})^{-1} \hess^{-1/2}\nabla \cost(\augtraj(\ctrls))\|_{2} 
		\leq  \frac{\lambda_\cost(\augtraj(\ctrls))	}{1+  \regscaled\lip^{-2}\lambda_\cost(\augtraj(\ctrls))}.
		\]
		So if $\lambda_\cost(\augtraj(\ctrls)) \leq 1 /(2\localconcord c_2)$, we get
		$\|\newton\|_\states \leq 1/(2\localconcord)$ and, since $\cost$ is
		self-concordant, by Corollary 5.1.5 of \citet{nesterov2018lectures}, we
		have, denoting $J = \int_{0}^{1}(\nabla^2
		\cost(\states+t\newton)-\nabla^2\cost(\states))dt$, 
		\[
		(-\|\newton\|_{\states}\localconcord + \|\newton\|_{\states}^2\localconcord^2/3 ) \nabla^2 \cost(\states)\preceq  J\preceq \frac{\|\newton\|_{\states}\localconcord}{1-\|\newton\|_{\states}\localconcord} \nabla^2 \cost(\states).
		\]
		Moreover, since  $\|\newton\|_\states  <1/(2\localconcord)$, we have
		$\|\newton\|_{\states}\localconcord-
		\|\newton\|_{\states}^2\localconcord^2/3  \leq
		\frac{\|\newton\|_{\states}\localconcord}{1-\|\newton\|_{\states}\localconcord}$.
		Hence, we get
	\begin{equation}\label{eq:bound_alanewton}
		\|\nabla \cost(\states+\newton) - \nabla \cost(\states) - \hess \newton \|_{\states}^* 
		\leq \frac{\|\newton\|_{\states}^2\localconcord}{1-\|\newton\|_{\states}\localconcord}
		\leq\frac{2\lambda_\cost(\augtraj(\ctrls))^2\localconcord}{(1+  \regscaled\lip^{-2}\lambda_\cost(\augtraj(\ctrls)))^2 }.
	\end{equation}
	On the other hand, since $\sigma \leq \sigma_{\min}(\nabla \augtraj(\ctrls)
	\nabla^2\cost(\augtraj(\ctrls))^{1/2})$,  we have 
	\begin{equation}\label{eq:bound_newton_step}
\|(\grad^\top \grad)^{-1} \newton\|_{\augtraj(\ctrls)}^* 
= \|(\hess^{1/2}\grad^\top \grad \hess^{1/2} +\reg \idm)^{-1} \hess^{-1/2} \nabla \cost(\augtraj(\ctrls))\|_{2}
\leq  
\frac{	\lambda_\cost(\augtraj(\ctrls))}{\sigma^2 + \regscaled  \lambda_\cost(\augtraj(\ctrls))}  .
	\end{equation}
	So combining~\eqref{eq:bound_newton_step} and~\eqref{eq:bound_alanewton}
	into~\eqref{eq:B_split} and then~\eqref{eq:bound_B_init} we get
	\begin{equation}\label{eq:bound_B}
		B \leq 2\left(\frac{2\localconcord}{(1+ \regscaled\lip^{-2}  \lambda_\cost(\augtraj(\ctrls)))^2} + \frac{\regscaled}{\sigma^2 + \regscaled  \lambda_\cost(\augtraj(\ctrls))}\right) \lambda_\cost(\augtraj(\ctrls))^2.
	\end{equation}
	\paragraph{Local quadratic convergence rate}
	By combining~\eqref{eq:bound_A} and~\eqref{eq:bound_B}
	into~\eqref{eq:initial_split}, we get, as long as
	$\lambda_\cost(\augtraj(\ctrls))  \leq 1/\max\{
	\sqrt{2\localconcord\localscaling}c_1  , 2\localconcord\localcondnb c_2,
	2\localconcord c_2 \}$,
	\begin{align}\nonumber
		\lambda_\cost(\augtraj(\ctrls{+}\auxctrls))&  \leq \left( 
		\frac{2\smooth \lip^2 }{(\lip\sigma {+} \regscaled  \lambda_\cost(\augtraj(\ctrls)))^2}
		{+} \frac{4\localconcord}{(1{+} \regscaled \lip^{-2} \lambda_\cost(\augtraj(\ctrls)))^2} {+} \frac{2\regscaled}{\sigma^2 {+} \regscaled  \lambda_\cost(\augtraj(\ctrls))}
		\right) \lambda_\cost(\augtraj(\ctrls))^2.
	\end{align}
	Note that $c_1, c_2 \leq 1$ and that  $2\localscaling + 4\localconcord+
	2\regscaled /\sigma^{2} \geq \max\{2\localconcord,
	\sqrt{2\localconcord\localscaling}  \}$, using the arithmetic-geometric mean
	inequality.  
	Hence, for  
	\[
	\lambda_\cost(\augtraj(\ctrls))<\lambdaquad =1/ \max\{2\localscaling + 4\localconcord+ 2\regscaled/\sigma^{2},  2\localconcord\localcondnb\},
	\]
	we get
$
		\lambda_\cost(\augtraj(\ctrls+\auxctrls))\leq \bar\lambda^{-1} \lambda_\cost(\augtraj(\ctrls))^2 <\lambda_\cost(\augtraj(\ctrls)),
$
	that is, we reach the region of quadratic convergence for $\augtraj(\ctrls)$.
\end{proof}

\subsubsection{Total Complexity}\label{ssec:overall_conv} Given Assumption~\ref{asm:conv}, if the total
cost  is strongly convex then it satisfies the condition of
Theorem~\ref{thm:global_conv} and Assumption~\ref{asm:self_concord} is satisfied
with the estimates given in~\eqref{eq:constants_strgly_cvx}. We can then bound
the number of iterations to local quadratic convergence and obtain the total
complexity bound in this case. The following theorem is the detailed version of
Theorem~\ref{thm:intro}.

\begin{theorem}\label{thm:global_local_conv} Consider the costs $\cost_t$ to be
	$\strgcvx_\cost$-strongly convex and Assumption~\ref{asm:conv} to be
	satisfied. Then condition~\eqref{eq:suff_cond_descent} is satisfied for a
	regularization
	\[
	\reg(\ctrls)=\left(1+\frac{\simp}{2(1 + \scaling \|\nabla \cost(\augtraj(\ctrls))\|_2/(\sqrt{\strgcvx_\cost} \condnb_\augtraj))}\right)\smooth_\augtraj \|\nabla \cost(\augtraj(\ctrls))\|_2
	\] 
	With such regularization, the number of iterations of the \ref{eq:ilqr_algo}
	algorithm to reach an accuracy $\varepsilon$ is at most
	\begin{align}\label{eq:main_cplxity}
		 k(\delta_0, \varepsilon) & = 
		   4\scaling( \sqrt{\delta_0}  - \sqrt{\varepsilon}) 
		   + 2\condnb_\cost \ln\left(\frac{\delta_0}{ \varepsilon}\right) 
		+ 2 \simp \ln\left(\frac{\scaling\sqrt{\delta_0} + \condnb_\augtraj}{\scaling\sqrt{\varepsilon} + \condnb_\augtraj}\right),
	\end{align}
	where $\condnb_\cost = \smooth_\cost/\plcst_\cost$, $\condnb_\augtraj =
	\lip_{\augtraj}/\sigma_{\augtraj}$, $\scaling =
	\smooth_\augtraj/(\sigma_\augtraj^2\sqrt{\plcst_\cost})$, $\concord =
	\smoothess_\cost/(2\plcst_\cost^{3/2})$, $\simp = 4\condnb_\augtraj^2
	\condnb_\cost ( \newsimp + 1)$, $\newsimp =  \smoothess_\cost
	\lip_\augtraj^2/(3 \smooth_\augtraj \smooth_\cost)$, and $\delta_0 =
	\obj(\ctrls^{(0)})-\obj^*$. 
	
	If in addition the target accuracy $\varepsilon$ is smaller than $ \deltaquad
	= 1/(32 \condnb_\cost (\concord(1+ \sqrt{\condnb_\cost}\condnb_\augtraj^3/3) +
	\sqrt{\condnb_\cost} \scaling (1+ \condnb_\augtraj\condnb_\cost))^2)$ which
	determines a quadratic convergence phase,  the number of iterations of an ILQR
	algorithm to reach the accuracy $\varepsilon$ is at most  $k(\delta_0,
	\deltaquad) + O(\ln\ln(\varepsilon^{-1}))$.
	
	The total computational complexity of the algorithm in terms of basic
	operations is then of the order of $(k(\delta_0, \deltaquad) +
	O(\ln\ln(\varepsilon^{-1}))) \times \mathcal{C}(\dimstate, \dimctrl,
	\horizon)$ with  $\mathcal{C}(\dimstate, \dimctrl, \horizon)$ defined
	in~\eqref{eq:time_cost}.
\end{theorem}
\begin{remark}
	The rate of convergence can now be separated between three phases, (i) the
	number of iterations to reach some linear convergence determined by the first
	term in the complexity bound~\eqref{eq:main_cplxity}, (ii) the number of
	iterations to reach the quadratic convergence rate that is captured by the
	logarithmic terms in the complexity bound~\eqref{eq:main_cplxity}, (iii) the
	quadratic convergence phase once $\delta_k$ is smaller than the gap of local
	quadratic convergence $\deltaquad$.
\end{remark}
\begin{proof}[Proof of Theorem~\ref{thm:global_local_conv}]
	By using the strong convexity of the costs $\cost$, we can refine the choice
	of the regularization to ensure~\eqref{eq:suff_cond_descent}. The validity of
	the proposed regularization to ensure condition~\eqref{eq:suff_cond_descent}
	is shown in Lemma~\ref{lem:reg_strg_cvx} in Appendix~\ref{app:lemmas}. With
	the proposed regularization, Lemma~\ref{lem:detailed_cplxity_strgcvx}  in
	Appendix~\ref{app:lemmas} shows, following the same reasoning as in the proof
	of Theorem~\ref{thm:global_conv}, that the number of iterations of the ILQR
	algorithm needed to reach an accuracy $\varepsilon$ is bounded by 
	\begin{align}\label{eq:global_strg_cvx_iter}
	k &	\leq   
	2\condnb_\cost \ln\left(\frac{\delta_0}{\varepsilon}\right) 
	+ 4\scaling \left(\sqrt{\delta_0} - \sqrt{\varepsilon }\right) 
	+ 2 \simp \ln\left(\frac{\scaling\sqrt{\delta_0} + \condnb_\augtraj}{\scaling\sqrt{\varepsilon} + \condnb_\augtraj}\right),
	\end{align}
	with $\condnb_\cost$, $\condnb_\augtraj$, $\concord$, $\scaling$, $\simp$
	defined as in Theorem~\ref{thm:global_conv}.

		For the local convergence,  the constants in Theorem~\ref{thm:conv_ggn} can
		be expressed in terms of the constants in Theorem~\ref{thm:global_conv} as   
		$\sigma = \sqrt{\strgcvx_\cost}\sigma_{\augtraj}, \localconcord = \concord,
		\localscaling =  \sqrt{\condnb_\cost} \scaling,  \localcondnb =
		\sqrt{\condnb_\cost} \condnb_\augtraj $. From the proof of
		Theorem~\ref{thm:conv_ggn},  if $\lambda_\cost(\augtraj(\ctrls^{(k)}))  \leq
		1/\max\{ \sqrt{2\localconcord \localscaling}  ,  2\localconcord\localcondnb
		,  2\localconcord\}$, then
		 	\begin{align}\nonumber
		 	\lambda_\cost(\augtraj(\ctrls^{(k+1)}))&  \leq \left( 
		 	2 \localscaling
		 	+ 4\localconcord + \frac{2\regscaled_k}{\sigma^2}
		 	\right) \lambda_\cost(\augtraj(\ctrls^{(k)}))^2,
		 \end{align}
		 where  $\regscaled_k = \reg(\ctrls^{(k)}) /
		 \lambda_\cost(\augtraj(\ctrls^{(k)})) \leq \sqrt{\smooth_\cost}
		 (\smooth_\augtraj  + 2 \lip_{\augtraj}( \smoothess_\cost
		 \lip_{\augtraj}^2/3 {+} \smooth_\augtraj\smooth_\cost )  /(\sigma_\augtraj
		 \strgcvx_\cost))$. Define then
		 \[
		 \lambdaquad = \frac{1}{4(\concord(1+ \sqrt{\condnb_\cost}\condnb_\augtraj^3/3) + \sqrt{\condnb_\cost} \scaling (1+ \condnb_\augtraj\condnb_\cost))}.
		 \]
		 We have that $\lambdaquad \leq1/\max\{ \sqrt{2\localconcord \localscaling}
		 ,  2\localconcord\localcondnb ,  2\localconcord\}$. So, if
		 $\lambda_\cost(\augtraj(\ctrls^{(k)}))\leq \lambdaquad$, quadratic
		 convergence is ensured. 
		 
		 It remains to link the objective gap to the Newton decrement. By
		  considering a gradient step with step-size $1/\smooth_\cost$, we have
		  $\|\nabla \cost(\states) \|^2 \leq 2 \smooth_\cost (\cost(\states)
		  -\cost^*)$ for any $\states$, hence $\lambda_\cost(\states) \leq
		  \sqrt{2\condnb_\cost  (\cost(\states) -\cost^*)}$. So, the number of
		  iterations to reach quadratic convergence is bounded by the number of
		  iterations to get an accuracy $\deltaquad =\lambdaquad^2
		  /(2\condnb_\cost)$. Once quadratic convergence is reached the remaining
		  number of iterations is of the order of $O(\ln\ln\varepsilon^{-1})$. 
	 \end{proof}
 Theorem~\ref{thm:global_local_conv} presents an ideal implementation of the
 ILQR algorithm given the knowledge of all constants to define the
 regularizations. This ideal implementation informs us on an appropriate line
 search strategy for the regularization, namely searching over $\bar \reg$ for
 regularizations of the form $\reg_k =\regscaled \|\nabla
 \cost(\augtraj(\ctrls^{(k)}))\|_2$.  We present in
 Algorithm~\ref{algo:RGGN_linesearch} an implementation of the ILQR algorithm with
 an adequate line-search procedure that is guaranteed to terminate and maintain
 the complexity bounds presented in Theorem~\ref{thm:global_local_conv} as
 formally stated in Corollary~\ref{cor:line_search}, whose proof is given in
 Appendix~\ref{app:lemmas}.
 
 \begin{restatable}{corollary}{linesearch}\label{cor:line_search} Consider the
 	assumptions  and notations of Theorem~\ref{thm:global_local_conv} on
 	problem~\eqref{eq:discrete_pb} and Algorithm~\ref{algo:RGGN_linesearch} with an
 	initial scaled regularization guess $\bar \reg_{-1} \leq \left(1+ {\simp}/{(2
 	+ 2\scaling \sqrt{\delta_0}/\condnb_\augtraj)}\right)\smooth_\augtraj$. 
 	The total number of calls to ILQR oracles  of Algorithm~\ref{algo:RGGN_linesearch}
 	to reach an accuracy~$\varepsilon$ is at most
 	$
 	2 k(\delta_0, \deltaquad') + \ln\ln(\varepsilon^{-1}) + \left\lceil\log_2\left( (1+\simp/2) \smooth_\augtraj /\bar \reg_{-1}\right)\right\rceil,
 	$
 	where $k(\delta_0, \deltaquad')$ is defined as in
 	Theorem~\ref{thm:global_local_conv} and $\deltaquad' =  1/(32 \condnb_\cost
 	(\concord(1+ 2\sqrt{\condnb_\cost}\condnb_\augtraj^3/3) + \sqrt{\condnb_\cost}
 	\scaling (1+ 2 \condnb_\augtraj\condnb_\cost))^2)$ is a gap of quadratic
 	convergence for Algorithm~\ref{algo:RGGN_linesearch}.
 \end{restatable}

\begin{algorithm}\caption{ILQR with Line-Search \label{algo:RGGN_linesearch}}
	\begin{algorithmic}
	\State {\bf Inputs:} Initial point $\ctrls^{(0)}$, initial scaled regularization
	$\bar \reg_{-1} > 0$, costs and dynamics summarized as $\cost$ and $\augtraj$
	as in~\eqref{eq:total_cost}, $\LQR_{\reg}(\obj)$ oracle for $\obj = \cost
	\circ \augtraj$. 

	\For{$k=0, \ldots $} \State Set $\bar \reg_k = \bar \reg_{k-1}$, $\reg_k =
	\bar \reg_k \|\nabla \cost(\augtraj(\ctrls^{(k)}))\|_2$
	
	\State Compute $ \ctrls^{(k+1)} = \ctrls^{(k)} +
	\LQR_{\reg_k}(\obj)(\ctrls^{(k)})$ \While{ $\obj\left(\ctrls^{(k+1)}\right) >
	\obj(\ctrls^{(k)}) + \nabla \obj(\ctrls)^\top(\ctrls^{(k+1)}-\ctrls^{(k)} )/2$
	} \State Set $\bar \reg_k \leftarrow 2 \bar \reg_k$, $\reg_k \leftarrow  \bar
	\reg_k \|\nabla \cost(\augtraj(\ctrls^{(k)}))\|_2 $
	
	\State Set $ \ctrls^{(k+1)} \leftarrow \ctrls^{(k)} +
	\LQR_{\reg_k}(\obj)(\ctrls^{(k)})$ \EndWhile \EndFor
	\end{algorithmic}
\end{algorithm}

\subsection{Convergence Analysis of IDDP}\label{ssec:ddp_conv} The IDDP
algorithm departs from the implementation of usual optimization algorithms for
compositional problems as it cannot be formulated as the minimization of an
approximation of the objective but rather as an approximate minimization of the
objective by dynamic programming, see e.g. \citet{roulet2021techreport} for a
detailed overview. Its analysis can nevertheless be pursued by analogy of its
implementation with the ILQR algorithm. Namely, the technical
Lemmas~\ref{lem:approx_ilqrs} and~\ref{lem:bounded_policies_restricted_case} in
Appendix~\ref{app:conv_iddp} decompose the implementation of the IDDP algorithm
into the dynamical structure of the problem to quantify an approximation bound
between the oracles returned by the ILQR and IDDP algorithm of the form
$\|\DDP_\reg(\obj)(\ctrls) - \LQR_{\reg}(\obj)(\ctrls)\|_2 \leq
\ddpbound\|\LQR_\reg(\obj)(\ctrls)\|_2^2$ for some constant $\ddpbound$
independent of $\ctrls$ and $\reg$, provided that the costs are strongly convex.

Equipped with this approximation bound, we consider selecting the
regularization of the IDDP algorithm  such that 
\begin{equation}\label{eq:ddp_suff_cond}
\obj(\ctrls + \DDP_\reg(\obj) (\ctrls))\leq  \obj(\ctrls)  + \frac{1}{2}\nabla \obj(\ctrls) ^\top \LQR_\reg(\obj)(\ctrls),
\end{equation}
i.e., we use the same criterion as for the ILQR
algorithm~\eqref{eq:suff_cond_descent} to ensure a sufficient decrease. This
choice of regularization is motivated by the  implementation of the ILQR and
IDDP algorithms which both compute $ \frac{1}{2}\nabla \obj(\ctrls) ^\top
\LQR_\reg(\obj)(\ctrls)$ by dynamic programming;
see~\citet{roulet2021techreport} for more details. For strongly convex costs,
the rule provided in~\eqref{eq:ddp_suff_cond} to select the stepsize together
with the quadratic approximation bound between the oracles of the ILQR and IDDP
algorithm enable us to state a global convergence result for the IDDP algorithm.

\begin{theorem}\label{thm:ddp_conv} Consider the costs to be
	 $\strgcvx_\cost$-strongly convex and Assumption~\ref{asm:conv} to be
	 satisfied. Then the constant $\ddpbound {=} \sup_{\ctrls \in
	 \reals^{\horizon\dimctrl}, \reg>0} \|\DDP_\reg(\obj)(\ctrls) -
	 \LQR_{\reg}(\obj)(\ctrls)\|_2 /\|\LQR_\reg(\obj)(\ctrls)\|_2^2$ is finite.
	 Condition~\eqref{eq:ddp_suff_cond} is satisfied for a regularization
	 \[
	 \reg(\ctrls) = \smooth_\augtraj\cst\|\nabla \cost(\augtraj(\ctrls))\|_2 + \condnb_\cost\sigma_\augtraj^2 \scaling^2 \condnbddp^2 \|\nabla \cost(\augtraj(\ctrls))\|_2^2,
	 \]
	 where $\cst = (1+\condnb_\cost \condnb_\augtraj) (1 + 2\condnbddp) +
	 \condnb_\augtraj^3(2\concord)/(3\scaling)$,  $\condnbddp = \lip_\augtraj
	 \ddpbound/\smooth_\augtraj$ and $\condnb_\cost$, $\condnb_\augtraj$,
	 $\concord$, $\scaling$ are defined in Theorem~\ref{thm:global_local_conv}. 
	 
	 With such regularization, the number of iterations of the \ref{eq:iddp_algo}
	 algorithm to reach an accuracy $\varepsilon$ is at most
	 \[
	 k \leq 2\condnb_\cost \ln\left(\frac{\delta_0}{\deltaquad}\right)+ 4 \scaling\cst( \sqrt{\delta_0} - \sqrt{\deltaquad}) + 2\condnb_\cost \scaling^2 \condnbddp^2 (\delta_0- \deltaquad) + O(\ln\ln(\varepsilon^{-1})),
	 \]
	 where $\deltaquad = 1/(32\condnb_\cost(\scaling\sqrt{\condnb_\cost}(2+ 2 \cst
	 + \sqrt{\condnb_\cost}\condnbddp) + 4\concord)^2)$ is  the value of the gap
	 determining the quadratic convergence phase.
\end{theorem}
\begin{remark}
	The complexity bounds for the IDDP algorithm in Theorem~\ref{thm:ddp_conv}
	take then the same form as the complexity bounds obtained for the ILQR
	algorithm in Theorem~\ref{thm:global_local_conv} up to some additional
	multiplicative factors. Our proof is built on considering IDDP to approximate
	ILQR. In practice, IDDP appears more efficient than ILQR as illustrated in
	Figure~\ref{fig:conv} and other works~\citep{liao1992advantages,
	roulet2021techreport} and alternative proofs may better explain this
	phenomenon. On the other hand, our implementation and analysis provide
	theoretical convergence guarantees.
\end{remark}

\begin{proof}[Proof of Theorem~\ref{thm:ddp_conv}]
	We sketch the proof of the first part of the claim, whose technical details
	can be found in Lemma~\ref{lem:approx_ilqrs}. Given a command $\ctrls=
	(\ctrl_0;\ldots;\ctrl_{\horizon-1})$ with associated trajectory $\states =
	\augtraj(\ctrls) = (\state_1;\ldots;\state_\horizon)$, denote $\pi_t:
	\auxstate_t \rightarrow K_t \auxstate_t + k_t$ for $t\in \{0, \ldots,
	\horizon-1\}$ the affine policies computed in Algorithm~\ref{algo:lqr_ddp}  and
	define for $\auxstates = (\auxstate_1;\ldots;\auxstate_\horizon)$,
	$\pi(\auxstates) = (\pi_0(0);
	\pi_1(\auxstate_1);\ldots;\pi_{\horizon-1}(\auxstate_{\horizon-1}))$. Denoting
	then $ \auxctrls =	\LQR_\reg(\obj)(\ctrls)$, $\auxxctrls =
	\DDP_\reg(\obj)(\ctrls) $,  we have, after close inspection of the roll-outs, 
	\[
	\auxctrls =\pi(\nabla \augtraj(\ctrls)^\top \auxctrls), \qquad \auxxctrls = \pi( \augtraj(\ctrls + \auxxctrls) - \augtraj(\ctrls)).
	\]
	
	Denoting, for $e_i$ the i\textsuperscript{th} canonical vector in
	$\reals^\horizon$, $K = \sum_{i=2}^\horizon e_ie_{i-1}^\top \otimes K_{i-1}
	\in \reals^{\horizon\dimctrl \times \horizon\dimstate}$, $k = (k_0;\ldots;
	k_{\horizon-1})$, $\grad = \nabla \augtraj(\ctrls)$, we get that $\auxctrls =
	k + K \grad^\top \auxctrls$. Since $\grad^\top$ is lower block triangular and
	$K$ is strictly lower block triangular, $K \grad^\top$ is strictly lower block
	triangular and so $\idm- K \grad^\top$ is invertible. Therefore, we can
	express the LQR oracle as $\auxctrls = (\idm- K \grad^\top)^{-1} k$. For the
	IDDP oracle, a similar expression can be found by using the mean value theorem
	as formally shown in Lemma~\ref{lem:approx_ilqrs}. Informally, there exists a
	matrix $D$ which can be thought as $\nabla \augtraj(\ctrls+ \zeta)$ for some
	$\|\zeta\|_2 \leq \|\auxxctrls\|_2$ such that $\auxxctrls = (\idm- K
	D^\top)^{-1} k$. The difference $\auxctrls-\auxxctrls$ can be bounded by $c_0
	\|k\|_2 \|C^\top-D^\top\|_2$ for some constant $c_0$  and $\|C^\top-D^\top\|_2$
	can be bounded as $c_1 \|\auxxctrls\|_2$ such that we get in total a quadratic
	error bound in $\|k\|_2$ which can  be converted in a quadratic bound in terms
	of $\|\auxctrls\|_2$.
	
	For $\ctrls \in\real^{\horizon\dimctrl}$ denote  $ \auxctrls =
	\LQR_\reg(\obj)(\ctrls)$, $\auxxctrls = \DDP_\reg(\obj)(\ctrls) $. By
	definition of $\auxctrls$, condition~\eqref{eq:ddp_suff_cond} is satisfied if 
	\[
	\obj(\ctrls+\auxxctrls)  \leq \obj(\ctrls) 
	+ \qua_\cost^{\augtraj(\ctrls)}\circ\lin_\augtraj^\ctrls(\auxctrls) + \frac{\reg}{2} \|\auxctrls\|_2^2.
	\]
	We proceed by first observing that, by Lipschitz continuity of the gradients
	of $\cost$,
\[
	\obj(\ctrls{+}\auxxctrls) 
	{-} \obj(\ctrls{+}\auxctrls)  \leq 
	\nabla \cost(\augtraj(\ctrls {+} \auxctrls))^\top (\augtraj(\ctrls {+} \auxxctrls) {- }\augtraj(\ctrls {+} \auxctrls ) )
	+ \smooth_\cost\|\augtraj(\ctrls {+} \auxxctrls) {-} \augtraj(\ctrls {+} \auxctrls)\|_2^2/2, 
\]
and
$
\|\nabla  \cost(\augtraj(\ctrls{+}\auxctrls)) \|_2 \leq \|\nabla \cost( \augtraj(\ctrls))\|_2 + \smooth_\cost\| \augtraj(\ctrls {+} \auxctrls) {-} \augtraj(\ctrls)\|_2.
$
Hence,  using the Lipschitz continuity of $\augtraj$ and  the definition of
$\ddpbound$, we have 
\[
	\obj(\ctrls{+}\auxxctrls) 
	{-} \obj(\ctrls{+}\auxctrls) \leq  (\|\nabla \cost(\augtraj(\ctrls))\|_2 
	+ \smooth_\cost\lip_{\augtraj}\|\auxctrls\|_2)\lip_\augtraj \ddpbound \|\auxctrls\|_2^2
	+  \smooth_\cost\lip_\augtraj^2 \ddpbound^2 \|\auxctrls\|_2^4/2.
\]
On the other hand, the term $\obj(\ctrls{+}\auxctrls) {-} \obj(\ctrls)  
{-}\qua_\cost^{\augtraj(\ctrls)}\circ\lin_\augtraj^\ctrls(\auxctrls)$ can be
bounded using Lemma~\ref{lem:bound_approx_self_concord}. Hence, using that
$\|\auxctrls\|_2 \leq \|\nabla \cost(\augtraj(\ctrls))\|_2
/(\strgcvx_\cost\sigma_\augtraj)$ (see the first paragraph of the proof of
Theorem~\ref{thm:global_local_conv}), we get that
condition~\eqref{eq:ddp_suff_cond} is satisfied for 
\[
\reg(\ctrls) = \smooth_\augtraj\cst\|\nabla \cost(\augtraj(\ctrls))\|_2 + \condnb_\cost\sigma_\augtraj^2 \scaling^2 \condnbddp^2 \|\nabla \cost(\augtraj(\ctrls))\|_2^2,
\]
for $\cst = (1+\condnb_\cost \condnb_\augtraj) (1 + 2\condnbddp) +
\condnb_\augtraj^3(2\concord)/(3\scaling)$,  $\condnbddp = \lip_\augtraj
\ddpbound/\smooth_\augtraj$, where $\condnb_\cost$, $\condnb_\augtraj$,
$\concord$, $\scaling$ are defined in Theorem~\ref{thm:global_local_conv}. 

With such regularization choice, the convergence of the IDDP method follows from
the proof of Theorem~\ref{thm:global_conv} by using that
condition~\eqref{eq:ddp_suff_cond} is satisfied. Namely, we get that the number
of iterations of an IDDP algorithm with regularizations $\reg_k =
\reg(\ctrls^{(k)})$ to ensure an objective less than $\varepsilon$ is at most
(see Appendix~\ref{app:comput} for the detailed derivation)
\begin{equation}\label{eq:conv_ddp_iter}
k  \leq 2\condnb_\cost \ln\left(\delta_0/\varepsilon\right)+ 4 \scaling\cst( \sqrt{\delta_0} - \sqrt{\varepsilon}) + 2\condnb_\cost \scaling^2 \condnbddp^2 (\delta_0- \varepsilon).
\end{equation}

For the local convergence, define $\lip, \sigma, \smooth, \localconcord,
\localscaling$ as in the proof of Theorem~\ref{thm:global_local_conv}. We have 
\[
\lambda_\cost(\augtraj(\ctrls+ \auxxctrls)) \leq 
\|\nabla \cost(\augtraj(\ctrls+ \auxxctrls)) - \nabla \cost(\augtraj(\ctrls+ \auxctrls))\|_{\augtraj(\ctrls+ \auxxctrls)}^* 
+ \|\nabla \cost(\augtraj(\ctrls + \auxctrls))\|_{\augtraj(\ctrls+\auxxctrls)}^*.
\]
If $\lambda_\cost(\augtraj(\ctrls)) \leq \sigma / \sqrt{2\localconcord \lip
\ddpbound}$, then, 
\[
\|\augtraj(\ctrls+ \auxxctrls) - \augtraj(\ctrls+\auxctrls)\|_{\augtraj(\ctrls+\auxxctrls)} 
\leq \lip \|\auxctrls- \auxxctrls\|_2 
\leq \lip\ddpbound \|\auxctrls\|_2^2 
\leq \lip\ddpbound \lambda_{\cost}(\augtraj(\ctrls))^2/\sigma^2 
\leq 1/(2\localconcord),
\] 
where we used that $\|\auxctrls\|_2 \leq \lambda_\cost(\augtraj(\ctrls))/\sigma$
as shown in the second paragraph of the proof of Theorem~\ref{thm:conv_ggn}.
Hence, using Lemma~\ref{lem:smooth_self_concord}, we have that $\|\nabla
\cost(\augtraj(\ctrls+ \auxxctrls)) - \nabla \cost(\augtraj(\ctrls+
\auxctrls))\|_{\augtraj(\ctrls+ \auxxctrls)}^* \leq 2 \lip \ddpbound
\lambda_\cost(\augtraj(\ctrls))^2/\sigma^2$ and using Theorem 5.1.7
of~\citet{nesterov2018lectures}, we have that $\|\nabla \cost(\augtraj(\ctrls +
\auxctrls))\|_{\augtraj(\ctrls+\auxxctrls)}^* \leq 2 \|\nabla
\cost(\augtraj(\ctrls+\auxctrls))\|_{\augtraj(\ctrls+\auxctrls)}^*$. We conclude
that if  $\lambda_\cost(\augtraj(\ctrls)) \leq 1 / \sqrt{2\localconcord
\localscaling\condnbddp}$, 
\[
\lambda_\cost(\augtraj(\ctrls+ \auxxctrls)) \leq 2\condnbddp \localscaling\lambda_\cost(\augtraj(\ctrls))^2 + 2 \lambda_\cost(\augtraj(\ctrls+\auxctrls)).
\]
Hence, using the bound derived in Theorem~\ref{thm:conv_ggn} for
$\lambda_\cost(\augtraj(\ctrls+\auxctrls))$, we conclude that for  
\[
\lambda_\cost(\augtraj(\ctrls)) \leq 
1/\max\{ \sqrt{2\localconcord \localscaling}  , \sqrt{2\localconcord \localscaling\condnbddp},  2\localcondnb\localconcord ,   2\localconcord\},
\]
we have that
\begin{align*}
\lambda_\cost(\augtraj(\ctrls+\auxxctrls)) & 
\leq \left(2(2+\condnbddp) \localscaling + 8\localconcord + 4 \regscaled \sigma^{-2}\right)\lambda_\cost(\augtraj(\ctrls))^2\\
& \leq \left(2\scaling \sqrt{\condnb_\cost} (2+ 2 \cst + \condnbddp)  + 8\concord + 4\condnb_\cost^2\condnbddp^2 \scaling^2 \lambda_\cost(\augtraj(\ctrls))\right)\lambda_\cost(\augtraj(\ctrls))^2,
\end{align*}
where  we used that $\regscaled
=\reg(\ctrls)/\lambda_\cost(\augtraj(\ctrls))\leq
\smooth_\augtraj\sqrt{\smooth_\cost} \cst +\smooth_\cost \condnb_\cost
\sigma_\augtraj^2 \scaling^2 \condnbddp^2 \lambda_\cost(\augtraj(\ctrls))$.
Denote  
\[
\lambdaquad = 1/(4(\scaling\sqrt{\condnb_\cost}(2+ 2 \cst + \sqrt{\condnb_\cost}\condnbddp) + 4\concord)),
\]
s.t.  $\lambdaquad \leq 1/\max\{ \sqrt{2\localconcord \localscaling}  ,
\sqrt{2\localconcord \localscaling\condnbddp},  2\localcondnb\localconcord ,
2\localconcord\}$. For $\lambda_\cost(\augtraj(\ctrls)) <\lambdaquad$, quadratic
convergence is ensured, i.e., $\lambda_\cost(\augtraj(\ctrls+\auxxctrls)) \leq
\lambdaquad^{-1} \lambda_\cost(\augtraj(\ctrls)) ^2 <
\lambda_\cost(\augtraj(\ctrls)) $. The conclusion follows as in the proof of
Theorem~\ref{thm:global_local_conv}.
\end{proof}

\section{Numerical Evaluations}\label{sec:exp}
We illustrate numerically the theoretical findings to examine their relevance.
In all experiments, we implemented gradient descent (GD), ILQR, IDDP, with a
line-search on either the stepsize for GD or the scaled regularization for ILQR
and IDDP as in Algorithm~\ref{algo:RGGN_linesearch}. The algorithms are run at
double precision. They stop if (i) the norm of the gradient of the objective is
smaller than $10^{-16}$, (ii) the linesearch does not find a valid stepize
bigger than $10^{-24}$, (iii) the relative change in costs ($|c_k -
c_{k-1}|/|c_k|$) is smaller than $10^{-24}$. The code is publicly
available at {\small \coderef}.
A tutorial notebook is available at 
{\small \url{https://github.com/vroulet/ilqc/ilqc.ipynb}}

\subsection{Settings Considered}
We consider two simple synthetic control environments: swinging up a pendulum,
and controlling a simplified model of a car. Experiments on a more realistic
model of a car are presented in Appendix~\ref{app:exp_sup}. In all experiments we consider only
a cost on the state variables, i.e., $h_t(\state_t, \ctrl_t) = h_t(\state_t)$.
See \citet{roulet2021techreport} for additional experiments with costs on the
control variables and other settings.

\paragraph{Swinging up pendulum}
We consider swinging up a pendulum vertically through the control of a torque.
The state $\state=(\theta,\omega)$ consists in the angle $\theta$ with the
vertical axis and the angular speed $\omega$ as illustrated in
Fig~\ref{fig:conv}. The dynamics in continuous time are
\begin{align}\label{eq:pendulum_cont_dyn}
    \dot \theta(t) & = \omega(t), \qquad
    ml^2 \dot \omega(t) = -m l g \sin \theta(t) - \mu \omega(t) + \ctrl(t),
\end{align}
where $m=1$ is the mass of the blob, $l=1$ is the length of the blob, $\mu=0.01$ is a
friction coefficient, $g=10$ is the gravitational constant. The system is
controlled through a torque, $\ctrl(t)$, applied to the pendulum. 
We use an Euler discretization scheme~\citep[Chapter 4]{gautschi2011numerical}
for the continuous dynamics~\eqref{eq:pendulum_cont_dyn} with a discretization
step $\Delta = T/\horizon$ for a total time $T=2$ and a number of discretization
steps $\horizon=100$.

For Figure~\ref{fig:conv}, we consider a single cost on the last state. Namely,
the objective is to swing up the pendulum to be vertical with
\[
h_\horizon(\state_\horizon) 
= (\theta_\horizon - \pi)^2 + \omega_\horizon^2, \qquad
h_t(\state_t)= 0 \ \mbox{for} \ t \in \{1, \ldots, \horizon-1\},
\]
for $\state_\horizon = (\theta_\horizon, \omega_\horizon)$. In other words, we
target $\theta(T) = \pi, \omega(T) = 0$ for some time horizon $T$, given
$\theta(0)=0$, $\omega(0) = 0$. In some experiments below, we consider
variations of the costs, such as considering a cost for each time step or a
subsampled cost.

\paragraph{Simple model of a car with tracking costs}
We consider a simple model of the car, illustrated in Figure~\ref{fig:conv}.
The state consists in $\state = (z_x, z_y, \theta, v)$, where $z = (z_x, z_y)$
is the position of the car, $\theta$ is the angle between the orientation of the
car and the horizontal axis, a.k.a., the yaw, and $v$ is the longitudinal speed.
The controls $\ctrl = (a, \delta)$ consist of the longitudinal acceleration $a$
of the car, and the steering angle $\delta$. For a car of length $l=1$, the
continuous time dynamics of this simplified model of the car are
\begin{align*}
\begin{array}{lll}
  \dot z_x(t) = v(t) \cos \theta(t) & & \dot \theta(t)  = {v(t) \tan \delta(t)}/{l} \\
  \dot z_y(t) = v(t) \sin \theta(t) & & \dot v(t) = a(t).
\end{array}
\end{align*}
We use a Runge-Kutta method of order 4~\citep[Chapter 4]{gautschi2011numerical}, a discretization step $\Delta = T/\tau$ for a total time $T=2$, and
a number of discretization steps $\horizon=25$.

The objective consists in minimizing the distance between the position of the
car and a reference position on a track. We define a reference track $z^*(t)$ as
a continuous spline using a simple track presented by~\citet[Figure
13]{roulet2021techreport}. The discrete time reference positions are defined as
$z_t^* = z^*(\Delta t)$. The costs consist then
\[
  h_t(\state_t)
  = \|z_t - z_t^*\|_2^2 \quad \mbox{for} \ t \in \{1, \ldots, \horizon\}.
\]
For Figure~\ref{fig:conv}, we consider a subsampled cost equivalent to consider a
multistep discretization strategy detailed in Section~\ref{sec:suff_cond}.
Namely, we subsample the cost every $k=3$ steps such that the costs are then 
\begin{equation}\label{eq:cost_car_subsampled}
  h_t(\state_t)
= \begin{cases}
  \|z_t - z_t^*\|_2^2 & \mbox{if} \ t \bmod{k} = 0 \\
  0 & \mbox{otherwise}
\end{cases}
\end{equation}
with $\Delta = T/(k\tau)$. Below, we consider also costs on every time-step,
i.e., $k=1$.

\subsection{Evaluations}

\paragraph{Costs along iterations for the pendulum}

\begin{figure}[t]
  \begin{center}
    \includegraphics[width=0.7\linewidth]{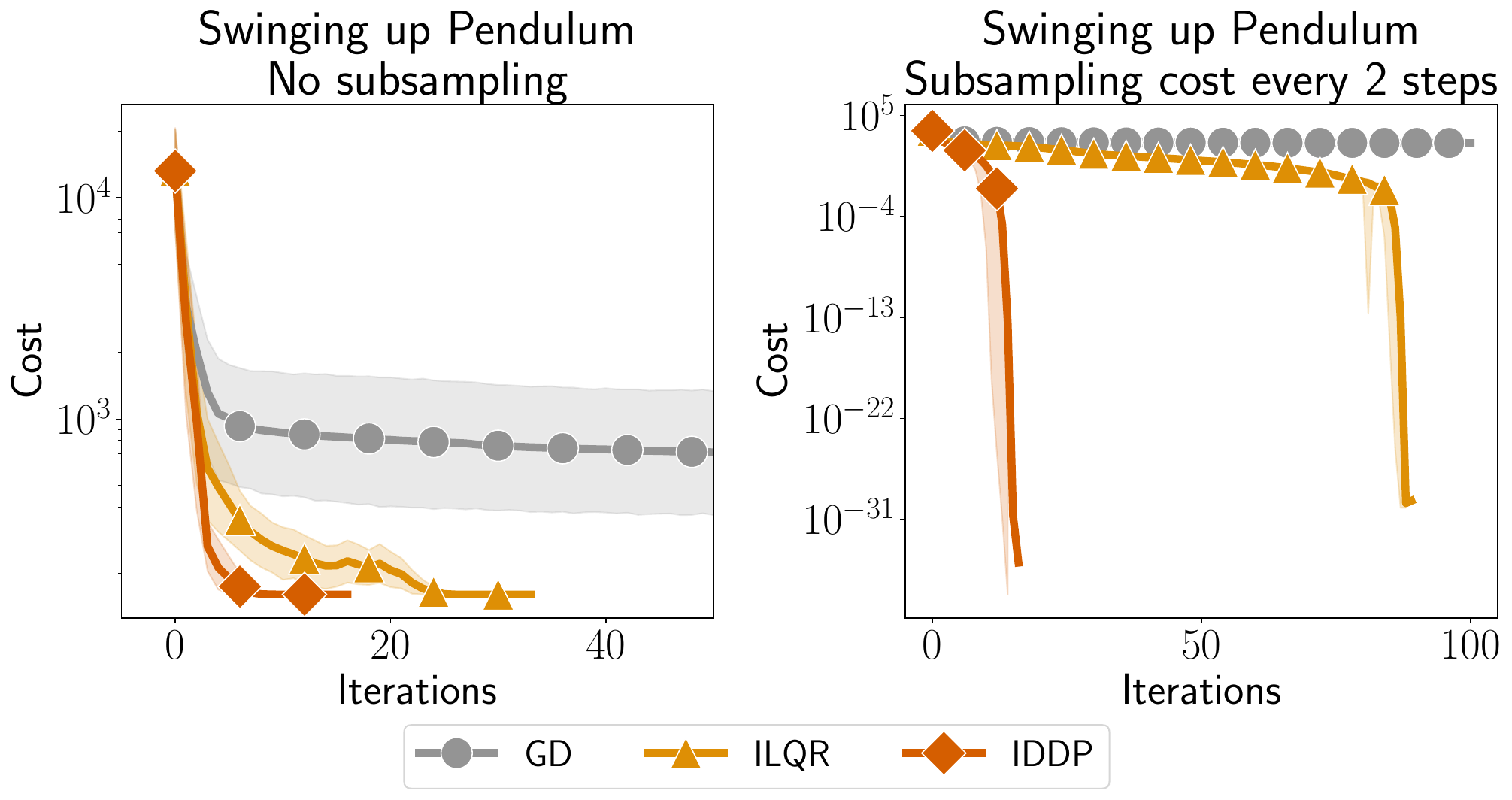}
    \caption{
      Cost along iterations of ILQR, IDDP and Gradient Descent (GD) on the
      pendulum problem using a cost at each time step or every two time steps.
      Shaded areas correspond to a 95\% confidence intervals over 10 random
      initializations of the control sequences. \label{fig:pend_subsampled}
    }
  \end{center}
\end{figure}

For a single final cost, the problem of swinging up the pendulum is equivalent
to minimizing the composition of a strongly convex cost with the control in
$\tau$ steps of the discrete dynamics of the pendulum. With an Euler
discretization of the continuous dynamics of the
pendulum~\eqref{eq:pendulum_cont_dyn}, one easily observes that the control in
any $k\geq 2$ steps of the discrete dynamics has surjective linearizations as
outlined in Section~\ref{sec:cond}. Hence, with a single final quadratic cost,
this problem falls under the assumptions of Section~\ref{sec:cvg}. The
convergence of both ILQR and IDDP algorithms towards a global minimum cost,
namely a null cost, is observed in Figure~\ref{fig:conv}.

In Figure~\ref{fig:pend_subsampled}, we consider a cost every $k$ steps, that is 
\begin{align*}
h_t(\state_t) & = \begin{cases}
(\theta_t - \pi)^2 + \omega_t^2 & \mbox{if} \ t \bmod{k} = 0 \\
0 & \mbox{otherwise}
\end{cases} \quad \mbox{for} \ t \in \{1, \ldots, \horizon\},
\end{align*}
for $k\in \{1, 2\}$. We also consider $10$ random initial sequence of control
variables, i.e., $u_t^{(0)} \sim \mathcal{N}(0, \sigma)$, for
$\sigma=1/\Delta=100$, $t\in \{0, \ldots, \horizon-1\}$.

By taking $k=2$, we observe that ILQR and IDDP both converge to a $0$ cost,
hence a global minimum, across random initializations. As mentioned above, by
taking $k>1$ convergence to a global minimal cost is predicted by the theory in
Section~\ref{sec:cond} and~\ref{sec:cvg}.

For $k=1$, none algorithm converges to 0. However, this does not mean that they
do not converge to a global minimum. In fact, one observes that across random
initializations, both ILQR and IDDP converge to the same cost. Namely, the
standard deviation of the minimum cost computed by these algorithms across
random initializations is $10^{-14}$. This suggests a global convergence
behavior to a same minimum. While the theory developed in Section~\ref{sec:cond}
and~\ref{sec:cvg} explains the behavior for $k=2$, the results for $k=1$ suggest
that convergence to a global minimum may be ensured beyond the sufficient
condition~\eqref{eq:global_conv_cond}. Note that global convergence of ILQR and
IDDP to, e.g., stationary points, can be verified on generic
problems~\eqref{eq:discrete_pb_ctrl_cost} (Section~\ref{sec:algos}). Such global
convergence properties are not sufficient to ensure convergence to global
minima. Convergence to global minima require additional properties of the
problem itself.

\paragraph{Costs along iteration for the simple model of a car}

\begin{figure}[t]
  \begin{center}
    \includegraphics[width=0.7\linewidth]{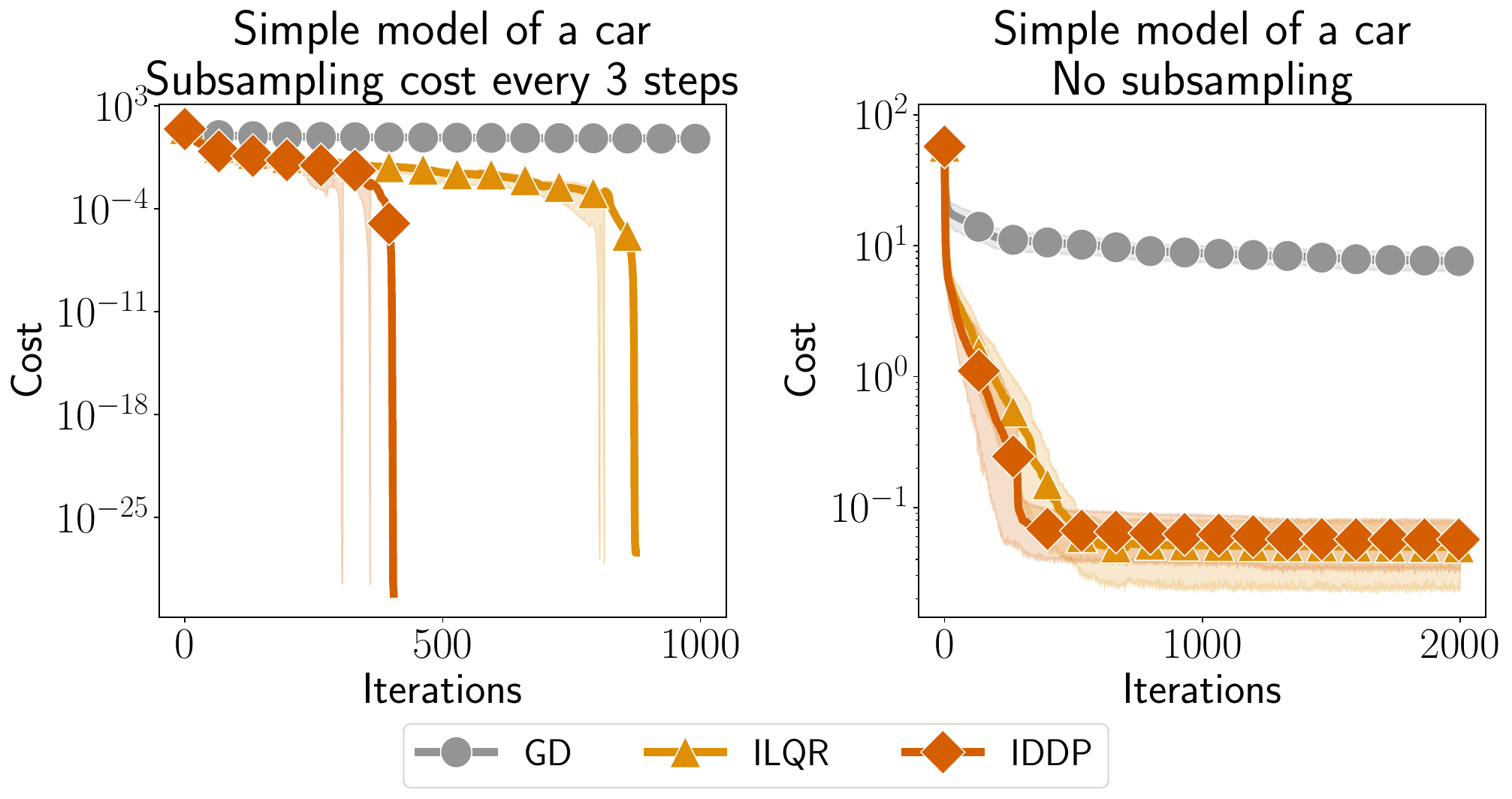}
    \caption{
        Cost along iterations of ILQR, IDDP and Gradient Descent (GD) on
        the car problem using a cost at each time step or every three time steps
        with varying initial controls. Shaded areas correspond to a 95\% confidence
        intervals over 10 random initializations of the control sequences.
        \label{fig:car_subsampled}
    }
  \end{center}
\end{figure}

In Figure~\ref{fig:conv}, we considered a subsampled cost, such that a sufficient
condition for convergence to global minima outlined in
Section~\ref{sec:suff_cond} may be satisfied. We observe in Figure~\ref{fig:conv}
convergence to a global minimal cost, namely a null cost, for both ILQR and IDDP
algorithms.

In Figure~\ref{fig:car_subsampled}, we consider a cost at each time step (no
subsampling of the costs, i.e., $k=1$ in~\eqref{eq:cost_car_subsampled}) with
$10$ random initial control sequences, i.e., $u_t^{(0)} \sim \mathcal{N}(0,
\sigma)$ for $\sigma = 2/\Delta = 25$, $t \in \{0, \ldots, \horizon-1\}$. We
also repeat the experiment with costs subsampled every $3$ time steps with the
same random initializations schemes. 

For subsampled costs, i.e., $k=3$ in~\eqref{eq:cost_car_subsampled}, we observe
convergence to global minimal costs (null costs) for both IDDP and ILQR
algorithms across random initializations.

For non-subsampled costs, i.e., $k=1$ in ~\eqref{eq:cost_car_subsampled}, the
costs do not converge to 0. Contrarily to the pendulum case, we observed a
discrepancy in the minimal cost reached after $2000$ iterations. ILQR and IDDP
reach, on average across initializations, costs of, respectively, $4.61 \cdot
10^{-2}$ and $5.68\cdot 10^{-2}$ with standard deviations across initializations
of, respectively, $3.93\cdot 10^{-2}$ and $3.75 \cdot 10^{-2}$.

Finally, IDDP converges faster than ILQR in all pendulum examples and in the
example of the car with subsampled costs. A similar observation was also made
by~\citet{liao1991convergence} and in the companion
paper~\citep{roulet2021techreport}.

\begin{figure}[t]
  \begin{center}
    \includegraphics[width=0.7\linewidth]{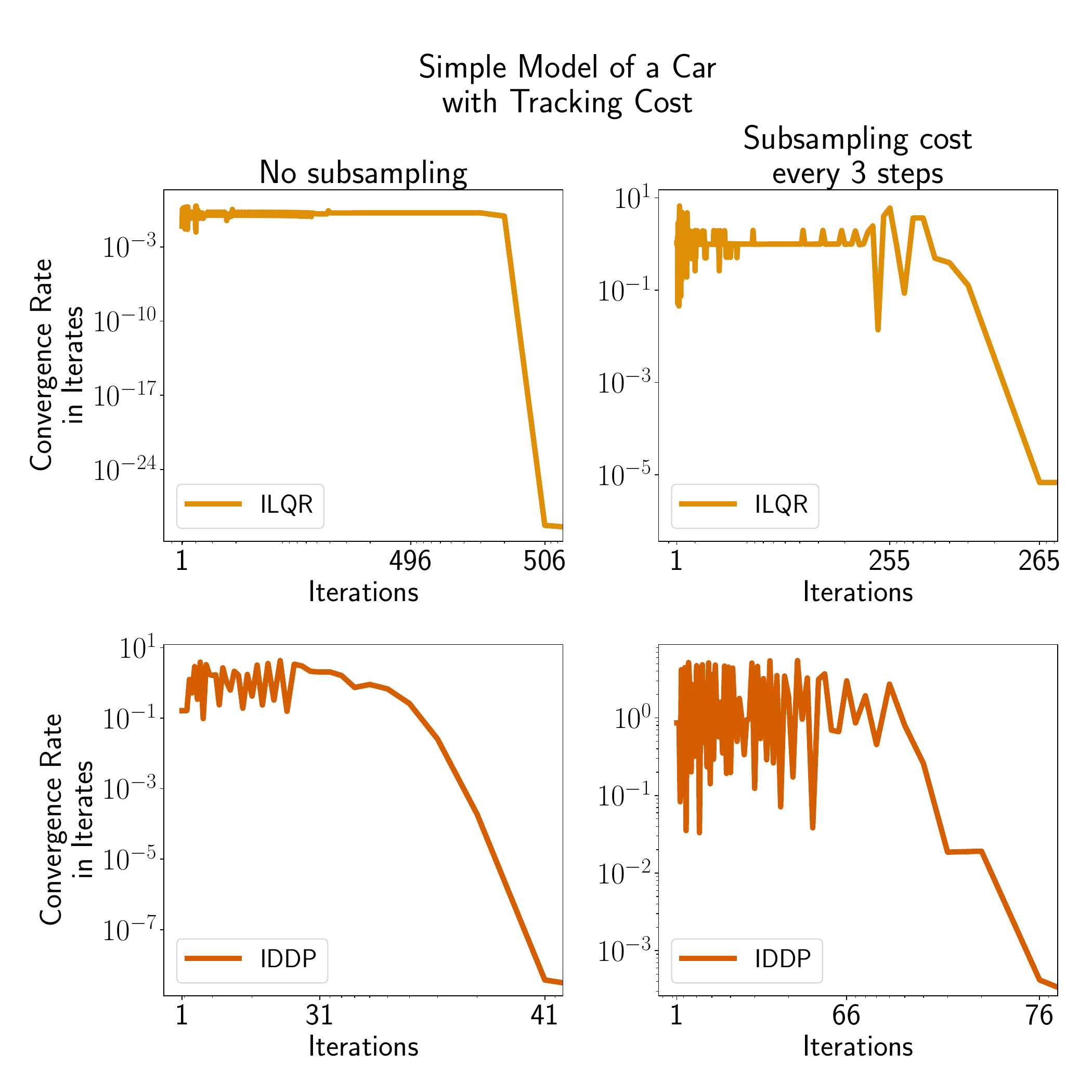}
    \caption{ Convergence rate in iterates, $\kappa^{(k)} =
      \|\ctrls^{(k+1)}-\ctrls^{(k)}\|_2/\|\ctrls^{(k)} - \ctrls^{(k-1)}\|_2$,
      along iterations of ILQR and IDDP algorithms for the simple model of a car
      with or without subsampling the costs. For each algorithm and each setting
      we plot the convergence rate up to the final iterate before the algorithm
      stopped and use a log scale x-axis to zoom on the final iterates.
      \label{fig:iter_rates}
    }
  \end{center}
\end{figure}

\begin{figure}
  \begin{center}
    \includegraphics[width=0.7\linewidth]{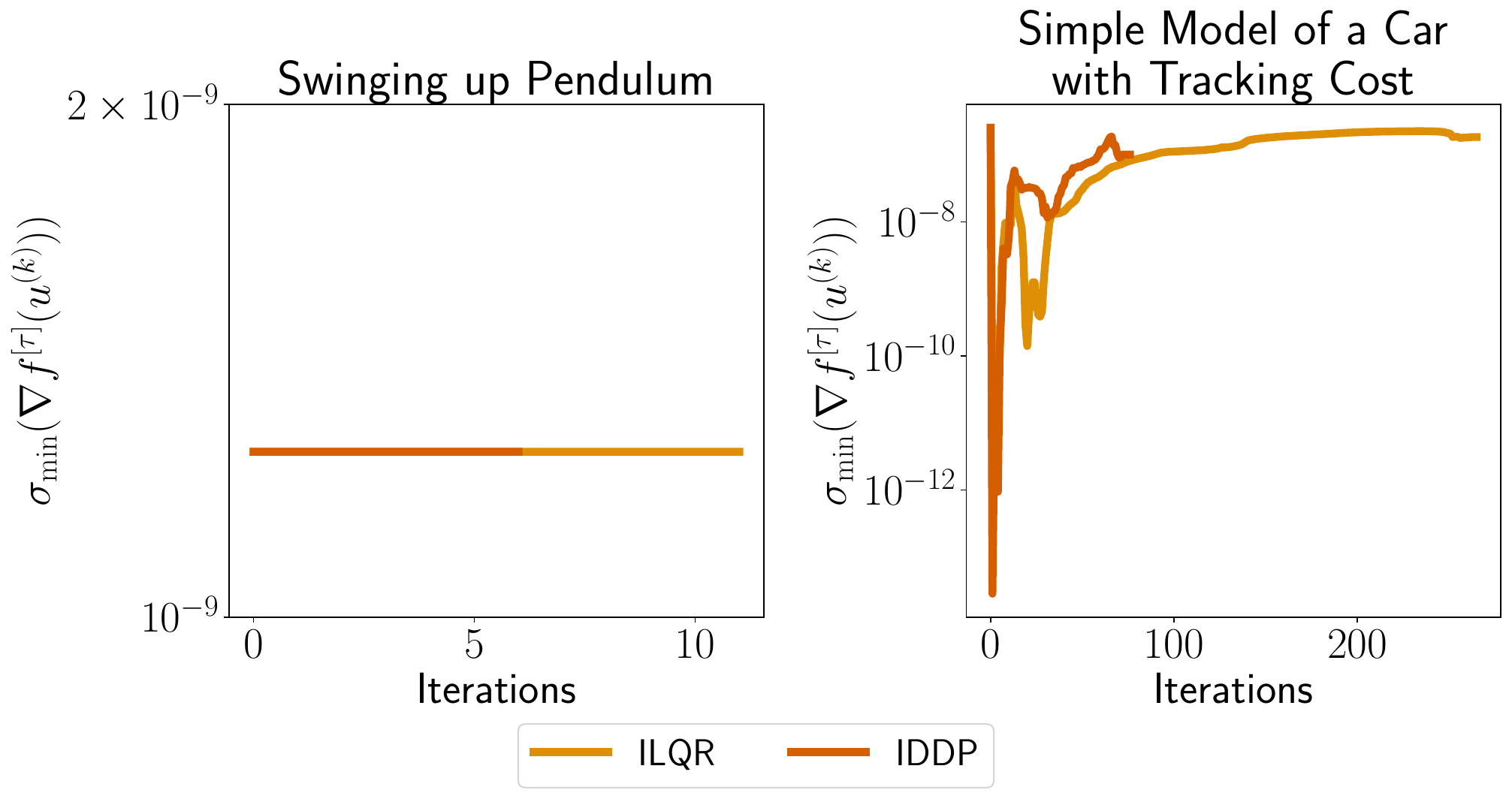}
    \caption{Minimal singular value of the transpose Jacobian of the control in
      $\horizon$ steps of the discrete dynamics along the iterations of ILQR and
      IDDP algorithms. We consider discrete dynamics of a pendulum or of a
      simple model of a car stemming from the control in $2$ and $3$ steps
      respectively of the original discretization of the dynamics.
      \label{fig:min_sev} }
  \end{center}
\end{figure}

\paragraph{Instantaneous rate of convergence}
The theoretical findings of Section~\ref{sec:cvg} outline a priori three phases of
convergence, sublinear, linear and quadratic. Convergence rates of ILQR and IDDP
can be assessed through convergence rates in function values $\rho^{(k)} =
(c^{(k+1)} - c^*)/(c^{(k)}-c^*)$ for $c^*$ the minimal cost as done in
Appendix~\ref{app:exp_sup}, or by considering convergence in iterates through
$\kappa^{(k)} = \|\ctrls^{(k+1)}-\ctrls^{(k)}\|_2/\|\ctrls^{(k)} -
\ctrls^{(k-1)}\|_2$ as done in Figure~\ref{fig:iter_rates}.

For the simple model of a car, in Figure~\ref{fig:iter_rates}, we observe that the
convergence rate in iterations of these algorithms remain close to $1$ for many
iterations (the x-axis in Figure~\ref{fig:iter_rates} is in reverted log-scale).
This rate suddenly drops close to convergence akin to a local quadratic local
convergence. This shows that the main difficulty of the problem arises for a
long first phase of slow convergence.

\paragraph{Surjectivity of the Jacobian}

The sufficient condition for convergence~\eqref{eq:global_conv_cond} to global
minima can be assessed by computing the minimal singular value
$\sigma_{\min}(\nabla \traj(\ctrls^{(k)}))$ of the transpose Jacobian of the
control of $\horizon$ steps of the discrete dynamics. In Figure~\ref{fig:min_sev},
we plot this minimal singular value along the iterations of the ILQR and IDDP
algorithms. We consider discrete dynamics defined as the control in $k=2$ and
$k=3$ steps of the discretization of the continuous dynamics of, respectively,
the pendulum and the simple model of a car. Considering discrete dynamics in
multiple steps amount to the subsampling of the costs presented in previous
experiments. 

We observe in Figure~\ref{fig:min_sev} that $\sigma_{\min}(\nabla
\traj(\ctrls^{(k)}))$ is small yet bounded away from $0$ along the iterations.
This result concurs with the convergence to global minimal costs of these
algorithms observed in the right panels of Figure~\ref{fig:pend_subsampled} and
Figure~\ref{fig:car_subsampled}.

\section{Related Work}\label{sec:related_work}
\paragraph{Nonlinear control approaches} 
Nonlinear control problems of the
form~\eqref{eq:discrete_pb_ctrl_cost} stem from the discretization of generic
optimal control problems in continuous time of the form
\begin{align}
  \min_{\state(\cdot), \ctrl(\cdot)} \quad
  & \int_0^T h(x(t), u(T)) + h_T(x(T)) \label{eq:cont_ctrl_pb}\\
  \mbox{subject to} \quad  & \dot x(t) = f(x(t), u(t)), \quad x(0) = \bar x_0. \nonumber
\end{align}
Continuous optimal control problems of the form~\eqref{eq:cont_ctrl_pb} can be
tackled in various ways~\citep{diehl2006fast}. One can approach the problem from
a \emph{dynamic programming} perspective to derive the Hamilton-Jacobi-Bellman
equation, a partial differential equation in state
space~\citep{lions1982generalized}. Alternatively, one can derive necessary
optimality conditions for~\eqref{eq:cont_ctrl_pb} to derive a boundary value
problem. Such a method is referred to as an \emph{indirect method} and amounts
to an ``optimize then discretize'' approach~\citep{farshidian2017efficient}.
Finally, problem~\eqref{eq:cont_ctrl_pb} can be tackled by \emph{direct methods}
that consider finite dimensional approximations of the original infinite
dimensional problem~\eqref{eq:cont_ctrl_pb}. Direct methods amount to a
``discretize then optimize'' approach~\citep{diehl2006fast}, they can further be
split into different approaches. First, one may consider a finite representation
of the continuous control $u(t)$ as piecewise constant functions whose values
$q_1, \ldots, q_\tau$ at each piece define the finite number of degrees of
freedom. The problem still involves an ODE in the state variable, $\dot x(t) =
f(x(t), u_{q_{1:\tau}}(t))$, albeit a simpler one. Tackling the problem with
such a partial discretization is referred to as a \emph{single shooting}
method~\citep{diehl2006fast, bock1984multiple}. \emph{Collocation
methods}~\citep{von1993numerical} consider discretizing both the states and
controls, leading to a formulation like~\eqref{eq:discrete_pb_ctrl_cost}, that
can benefit from advanced numerical integration methods. Finally, \emph{multiple
shooting}~\citep{diehl2006fast,bock1984multiple} combines both approaches. The
system is split in multiple windows and for each window a single shooting method
is used. We focus solely on the resulting discrete time nonlinear control
problems~\eqref{eq:discrete_pb_ctrl_cost} and refer the interested reader to,
e.g.,~\citet[Section 8.5]{rawlings2017model} for an overview of the approaches
mentioned above.

\paragraph{Nonlinear control with local approximations and iterative
refinements} One of the first approaches for nonlinear discrete time control
problems~\eqref{eq:discrete_pb_ctrl_cost} appear to be the Differential Dynamic
Programming (DDP) methods developed by~\citet{mayne1966second,
jacobson1970differential, mayne1975first}. Its principle is to apply a dynamic
programming procedure to the nonlinear system. The associated Bellman equation
is approximately solved by considering its quadratic approximation around the
current trajectory. A set of policies is computed along this process and applied
to the original dynamics as if the true solutions of the Bellman equations were
found. A modern account is provided in the companion
paper~\citep{roulet2021techreport} for reference; see
also~\citet{liao1992advantages}. Numerous variants of DDP have been developed to
account for constraints or noise in the dynamics~\citep{li2004iterative,
tassa2007receding, tassa2014control, giftthaler2018family}. Among those, IDDP,
a.k.a. iLQR, can be seen to follow the same principle as DDP except that
\emph{linear-quadratic} approximations \`a la Gauss-Newton are used in place of
the quadratic approximations of the Bellman equation akin to Newton's method.

DDP approaches differ from the implementation of classical
optimization algorithms such as a Newton, quasi-Newton or Gauss-Newton method
for discrete nonlinear control problems. \citet{bock1983recent, bock1984multiple} 
first presented such approaches referred to as \emph{direct multiple shooting}.
Detailed and efficient implementations of Newton's method exploiting the
dynamical structure of the problem were presented
by~\citet{de1988differential,dunn1989efficient}. A linear algebraic viewpoint on
these implementations was presented by~\citet{wright1990solution,
wright1991partitioned}, that enabled the use of fast linear solvers exploiting
the structure of nonlinear control problems~\citep{wright1991structured,
jerez2011condensed, rao1998application}. In particular,
\citet{wright1991partitioned} presents alternative resolutions of the
linear quadratic subproblem using a ``Riccati-like'' recursion that slightly
differs from the resolution by dynamic programming presented here. 
\citet{wright1991partitioned} further developed parallel implementations of
algorithms solving the LQR problems. We do not delve into the specific
implementations of the oracles used in ILQR or IDDP and rather focus on the
global behavior of the algorithms.

This viewpoint was further generalized
to handle nonlinear inequalities in model predictive
control~\citep{diehl2009efficient} or even generic graphs of
computations~\citep{srinivasan2015graphical}. The ILQR algorithm can be seen as
an instance of direct multiple shooting, namely, an instance of a generalized
Gauss-Newton method~\citep{sideris2005efficient} which uses
\emph{linear-quadratic} approximations of the problem decomposed along the
dynamics.

Detailed implementations of the DDP (quadratic approximation of Bellman
equation), the IDDP (linear-quadratic approximation of Bellman equation),
Newton (quadratic approximation of the objective) and the ILQR (linear-quadratic
approximation of the objective) approaches are presented in the companion
paper~\citep{roulet2021techreport} to highlight their common points and
differences. 

The decomposition of the problem at several scales by means of 
some quadratic approximations have also been developed and 
studied by~\citet{messerer2021survey, frasch2015parallel, 
verschueren2016exploiting, houska2013quadratically}. 

\paragraph{Convergence analysis of Gauss-Newton methods}
Regularized Gauss-Newton methods, a.k.a. Levenberg-Marquardt
methods~\citep{more1978levenberg}, have been extensively
studied~\citep{yamashita2001rate, fan2005quadratic, dan2002convergence,
zhao2016global, bergou2020convergence}. Global convergence to stationary points
at a polynomial rate is established by, e.g.,~\citet[Theorem
3.1]{bergou2020convergence}. The results may be extended, provided
that the non-linear mappings have surjective Jacobians~\citep[Corollary
2.1]{ueda2010global}. Our approach improves on previous results with polynomial
rates and our complexity bounds provide explicit dependencies on the initial gap
and the region of quadratic convergence. We also depart from previous results
using error bounds, such as the ones of~\citet[Assumption
4.2]{bergou2020convergence} and \citet[Eq. (1.6)]{yamashita2001rate}, in that
our assumption on surjective Jacobians is stronger than an error bound. 

Closer to our approach is the work of~\citet{nesterov2007modified} where the
assumption of surjective Jacobians is used to provide global convergence
guarantees of a \emph{modified} Gauss-Newton method also known as the
prox-linear method~\citep{drusvyatskiy2019efficiency} for nonlinear fitting.
\citet{nesterov2007modified} argues in favor of least un-squared norms methods,
as opposed to least squared norms methods, by reasoning in terms of condition
numbers irrespective of local subroutine computational complexity. In contrast,
we consider twice differentiable costs, for which we build a quadratic model,
leading to \emph{generalized} Gauss-Newton methods. In nonlinear control,
generalized Gauss-Newton oracles can be implemented efficiently by exploiting
the dynamical structure of the problem, while modified Gauss-Newton method
oracles may require a computationally expensive line-search.
\citet{messerer2021survey} considered also convergence of generalized
Gauss-Newton methods. However, \citet{messerer2021survey} analyzes such
algorithms without regularization, nor linesearch or trust-region techniques,
resulting in possibly divergent algorithms or only local convergence guarantees.
By adding a regularization scheme, we are able to ensure global convergence, and
to provide practical guidance on the choice of regularization
(Algorithm~\ref{algo:RGGN_linesearch}). \citet{baumgartner2023local} also
considered the local convergence properties of ILQR, IDDP to determine that they
share the same linear convergence rate locally. We consider more general
convergence properties towards stationary points, or global minima given
additional assumptions. Finally, our results are quantitative, relating the
region of quadratic convergence to the smallest singular value of the transposed
Jacobian. 

As mentioned earlier, a Newton's method could just as well be implemented to
exploit the dynamical structure of the problem~\citep{dunn1989efficient}.
Several caveats lend still in favor of a Gauss-Newton method. First, a Newton's
method (or a DDP approach) requires computing and storing the second order
information associated to the dynamics at the intermediate states, although the
storage issue can be mitigated by an adequate implementation in a differentiable
programming framework~\citep{nganga2021accelerating, roulet2021techreport}.
Second, Newton's method does not compute a priori descent directions if
the Hessian is not positive definite. Hessian modifications~\citep[Section
3.4]{nocedal2006numerical} may be necessary to ensure a descent direction such
that a linesearch can be used.
On the other hand, for generic functions, Newton's method is known to converge
locally at a quadratic rate~\citep{nesterov2018lectures}, which is a priori not
true for a generalized Gauss-Newton method. Our analysis shows that in some
nonlinear control problems generalized Gauss-Newton methods can converge with
such a local quadratic rate, just as observed empirically. Our analysis stems in
fact from considering a generalized Gauss-Newton method as an approximate Newton
method in the space of the trajectories which enable us to recover the fast
local rate of convergence of Newton's method by appropriately controlling the
approximation error.

\paragraph{Convergence analysis of differentiable dynamic programming methods}

Previous work mainly focused on local convergence
guarantees~\citep{mayne1975first,murray1984differential,liao1991convergence} or
convergence guarantees towards controls satisfying first-order necessary
optimality conditions~\citep{polak2011role}. The local quadratic convergence
analysis of DDP is based on viewing DDP as an approximate Newton
method~\citep{de1988differential, di2019newton}. An alternative proof of local
quadratic convergence~\citep{liao1991convergence} and an approach based on the
method of strong variations~\citep{mayne1975first} are also worth mentioning.
Previous work~\citep{de1988differential, di2019newton} considers additional
costs on the control variables and assumes that the Hessian of the overall
objective~\eqref{eq:discrete_pb} is invertible; see~\cite[Theorem
4.1]{de1988differential} or~\cite[Assumption 2.2]{di2019newton}. In contrast to
previous work, we do not consider additional costs on the control variable, and
we consider the IDDP algorithm which uses linear-quadratic approximations
developed by~\citet{tassa2012synthesis} and extended
by~\citet{giftthaler2018family}. The IDDP algorithm benefits from a smaller
per-iteration cost compared to DDP, as IDDP does not require computing
intermediate second-order information associated to the dynamics.

\paragraph{Sufficient conditions for convergence to global minima}
Discrete time nonlinear control problems of the form~\eqref{eq:discrete_pb} stem
from the time discretization of continuous time problems. Necessary optimality
conditions for the continuous time control problems are characterized by
Pontryagin's maximum principle~\citep{pontryagin1961mathematical}. However,
these optimality conditions cannot be used for the discretized problems since
Pontryagin variations in finite dimensional space do not
exist~\citep{polak2011role}.  Necessary optimality conditions can be derived
from the Karush-Kuhn-Tucker conditions for problem~\eqref{eq:discrete_pb}, which
are  equivalent to  first order optimality conditions of the objective in terms
of control variables. Sufficient optimality conditions for the continuous time
nonlinear control problem were also derived by~\citet{mangasarian1966sufficient,
arrow1968control, kamien1971sufficient}. We translate these conditions for the
discrete time nonlinear control problem in Appendix~\ref{app:opt_cond}.
Unfortunately, such conditions require convexity assumptions of implicitly
defined functions that seem difficult to verify in practice. We argue in
Section~\ref{sec:suff_cond} that our assumption~\eqref{eq:global_conv_cond} can be
verified on simple instances. 

Our assumption is based on analyzing the gradient dominating property of the
objective of problem~\eqref{eq:discrete_pb} in terms of the properties of the
dynamic. The gradient dominating property was introduced
by~\citet{polyak1964some, lojasiewicz1963topological} as a sufficient condition
to ensure convergence of gradient descent to global minima. Here, we exploit
this property to ensure global and local quadratic convergence to global minima
of a regularized generalized Gauss-Newton algorithm. From a nonlinear control
viewpoint, our assumption translates as the controllability of the discrete
linearized trajectories in one step. In a similar spirit, a controllability
assumption on the discrete linearized trajectories in \emph{several} steps was
considered to analyze the local convergence of MPC controllers
by~\citet[Assumption 2]{na2020superconvergence}
following~\citet{xu2019exponentially}. Compared
to~\citet{xu2019exponentially,na2020superconvergence}, we consider \emph{global}
convergence results to minimizers, which justifies a stronger assumption. In
addition, compared to~\citet{xu2019exponentially,na2020superconvergence},  we
formally relate our condition to feedback linearization schemes well understood
in continuous time~\citep{isidori1995nonlinear, sontag2013mathematical} and
further developed in discrete time by \citet{ jakubczyk1990controllability,
jakubczyk1987feedback, jayaraman1993feedback, aranda1996linearization,
belikov2017global}. In particular, we exploit the existence of a feedback
linearization scheme by considering a multi-rate sampling scheme to ensure our
sufficient condition. Using multi-rate sampling was proposed in the early work
of~\citet{grizzle1988feedback} on discrete time feedback linearization schemes. 

We consider only understanding the performance of two popular algorithms, ILQR
and IDDP. Several variants can be considered. In particular, given the
surjectivity of the Jacobian of the dynamics~\eqref{eq:global_conv_cond}, the
problem may also be rephrased as a feasibility problem and tackled differently.
Namely, the minimizers $\state_t^*$ of the costs $h_t$ could be computed
offline, and the problem would reduce to fit a nonlinear model of the states,
here described by the trajectories given by the dynamics, to the minimizers
$\state_t^*$. Such feasibility problems may be tackled for example by penalty
method as done by~\citet{kim2016s}. However, such penalty methods may dismiss
the dynamical structure of the problem. Moreover, ILQR or IDDP methods can
tackle the original problem at once, rather than deriving a two-stage method
consisting in computing first the minimizers of the costs.

\section{Conclusion}
We have detailed computational complexities of the ILQR and IDDP algorithms for
discrete time nonlinear control problems~(\ref{eq:discrete_pb_ctrl_cost},
\ref{eq:discrete_pb}). Our analysis decomposes at several scales. At the scale
of the whole trajectory, the problem can be summarized as a compositional
objective and analyzed as a Gauss-Newton type algorithm. The trajectories can be
detailed at the scale of the dynamic, which reveals the low computational cost
of the optimization oracles. Finally, the dynamics can further be detailed in
terms of the discretization scheme in order to ensure sufficient conditions for
convergence of the algorithms towards global optima. 

The sufficient conditions for global convergence are restricted to problems
without costs or constraints on the control variables. Moreover, they
may not be applicable in usual scenarios with costs that are not subsampled. As
future work, one may analyze constraints on the control variables while ensuring
a gradient dominating-like property on the objective. Analyzing further the
links between feedback linearization schemes and sufficient conditions for
global optimality may also reveal the impact of the discretization stepsize on
the overall condition number of the problem.

\acks{This work was supported by NSF DMS-1839371, DMS-2134012, CCF-2019844,
CIFAR-LMB, NSF TRIPODS II DMS-2023166 and faculty research awards. This work was
done when Vincent Roulet was at the University of Washington, minor revisions
where done when he was at Google. The authors thank Dmitriy Drusvyatskiy,
Alexander Liniger, Krishna Pillutla and John Thickstun for fruitful discussions
on the paper and their help to develop the numerical experiments. The authors are sincerely grateful to the action editor and the reviewers for their thorough work and the numerous comments that helped us improve on the original manuscript. 
}

\clearpage
\bibliography{refs}
\clearpage

\appendix
\addcontentsline{toc}{section}{Appendix}
\addtocontents{toc}{\protect\setcounter{tocdepth}{-10}}
\part{Appendix}
\parttoc

\section{Index of Constants}\label{app:index}
Table~\ref{tab:index} presents an index of the constants used in the main
results of the paper in Section~\ref{sec:cvg} with their units. We denote the unit
of the control variables, the states and the costs as, respectively, $\ctrl$,
$\state$ and $\cost$ and use $1$ if the constant has no dimension. 

Note that all constants are rooted in assumptions about the dynamic $\dyn$ and
the individual costs $\cost_t$ of problem~\eqref{eq:discrete_pb}. In particular,
constants governing the compositional problem~\eqref{eq:total_cost} defined by
the total cost $\cost$ and the control $\augtraj$ in $\horizon$ steps of $\dyn$
for fixed initial state (see~\eqref{eq:total_cost} and
Def.~\ref{def:traj_func}), are all explicitly given in terms of the constants of
$\dyn$, $\cost_t$. Moreover, note that the constants governing the dynamic
$\dyn$ can be further decomposed by considering the dynamic as the control in
$\ksteps$ of a dynamic as presented in Section~\ref{sec:suff_cond}.

For simplicity, we present only the strongly convex case. For the gradient
dominating case with  exponent $\pl\neq 1/2$ we refer the reader to
Theorem~\ref{thm:global_conv}. For the local convergence, constants $\sigma,
\lip, \smooth, \concord$ can be defined without strong convexity as presented in
Assumption~\ref{asm:self_concord}. 

\begin{table}
	\begin{center}
	\begin{tabular}{c|c|c|c}
		Notation & Definition& Interpretation & Unit \\
		\hline
		$\sigma_\dyn$ & $\inf_{\state, \ctrl } \sigma_{\min} (\nabla_{\ctrl}
		\dyn(\state, \ctrl))$ & Surj. param. of  $\auxctrl \rightarrow \nabla_\ctrl
		\dyn(\state, \ctrl)^\top\auxctrl$& $\state/\ctrl$\\
		$\lipdynstate$ & $\sup_{\ctrl} \lip_{\dyn(\cdot, \ctrl)}$ & Lip. cont. of
		$\dyn(\cdot, \ctrl)$ for any $\ctrl$ & $1$\\
		$\lipdynctrl$ &$\sup_{\state} \lip_{\dyn(\state, \cdot)}$  & Lip. cont.  of
		$\dyn(\state, \cdot)$ for any $\state$ & $\state/\ctrl$\\
		$\smoothdynstate$ &$ \sup_{\ctrl} \lip_{\nabla_\state \dyn(\cdot, \ctrl)}$
		& Bound on $\|\nabla_{\state\state} \dyn^2(\state, \ctrl)\|$  & $1/\state$\\
		$\smoothdynctrl$ &$ \sup_{\state} \lip_{\nabla_\ctrl \dyn(\state, \cdot)}$
		& Bound on $\|\nabla_{\ctrl\ctrl} \dyn^2(\state, \ctrl)\|$  &
		$\state/\ctrl^2$\\
		$\smoothdynstatectrl$ &$ \sup_{\state}\lip_{\nabla_\ctrl \dyn(\cdot,
		\ctrl)}$  & Bound on $\|\nabla_{\state\ctrl} \dyn^2(\state, \ctrl)\|$ &
		$1/\ctrl$\\
				$\sigma_\augtraj, \sigma_{\traj}$ & $\sigma_\dyn/(1+\lipdynstate)$
				&Lower bound on $ \sigma_{\min}(\nabla \traj(\state_0, \ctrls))$  &
				$\state/\ctrl$\\
		$\lip_\augtraj, \lip_\traj$ &  $\lipdynctrl S$ & Lip. cont.  of
		$\traj(\state_0, \ctrl)$ & $\state/\ctrl$\\
		$\smooth_\augtraj, \smooth_\traj $ & {\footnotesize  $\smoothdynstate
		(\lipdynctrl S)^2 {+} 2 \smoothdynstatectrl\lipdynctrl S {+} \smoothdynctrl
		S $} & Lip. cont.  of $\nabla_{\ctrls}\traj(\state_0, \ctrl)$ &
		$\state/\ctrl^2$ \\
		$S$ & {\small $\sum_{t=0}^{\horizon-1}(\lipdynstate)^t$} & Auxiliary
		constant & $1$\\
		$\strgcvx_\cost$ & $\inf_\state \sigma_{\min}(\nabla^2\cost_t(\state))$ &
		Strong convexity param. of the costs & $\cost/\state^2$\\
		$\smooth_\cost$ & $\sup_\state \sigma_{\max}(\nabla^2\cost_t(\state))$ &
		Lip. cont.  of gradients of the costs & $\cost/\state^2$\\
		$\smoothess_\cost$ & $\lip_{\nabla^2\cost_t}$ & Lip. cont. of  Hessians  of
		the costs & $\cost/\state^3$ \\
		$\condnb_\augtraj$ & $\lip_\augtraj/\sigma_\augtraj$ & Cond. nb of $\nabla
		\augtraj(\ctrls)$ & $1$ \\
		$\condnb_\cost$ & $\smooth_\cost/\strgcvx_\cost$ & Cond. nb of the costs &
		$1$ \\
		$\concord$, $\localconcord$ & $\smoothess_\cost/(2\strgcvx_\cost^{3/2})$ &
		Self-concordance of the costs & $1/\sqrt{\cost}$\\
		$\scaling$ & $\smooth_\augtraj/(\sigma_\augtraj^2\sqrt{\strgcvx_\cost})$ &
		Scaling param. for $g$ & $1/\sqrt{\cost}$\\
		$\newsimp$ & { $\smoothess_\cost \lip_\augtraj^2/(3 \smooth_\augtraj
		\smooth_\cost)$ } & Cond. nb for global conv.  of ILQR& $1$ \\
		$\simp$ & { $4\condnb_\augtraj^2 \condnb_\cost ( \newsimp +1)$ } & Cond. nb
		for global conv. of ILQR& $1$ \\
		$\lip$ & $\sqrt{\smooth_\cost} \lip_\augtraj$ & Lip. cont. of $g$ w.r.t. $h$
		in Asm.~\ref{asm:conv} &  $\sqrt{\cost}/\ctrl$\\
		$\smooth$ & $\sqrt{\smooth_\cost} \smooth_\augtraj$& Lip. cont. of grad. of
		$g$  w.r.t. $h$ in Asm.~\ref{asm:conv} & $\sqrt{\cost}/\ctrl^2$ \\
		$\sigma$ & $\sqrt{\strgcvx_\cost} \sigma_\augtraj$ & Surj. param. of
		$\augtraj$ w.r.t. $\cost$ & $\sqrt{\cost}/\ctrl$\\
		$\localcondnb$ & $\lip/\sigma = \sqrt{\rho_\cost} \rho_\augtraj $ & Cond. nb
		of $\augtraj$ w.r.t. $\cost$ & $1$ \\
		$\localscaling$ & $\smooth/\sigma^2 = \sqrt{\rho_\cost} \scaling$ & Scaling
		param. of $\augtraj$ w.r.t. $\cost$ & $1/\sqrt{\cost}$\\
		$\ddpbound$ & See Corollary~\ref{cor:approx_ilqrs}& \parbox{160pt}{Relative
	bound btw DDP \& LQR: \\ \hspace*{10pt} $\frac{\|\DDP_\reg(\obj)(\ctrls) {-}
	\LQR_{\reg}(\obj)(\ctrls)\|_2}{\|\LQR_\reg(\obj)(\ctrls)\|_2^2} {\leq}
	\ddpbound$} & $1/\ctrl$ \\
		$\condnbddp$ & $\lip_\augtraj\ddpbound/\smooth_\augtraj$ & Factor of
		smoothness for IDDP &  $1$
	\end{tabular}
\caption{Index of constants used in the paper. \label{tab:index}}
	\end{center}
\end{table}

\section{Optimality Conditions}\label{app:opt_cond}
\subsection{Necessary Optimality Conditions}
We recall necessary optimality conditions for nonlinear control problems in
continuous and discrete time to underline their discrepancies. The problem we
consider in continuous time is
\begin{align}\label{eq:contpb_app}
\min_{\substack{\state \in \mathcal{C}^1([0, 1], \reals^\dimstate)\\ \ctrl \in \mathcal{C}([0, 1], \reals^\dimctrl)}} &   \int_0^1 \cost(\state(t), \ctrl(t),  t)dt + \cost(\state(1), 1)\\
\mbox{subject to} \quad&  \dot \state(t) = \contdyn(\state(t), \ctrl(t), t), \quad \mbox{for} \ t \in [0, 1] \quad \state(0) = \initstate, \nonumber
\end{align}
where $\mathcal{C}([0, 1], \reals^d)$ and $\mathcal{C}^1([0, 1], \reals^d)$
denote the set of continuous and continuously differentiable functions from $[0,
1]$ onto $\reals^d$ respectively, and we assume $\contdyn$ and $\cost$ to be
continuously differentiable. By using an Euler discretization scheme with
discretization stepsize $\Delta = 1/\horizon$, we get the discrete time control
problem
\begin{align}
	\min_{\substack{\state_0, \ldots, \state_\horizon \in \reals^\dimstate \\ \ctrl_0 \ldots, \ctrl_{\horizon-1} \in \reals^\dimctrl}}
& \sum_{t=0}^{\horizon-1} \cost_t(\state_t, \ctrl_t)  + \cost_\horizon(\state_\horizon) \label{eq:discretized_pb}\\
\mbox{subject to} \quad & \state_{t+1} =  \state_t + \contdyn_t(\state_t, \ctrl_t), \quad \mbox{for} \ t \in \{0, \ldots, \horizon-1\},  \quad \state_0 = \initstate, \nonumber
\end{align}
where $\state_t = \state(\Delta t)$, $\ctrl_t = \ctrl(\Delta t)$, $\cost_t =
\Delta \cost(\cdot, \cdot, \Delta t)$, $\cost_\horizon = \cost(\cdot, 1)$, $
\contdyn_t =\Delta  \contdyn (\cdot, \cdot, \Delta t)$. Compared to
problem~\eqref{eq:discrete_pb_ctrl_cost}, we have $\state_t +
\contdyn_t(\state_t, \ctrl_t) = \dyn_t(\state_t, \ctrl_t)$. 

\paragraph{Continuous time necessary optimality condition}
Necessary optimality conditions for the continuous time control problem are
known as  Pontryagin's maximum principle, recalled below.
See~\citet{arutyunov2004simple} for a recent proof and~\citet{lewis2006maximum}
for a comprehensive overview. 
\begin{theorem}[Pontryagin's maximum
	principle~\citep{pontryagin1961mathematical}] \label{thm:pontryagin} 
  Define the Hamiltonian associated with problem~\eqref{eq:contpb_app} as
	\[
	\ham(\state(t), \ctrl(t), \lambda(t), t) = \lambda(t)^\top \contdyn(\state(t), \ctrl(t), t) - \cost(\state(t),\ctrl(t),  t).
	\]
	A  trajectory $\state \in \mathcal{C}^1([0, 1], \reals^\dimstate)$ and a
	control function $\ctrl \in \mathcal{C}([0,1], \reals^\dimctrl)$ are optimal
	if there exists  $\lambda \in \mathcal{C}^1([0, 1], \reals^\dimstate)$ such
	that 
	\begin{flalign}
\quad \dot \state(t) & = \nabla_{\lambda(t)} \ham(\state(t), \ctrl(t),
\lambda(t), t) \hspace{11pt} \mbox{for all} \ t\in [0, 1], &  \label{eq:cont_cond1} \tag{C1}\\
\mbox{with} \
\state(0) & = \initstate & \nonumber \\ 
\quad \dot \lambda(t) & = -\nabla_{\state(t)} \ham(\state(t), \ctrl(t),
\lambda(t), t) \ \mbox{for all} \ t\in [0, 1],  & \label{eq:cont_cond2} \tag{C2} \\
\mbox{with} \ \lambda(1) & =
-\nabla_{\state(1)} \cost(\state(1), 1)& \nonumber \\
\quad \ham(\state(t), \ctrl(t), \lambda(t), t) & = \max_{\ctrl\in
\reals^\dimctrl} \ham(\state(t), \ctrl, \lambda(t), t)\hspace{17.5pt}  \mbox{for
all} \ t\in [0, 1].& \label{eq:cont_cond3} \tag{C3}
	\end{flalign}
\end{theorem} 

\paragraph{Discrete time necessary optimality conditions}
In comparison, necessary optimality conditions for the discretized
problem~\eqref{eq:discretized_pb} are given by considering the
Karush-Kuhn-Tucker conditions of the problem, or equivalently by considering a
sequence of controls such that the gradient of the objective is
null~\citep{bertsekas1997nonlinear}.
\begin{lemma}\label{lem:discrete_time_necessary_app} Define the Hamiltonian
	associated with problem~\eqref{eq:discretized_pb} as
\[
		\ham_t(\state_t, \ctrl_t, \costate_{t+1}) =  \costate_{t+1}^\top \contdyn_t(\state_t, \ctrl_t) -\cost_t(\state_t, \ctrl_t)
\]
	A trajectory $\state_0, \ldots, \state_\horizon \in \reals^{\dimstate}$ and  a
	sequence of controls $\ctrl_0,\ldots, \ctrl_{\horizon-1} \in
	\reals^{\dimctrl}$ are optimal if there exists $\lambda_1, \ldots,
	\lambda_{\horizon} \in \reals^{\dimstate}$ such that
	\begin{flalign}
\quad \state_{t+1} - \state_t  &= \nabla_{\lambda_{t+1}} 	\ham_t(\state_t,
\ctrl_t, \costate_{t+1}) \quad \mbox{for all}\  t\in \{0, \ldots, \horizon-1\},
\ \mbox{with} \ \state_0 = \initstate \tag{D1} \label{eq:disc_cond1} & \\
\quad \lambda_{t+1} - \lambda_t  & = - \nabla_{\state_t} \ham_t(\state_t,
\ctrl_t, \costate_{t+1}) \quad  \mbox{for all} \ t\in \{1, \ldots, \horizon-1\},
\  \mbox{with} \ \lambda_\horizon = -\nabla \cost_{\horizon} (\state_\horizon)
\tag{D2} \label{eq:disc_cond2} & \\
\quad 0 & =  \nabla_{\ctrl_t}\ham_t(\state_t, \ctrl_t, \costate_{t+1})
\hspace{17pt} \mbox{for all} \ t\in \{0, \ldots, \horizon-1\}. \tag{D3}
\label{eq:disc_cond3} &
	\end{flalign}
\end{lemma}
\begin{proof}
	Necessary optimality conditions are given by considering stationary points of
	the Lagrangian~\citep{bertsekas1997nonlinear}. The Lagrangian of
	problem~\eqref{eq:discretized_pb} is given for $\costates = (\costate_1;
	\ldots; \costate_{\horizon})^\top$, $\states =(\state_1; \ldots;
	\state_\horizon)$, $\ctrls=(\ctrl_0; \ldots; \ctrl_{\horizon-1})$ as, for
	$\state_0 = \bar \state_0$ fixed,  
	\begin{align*}
		L(\states, \ctrls, \costates) & = \sum_{t=0}^{\horizon-1} \cost_t(\state_t, \ctrl_t) + \sum_{t=0}^{\horizon-1} \costate_{t+1}^\top (\state_{t+1} - \state_t - \contdyn_t(\state_t, \ctrl_t)) + \cost_\horizon(\state_\horizon) \\ 
		& = \sum_{t=0}^{\horizon-1} \cost_t(\state_t, \ctrl_t) + \sum_{t=1}^{\horizon-1} \left(\state_t^\top (\lambda_t - \lambda_{t+1}) -\lambda_{t+1}^\top  \contdyn_t(\state_t, \ctrl_t) \right) \\
            & \quad  + \cost_\horizon(\state_\horizon) + \lambda_\horizon^\top \state_\horizon  - \lambda_1^\top (\state_0 + \contdyn_0(\state_0, \ctrl_0)).
	\end{align*}
	We have then, for $t\in \{0, \ldots, \horizon-1\} $,
	\begin{align*}
		\nabla_{\costate_{t+1}} L(\states, \ctrls, \costates)  = 0 \iff \state_{t+1} - \state_t & = \contdyn_t(\state_t, \ctrl_t) \\
        & = \nabla_{\costate_{t+1}} \ham_t(\state_t, \ctrl_t, \costate_{t+1}),
        \end{align*}
        \begin{align*}
		\nabla_{\ctrl_t} L(\states, \ctrls, \costates)  = 0 \iff 0 & = -\nabla_{\ctrl_t}\contdyn_t(\state_t, \ctrl_t) \lambda_{t+1} + \nabla_{\ctrl_t} \cost_t(\state_t, \ctrl_t) \\
        & = -\nabla_{\ctrl_t} \ham_t(\state_t, \ctrl_t, \costate_{t+1}),
	\end{align*} 
	We have, for $t\in \{1, \ldots, \horizon-1\} $,
	\begin{align*}
		\nabla_{\state_t} L(\states, \ctrls, \costates)  = 0 \iff \costate_{t+1} - \costate_t & = -\nabla_{\state_t}\contdyn_t(\state_t, \ctrl_t) \lambda_{t+1} + \nabla_{\state_t} \cost_t(\state_t, \ctrl_t) \\
        & = -\nabla_{\state_t} \ham_t(\state_t, \ctrl_t, \costate_{t+1}), 
	\end{align*}
	Finally, for $t=\horizon$, we have $\nabla_{\state_\horizon} L(\states, \ctrls, \costates)
	= 0 \iff \nabla \cost_\horizon(\state_\horizon) + \lambda_\horizon = 0$. 
\end{proof}

\paragraph{Common points and discrepancies between continuous and discrete time}
The first two necessary optimality conditions~\eqref{eq:disc_cond1}
and~\eqref{eq:disc_cond2} for the discretized problem correspond to the
discretizations of the first two necessary optimality
conditions~\eqref{eq:cont_cond1} and~\eqref{eq:cont_cond2} for the continuous
time problem. The third condition differs since, in discrete time, the control
variables only need  to be stationary points of the Hamiltonian. One may wonder
whether condition~\eqref{eq:disc_cond3} could be replaced by a stronger
necessary optimality condition of the form
\begin{equation}
	\ctrl_t \in \argmax_{\ctrl \in \reals^\dimctrl}
  \ham_t(\state_t, \ctrl_t, \costate_{t+1}). \tag{D4} \label{eq:disc_cond3prim}
\end{equation}
If the Hamiltonian is convex w.r.t. to the control variable, i.e.,
$\ham_t(\state_t, \cdot, \costate_{t+1})$ is concave. If, e.g., the costs
$\cost_t(\state_t, \cdot)$ are convex and if the dynamics are affine input of
the form $\contdyn_t(\state_t, \ctrl_t) = a_t(\state_t) + B_t(\state_t)\ctrl_t$,
then  condition~\eqref{eq:disc_cond3} is equivalent to
condition~\eqref{eq:disc_cond3prim}. However, generally,
condition~\eqref{eq:disc_cond3prim} is not a necessary optimality condition for
the discrete-time control problem as shown in the
counter-example~\ref{exm:discrete_vs_continuous}. 
\begin{example}\label{exm:discrete_vs_continuous}
	Consider the continuous time control problem  
	\begin{align*}
		\min_{\state(t), \ctrl(t) \in \mathcal{C}([0, 1], \reals)}  &  \int_0^1 (a \state(t)^2 - \ctrl(t)^2)dt + a \state(1)^2 \\
		\mbox{subject to}  \quad& \dot \state(t) = \ctrl(t), \quad x(0) = 0,
	\end{align*}
	for some $a>0$ and the associated discrete time control problem, for an Euler
	scheme with discretization $\Delta = 1/\horizon$, 
	\begin{align*}
		\min_{\substack{\state_0, \ldots, \state_\horizon \in \reals \\ 
		\ctrl_0, \ldots, \ctrl_{\horizon-1} \in \reals} } &   \sum_{t=0}^{\horizon-1}  \Delta (a \state_t^2 - \ctrl_t^2) + a \state_\horizon^2\\
		\mbox{subject to} \quad & \state_{t+1} = \state_t + \Delta \ctrl_t, \quad \state_0 = 0.
	\end{align*}
	
	The Hamiltonians in continuous time, $\ham(\state(t), \ctrl, \costate(t)) =
	\costate(t)^\top \ctrl + \ctrl^2 - a\state(t)^2$, and in discrete time,
	$\ham_t(\state_t, \ctrl_t, \costate_{t+1}) = \Delta \costate_{t+1}^\top \ctrl
	+ \ctrl^2 -a\state_t^2$, are both strongly convex in $\ctrl$ such that neither
	condition~\eqref{eq:cont_cond3} nor~\eqref{eq:disc_cond3prim} can be satisfied. 
	
	According to Theorem~\ref{thm:pontryagin}, this means that the continuous time
	control problem has no solution. This can be verified by expressing the
	continuous time control problem uniquely in terms of the trajectory $x(t)$ as 
	\[
	\min_{x(t): x(0)=0}  \left\{C(x) =  \int_0^1 (a \state(t)^2 - \dot x(t)^2)dt + a \state(1)^2\right\}.
	\]
	By considering functions of the form $x_k(t) = \exp(t^k) -1$, we observe that
	the corresponding costs are unbounded below, namely, $C(x_k) \leq 2a (\exp(1)
	-1)^2 - k^2/(2k-1) \underset{k \rightarrow + \infty}{\rightarrow} -\infty$
	which shows that the problem is unbounded below and has no minimizer. 
	
	On the other hand, the discrete time control problem  can be expressed in
	terms of the control variables as 
	\[
	\min_{\ctrls \in \reals^{\horizon\dimctrl}} a\Delta^2 \ctrls^\top D^{-\top} J D^{-1} \ctrls - \Delta \|\ctrls\|_2^2,
	\]
	where $J = \diag(\Delta, \ldots, \Delta, 1)$, $D = \idm -
	\sum_{t=1}^{\horizon-1} e_{t+1}e_t^\top$. We have, using that $\Delta<1$ for
	the first inequality, $\ctrls^\top D^{-\top} J D^{-1} \ctrls \geq \Delta
	\|D^{-1} \ctrls\|_2^2 \geq \Delta \sigma_{\min}(D^{-1})^2 \|\ctrls\|_2^2  =
	\Delta\|\ctrls\|_2^2 /\|D\|_2^2 \geq \Delta\|\ctrls\|_2^2 /4$. Hence, for any
	$a$ such that $a\Delta^2/4>1$, the above problem is strongly convex and has a
	unique solution. Yet, if condition~\eqref{eq:disc_cond3prim} was necessary the
	discrete control problem should not have a solution since
	condition~\eqref{eq:disc_cond3prim} cannot be satisfied. 
\end{example}

\paragraph{Alternative derivation}
Necessary optimality conditions for the discretized
problem~\eqref{eq:discrete_pb_ctrl_cost} can be derived from
Lemma~\ref{lem:discrete_time_necessary_app} using the correspondence $\state_t +
\contdyn_t(\state_t, \ctrl_t) = \dyn_t(\state_t, \ctrl_t)$. We can also derive
the necessary optimality conditions simply by considering a sequence of control
variables $\ctrls = (\ctrl_0; \ldots; \ctrl_{\horizon-1})$ that minimize the
objective $\obj$ defined in~\eqref{eq:obj_full}.

Namely, the gradient of the objective $\obj$ on $\ctrls = (\ctrl_0; \ldots;
\ctrl_{\horizon-1})$ can be obtained by gradient back-propagation as follows.
First the states corresponding to the control variables are computed in a forward pass 
\[
  \state_{t+1} = \dyn_t(\state_t, \ctrl_t),
  \quad \mbox{for} \ t \in \{0, \ldots, \horizon-1\}
\]
starting from $\state_0 = \initstate$. Then the gradients $\nabla \obj(\ctrls) =
(g_0;\ldots; g_{\horizon-1})$ are computed in a backward pass as
\begin{align*}
  \lambda_\horizon & = \nabla \cost_\horizon(\state_\horizon), \\
  \lambda_t & = 
  \nabla_{\state_t} \dyn_t(\state_t, \ctrl_t)^\top \lambda_{t+1} 
  + \nabla_{\state_t} \cost_t(\state_t, \ctrl_t),
  \quad \mbox{for} \ t \in \{0, \ldots, \horizon-1\}, \\
  g_t & = \nabla_{\ctrl_t} \dyn_t(\state_t, \ctrl_t)^\top \lambda_{t+1}
  + \nabla_{\ctrl_t} \cost_t(\state_t, \ctrl_t),
  \quad \mbox{for} \ t \in \{0, \ldots, \horizon-1\}.
\end{align*}
One easily verifies then that having $g_t=0$ for all $t \in \{0, \ldots,
\horizon-1\}$ correspond to the optimality conditions presented in
Lemma~\ref{lem:discrete_time_necessary_app} with the correspondence $\state_t +
\contdyn_t(\state_t, \ctrl_t) = \dyn_t(\state_t, \ctrl_t)$.

\subsection{Sufficient Optimality Conditions}
Sufficient optimality conditions can also be derived following sufficient
optimality conditions in continuous time presented
by~\cite{mangasarian1966sufficient, arrow1968control, kamien1971sufficient}. We
start by rewriting problem~\eqref{eq:discrete_pb_ctrl_cost} as
\begin{align}
	\min_{\substack{\state_0,\ldots, \state_\horizon \in \reals^\dimstate \\ \delta_0, \ldots, \delta_{\horizon-1} \in \reals^\dimstate }}  & \sum_{t=0}^{\horizon-1} \intercost_t(\state_t,\delta_t)  + \finalcost(\state_\horizon),  \label{eq:bolza_pb_discrete} \quad \mbox{where} \ 	\intercost_t(\state_t, \delta_t) =  \inf_{\substack{\ctrl\in \reals^\dimctrl\\  \delta_t = \dyn(\state_t, \ctrl_t) - \state_t}}   \cost_t(\state_t, \ctrl)\\
	\mbox{subject to}  \quad & \delta_t = \state_{t+1} - \state_t, \    \state_0 =  \initstate. \nonumber
\end{align}
Sufficient conditions can be expressed through the true Hamiltonian, presented
by~\cite{clarke1979optimal}, and defined as the convex conjugate of
$\intercost_t(\state_t, \cdot)$, i.e., for $\state_t, \costate_{t+1} \in
\reals^\dimstate$,
\begin{align*}
   \trueham_t(\state_t, \costate_{t+1}) & = \sup_{\delta\in \reals^\dimstate} \costate_{t+1}^\top \delta - \intercost_t(\state_t, \delta) \\
   & = \sup_{\ctrl \in \reals^\dimctrl} \costate_{t+1}^\top (\dyn(\state_t, \ctrl_t) - \state_t) - \cost_t(\state_t, \ctrl)  \\
   & = \sup_{\ctrl \in \reals^\dimctrl} \ham_t(\state_t, \ctrl, \costate_{t+1}), 
\end{align*}
where 
\[
		\ham_t(\state_t, \ctrl_t, \costate_{t+1}) =  \costate_{t+1}^\top (\dyn_t(\state_t, \ctrl_t) - \state_t) -\cost_t(\state_t, \ctrl_t)
\]
is the Hamiltonian associated with problem~\eqref{eq:discrete_pb_ctrl_cost}.
\begin{theorem}\label{thm:suff_cond_bolza} Assume that $\intercost_t$ defined
	in~\eqref{eq:bolza_pb_discrete} is such that $\intercost_t(\state_t, \cdot)$
	is convex for any $\state_t$ and $\cost_\horizon$ is convex. If there exist
	$\state_0^*, \ldots, \state_{\horizon}^*$ and $\costate_1^*, \ldots,
	\costate_\horizon^*$ such that $\trueham_t(\cdot, \costate_{t+1}^*)$ is
	concave and
	\begin{align}
		\costate_t^* - \costate_{t+1}^* & \in \partial_{\state_t} \trueham_t(\state_t^*, \costate_{t+1}^*)   \quad  \hspace{9pt} \mbox{for} \ t \in \{1, \ldots, \horizon-1\}, \hspace{19pt}	\costate_\horizon^* = \nabla \finalcost(\state_\horizon^*) \label{eq:cond_bolza_1}\\
		\state_{t+1}^* - \state_t^* &  \in  \partial_{\costate_{t+1}} \trueham_t(\state_t^*, \costate_{t+1}^*) \quad \mbox{for} \ t \in \{0, \ldots, \horizon-1\}, \qquad \state_0^* =  \initstate,  \label{eq:cond_bolza_2}		
	\end{align}
	then $\state_0^*, \ldots, \state_{\horizon}^*$ is an optimal trajectory
	for~\eqref{eq:bolza_pb_discrete}. 
  Conditions~\eqref{eq:cond_bolza_1} and \eqref{eq:cond_bolza_2} amount to the existence of $\ctrl_t^* \in
\argmax_{\ctrl \in \reals^\dimctrl} \costate_{t+1}^\top(\dyn(\state_t, \ctrl_t)
- \state_t) -\cost_t(\state_t, \ctrl) , \auxctrl_t^* \in \argmax_{\auxctrl \in
\reals^\dimctrl} \costate_{t+1}^\top\contdyn_t(\state_t, \auxctrl) -
\cost_t(\state_t, \auxctrl) $ such that 
\begin{align*}
	\costate_t^* - \costate_{t+1}^* & =  \nabla_{\state_t} \contdyn_t(\state_t^*, \auxctrl_t^*) \costate_{t+1}^* - \nabla_{\state_t} \cost_t(\state_t^*, \auxctrl_t^*), \qquad
		\state_{t+1}^* - \state_t^*  = \contdyn_t(\state_t^*, \ctrl_t^*).
\end{align*}
\end{theorem}
\begin{proof}
	Since   $\intercost_t(\state_t, \cdot)$ is convex for any $\state_t$,
	problem~\eqref{eq:bolza_pb_discrete} can be rewritten
	\begin{align}\label{eq:bolza}
	&	\min_{\substack{\state_1,\ldots, \state_\horizon \in \reals^\dimstate \\ \state_0 = \hat \state_0}} \sup_{\costate_1, \ldots, \costate_\horizon \in \reals^\dimstate} \quad  \sum_{t=0}^{\horizon-1} \left({\costate_{t+1}}^\top(\state_{t+1} - \state_t) - \trueham_t(\state_t, \costate_{t+1})\right) + \finalcost(\state_\horizon).
	\end{align}
The above problem can be written as $\min_{\states \in
\reals^{\horizon\dimstate}} \sup_{\costates \in \reals^{\horizon\dimstate}}
c(\states, \costates)$ with $c(\states, \cdot)$ concave for any $\states$. The
assumptions amount to consider $\states^*, \costates^*$ such that (i) $0 \in
\partial_{\costates^*}  c(\states^*, \costates^*)$, (ii) $c(\cdot, \costates^*)$
convex and $0 \in \partial_{\states^*} c(\states^*, \costates^*)$. Then for any
$\states \in \reals^{\horizon \dimstate}$, 
\[
\sup_{\costates \in \reals^{\horizon\dimstate}} c(\states, \costates) \geq c(\states,  \costates^*) \stackrel{(ii)}{\geq} c(\states^*, \costates^*) \stackrel{(i)}{=} \sup_{\costates\in \reals^{\horizon\dimstate}} c(\states^*, \costates).
\]
Hence, $\states^* \in \argmin_{\states \in \reals^{\horizon \dimstate}}
\sup_{\costates \in \reals^{\horizon\dimstate}} c(\states, \costates)$, that is,
$\state_0^*, \ldots, \state_{\horizon}^*$ is an optimal trajectory.
\end{proof}
Theorem~\ref{thm:suff_cond_bolza} provides generic sufficient optimality
conditions for problem of the form~\eqref{eq:discrete_pb_ctrl_cost} inspired by
the continuous time viewpoint. However, as noted by~\citet{polak2011role},
optimality conditions in continuous time may not be informative for discrete
time counterparts. This is illustrated here by the difficulty to verify
convexity of $\intercost_t(\state_t, \cdot)$ or the concavity of
$\trueham_t(\cdot, \costate_{t+1}^*)$.

\section{Generic Convergence Results}\label{app:gen_cvg}
In this section, we present the convergence analysis for generic problems of the
form~\eqref{eq:discrete_pb_ctrl_cost}, recalled below. 
\begin{align}
	\min_{\substack{\ctrl_0,\ldots, \ctrl_{\horizon-1} \in \reals^\dimctrl\\ \state_0, \ldots, \state_\horizon \in \reals^\dimstate}} \quad 
	& \sum_{t=0}^{\horizon-1} \cost_t(\state_t, \ctrl_t) + \cost_\horizon(\state_\horizon)\label{eq:discrete_pb_ctrl_cost_app}\\
	\mbox{subject to} \quad & 
	\state_{t+1} = \dyn_t(\state_t, \ctrl_t) \quad \mbox{for} \ t \in \{0, \ldots, \horizon-1\}, 
	\qquad \state_0 = \initstate.\nonumber
\end{align}
We decompose the problem in a composition by defining first the control of
$\horizon$ discrete dynamics below.
\begin{definition}\label{def:traj_func_time_varying} We define the control of
  $\horizon$ discrete time dynamics $(\dyn_t:\reals^\dimstate \times
  \reals^\dimctrl \rightarrow \reals^\dimstate)_{t=0}^{\horizon-1}$ as the
  function $\traj: \reals^\dimstate \times \reals^{\horizon\dimctrl} \rightarrow
  \reals^{\horizon \dimstate}$, which, given an initial point $\state_0 \in
  \reals^{\dimstate}$ and a sequence of controls $\ctrls =
  (\ctrl_0;\ldots;\ctrl_{\horizon-1})\in \reals^{\horizon\dimctrl}$, outputs the
  corresponding trajectory $\state_1, \ldots, \state_\horizon$, i.e., 
	\begin{align}
		\label{eq:traj_time_varying}
		\traj(\state_0, \ctrls) & = ( \state_1;\ldots;\state_\horizon) \\
		\mbox{s.t.} \quad \state_{t+1} &= \dyn_t(\state_t, \ctrl_t) \quad \mbox{for} \ t \in \{0,\ldots, \horizon-1\}. \nonumber
	\end{align}
\end{definition}
We consider then the cost of a sequence of control variables $\ctrls = (\ctrl_0;
\ldots; \ctrl_{\horizon-1}) \in \reals^{\horizon \dimctrl}$ and associated
trajectory $\states = (\state_1; \ldots; \state_\horizon) \in \reals^{\horizon
\dimstate}$, as
\begin{align*}
  \fullcost(\states, \ctrls) & = \sum_{t=0}^{\horizon-1} \cost_t(\state_t, \ctrl_t) + \cost_\horizon(\state_\horizon).
\end{align*}
Problem~\eqref{eq:discrete_pb_ctrl_cost_app} amounts then to a compositional
problem of the form
\begin{equation}\label{eq:comp_gen}
  \min_{\ctrls \in \reals^{\horizon \dimctrl}}
  \fullobj(\ctrls),
  \ \mbox{for} \
  \fullobj(\ctrls) = \fullcost(\fullaugtraj(\ctrls)), 
  \
  \fullaugtraj(\ctrls) = (\traj(\ctrls, \initstate), \ctrls)
\end{equation}
We make the following regular smoothness assumptions.
\begin{assumption}\label{asm:gen_cvg} We consider all costs $\cost_t$ to be
  $\lip_{\cost}$ Lipschitz continuous, with $\smooth_{\cost}$ Lipschitz
  continuous gradients and $\smoothess_{\cost}$ Lipschitz continuous Hessians,
  then the cost function $\fullcost$ is $\lip_{\cost}$
  Lipschitz continuous, with $\smooth_{\cost}$ Lipschitz
  continuous gradients and $\smoothess_{\cost}$ Lipschitz
  continuous Hessians.

  We consider the dynamics to be Lipschitz-continuous with Lipschitz-continuous
  Jacobians such that the control $\traj$ of these dynamics is
  $\lip_{\traj}$ Lipschitz continuous with $\smooth_{\traj}$ Lipschitz
  continuous Jacobians as detailed in
  Lemma~\ref{lem:smooth_traj_from_dyn_time_varying}. The augmented function
  $\fullaugtraj$ is then $\lip_{\fullaugtraj}= \sqrt{\lip_{\traj}^2 + 1}$
  Lipschitz continuous with $\smooth_{\fullaugtraj}=\sqrt{\smooth_{\traj}^2 +
  1}$ Lipschitz continuous Jacobians.
\end{assumption}

\textcolor{BrickRed}{ \textbf{Note:} The notations $\fullcost, \fullobj,
\fullaugtraj$ used in Appendix~\ref{app:gen_cvg} for problems of the
form~\eqref{eq:discrete_pb_ctrl_cost_app} pertain only to Appendix~\ref{app:gen_cvg},
Lemma~\ref{lem:bounded_policies_gen_cvg} and
Corollary~\ref{cor:approx_ilqrs_gen}. In particular, the definition of
$\augtraj$ and its smoothness properties differ here than from the main text
(see~\eqref{eq:total_cost}).}

\subsection{Generic Convergence Results for ILQR}

As explained in the main text for the problem without control costs
(Section~\ref{sec:cvg}), the ILQR algorithm amounts to linearizing $\fullaugtraj$
and taking a quadratic approximation of $\fullcost$ such that, provided that the
minimum exists,
\begin{align}
  \LQR_\reg(\obj)(\ctrls) 
  & = \argmin_{\auxctrls \in \reals^{\horizon\dimctrl}} 
  \qua_{\fullcost}^{\fullaugtraj(\ctrls)}(\lin_{\fullaugtraj}^{\ctrls}(\auxctrls)) 
  + \frac{\reg}{2} \|\auxctrls\|_2^2\nonumber \\
  & = -(\nabla \fullaugtraj(\ctrls)\nabla^2 \fullcost(\fullaugtraj(\ctrls)) \nabla \fullaugtraj(\ctrls)^\top +\reg \idm)^{-1}  \nabla \fullaugtraj(\ctrls)\nabla \fullcost(\fullaugtraj(\ctrls)),\label{eq:rggn_oracle_app}
\end{align}
where $\lin_\fullaugtraj^{\ctrls}$ and $\qua_\fullcost^{\fullaugtraj(\ctrls)}$
are the linear and quadratic expansions of, respectively, the control in
$\horizon$ steps around $\ctrls$ and the total costs around
$\fullaugtraj(\ctrls)$. In other words the ILQR algorithm is a generalized
Gauss-Newton algorithm that exploits the compositional structure of the problem.
Lemma~\ref{lem:gen_stat_cvg_rggn} below presents then the convergence to
stationary point of the ILQR algorithm from the lens of a generalized
Gauss-Newton algorithm. Lemma~\ref{lem:local_gen_cvg_rggn} presents local
convergence guarantees.

\begin{lemma}\label{lem:gen_stat_cvg_rggn}
  Under assumption~\ref{asm:gen_cvg}, provided that the regularization $\reg$
  satisfies 
  \[\nu
  \geq 
  \max\left\{
    2\lip_{\fullaugtraj}^2 \smooth_\fullcost,
    \frac{\smooth_{\fullaugtraj} \lip_{\fullcost}}{2}
    \gamma\left(\frac{\smooth_\fullaugtraj\lip_\fullcost}{4\smooth_\fullcost \lip_\fullaugtraj(1+\beta)}\right)
  \right\},
  \]
  for $\gamma(x) = 1 + \sqrt{1+ 1/x}$ and $\beta = \smoothess_\fullcost \lip_\fullaugtraj^2/(3 \smooth_\fullaugtraj \smooth_\fullcost)$,
  the iterations of the \ref{eq:ilqr_algo} algorithm satisfy
  \[
    \min_{k\in \{0, \ldots, K\}} \|\nabla \fullobj(\ctrls^{(k)})\|_2 \leq \sqrt{\frac{2(\lip_{\fullaugtraj}^2\smooth_{\fullcost} + \nu)\left(\fullobj(\ctrls^{(0)}) - \min_{\ctrls \in \reals^{\horizon \dimctrl}} \fullobj(\ctrls)\right)}{K+1}},
  \]
  for $\fullobj, \fullcost, \fullaugtraj$ defining the objective in~\eqref{eq:comp_gen}.
\end{lemma}
\begin{proof}
  Using Lemma~\ref{lem:bound_approx_self_concord} adapted to
  Assumption~\ref{asm:gen_cvg}, we have that for any $\ctrls, \auxctrls \in
  \reals^{\horizon\dimctrl}$,  
  \begin{align*}
    (\fullcost\circ \fullaugtraj)(\ctrls + \auxctrls)
    & \leq (\fullcost \circ \fullaugtraj)(\ctrls) + 
    \qua_{\fullcost}^{\fullaugtraj(\ctrls)}\circ \lin_{\fullaugtraj}^\ctrls(\auxctrls) 
    + \frac{a_1 + a_2 \|\auxctrls\|_2}{2} \|\auxctrls\|_2^2,
  \end{align*}
  for 
  $
  a_1 = \smooth_{\fullaugtraj} \lip_{\fullcost}, a_2 =
  \smoothess_{\fullcost}\lip_{\fullaugtraj}^3/3 +
  \smooth_{\fullaugtraj}\smooth_{\fullcost} \lip_{\fullaugtraj}.
  $ 
  For $\reg > \lip_{\fullaugtraj}^2 \smooth_{\fullcost}$ the minimizer
  in~\eqref{eq:rggn_oracle_app} is uniquely defined. The oracle $\auxctrls =
  \LQR_\reg(\obj)(\ctrls)$ satisfies then $\|\auxctrls\|_2 \leq 
  \lip_{\fullaugtraj} \lip_{\fullcost}/\reg$. For 
  $\reg 
  \geq 
  a_1 (1+ \sqrt{1+ 4
  a_2 \lip_{\fullaugtraj} \lip_{\fullcost}/a_1^2})/2$, 
  we have $a_1 +
  a_2 \lip_{\fullaugtraj} \lip_{\fullcost}/\reg \leq \reg$. Expanding $a_1, a_2$, the condition $\reg \geq a_1 (1+ \sqrt{1+ 4
  a_2 \lip_{\fullaugtraj} \lip_{\fullcost}/a_1^2})/2$ reads
  \[
  \nu \geq
    \frac{\smooth_{\fullaugtraj} \lip_{\fullcost}}{2}
    \gamma\left(\frac{\smooth_\fullaugtraj\lip_\fullcost}{4\smooth_\fullcost \lip_\fullaugtraj(1+\beta)}\right)
  \]
  for $\gamma(x) = 1 + \sqrt{1+ 1/x}$ and $\beta = \smoothess_\fullcost \lip_\fullaugtraj^2/(3 \smooth_\fullaugtraj \smooth_\fullcost)$.
  Then, for
  $\auxctrls = \LQR_\reg(\obj)(\ctrls)$, we have
  \begin{align*}
    (\fullcost\circ \fullaugtraj)(\ctrls + \auxctrls)
    & \leq (\fullcost \circ \fullaugtraj)(\ctrls) + 
    \qua_{\fullcost}^{\fullaugtraj(\ctrls)}\circ \lin_{\fullaugtraj}^\ctrls(\auxctrls) 
    + \frac{\reg}{2} \|\auxctrls\|_2^2 \\
    & = (\fullcost \circ \fullaugtraj)(\ctrls) 
    - \frac{1}{2} \nabla (\fullcost \circ \fullaugtraj)(\ctrls)^\top
    (\nabla \fullaugtraj(\ctrls)\nabla^2 \fullcost(\fullaugtraj(\ctrls)) \nabla \fullaugtraj(\ctrls)^\top + \reg \idm)^{-1}
    \nabla (\fullcost \circ \fullaugtraj)(\ctrls) \\
    & \leq (\fullcost \circ \fullaugtraj)(\ctrls) - \frac{1}{2(\lip_{\fullaugtraj}^2\smooth_{\fullcost} + \nu)}\|\nabla (\fullcost \circ \fullaugtraj)(\ctrls)\|_2^2.
  \end{align*}
  We have then in terms of the ILQR iterations,
  \begin{align*}
    \frac{1}{2(\lip_{\fullaugtraj}^2\smooth_{\fullcost} + \nu)}\|\nabla (\fullcost \circ \fullaugtraj)(\ctrls^{(k)})\|_2^2 \leq 
    (\fullcost \circ
  \fullaugtraj)(\ctrls^{(k)}) - (\fullcost \circ
  \fullaugtraj)(\ctrls^{(k+1)}).
  \end{align*}
  Summing over $k=0, \ldots, K-1$ and taking the minimum on the left-hand side gives the result.
\end{proof}

\begin{lemma}\label{lem:local_gen_cvg_rggn}
  Consider Assumption~\ref{asm:gen_cvg}. Assume in addition that the dynamics are twice differentiable with Lipschitz-continuous Hessian such that $\fullobj = \fullcost \circ \fullaugtraj$ has
  $\smoothess_{\fullobj}$-Lipschitz continuous Hessians.
  Consider $\ctrls^{(k)}$ to be close to a minimum $\ctrls^*$ of $\fullobj = \fullcost \circ \fullaugtraj$ satisfying $\mu_{\fullobj} = \lambda_{\min}(\nabla^2 \fullobj(\ctrls^*)) > 0$. If  $\|\ctrls^{(k)} - \ctrls^*\|_2 \leq \strgcvx_\fullobj/\smoothess_\fullobj$, and the regularization satisfies 
  \[
    \reg \geq \kappa_\fullobj \max\{5\smooth_{\fullcost}\lip_{\fullaugtraj}^2 , 8 \smooth_{\fullobj}\},
  \]
  for $\kappa_\fullobj = \smooth_\fullobj/\strgcvx_\fullobj$,
  then the iterations of the \ref{eq:ilqr_algo} algorithm converge linearly to $\ctrls^*$ as 
  \[
    \|\ctrls^{(k+1)} -\ctrls^*\|_2
    \leq  \left(1 - \frac{\strgcvx_\fullobj}{16\reg}\right)\|\ctrls^{(k)}-\ctrls^*\|_2.
  \]
\end{lemma}
\begin{proof}
  Denote $R=\|\ctrls^{(k)} - \ctrls^*\|_2$ and $H = \int_{0}^{1} \nabla^2 \fullobj(\ctrls^* + t(\ctrls^{(k)} -\ctrls^*))dt$ such that 
  $\nabla \fullobj(\ctrls^{(k)})
  = \nabla \fullobj(\ctrls^{(k)}) - \nabla \fullobj(\ctrls^*)
  = H(\ctrls^{(k)} - \ctrls^*)$.
  Note that $(\mu_{\fullobj}/2) \idm \leq (\mu_{\fullobj} - \smoothess_{\fullobj}R/2)\idm \preceq H \preceq \smooth_{\fullobj}\idm$, where $\smooth_{\fullobj} \leq \smooth_{\fullcost} \lip_{\fullaugtraj}^2 + \smooth_{\fullaugtraj} \lip_{\fullcost}$ is the Lipschitz continuity parameter of $\fullobj = \fullcost \circ \fullaugtraj$.

  Denote $P = \nabla \fullaugtraj(\ctrls)\nabla^2 \fullcost(\fullaugtraj(\ctrls)) \nabla \fullaugtraj(\ctrls)^\top +\reg \idm$.
  We have 
  \begin{align}
    \|\ctrls^{(k+1)} -\ctrls^*\|_2^2 
    & \leq \|\ctrls^{(k+1)} - \ctrls^{(k)}\|_2^2 
    + (\ctrls^{(k+1)} - \ctrls^{(k)})^\top (\ctrls^{(k)} - \ctrls^*)
    + \|\ctrls^{(k)} - \ctrls^*\|_2^2 \nonumber \\
    & \leq \|P^{-1}H(\ctrls^{(k)} - \ctrls^*)\|_2^2
    - (\ctrls^{(k)} - \ctrls^*)^\top H P^{-1} (\ctrls^{(k)} - \ctrls^*)
    + \|\ctrls^{(k)} - \ctrls^*\|_2^2 \nonumber \\
    & \leq (1+ \smooth_{\fullobj}^2\reg^{-2})\|\ctrls^{(k)} - \ctrls^*\|_2^2 
    - (\ctrls^{(k)} - \ctrls^*)^\top H P^{-1} (\ctrls^{(k)} - \ctrls^*). \label{eq:local_cvg_gen_expansion}
  \end{align}
  Denote $C = \nabla \fullaugtraj(\ctrls)\nabla^2 \fullcost(\fullaugtraj(\ctrls)) \nabla \fullaugtraj(\ctrls)^\top$. For $\reg$ such that $\|C\|_2/\reg <1$, we have
  \begin{align*}
    P^{-1} = \reg^{-1}(\idm + \reg^{-1}C)^{-1} 
    = \reg^{-1} \idm + \reg^{-1} \sum_{j=1}^{+\infty} \reg^{-j} C^j 
  \end{align*}
  Denoting $G = \sum_{j=1}^{+\infty} \reg^{-j} C^j$, for $\|C\|_2/\reg \leq (\kappa^{-1}/4)/(1+\kappa^{-1}/4)$ with $\kappa = \smooth_{\fullobj}/\strgcvx_{\fullobj} \geq 1$,  we have
  $\|G\| \leq \kappa^{-1}/4$.
  We then have
  \begin{align*}
    -(\ctrls^{(k)} - \ctrls^*)^\top H P^{-1} (\ctrls^{(k)} - \ctrls^*) 
    & = -\reg^{-1} (\ctrls^{(k)} - \ctrls^*)^\top H (\ctrls^{(k)} \\
    & \quad - \ctrls^*) 
    -\reg^{-1} (\ctrls^{(k)} - \ctrls^*)^\top GH (\ctrls^{(k)} - \ctrls^*) \\
    & \leq - \frac{\reg^{-1} \mu_{\fullobj}}{2}\|\ctrls^{(k)} - \ctrls^*\|_2^2 
    + \frac{\reg^{-1} \smooth_{\fullobj}\kappa^{-1}}{4}\|\ctrls^{(k)} - \ctrls^*\|_2^2 \\
    & = - \frac{\reg^{-1} \mu_{\fullobj}}{4} \|\ctrls^{(k)} - \ctrls^*\|_2^2.
  \end{align*}
  Plugging the above equation into~\eqref{eq:local_cvg_gen_expansion}, we have that if $\reg$ satisfies in addition $\reg \geq 8 \smooth_{\fullobj}\kappa_{\fullobj}$, the iterates of the ILQR algorithm converge linearly to $\ctrls^*$ as
  \begin{align*}
    \|\ctrls^{(k+1)} -\ctrls^*\|_2^2 
    & \leq (1 + \smooth_{\fullobj}^2 \reg^{-2} - \reg^{-1} \mu_{\fullobj}/4)
    \|\ctrls^{(k)} -\ctrls^*\|_2^2 
    \leq \left(1 - \frac{\strgcvx_\cost}{8\reg}\right)\|\ctrls^{(k)}-\ctrls^*\|_2^2.
  \end{align*}
\end{proof}

\subsection{Generic Convergence Results for IDDP}
We analyze the convergence of the IDDP through the lens of the ILQR algorithm.
As these two algorithms differ simply by the roll-out procedure,
Lemmas~\ref{lem:approx_ilqrs} and~\ref{lem:bounded_policies_gen_cvg}, summarized in Corollary~\ref{cor:approx_ilqrs_gen}, show that their oracles differ by at most
\[
  \|\auxxctrls - \auxctrls\|_2  \leq \xi \|\auxctrls\|_2^2
\]
for $\xi$ independent of $\reg$ provided that $\reg$ is sufficiently large.

We can then show the convergence of IDDP to stationary points in Lemma~\ref{lem:iddp_stat_cvg}, as well as its local convergence behavior in Lemma~\ref{lem:iddp_local_gen_cvg}.
\begin{lemma}\label{lem:iddp_stat_cvg}
  Under Assumption~\ref{asm:gen_cvg}, provided that the regularization $\reg$ is larger than $2\lip_\fullaugtraj^2 \smooth_\fullcost$ the oracle $\auxctrls=\LQR_\reg(\fullobj)(\ctrls)$ returned by the \ref{eq:ilqr_algo} algorithm and the oracle $\auxxctrls=\LQR_\reg(\fullobj)(\ctrls)$ returned by the \ref{eq:iddp_algo} algorithm differ as 
  \[
    \|\auxctrls - \auxxctrls\|_2 \leq \xi \|\auxctrls\|_2^2,
  \] 
  for $\xi$ a constant independent of $\reg$.
  Moreover, if the regularization $\reg$ satisfies 
  \[
    \reg \geq 
    \max\left\{
      2\lip_\fullaugtraj^2 \smooth_\fullcost,
    \frac{\smooth_{\fullaugtraj} \lip_{\fullcost}}{2}
    \gamma\left(\frac{\smooth_\fullaugtraj\lip_\fullcost}{4\smooth_\fullcost \lip_\fullaugtraj(1+\beta + \beta')}\right)
    \right\}
  \]
  for $\gamma(x) = 1 + \sqrt{1+ 1/x}$ and $\beta = \smoothess_\fullcost \lip_\fullaugtraj^2/(3 \smooth_\fullaugtraj \smooth_\fullcost)$, $\beta' = 2\lip_\cost \xi /(\smooth_\fullcost\smooth_\fullaugtraj)$,
  the iterations of the \ref{eq:iddp_algo} algorithm satisfy
  \[
    \min_{k\in \{0, \ldots, K\}} \|\nabla \fullobj(\ctrls^{(k)})\|_2 \leq \sqrt{\frac{2(\lip_{\fullaugtraj}^2\smooth_{\fullcost} + \nu)\left(\fullobj(\ctrls^{(0)}) - \min_{\ctrls \in \reals^{\horizon \dimctrl}} \fullobj(\ctrls)\right)}{K+1}},
  \]
  for $\fullobj, \fullcost, \fullaugtraj$ defining the objective in~\eqref{eq:comp_gen}.
\end{lemma}
\begin{proof}
  To show the convergence of the \ref{eq:iddp_algo}, we consider selecting $\reg$ such that 
  \[
  \fullobj(\ctrls + \auxxctrls) \leq 
  \fullobj(\ctrls) 
  + 
  \qua_{\fullcost}^{\fullaugtraj(\ctrls)}
  \lin_{\fullaugtraj}^{\ctrls}(\auxctrls) 
  + \frac{\reg}{2}\|\auxctrls\|_2^2
  \]
  for $\auxxctrls = \DDP_\reg(\fullobj)(\ctrls)$ and $\auxctrls = \LQR_\reg(\fullobj)(\ctrls)$. 
  Using that costs and dynamics are Lipschitz continuous, we have 
  \[
    |\fullobj(\ctrls + \auxxctrls) - \fullobj(\ctrls + \auxctrls)| 
    \leq 
    \lip_{\fullcost}\lip_{\fullaugtraj}
    \|\auxxctrls - \auxctrls\|_2.
  \]
  On the other hand, by Corollary~\ref{cor:approx_ilqrs_gen} for $\reg \geq 2\smooth_\fullcost\lip_\fullaugtraj^2$, there exists a constant $\xi$ independent of $\reg$ such that $\|\auxxctrls - \auxctrls\|_2\leq \xi \|\auxctrls\|_2^2$ and so
  \[
    |\fullobj(\ctrls + \auxxctrls) - \fullobj(\ctrls + \auxctrls)| 
    \leq 
    \lip_{\fullcost}\lip_{\fullaugtraj}\xi
    \|\auxctrls\|_2^2.
  \]
  Now using Lemma~\ref{lem:bound_approx_self_concord} adapted to
  Assumption~\ref{asm:gen_cvg}, we have that for any $\ctrls, \auxctrls \in
  \reals^{\horizon\dimctrl}$,  
  \begin{align*}
    \fullobj(\ctrls + \auxctrls)
    & \leq \fullobj(\ctrls) + 
    \qua_{\fullcost}^{\fullaugtraj(\ctrls)}\circ \lin_{\fullaugtraj}^\ctrls(\auxctrls) 
    + \frac{a_1 + a_2 \|\auxctrls\|_2}{2} \|\auxctrls\|_2^2,
  \end{align*}
  for 
  $
  a_1 = \smooth_{\fullaugtraj} \lip_{\fullcost}, a_2 =
  \smoothess_{\fullcost}\lip_{\fullaugtraj}^3/3 +
  \smooth_{\fullaugtraj}\smooth_{\fullcost} \lip_{\fullaugtraj}.
  $ 
  Hence, we have 
  \[
    \fullobj(\ctrls + \auxxctrls) \leq 
    \fullobj(\ctrls) 
    + 
    \qua_{\fullcost}^{\fullaugtraj(\ctrls)}
    \lin_{\fullaugtraj}^{\ctrls}(\auxctrls) 
    + \frac{a_1 + a_3 \|\auxctrls\|_2}{2}\|\auxctrls\|_2^2  
  \]
  for $a_3 = a_2 + 2\lip_\fullcost\lip_\fullaugtraj \xi$. Selecting $\reg \geq a_1 (1+ \sqrt{1+ 4
  a_3 \lip_{\fullaugtraj} \lip_{\fullcost}/a_1^2})/2$, that is, 
  \[
  \reg \geq 
    \frac{\smooth_{\fullaugtraj} \lip_{\fullcost}}{2}
    \gamma\left(\frac{\smooth_\fullaugtraj\lip_\fullcost}{4\smooth_\fullcost \lip_\fullaugtraj(1+\beta + \beta')}\right)
  \]
  for $\gamma, \beta$ defined as in Lemma~\ref{lem:gen_stat_cvg_rggn} and $\beta' = 2\lip_\cost \xi /(\smooth_\fullcost\smooth_\fullaugtraj)$
  ensures that $a_1 + a_3 \|\auxctrls\|_2\leq \reg$. So we get that 
  \[
    \fullobj(\ctrls + \auxxctrls) \leq 
  \fullobj(\ctrls) 
  + 
  \qua_{\fullcost}^{\fullaugtraj(\ctrls)}
  \lin_{\fullaugtraj}^{\ctrls}(\auxctrls) 
  + \frac{\reg}{2}\|\auxctrls\|_2^2.
  \] 
  The rest of the proof follows exactly the proof of Lemma~\ref{lem:gen_stat_cvg_rggn}.
\end{proof}

\begin{lemma}\label{lem:iddp_local_gen_cvg}
  Consider Assumption~\ref{asm:gen_cvg}. Assume in addition that the dynamics are twice differentiable with Lipschitz-continuous Hessian such that $\fullobj = \fullcost \circ \fullaugtraj$ has
  $\smoothess_{\fullobj}$-Lipschitz continuous Hessians.
  Consider $\ctrls^{(k)}$ to be close to a minimum $\ctrls^*$ of $\fullobj = \fullcost \circ \fullaugtraj$ satisfying $\mu_{\fullobj} = \lambda_{\min}(\nabla^2 \fullobj(\ctrls^*)) > 0$. If  $\|\ctrls^{(k)} - \ctrls^*\|_2 \leq \strgcvx_\fullobj/\smoothess_\fullobj$, and the regularization satisfies 
  \[
    \reg \geq  \max\{
      2 \lip_\fullaugtraj^2 \smooth_\fullcost, 5\kappa_\fullobj\smooth_{\fullcost}\lip_{\fullaugtraj}^2 ,
      8 \kappa_\fullobj\smooth_{\fullobj},
      32 \xi \frac{\smooth_\fullobj^2}{\smoothess_\fullobj}
      \},
  \]
  for $\kappa_\fullobj = \smooth_\fullobj/\strgcvx_\fullobj$,
  then the iterations of the \ref{eq:iddp_algo} algorithm converge linearly to $\ctrls^*$ as 
  \[
    \|\ctrls^{(k+1)} -\ctrls^*\|_2
    \leq  \left(1 - \frac{\strgcvx_\fullobj}{32\reg}\right)\|\ctrls^{(k)}-\ctrls^*\|_2.
  \]
\end{lemma}
\begin{proof}
  Given the $k$\textsuperscript{th} iteration $\ctrls^{(k)}$ of the \ref{eq:iddp_algo}, denote 
  \[
    \ctrls^{(k+1)}_{\mathrm{IDDP}} = \ctrls^{(k)} 
    + \DDP(\fullobj)(\ctrls^{(k)}), 
    \quad 
    \ctrls^{(k+1)}_{\mathrm{ILQR}} =
    \ctrls^{(k)} 
    + \LQR(\fullobj)(\ctrls^{(k)}), 
  \] 
  the next iteration if the $\LQR$ or the $\DDP$ oracles are used respectively.
  We have using Lemma~\ref{lem:gen_stat_cvg_rggn} and Lemma~\ref{lem:approx_ilqrs},
  \begin{align*}
    \|\ctrls^{(k+1)}_{\mathrm{IDDP}} - \ctrls^*\|_2
    & \leq \|\ctrls^{(k+1)}_{\mathrm{ILQR}} - \ctrls^*\|_2
    + \xi \|\LQR(\fullobj)(\ctrls^{(k)})\|_2^2 \\
    & \leq \left(
      1
      - \frac{\strgcvx_\fullobj}{16\reg}
      + \xi \frac{\smooth_\fullobj^2}{\reg^2}\|\ctrls^{(k)}- \ctrls^*\|_2
    \right)\|\ctrls^{(k)}- \ctrls^*\|_2.
  \end{align*}
  The result follows by using that  $\|\ctrls^{(k)}- \ctrls^*\|_2 \leq \strgcvx_\fullobj/\smoothess_\fullobj$ and taking
  \[
  \reg \geq \frac{32\xi \smooth_\fullobj}{\smoothess_\fullobj}.
  \]
\end{proof}

\section{Conditioning Analysis}\label{app:cond_analysis}
\subsection{Smoothness Estimations}\label{ssec:smoothness}
To derive simple bounds on the Lipschitz-continuity constants of the trajectory
function $\traj$, we present first a compact formulation of the first and second
order information of $\traj$ with respect to the first and second order
information of the dynamics $(\dyn_t)_{t=0}^{\horizon-1}$ in
Lemma~\ref{lem:grad_hess_detailed}. We require the following tensor notations in
this subsection.

A tensor $\mathcal{A} = (a_{i,j,k})_{1\leq i\leq d, 1\leq j\leq p, 1\leq k\leq
n} \in \reals^{d \times p \times n}$ is represented as a list of matrices
$\mathcal{A} = (A_1,\ldots, A_n)$ where $A_k = (a_{i,j,k})_{1 \leq i \leq d, 1
\leq j \leq p} \in \reals^{d\times p}$ for $ k\in \{1,\ldots n\}$. Given
$\mathcal{A}  \in \reals^{d\times p\times n}$ and  $P \in \reals^{d \times d'},
Q \in \reals^{p \times p'}, R \in \reals^{n \times n'}$, we denote
\[
\mathcal{A}[P, Q, R] = \left(\sum_{k=1}^{n} R_{k,1}P^\top A_k Q,  \ldots,  \sum_{k=1}^{n} R_{k,n'}P^\top A_k Q \right) \in \reals^{d'\times p'\times n'}.
\]
For $\mathcal{A}_0 \in \reals^{d_0\times p_0\times n_0}$, $P \in \reals^{d_0
\times d_1}, Q \in \reals^{p_0 \times p_1}, R \in \reals^{n_0 \times n_1}$
denote $\mathcal{A}_1= \mathcal{A}_0[P, Q, R] \in \reals^{d_1\times p_1\times
n_1}$. Then, for $S \in \reals^{d_1\times d_2}, T \in \reals^{p_1 \times p_2}, U
\in \reals^{ n_1\times n_2}$, we have 
\[
\mathcal{A}_1[S, T, U] = \mathcal{A}_0[PS, QT, RU] \in \reals^{d_2\times p_2\times n_2}.
\]
If $P, Q$ or $ R$ are identity matrices, we use the symbol ``$\: \cdot\: $'' in
place of the identity matrix. For example, we denote $\mathcal{A}[P, Q, \idm_n]
= \mathcal{A}[P,Q, \cdot] = \left(P^\top A_1 Q,  \ldots,  P^\top A_n Q \right)$.
If $P, Q$ or $R$ are vectors we consider the flattened object. In particular, for
$x\in \reals^d, y\in \reals^p$, we denote
\[
\mathcal{A}[x, y, \cdot] =  \left(\begin{matrix}
	x^\top A_1 y,\ldots, x^\top A_ny
\end{matrix}
\right)^\top\in \reals^n,
\]
rather than having $\mathcal{A}[x, y, \cdot] \in \reals^{1 \times 1\times n}$.
Similarly, for $z\in \reals^n$, we denote
\[
\mathcal{A}[\cdot, \cdot, z] = \sum_{k=1}^n z_kA_k \in \reals^{d\times p}.
\]
We denote $\|a\|_2$ the Euclidean norm for $a \in \reals^d$, $\|A\|_{2, 2}$ the
spectral norm of a matrix $A\in \reals^{d\times p}$, and we define the norm of a
tensor $\mathcal{A}$ induced by the Euclidean norm  as
$
\|\mathcal{A}\|_{2, 2, 2} = \sup_{x \neq 0, y \neq 0, z \neq 0} \mathcal{A}[x,y,z] /(\|x\|_2 \|y\|_2 \|z\|_2).
$

We refer to the third-order tensor agglomerating all second derivatives
$\partial_{x_i x_j}^2 f_k(x)$ of a vector function $f: \reals^a \rightarrow
\reals^b$ at a point $x \in \reals^a$ as simply the Hessian of $f$ at $x \in
\reals^a$.
\begin{lemma}\label{lem:grad_hess_detailed}
	Consider the control  $\traj$ of $\horizon$ dynamics $(\dyn_t)_{t=0}^{\horizon-1}$ as defined in Def.~\ref{def:traj_func_time_varying} and an initial point $\state_0 \in \reals^\dimstate$.
	For $\states = (\state_1;\ldots; \state_\horizon)$ and $\ctrls=(\ctrl_0; \ldots;\ctrl_{\horizon-1})$, define
	\[
	\concatdyn(\states, \ctrls) = (\dyn_0(\state_0, \ctrl_0); \ldots;\dyn_{\horizon-1}(\state_{\horizon-1}, \ctrl_{\horizon-1})),
	\]
  such that $\states = \traj(\state_0, \ctrls)$ is the unique solution of the
	implicit equation $\states = \concatdyn(\states, \ctrls)$.
  The transpose Jacobian of the control $\traj$ of the dynamics
	$(\dyn_t)_{t=0}^{\horizon-1}$  on $\ctrls\in \reals^{\horizon\dimctrl}$ can be
	written
	\[
	\nabla_\ctrls \traj (\state_0, \ctrls) = \nabla_{\ctrls} \concatdyn(\states, \ctrls) (\idm - \nabla_\states \concatdyn(\states, \ctrls))^{-1}. 
	\]
	The Hessian of  the control  $\traj$ of the dynamics $(\dyn_t)_{t=0}^{\horizon-1}$  on $\ctrls\in  \reals^{\horizon\dimctrl}$ can be written
	\begin{align*}
		\nabla^2_{\ctrls\ctrls} \traj(\state_0,\ctrls)  & = 
		\nabla^2_{\states\states}\concatdyn(\states, \ctrls)[N , N,  M ] 
		+ \nabla^2_{\ctrls\ctrls} \concatdyn(\states, \ctrls)[\cdot, \cdot, M ]  \\ 
		& \quad + \nabla^2_{\states\ctrls}\concatdyn(\states, \ctrls)[N,  \cdot, M ]
		+ \nabla^2_{\ctrls\states}\concatdyn(\states, \ctrls)[\cdot, N,  M ],
	\end{align*}
	where $M =  (\idm - \nabla_\states \concatdyn(\states, \ctrls))^{-1}$ and $N=  \nabla_\ctrls \traj (\state_0, \ctrls) ^\top $.
\end{lemma}
\begin{proof}
	Denote simply, for $\ctrls \in \reals^{\horizon \dimctrl}$,  $\auxdyn(\ctrls) = \traj(\state_0, \ctrls)$ with $\state_0$ a fixed initial state. By definition, the function $\auxdyn$ can be decomposed, for $\ctrls\in \reals^{\horizon\dimctrl}$, as $\auxdyn(\ctrls) = (\auxdyn_1(\ctrls); \ldots;\auxdyn_\horizon(\ctrls))$, such that
	\begin{equation}\label{eq:detailed_traj_func}
		\auxdyn_{t+1}(\ctrls)  = \dyn_t(\auxdyn_{t}(\ctrls), E_t^\top \ctrls) \quad \mbox{for} \ t \in \{0, \ldots, \horizon-1\},
	\end{equation}
	with $\auxdyn_{0}(\ctrls) = \state_0$ and for $t \in \{0, \ldots, \horizon-1\}$,  $E_t = e_t \otimes \idm_{\dimctrl}$ is such that $E_t^\top \ctrls = \ctrl_t$, with $e_t$ the $t+1$\textsuperscript{th} canonical vector in $\reals^{\horizon}$, $\otimes$ the Kronecker product and $\idm_{\dimctrl} \in \reals^{\dimctrl\times\dimctrl}$ the identity matrix. By taking the derivative of \eqref{eq:detailed_traj_func}, we get,
	denoting $\state_t = \auxdyn_{t}(\ctrls)$ for $t\in \{0, \ldots, \horizon\}$  and
	using that $E_t^\top \ctrls= \ctrl_t$, 
	\begin{equation}\nonumber
		\nabla\auxdyn_{t+1}(\ctrls)  = \nabla \auxdyn_{t}(\ctrls) \nabla_{\state_t}\dyn_t(\state_t, \ctrl_t) + E_t \nabla_{\ctrl_t} \dyn_t(\state_t,\ctrl_t) \quad \mbox{for} \  t\in \{0, \ldots, \horizon-1\}.
	\end{equation}
	So, for $\auxctrls= (\auxctrl_0;\ldots;\auxctrl_{\horizon-1}) \in \reals^{\horizon\dimctrl}$, denoting $\nabla \auxdyn(\ctrls)^\top \auxctrls = (\auxstate_1;\ldots;\auxstate_\horizon)$ s.t. $\nabla \auxdyn_{t}(\ctrls)^\top \auxctrls = \auxstate_t$ for $t\in \{1, \ldots, \horizon\}$, 
	we have, with $\auxstate_0=0$,
	\begin{equation}\label{eq:detailed_grad_traj_func_simp}
		\auxstate_{t+1} = \nabla_{\state_t}\dyn_t(\state_t, \ctrl_t)^\top \auxstate_t + \nabla_{\ctrl_t} \dyn_t(\state_t,\ctrl_t)^\top \auxctrl_t \quad \mbox{for} \  t\in \{0, \ldots, \horizon-1\}.
	\end{equation}
	Denoting $\auxstates = (\auxstate_1;\ldots;\auxstate_{\horizon})$, we have then 
	\[
	(\idm-A)\auxstates = B\auxctrls, \quad \mbox{i.e.},  \quad \nabla \auxdyn(\ctrls)^\top\auxctrls = (\idm-A)^{-1} B\auxctrls,
	\]
	where $A = \sum_{t=1}^{\horizon-1} e_te_{t+1}^\top \otimes A_t$ with $A_t = \nabla_{\state_t} \dyn_t(\state_t, \ctrl_t) ^\top $ for $t \in \{1, \ldots, \horizon-1\}$ and $B = \sum_{t=1}^{\horizon} e_te_t^\top \otimes B_{t-1}$ with $B_t = \nabla_{\ctrl_t} \dyn_t(\state_t, \ctrl_t)^\top$ for $t\in \{0, \ldots, \horizon-1\}$, i.e.
	\begin{align*}
		A = \begin{pmatrix}
			0 & \ldots &  &\ldots & 0 \\
			A_1 & \ddots &  &  & \vdots \\
			0 & \ddots& &  &  \\
			\vdots & \ddots & \ddots & \ddots & \vdots \\
			0 & \ldots & 0 & A_{\horizon-1} & 0 
		\end{pmatrix}, \quad 
    B = \begin{pmatrix}
			B_0 & 0 & \ldots & 0 \\
			0 & \ddots & \ddots & \vdots \\
			\vdots & \ddots & \ddots & 0 \\
			0 & \ldots & 0 & B_{\horizon-1}
		\end{pmatrix}.
	\end{align*}
	By definition of  $\concatdyn$ in the claim,  one easily check that $A = \nabla_\states \concatdyn(\states, \ctrls)^\top$ and $B = \nabla_\ctrls \concatdyn(\states, \ctrls)^\top$. Therefore, we get 
	\[
	\nabla_{\ctrls} \traj(\state_0, \ctrls)	= \nabla \auxdyn(\ctrls) = \nabla_\ctrls \concatdyn(\states, \ctrls) (\idm- \nabla_\states \concatdyn(\states, \ctrls))^{-1}.
	\]	
	For the Hessian, note that for $g:\reals^d \rightarrow \reals^p$, $f:\reals^p \rightarrow \reals$, $x\in \reals^d$,  we have  
	$
	\nabla^2 (f\circ g)(x) = \nabla g(x)\nabla^2f(x)\nabla g(x)^\top + \nabla^2 g(x) [\cdot, \cdot, \nabla f(x)] \in \reals^{d\times d}.
	$
	If $f :\reals^p \rightarrow \reals^n$, we have 
	$
	\nabla^2 (f\circ g)(x) = \nabla^2 f(x)[\nabla g(x)^\top, \nabla g(x)^\top, \cdot]+ \nabla^2 g(x) [\cdot, \cdot, \nabla f(x)]\in \reals^{d\times d \times n}.
	$
	Applying this on $\dyn_t\circ g_t$ for $g_t(\ctrls) = (\auxdyn_{t}(\ctrls), E_t^\top\ctrls)$, we get from~\eqref{eq:detailed_traj_func}, using that $\nabla g_t(\ctrls)= (\nabla \auxdyn_{t}(\ctrls), E_t)$,
	\begin{align} \nonumber
		\nabla^2\auxdyn_{t+1}(\ctrls) & = 
		\nabla^2\auxdyn_{t}(\ctrls)[\cdot, \cdot, \nabla_{\state_t} \dyn_t(\state_t, \ctrl_t)]
		\\ 
		& 	\quad + \nabla^2_{\state_t\state_t} \dyn_t(\state_t, \ctrl_t)[\nabla \auxdyn_{t}(\ctrls)^\top, \nabla \auxdyn_{t}(\ctrls)^\top, \cdot]+ \nabla^2_{\ctrl_t\ctrl_t}\dyn_t(\state_t, \ctrl_t)[E_t^\top , E_t^\top, \cdot]  \nonumber\\
		&\quad  +	\nabla^2_{\state_t\ctrl_t}\dyn_t(\state_t, \ctrl_t)[\nabla \auxdyn_{t}(\ctrls)^\top, E_t^\top , \cdot] 
		+	\nabla^2_{\ctrl_t\state_t}\dyn_t(\state_t, \ctrl_t)[E_t^\top , \nabla \auxdyn_{t}(\ctrls)^\top, \cdot],  \nonumber
	\end{align}
	for $t\in \{0, \ldots, \horizon-1\}$, with $\nabla^2\auxdyn_0(\ctrls) = 0$.
 Therefore, for $\auxctrls= (\auxctrl_0;\ldots;\auxctrl_{\horizon-1}), \auxxctrls = (\auxxctrl_0;\ldots;\auxxctrl_{\horizon-1}) \in \reals^{\horizon\dimctrl}$, $\dualvars = (\dualvar_1;\ldots;\dualvar_{\horizon}) \in \reals^{\horizon \dimstate}$, we get
	\begin{align}\nonumber
		\nabla^2\auxdyn(\ctrls)[\auxctrls, \auxxctrls, \dualvars] &= \sum_{t=0}^{\horizon-1}\nabla^2\auxdyn_{t+1}(\ctrls)[\auxctrls, \auxxctrls, \dualvar_{t+1}]\\
		& = \sum_{t=0}^{\horizon-1} \Big( \nabla^2_{\state_t\state_t}\dyn_t(\state_t, \ctrl_t)[\auxstate_{t}, \auxxstate_{t}, \costate_{t+1}] 
		+ \nabla^2_{\ctrl_t\ctrl_t}\dyn_t(\state_t, \ctrl_t)[\auxctrl_t, \auxxctrl_t, \costate_{t+1} ]\label{eq:detailed_hessian_traj}\\
		& \hspace{35pt} + \nabla^2_{\state_t\ctrl_t}\dyn_t(\state_t, \ctrl_t)[\auxstate_t,\auxxctrl_t, \costate_{t+1} ] + \nabla^2_{\ctrl_t\state_t}\dyn_t(\state_t, \ctrl_t)[\auxctrl_t,\auxxstate_t, \costate_{t+1} ]
		\Big),  \nonumber
	\end{align}
	where $\auxstates = (\auxstate_1;\ldots;\auxstate_\horizon) = \nabla \auxdyn(\ctrls)^\top\auxctrls$,  $\auxxstates = (\auxxstate_1;\ldots;\auxxstate_\horizon) = \nabla \auxdyn(\ctrls)^\top\auxxctrls$, with $\auxstate_0 = \auxxstate_0=0$ and  $\costates = (\costate_1;\ldots;\costate_\horizon) \in \reals^{\horizon\dimstate}$ is defined by
	\begin{align*}
		\costate_{t} & =  \nabla_{\state_t} \dyn_{t}(\state_{t}, \ctrl_{t}) \costate_{t+1} + \dualvar_t\qquad \mbox{for} \ t\in\{1,\ldots, \horizon-1\}, \quad \costate_\horizon  =\dualvar_\horizon.
	\end{align*}
	On the other hand, denoting $\concatdyn_t(\states, \ctrls) = \dyn_t(\state_t, \ctrl_t)$ for $t\in \{0, \ldots, \horizon-1\}$, the Hessian of $\concatdyn$ with respect to the variables $\ctrls$ can be decomposed as
	\[
	\nabla^2_{\ctrls\ctrls}\concatdyn(\states,\ctrls)[\auxctrls, \auxxctrls, \costates] = \sum_{t=0}^{\horizon-1} \nabla^2_{\ctrls\ctrls} \concatdyn_t(\states, \ctrls)[\auxctrls, \auxxctrls, \costate_{t+1}] = \sum_{t=0}^{\horizon-1} \nabla^2_{\ctrl_t\ctrl_t} \dyn_t(\state_t, \ctrl_t)[\auxctrl_t, \auxxctrl_t, \costate_{t+1}].
	\]
	The Hessian of $\concatdyn$ with respect to the variable $\states$ can be decomposed as 
	\begin{align*}
		\nabla^2_{\states\states}\concatdyn(\states, \ctrls) [\auxstates, \auxxstates, \costates] & = \sum_{t=0}^{\horizon-1} \nabla^2_{\states\states} \concatdyn_t(\states,\ctrls)[\auxstates, \auxxstates, \costate_{t+1}] =
		\sum_{t=1}^{\horizon-1} \nabla^2_{\state_t\state_t} \dyn_t(\state_t,\ctrl_t)[\auxstate_t, \auxxstate_t, \costate_{t+1}].
	\end{align*}
	A similar decomposition can be done for $\nabla^2_{\states\ctrls}\concatdyn(\states, \ctrls)$. From~\eqref{eq:detailed_hessian_traj}, we then get 
	\begin{align*}
		\nabla^2\auxdyn(\ctrls)[\auxctrls, \auxxctrls, \dualvars]  & 
		= \nabla^2_{\states\states} \concatdyn(\states, \ctrls)[\auxstates, \auxxstates, \costates] 
		+  \nabla^2_{\ctrls\ctrls}\concatdyn(\states,\ctrls)[\auxctrls, \auxxctrls, \costates]  \\
		& + \nabla^2_{\states\ctrls}\concatdyn(\states,\ctrls)[\auxstates, \auxxctrls, \costates] 
		+ \nabla^2_{\ctrls\states} \concatdyn(\states, \ctrls)[\auxctrls, \auxxstates, \costates].
	\end{align*}
	Finally, by noting that 
 \begin{align*}
     \auxstates & = (\nabla_\ctrls \concatdyn(\states, \ctrls)(\idm -\nabla_\states \concatdyn(\states, \ctrls))^{-1} )^\top \auxctrls, \\
     \auxxstates & =(\nabla_\ctrls \concatdyn(\states, \ctrls)(\idm -\nabla_\states \concatdyn(\states, \ctrls))^{-1} )^\top \auxxctrls \\
     \costates & = (\idm -\nabla_\states \concatdyn(\states, \ctrls))^{-1} \dualvars
 \end{align*}
 the claim is shown.
\end{proof}

Lemma~\ref{lem:grad_hess_detailed} can be used to get estimates on the smoothness properties of the control of $\horizon$ dynamics given the smoothness properties of each individual dynamics. 
\begin{lemma}\label{lem:smooth_traj_from_dyn_time_varying}
	If $\horizon$  dynamics $(\dyn_t)_{t=0}^{\horizon-1}$ are Lipschitz continuous with Lipschitz continuous Jacobians, then the function $\ctrls\rightarrow \traj(\state_0, \ctrls)$, with $\traj$ the control of the $\horizon$ dynamics $(\dyn_t)_{t=0}^{\horizon-1}$, is $\lip_\traj$-Lipschitz continuous and has $\smooth_\traj$-Lipschitz continuous Jacobians with 
	\begin{equation}\label{eq:smooth_traj_from_dyn_time_varying}
		\lip_{\traj} \leq \lipdynctrl S, \qquad \smooth_{\traj} \leq 
		S(\smoothdynstate \lip_{\traj}^2 + 2 \smoothdynstatectrl\lip_{\traj} + \smoothdynctrl),
	\end{equation}
	where
	$\lip_{\dyn_t}^\ctrl {=} \sup_{\state, \ctrl } \|\nabla_\ctrl \dyn(\state, \ctrl)\|_{2, 2}$, 
	$\lip_{\dyn_t}^\state {=} \sup_{\state, \ctrl } \|\nabla_\state \dyn(\state, \ctrl)\|_{2, 2}$, 
	$\smooth_{\dyn_t}^{\state\state} {=} \sup_{\state, \ctrl } \|\nabla_{\state\state}^2 \dyn(\state, \ctrl)\|_{2, 2, 2}$, 
	$\smooth_{\dyn_t}^{\ctrl\ctrl} = \sup_{\state, \ctrl } \|\nabla_{\ctrl\ctrl}^2 \dyn(\state, \ctrl)\|_{2, 2, 2}$, 
	$\smooth_{\dyn_t}^{\state\ctrl} = \sup_{\state, \ctrl } \|\nabla_{\state\ctrl}^2 \dyn(\state, \ctrl)\|_{2, 2, 2}$, $S {=} \sum_{t=0}^{\horizon-1}(\lipdynstate)^t$,  and we drop the index $t$ to denote the maximum over all dynamics such as $\lip_{\dyn}^\state = \max_{t\in \{0, \ldots, \horizon-1\}} \lip_{\dyn_t}^\state$.
\end{lemma}
\begin{proof}
	The Lipschitz continuity constant of $ \ctrls\rightarrow \traj(\state_0, \ctrls)$ and its Jacobians can be estimated by upper bounding  the norm of the Jacobians and the Hessians.
	With the notations of Lemma~\ref{lem:grad_hess_detailed}, $\nabla_\states \concatdyn(\states, \ctrls)$ is nilpotent of degree $\horizon$ since it can be written $\nabla_\states \concatdyn(\states, \ctrls)=  \sum_{t=1}^{\horizon-1} e_{t+1}e_t^\top \otimes \nabla_{\state_t} \dyn_t(\state_t, \ctrl_t)$ and $(A\otimes B) (C\otimes D) = (AC \otimes BD)$. Hence, we have
	\[
	(\idm - \nabla_\states \concatdyn(\states, \ctrls))^{-1} = \sum_{t=0}^{\horizon-1} \nabla_\states \concatdyn(\states, \ctrls)^t.
	\]
	The Lipschitz continuity constant of $\traj$ is then estimated by
	\[
	\|\nabla_\ctrls \traj(\state_0, \ctrls)\|_{2, 2} \leq \|\nabla_\ctrls\concatdyn(\states, \ctrls)\|_{2, 2} \|(\idm - \nabla_\states \concatdyn(\states, \ctrls))^{-1}\|_{2, 2}\leq
	\lipdynctrl \sum_{t=0}^{\horizon-1} (\lipdynstate)^t.
	\]
	As shown in Lemma~\ref{lem:grad_hess_detailed}, the Hessian of $ \ctrls\rightarrow \traj(\state_0, \ctrls)$ can be decomposed as
	\begin{align*}
		\nabla^2_{\ctrls\ctrls} \traj(\state_0,\ctrls)  & 
		= \nabla^2_{\states\states}\concatdyn(\states, \ctrls)[N , N,  M ] 
		+ \nabla^2_{\ctrls\ctrls} \concatdyn(\states, \ctrls)[\cdot, \cdot, M ]  \\
		& + \nabla^2_{\states\ctrls}\concatdyn(\states, \ctrls)[N,  \cdot, M ]
		+ \nabla^2_{\ctrls\states}\concatdyn(\states, \ctrls)[\cdot, N,  M ],
	\end{align*}
	where $M =  (\idm - \nabla_\states \concatdyn(\states, \ctrls))^{-1}$ and $N=  \nabla_\ctrls \traj(\initstate, \ctrls)^\top$. Given the structure of  $\concatdyn$, bounds on the Hessians are  $\|\nabla^2_{ab}\concatdyn(\states, \ctrls)\|_{2, 2, 2}\leq \smooth_\dyn^{ab}$ for $a, b \in \{\states, \ctrls\}$, where $\|\mathcal{A} \|_{2, 2, 2}$ is the norm of a tensor $\mathcal{A}$ w.r.t. the Euclidean norm as defined in the notations.  Note that for a given tensor $\mathcal{A} \in \reals^{d\times p\times n}$ and  $P,Q,R$ of  appropriate sizes, we have $\|\mathcal{A}[P,Q,R]\|_{2, 2, 2} \leq  \|\mathcal{A}\|_{2, 2, 2} \|P\|_{2, 2}\|Q\|_{2, 2}\|R\|_{2, 2}$. We then get
	\begin{align*}
	   ||\nabla^2_{\ctrls\ctrls} \traj(\state_0,\ctrls)||_{2, 2, 2} 
            \leq \smoothdynstate \|N\|_{2, 2}^2 \|M\|_{2, 2} 
            + \smoothdynctrl\|M\|_{2, 2} 
            + 2\smoothdynstatectrl \|M\|_{2, 2} \|N\|_{2, 2},
	\end{align*}
	where for twice differentiable functions we used that $\smoothdynstatectrl = \smooth_\dyn^{\ctrl\state}$.
\end{proof}

\subsection{Time-varying Dynamics Case} Lemma~\ref{lem:injectivity_time_varying}  presents a simple extension
of Lemma~\ref{lem:injectivity} for time-varying dynamics. Note that provided that
condition~\eqref{eq:injectivity_time_varying} is satisfied, the analysis of the
ILQR and IDDP algorithms remain essentially unchanged, up to different
constants. 

\begin{lemma}\label{lem:injectivity_time_varying} Consider the control of
	$\horizon$ discrete time dynamics $(\dyn_t:\reals^\dimstate \times
	\reals^\dimctrl \rightarrow \reals^\dimstate)_{t=0}^{\horizon-1}$ as defined
	in Definition~\ref{def:traj_func_time_varying}.
If the dynamics $\dyn_t$ are Lipschitz continuous and satisfy 
\begin{equation}\label{eq:injectivity_time_varying} 
	\forall \state, \ctrl \in \reals^\dimstate \times \reals^\dimctrl, \quad 	\sigma_{\min} (\nabla_\ctrl \dyn_t(\state, \ctrl))\geq \sigma_{\dyn_t} >0,
\end{equation}
then the control $\traj$ of these dynamics satisfy for all $t\in \{0, \ldots,
\horizon-1\}$, 
\begin{equation}\nonumber
	\forall \state_0, \ctrls \in \reals^\dimstate \times \reals^{\horizon\dimctrl}, \quad 	\sigma_{\min} (\nabla_\ctrl \traj(\state_0, \ctrls))\geq
\sigma_\traj := \frac{ \min_{t\in \{0, \ldots, \horizon-1\}} \sigma_{\dyn_t}}{1+ \max_{t\in \{0, \ldots, \horizon-1\}} \lip_{\dyn_t}^\state} >0,
\end{equation}
where $\lip_{\dyn_t}^\state = \sup_{\ctrl \in \reals^\dimctrl}
\lip_{\dyn_t(\cdot, \ctrl)}$ is the maximal Lipschitz-continuity constant of the
functions $\dyn_t(\cdot, \ctrl)$ for any $\ctrl \in \reals^{\dimctrl}$. 
\end{lemma}
\begin{proof}
	With the notations of the proof of Lemma~\ref{lem:injectivity} we have 
	\begin{align*}
		\sigma_{\min} (\nabla_\ctrls \traj(\state_0, \ctrls)) \geq \frac{\sigma_{\min}(\nabla_\ctrls \concatdyn(\states, \ctrls))}{\sigma_{\max}( \idm - \nabla_\states \concatdyn(\states, \ctrls))},
	\end{align*}
where here $ \concatdyn(\states, \ctrls) {=} (\dyn_0(\state_0, \ctrl_0) ;
\ldots; \dyn_{\horizon-1}(\state_{\horizon-1}, \ctrl_{\horizon-1}))$. Noting
that  $\sigma_{\min}(\nabla_\ctrls \concatdyn(\states, \ctrls)) \geq \min_{t\in
\{0, \ldots, \horizon-1\}} \sigma_{\dyn_t}$ and  $\sigma_{\max}( \idm -
\nabla_\states \concatdyn(\states, \ctrls)) \leq  1+ \max_{t\in \{0, \ldots,
\horizon-1\}} \lip_{\dyn_t}^\state$ concludes the proof.
\end{proof}

\section{Linearization by Static Feedback and Brunovsky's Form}\label{app:static_lin}
We briefly recall here the rationale behind the parameterization of a  system in
Brunovsky's form and the theory underlying the existence of static feedback
linearization in continuous time, see, e.g.~\citet{isidori1995nonlinear,
busawon2009algorithms} for more details on these subjects
and~\citet{aranda1996linearization} for an analysis of  feedback linearization
in discrete time. 

\subsection{Brunovsky's Form}
We start by understanding the relevance of the parameterization in Brunovsky's
form for discrete time linear systems of the form
\begin{equation}\label{eq:brunowsky_discrete_linear}
\state_{t+1} = A\state_t + B\ctrl_t \quad \mbox{for} \ t =0, 1,\ldots
\end{equation}
for $\state_t\in \reals^\dimstate$, $\ctrl_t \in \reals^\dimctrl$ with
$\dimctrl=1$ for ease of presentation, where $A \in \reals^{\dimstate \times
\dimstate}$, $B\in \reals^{\dimstate \times 1}$ and we denote $\dimstate=n$ for
more readability. An important property that can be investigated for such system
is its controllability, i.e., whether, from any initial state $\state_0$, we can
reach any state $x^*$ after a sufficient number of steps of the discrete
dynamical system and appropriate control variables. This question can be
answered by examining the controllability matrix $C=[B, AB, \ldots A^{n-1}B]$
associated to the system~\eqref{eq:brunowsky_discrete_linear}. If $C$ has full
row rank, i.e., $\operatorname{rank}(C) = n$, then the system is controllable in
at most $n$ steps, as observed from standard linear algebra considerations. For
a controllable system~\eqref{eq:brunowsky_discrete_linear}, we can investigate
whether there exists a reparameterization of the system in variables
$\auxstate_t = M\state_t, \auxctrl_t = N\ctrl_t + P\auxstate_t$, in which the
notion of controllability is transparent in the reparameterized system
$\auxstate_{t+1}= F\auxstate_t + G\auxctrl_t$, for $F, G$ defined appropriately
from $A, B, M, N, P$. One ideal reparameterization is given by Brunovsky's form, 
\[
\begin{matrix*}[l]
	\auxstate_{t+1}^{(1)} &= \auxstate_{t}^{(2)} \\
	\auxstate_{t+1}^{(2)} &= \auxstate_{t}^{(3)} \\
& \hspace{3pt} \vdots \\
	\auxstate_{t+1}^{(n-1)} &= \auxstate_t^{(n)}\\
	\auxstate_{t+1}^{(n)} & = \auxctrl_t
\end{matrix*} \quad\mbox{that is}\quad  \auxstate_{t+1} = D\auxstate_t + E \auxctrl_t
\]
for $D = \sum_{i=1}^{n-1} e_ie_{i+1}^\top$ the upper-shift matrix in $\reals^n$
and $e_i$ the $i$\textsuperscript{th} canonical vector in $\reals^n$ with $E =
e_n$. In this reparameterized system of equations, after $n$ steps of the linear
system we naturally have that $\auxstate_n^{(i)} = \auxctrl_{i-1}$ for $i \in
\{1, \ldots, n\}$, that is, the operator that, at $n$ control variables
associates the state of the system after $n$ steps is just the identity
operator, which clearly satisfies the definition of controllability.  

To get such a parameterization, consider state variables of the form
$\auxstate_t = M\state_t$  for $M$ invertible. The resulting linear system has
the form $\auxstate_{t+1} = MAM^{-1} \auxstate_t + M B \ctrl_t$. We then need to
choose $M$ such that $MAM^{-1} = D + EJ$ for some $J\in \reals^{1\times n}$ and
$FB =E$ such that by defining $\auxctrl_t = EJ \auxstate_t  + E\ctrl_t $, we get
that $\auxstate_{t+1} = D\auxstate_t + E\auxctrl_t$. Such invertible matrix $M$
can be computed in closed form from the expressions of $A$, $B$, $D$ and $E$ as
$CK^{-1}$ for $C$ the controllability matrix associated to the pair $(A, B)$
defining the linear system~\eqref{eq:brunowsky_discrete_linear} and $K$ the
controllability matrix associated to the pair $D, E$ defining the linear system
in Brunovsky's form. This is essentially the approach taken in
Lemma~\ref{lem:ode}.

\subsection{Static Feedback Linearization for Continuous Time Systems.}
Consider a continuous dynamical system of the form
\begin{equation}\label{eq:cont_brunovsky}
\dot \state = f(\state) + g(\state)\ctrl
\end{equation}
for $f:\reals^\dimstate \rightarrow \reals^\dimstate$ and $g:\reals^\dimstate
\rightarrow \reals^\dimstate$ with $\ctrl \in \reals$, and we denote here
$\dimstate = n$ for more readability. Static feedback linearization schemes aim
to find a reparameterization of this system around an initial state $\state_0$
such that the system is linear and controllable in the reparameterized variables
under suitable assumptions on $f$ and $g$. The point of departure of the
analysis of conditions for the existence of a static feedback linearization
scheme~\citep[Section 4]{isidori1995nonlinear} is to consider a function $h$
which defines the output of the system~\eqref{eq:cont_brunovsky} as $\auxxstate
= h(\state)$. We may then analyze the influence of the control on this output
through the derivatives of $\auxxstate$. Namely, we have that $\frac{\partial
\auxxstate}{\partial t} = \frac{\partial h}{\partial \state} \frac{\partial
\state}{\partial t} = \frac{\partial h}{\partial x} (f(\state) + g(\state)\ctrl)
= L_fh(\state) + L_g h(\state) \ctrl$, where we defined the derivative of
$h:\reals^n \rightarrow \reals$ along $f:\reals^n \rightarrow\reals^n$ as $L_f
h(x) = \sum_{i=1}^n \frac{\partial h}{\partial x_i}(x) f_i(x)$ for $f_i(x)$ the
$i$\textsuperscript{th} coordinate of $f(x)$. If $ L_g h(\state)=0$ in a
neighborhood of the initial state $\state_0$, i.e., the derivative of the output
function $h$ along $g$ is zero, then the control has no effect on the first
derivative of the output for $t$ small enough, i.e.,  $\frac{\partial
\auxxstate}{\partial t}(t)  = L_fh(\state(t))$. Analyzing the second derivative
of the output, we have $\frac{\partial^2 \auxxstate}{\partial t^2} = L_f^2
h(\state) + L_gL_fh(\state)\ctrl$, where $L_gL_f h(x)= \sum_{i=1}^n
\frac{\partial L_f h}{\partial x_i}(x) g_i(x)$ and $L_f^2 h(\state)= L_fL_f
h(\state)$. If $L_gL_fh(\state) = 0$ in a neighborhood of the initial state,
then the control variable has no effect on the second derivative of the output
for $t$ small enough. Continuing this way, if for any $k\in \{0, \ldots, n
-1\}$, $L_gL_f^kh(\state) = 0$ around $\state_0$ and $L_gL_f^n h(\state_0) \neq
0$  then the derivatives of the output satisfy $\frac{\partial^k \auxxstate
}{\partial t^k} = L_f^k h(t)$ for $t$ small enough and $\frac{\partial^n
\auxxstate}{\partial t^n}(0) = L_f^n h(\state_0) +L_gL_f^{n-1} h(\state_0)
u(0)$. In other words, under the aforementioned conditions, the output of the
system can be seen as a dynamical system driven by its $n$\textsuperscript{th}
derivative. Given an output function $h$ satisfying the aforementioned
conditions, we can then consider the reparameterization $\auxstate_i =
L_f^{i-1}h(\state)$ (which corresponds to consider a system whose coordinates
are defined by the $i$\textsuperscript{th} derivative of the output) and define
$\auxctrl = L_f^n h(\state) +L_gL_f^{n-1} h(\state)\ctrl$  such that the
reparameterized system takes the form
\[
\begin{matrix*}[l]
	\dot \auxstate_1  &= \auxstate_2  \\
	\dot \auxstate_2 & = \auxstate_3 \\
	& \hspace{3pt} \vdots \\
		\dot \auxstate_{n-1} & = \auxstate_n \\
	\dot \auxstate_n &=  L_f^n h(\state) +L_gL_f^{n-1} h(\state)\ctrl = \auxctrl.
\end{matrix*}
\]
We recognize here again a parameterization in Brunovsky's form, here for the
continuous time system considered. Existence of an output function satisfying
the aforementioned assumptions and such that the reparameterization is a
diffeomorphism around the initial point can be verified by considering  the
involutivity and regularity of the vector field defined by repeated Lie brackets
of the function $f$ on the function $g$, see, e.g.~\citet[Lemma 4.2.2]{isidori1995nonlinear}.

\subsection{Reparameterization in Brunovsky Form}
Lemma~\ref{lem:ode} shows how a discrete time system driven by its
$k$\textsuperscript{th} derivative can be expressed in Brunovsky's
form~\eqref{eq:brunovsky}~\citep{brunovsky1970classification}.

\begin{lemma}\label{lem:ode} Consider the Euler discretization of a single-input
	continuous-time  system driven by its $\dimstate$\textsuperscript{th}
	derivative as presented in~\eqref{eq:ode_single_input}. If $|\partial_\auxctrl
	\contdyn(\auxstate,  \auxctrl)|>0$ for all $\auxstate \in \reals^\dimstate,
	\auxctrl \in \reals$ then the dynamical system~\eqref{eq:ode_single_input} can
	be linearized by static  feedback into the canonical
	form~\eqref{eq:brunovsky}.
\end{lemma}
\begin{proof}
	Denoting $A= \idm + \Delta D$, with $D$ the upper-shift matrix in
	$\reals^{\dimstate}$, the original dynamical
	system~\eqref{eq:ode_single_input} can be written as $\auxstate_{t+1} = A
	\auxstate_t + \Delta\contdyn(\auxstate_t,  \auxctrl_t)e$, with $e=
	e_\dimstate$ the $\dimstate$\textsuperscript{th} canonical vector in
	$\reals^\dimstate$. It suffices to note that the matrix $A$ is similar to  a
	matrix of the form $B= D + e c^\top$ for some vector $c$. Namely, denoting
	$P_{\dimstate} =(p_{\dimstate, 1}, \ldots, p_{\dimstate, \dimstate})^\top$ the
	$\dimstate$\textsuperscript{th} lower triangular Pascal matrix defined by rows
	$p_{\dimstate, i} = (\binom{i-1}{j-1})_{j=1}^{\dimstate}$ with the convention
	$\binom{i}{j} = 0$ if $i <j$ and $Q = P_\dimstate
	\diag((\Delta^{i-\dimstate})_{i=1}^\dimstate)$, we get that  $BQ = Q A $ for
	$B= D + e c^\top$ with  $c = ((-1)^{\dimstate
	-i}\binom{\dimstate}{i-1})_{i=1}^\dimstate$. 
	
	Hence, by considering the change of variable $\auxxstate_t =
	\diffeostate(\auxstate_t) =Q\auxstate_t$, we get that 
	\[
	\auxxstate_{t+1} = B\auxxstate_t  +\Delta \contdyn(\auxstate_t, \auxctrl_t) Q e 
	= D \auxxstate_t +  c^\top \auxxstate_t e + \Delta \contdyn(\auxstate_t, \auxctrl_t)  e,
	\]
	 using that $Qe = e$. By defining $\auxxctrl_t = \diffeoctrl(\auxstate_t,
	\auxctrl_t) = c^\top Q\auxstate_t + \Delta \contdyn(\auxstate_t, \auxctrl_t) $
	we get the desired form~\eqref{eq:brunovsky}. The transformation
	$\diffeostate$ is a diffeomorphism since $Q$ is invertible. The
	transformations $\diffeoctrl(\auxstate_t, \cdot)$ are also diffeomorphisms
	since $|\partial_\auxctrl \contdyn(\auxstate,  \auxctrl)|>0$ for all
	$\auxstate \in \reals^\dimstate, \auxctrl \in \reals$. 
\end{proof}

\section{Convergence Analysis of ILQR}\label{app:lemmas}
\subsection{Global Convergence Analysis}

In the statement of Theorem~\ref{thm:global_conv}, we used
Lemma~\ref{lem:plcst_decomp} to relate the constants associated to gradient
dominance properties of the costs on the sates to the constant associated to the
gradient dominance property of the cost on the trajectory with, in
Theorem~\ref{thm:global_conv}, compared to Lemma~\ref{lem:plcst_decomp}, we used
$\plcst_t = \plcst$ for all $t$ such that $\plcst_\cost =
\|\boldsymbol{\plcst}^{-1}\|_q^{-1} = \plcst \horizon^{-q} =
\plcst/\horizon^{2\pl/(2\pl-1)}$.
\begin{lemma}\label{lem:plcst_decomp} Let $\cost_1, \ldots, \cost_\horizon$ be
	differentiable functions from $\reals^\dimstate \rightarrow \reals$ such that 
	\[
	\|\nabla \cost_t(\state_t)\|_2 \geq \plcst_t^\pl(\cost_t(\state_t) - \cost_t^*)^\pl \quad \mbox{for} \ t \in \{1, \ldots, \horizon\},
	\]
	for some constants $\plcst_t \geq 0 $, $\pl\in [1/2, 1)$. The function $\cost:
	\states = (\state_1;\ldots;\state_\horizon) \rightarrow \sum_{t=1}^{\horizon}
	\cost_t(\state_t)$ satisfies
	\[
	\|	\nabla \cost(\states) \|_2 \geq \plcst_\cost^\pl (\cost(\states) - \cost^*)^\pl \quad \mbox{for} \ \plcst_\cost = \|\boldsymbol{\plcst}^{-1}\|_q^{-1},
	\]
	for $q = 2\pl/(2\pl-1)$ and $\boldsymbol{\plcst}^{-1} =
	(\plcst_1^{-1},\ldots,\plcst_\horizon^{-1})^{\top}$ with
	$\|\boldsymbol{\plcst}^{-1}\|_{+\infty}^{-1} = \min_{t\in \{1, \ldots,
	\horizon\}} \plcst_t$ if $\pl=1/2$.
\end{lemma}
\begin{proof}
	Denoting for simplicity $\delta_t = \cost_t(\state_t)- \cost^*$, we have 
	\begin{align*}
		\|\nabla \cost(\states) \|_2^2 & =  \sum_{t=1}^\horizon \|\nabla \cost_t(\state_t)\|_2^2 \geq  \sum_{t=1}^{\horizon} (\plcst_t \delta_t )^{2\pl}  = \| \boldsymbol{\plcst} \odot \boldsymbol{\delta} \|_{2\pl}^{2\pl} \geq \frac{1}{\|\boldsymbol{\plcst}^{-1} \|_q^{2\pl}} (\boldsymbol{\delta}^\top \ones)^{2\pl},
	\end{align*}
	for $q = 2\pl/(2\pl-1)$, where $\boldsymbol{\plcst} =
	(\plcst_1,\ldots,\plcst_{\horizon})^\top$, $\boldsymbol{\delta} = (\delta_1,
	\ldots, \delta_{\horizon})^{\top}$, $\odot$ denotes the element-wise product,
	and we used H\"older's inequality $\|x\|_p\|y\|_q \leq |x^\top y|$ for
	$p=2\pl$, $q=p/(p-1) =2\pl/(2\pl-1)$,   $x = \boldsymbol{\plcst} \odot
	\boldsymbol{\delta}$ and $y = \boldsymbol{\plcst}^{-1}$. Plugging the values
	of $\delta$ in the inequality above,  we get 
	\[
	\|\nabla \cost(\states) \|_2 \geq \|\boldsymbol{\plcst}^{-1}\|_q^{-\pl} \left(\sum_{t=1}^{\horizon} \cost_t(\state_t) - \cost_t^*\right)^{\pl} = \plcst_\cost^\pl (\cost(\states) - \cost^*)^\pl,
	\]
	where we used that, since $\cost$ is decomposable in the variables $\state_t$,
	$\cost^* = \sum_{t=1}^{\horizon} \cost_t^*$.
\end{proof} 

Lemma~\ref{lem:bound_approx_self_concord} states that a linear quadratic
approximation of the compositional objective in~\eqref{eq:total_cost}
approximates the objective up to a cubic error. 
\begin{lemma}\label{lem:bound_approx_self_concord} Given
	Assumption~\ref{asm:conv}, we have, for problem~\eqref{eq:total_cost},
	\[
	|(\cost \circ \augtraj)(\ctrls + \auxctrls) 
	{-} ( \cost \circ \augtraj)(\ctrls) 
	{-} \qua_{\cost}^{\augtraj(\ctrls)}\circ \lin_\augtraj^\ctrls(\auxctrls) |  {\leq}  \frac{\smooth_\augtraj \|\nabla \cost(\augtraj(\ctrls))\|_2 {+}  (\smoothess_\cost \lip_{\augtraj}^3/3 {+} \smooth_\augtraj\smooth_\cost \lip_\augtraj)\|\auxctrls\|_2}{2}\|\auxctrls\|_2^2.
	\]
\end{lemma}
\begin{proof}
We have for any $\ctrls, \auxctrls \in \reals^{\horizon\dimctrl}$,
	\begin{align*}
		|\cost(\augtraj(\ctrls {+} \auxctrls)) {-} \cost (\augtraj(\ctrls) ) {-} \qua_\cost^{\augtraj(\ctrls)}(\lin_\augtraj^\ctrls(\auxctrls))| & \leq |\cost(\augtraj(\ctrls {+} \auxctrls)) {-} \cost (\augtraj(\ctrls) ) {-} \qua_\cost^{\augtraj(\ctrls)}(\augtraj(\ctrls{+}\auxctrls){-} \augtraj(\ctrls))| \\
		& \quad + |\qua_\cost^{\augtraj(\ctrls)}(\augtraj(\ctrls{+}\auxctrls) {-} \augtraj(\ctrls) ) {-}\qua_\cost^{\augtraj(\ctrls)}(\lin_\augtraj^\ctrls(\auxctrls))|.
	\end{align*}
	On one hand, we have, by Taylor-Lagrange inequality, 
	\begin{align*}
		|\cost(\augtraj(\ctrls {+} \auxctrls)) {-} \cost (\augtraj(\ctrls) ) {-} \qua_\cost^{\augtraj(\ctrls)}(\augtraj(\ctrls{+}\auxctrls) {-} \augtraj(\ctrls))| 
		& \leq \frac{\smoothess_\cost}{6}\|\augtraj(\ctrls + \auxctrls) {-} \augtraj(\ctrls)\|^3_{2} 
		\leq \frac{\smoothess_\cost \lip_\augtraj^3}{6}\|\auxctrls\|_2^3.
	\end{align*}
	On the other hand, we have, 
	\begin{align*}
		|\qua_\cost^{\augtraj(\ctrls)}(\augtraj(\ctrls{+}\auxctrls){-}\augtraj(\ctrls)) {{-}}\qua_\cost^{\augtraj(\ctrls)}(\lin_\augtraj^\ctrls(\auxctrls))| & = \Big| (\augtraj(\ctrls{+} \auxctrls) {{-}} \augtraj(\ctrls){{-}} \nabla\augtraj(\ctrls)^\top \auxctrls)^\top \nabla \cost(\augtraj(\ctrls)) \\
		& \hspace{-95pt}  {+} \frac{1}{2}(\augtraj(\ctrls{+} \auxctrls) {{-}} \augtraj(\ctrls){{-}} \nabla\augtraj(\ctrls)^\top \auxctrls)^\top \nabla^2 \cost(\augtraj(\ctrls)) (\augtraj(\ctrls{+} \auxctrls) {{-}} \augtraj(\ctrls) {+}  \nabla\augtraj(\ctrls)^\top \auxctrls) \Big| \\
		& \leq \frac{\smooth_\augtraj \|\nabla \cost (\augtraj(\ctrls))\|_{2}}{2}\|\auxctrls\|^2_2 + \frac{\smooth_\cost \smooth_\augtraj\lip_\augtraj }{2}\|\auxctrls\|_2^3.
	\end{align*}
\end{proof}

\subsection{Local Convergence Analysis}
Lemma~\ref{lem:bound_ggn_oracle_self_concord} provides a bound on the oracle
returned by an ILQR method in terms of the constants introduced in
Theorem~\ref{thm:conv_ggn}.

\begin{lemma}\label{lem:bound_ggn_oracle_self_concord} Given
	Assumption~\ref{asm:self_concord} on problem~\eqref{eq:total_cost}, we have
	for any $\ctrls \in \reals^{\horizon\dimctrl}$,  $\reg\geq 0$, 
	\[
	\|\LQR_\reg(\obj)(\ctrls)\|_2 \leq \frac{\lip }{\lip \sigma  + \reg} \|\nabla \cost(\augtraj(\ctrls))\|_{\augtraj(\ctrls)}^*.
	\]
\end{lemma}
\begin{proof}
	For $\ctrls\in \reals^{\horizon\dimctrl}$, $\reg\geq 0$,  denoting $\nabla^2
	\cost(\augtraj(\ctrls)) = \hess$, $\nabla \augtraj(\ctrls) = \grad$, we have 
	\[
	\LQR_\reg(\obj)(\ctrls)  = - \grad\hess^{1/2}
	(\hess^{1/2}\grad^\top \grad\hess^{1/2}  + \reg\idm)^{-1}
	 \hess^{-1/2}\nabla \cost(\augtraj(\ctrls)).
	\]
	Recall that by definition of $\sigma$ and $\lip$, we have 
	$
	\sigma \leq \sigma_{\min}(\grad\hess^{1/2}),   \sigma_{\max}(\grad\hess^{1/2}) \leq \lip.
	$
	By considering the singular value decomposition of $\grad\hess^{1/2}$, we then
	have
	\begin{align*}
		\|\grad\hess^{1/2}(\hess^{1/2}\grad^\top \grad\hess^{1/2}  + \reg\idm)^{-1}\|_{2} \leq 
		\max_{x \in [\sigma, \lip]} \frac{x}{\reg + x^2} &= \begin{cases}
			\frac{\sigma }{\sigma^2  +\reg} &\mbox{if}\  \reg \leq \sigma^{2}  \\
			\frac{1}{2\sqrt{\reg}} & \mbox{if} \ \sigma^{2} \leq \reg \leq \lip^{2} \\
			\frac{\lip}{\lip^2 +\reg } & \mbox{if} \ \reg \geq \lip^{2}
		\end{cases}.
	\end{align*}
By analyzing each case, we get the claimed  inequality.
\end{proof}

Lemma~\ref{lem:smooth_self_concord} provides a bound on the differences of
gradients of a self-concordant function. It replaces the classical bound we can
have for Lipschitz continuous gradients.
\begin{lemma}\label{lem:smooth_self_concord} For a
	$\localconcord$-self-concordant  strictly convex function
	$\cost$~\cite[Definition 5.1.1]{nesterov2018lectures} and $y, x$ such that
	$\|y-x\|_x<1/\localconcord$, we have,  
	\[
	\|\nabla \cost(y) -\nabla \cost(x)\|_x^*\leq \frac{1}{1-\localconcord\|y-x\|_x} \|y-x\|_x.
	\]
\end{lemma}
\begin{proof}
	Denote $J = \int_0^1 \nabla ^2 \cost(x+ t (y-x)) dt$ and $\hess = \nabla^2
	\cost(x)$,  we have
$
		\|\nabla \cost(y) -\nabla \cost(x) \|_x^* = \|J(y-x)\|_x^* = \|\hess^{-1/2}J\hess^{-1/2}\|_2\|y-x\|_x.
$
	Now $\hess^{-1/2}J\hess^{-1/2}\succeq 0$ since $\cost$ is strictly convex and
	by~\cite[Corollary 5.1.5]{nesterov2018lectures}, we have $J \preceq \nabla^2
	\cost(x)/( 1- \localconcord\|y-x\|_x )$ hence $\|\hess^{-1/2}J\hess^{-1/2}\|_2
	\leq 1/(1- \localconcord\|y-x\|_x )$.
\end{proof}

\subsection{Total Complexity Bound}
Lemma~\ref{lem:reg_strg_cvx} refines the regularization choice of
Theorem~\ref{thm:global_conv} by exploiting an additional assumption of strong
convexity of the costs. 
\begin{lemma}\label{lem:reg_strg_cvx} Consider $\cost$ to be
	$\strgcvx_\cost$-strongly convex and Assumption~\ref{asm:conv} to be
	satisfied. Condition~\eqref{eq:suff_cond_descent} is satisfied by choosing a
	regularization
	\[
	\reg \geq 	\reg(\ctrls)  =\left(1+\frac{\simp}{2(1 + \scaling \|\nabla \cost(\augtraj(\ctrls))\|_2/(\sqrt{\strgcvx_\cost} \condnb_\augtraj))}\right)\smooth_\augtraj \|\nabla \cost(\augtraj(\ctrls))\|_2,
	\]
	for 	$\condnb_\cost = \smooth_\cost/\plcst_\cost$, 	$\condnb_\augtraj =
	\lip_{\augtraj}/\sigma_{\augtraj}$, 	$\scaling =
	\smooth_\augtraj/(\sigma_\augtraj^2\sqrt{\plcst_\cost})$, $\simp =
	4\condnb_\augtraj^2\condnb_\cost (\newsimp+1)$,
	$\newsimp=\smoothess_\cost\lip_{\augtraj}^2/(3\smooth_\cost\smooth_\augtraj)$.
\end{lemma}
\begin{proof}
	Let $\ctrls \in \reals^{\horizon\dimctrl}$,  $\grad =  \nabla \augtraj(\ctrls)
	$, $\hess = \nabla^2 \cost(\augtraj(\ctrls))$. We have using that $G^\top G \succeq \sigma_\augtraj^2 \idm$, i.e., $G^\top G$ invertible, 
	\begin{align*}
		\LQR_\reg(\obj)(\ctrls) & = - \grad (\grad^\top \grad)^{-1}(\hess + \reg (\grad^\top \grad)^{-1})^{-1} \nabla \cost(\augtraj(\ctrls))\\
		&  = - \grad (\grad^\top \grad)^{-1/2}((\grad^\top \grad)^{1/2}\hess (\grad^\top \grad)^{1/2}+ \reg \idm)^{-1}(\grad^\top \grad)^{1/2} \nabla \cost(\augtraj(\ctrls)) .
	\end{align*}
	By bounding each formulation, using $\|\grad (\grad^\top \grad)^{-1}\| \leq 1/\sigma_\augtraj$ for the first formulation, and, $\|\grad (\grad^\top \grad)^{-1/2}\|_2 \leq 1$, $(\grad^\top \grad)^{1/2}\hess (\grad^\top \grad)^{1/2} \succeq \strgcvx_\cost \sigma_\augtraj^2$ for the second formulation,
	\begin{align*}
		\|\LQR_\reg(\obj)(\ctrls)\|_2 /\|\nabla \cost(\augtraj(\ctrls)) \|_2 
		& \leq \min\{ \lip_\augtraj^2/(\strgcvx_\cost \sigma_{\augtraj}\lip_\augtraj^2 + \reg \sigma_{\augtraj}), \lip_\augtraj/(\reg + \strgcvx_\cost \sigma_{\augtraj}^2)\}\\
		& \leq 2\lip_\augtraj/(\reg(1+ \sigma_{\augtraj}/\lip_{\augtraj}) + \strgcvx_\cost \sigma_\augtraj (\sigma_{\augtraj} +\lip_{\augtraj})) \\
		& \leq 2\lip_{\augtraj}/(\reg + \strgcvx_\cost \sigma_{\augtraj} \lip_{\augtraj}),
	\end{align*}
	where we used that $\min\{a, b\} \leq 2/(1/a+1/b)$. Hence,
	condition~\eqref{eq:suff_cond_descent} is satisfied if $\reg$ satisfies
	$
	a_1 + a_2/(a_3 + \reg) \leq \reg
	$
	with $a_1 = \smooth_\augtraj \|\nabla \cost(\augtraj(\ctrls))\|_2$, $a_2 =
	2a_0 \lip_\augtraj \|\nabla \cost(\augtraj(\ctrls))\|_2$, $a_3
	=\sigma_\augtraj\lip_\augtraj \strgcvx_\cost$, $a_0 = \smoothess_\cost
	\lip_{\augtraj}^3/3 {+} \smooth_\augtraj\smooth_\cost \lip_\augtraj$. Hence,
	condition~\eqref{eq:suff_cond_descent} is satisfied for
	$
	\reg \geq \reg_0 = (a_1 - a_3 + (a_1+a_3)\sqrt{1+ 4a_2(a_1+a_3)^{-2}})/{2}.
	$
	Since $\sqrt{1+2x} \leq 1 + x $,  we have $\reg_0 \leq a_1 + a_2/(a_1+a_3)$, so
	it suffices to take a regularization larger than or equal to 
	\begin{align*}
		\reg(\ctrls) & =  \smooth_\augtraj \|\nabla \cost(\augtraj(\ctrls))\|_2 +  \frac{2\lip_{\augtraj}^2(\smoothess_\cost \lip_{\augtraj}^2/3 {+} \smooth_\augtraj\smooth_\cost )\|\nabla \cost(\augtraj(\ctrls))\|_2 }{\smooth_\augtraj\|\nabla \cost(\augtraj(\ctrls))\|_2 + \sigma_\augtraj\lip_\augtraj \strgcvx_\cost }  .
	\end{align*}
\end{proof}

Lemma~\ref{lem:detailed_cplxity_strgcvx} details the computations of the
complexity bounds of the ILQR algorithm in the case of strongly convex costs,
used in~\eqref{eq:global_strg_cvx_iter} before taking into account the local
quadratic convergence.
\begin{lemma}\label{lem:detailed_cplxity_strgcvx} Consider the notations and
	assumptions of Theorem~\ref{thm:global_local_conv}. The number of  iterations
	of the ILQR  algorithm with regularizations  
	\[
	\reg_k  =\left(1+\frac{\simp}{2(1 + \scaling \|\nabla \cost(\augtraj(\ctrls^{(k)}))\|_2/(\sqrt{\strgcvx_\cost} \condnb_\augtraj))}\right)\smooth_\augtraj \|\nabla \cost(\augtraj(\ctrls^{(k)}))\|_2,
	\]
	needed to reach an accuracy $\varepsilon$ is at most
	\[
	k \leq 	2\condnb_\cost \ln\left(\frac{\delta_0}{\varepsilon}\right) 
	+ 4\scaling \left(\sqrt{\delta_0} - \sqrt{\varepsilon }\right) 
	+ 2 \simp \ln\left(\frac{\scaling\sqrt{\delta_0} + \condnb_\augtraj}{\scaling\sqrt{\varepsilon} + \condnb_\augtraj}\right),
	\]
	where $\condnb_\cost = \smooth_\cost/\plcst_\cost$, $\condnb_\augtraj =
	\lip_{\augtraj}/\sigma_{\augtraj}$, $\scaling =
	\smooth_\augtraj/(\sigma_\augtraj^2\sqrt{\plcst_\cost})$, $\concord =
	\smoothess_\cost/(2\plcst_\cost^{3/2})$, $\simp =
	4\condnb_\augtraj^2\condnb_\cost (\newsimp+1)$,
	$\newsimp=\smoothess_\cost\lip_{\augtraj}^2/(3\smooth_\cost\smooth_\augtraj)$
\end{lemma}
\begin{proof}
	Let $\var \in \reals^{{\horizon\dimctrl}}$ and $\auxctrls =
	\LQR_{\reg(\var)}(\obj)(\var)$ for 
	\[
	\reg(\var)  =\left(1+\frac{\simp}{2(1 + \scaling \|\nabla \cost(\augtraj(\var))\|_2/(\sqrt{\strgcvx_\cost} \condnb_\augtraj))}\right)\smooth_\augtraj \|\nabla \cost(\augtraj(\var))\|_2 .
	\]
	As shown in Lemma~\ref{lem:reg_strg_cvx}, the chosen regularization ensures
	the sufficient decrease~\eqref{eq:suff_cond_descent}. As in~\eqref{eq:decrease_iter}, in the proof of Theorem~\ref{thm:global_conv},
	we get that 
	\[
	\obj(\var+\auxctrls) - \obj(\var) \leq - \frac{1}{2}\frac{\sigma_\augtraj^2}{\sigma_\augtraj^2 \smooth_\cost + \reg(\var)} \|\nabla \cost(\augtraj(\var))\|_2^2=-\frac{b_1 x^3 + b_2x^2}{b_3 x^2 + b_4 x + 1},
	\]
	where $x = \|\nabla \cost(\augtraj(\var))\|_2$, $b_1 =
	\smooth_\augtraj/(2\lip_\augtraj \strgcvx_\cost \smooth_\cost
	\sigma_\augtraj)$, $b_2 = 1/(2\smooth_\cost)$, $b_3 = \smooth_\augtraj^2/(
	\sigma_\augtraj^3 \lip_\augtraj \strgcvx_\cost\smooth_\cost)$, $b_4 =
	\smooth_\augtraj/(\sigma_\augtraj\lip_\augtraj \strgcvx_\cost) +
	\smooth_\augtraj/(\sigma_\augtraj^2\smooth_\cost) + 2a_0
	/(\sigma_\augtraj^3\strgcvx_\cost \smooth_\cost)$.  The function $f_1(x) =
	(b_1 x^3 + b_2x^2)/(b_3 x^2 + b_4 x + 1)$ is increasing and since $\cost$ is
	strongly convex, we have that $\|\nabla \cost(\augtraj(\var))\|_2^2 \geq
	\strgcvx_\cost (\cost(\augtraj(\var))- \cost^*) = \strgcvx_\cost \delta$ for
	$\delta = \obj(\var) - \obj^*$.  Hence,  as in the proof of
	Theorem~\ref{thm:global_conv},  we get that the total number of iterations to
	reach an accuracy $\varepsilon$ is at most $k \leq f_2(\delta_0) -
	f_2(\varepsilon)$ where 
	\[
	f'_2(\delta) = \frac{1}{f_1(\sqrt{\strgcvx_\cost\delta})} = \frac{1+c_1 \delta^{1/2} + c_2 \delta}{c_3 \delta + c_4\delta^{3/2}},
	\]
	where $c_1 = \scaling(\condnb_\augtraj^{-1} +2\condnb_\augtraj +
	\condnb_\cost^{-1}) + 4\condnb_\augtraj^3\concord /(3\condnb_\cost)$, $c_2 =
	\scaling^2/(\condnb_\augtraj\condnb_\cost)$, $c_3 = 1/(2\condnb_\cost)$, $c_4
	= \scaling/(2\condnb_\augtraj\condnb_\cost)$. By standard integration, we have
	that an antiderivative of $f'_2$ is 
	\begin{align*}
		f_2(x) & = \frac{\ln(\delta)}{c_3} + \frac{2c_2}{c_4} \sqrt{\delta} - 2\frac{(c_2c_3^2 -c_4c_1c_3 + c_4^2)}{c_3c_4^2}\ln(c_4\sqrt{\delta} + c_3)\\
		& = 2 \condnb_\cost \ln(\delta) + 4 \scaling\sqrt{\delta} + 8\condnb_\augtraj^2 ( \condnb_\cost + 2 \condnb_\augtraj^2 \concord/(3\scaling) ) \ln(\scaling\sqrt{\delta}/(2\condnb_\cost\condnb_\augtraj) + 1/(2\condnb_\cost)).
	\end{align*}
	The result follows.
\end{proof}

We present below the proof of Corollary~\ref{cor:line_search} that ensures the
validity of the line-search procedure presented in
Algorithm~\ref{algo:RGGN_linesearch}.
\linesearch*
\begin{proof}
	Define for $\ctrls \in \reals^{\horizon\dimctrl}$, 
	\[
	\bar \reg(\ctrls) = \left(1+\frac{\simp}{2(1 + \scaling \sqrt{\obj(\ctrls) - \min_{\auxctrls \in \reals^{\horizon\dimctrl}} \obj(\auxctrls)} / \condnb_\augtraj))}\right)\smooth_\augtraj
	\]
	Since $\cost$ is strongly convex, we have that $\|\nabla
	\cost(\augtraj(\ctrls))\|_2 \geq \sqrt{\strgcvx_\cost}(\cost(\augtraj(\ctrls))
	-\min_{\auxstates \in \reals^{\horizon\dimstate}} \cost(\auxstates)) =
	\sqrt{\strgcvx_\cost}(\obj(\ctrls) -\min_{\auxctrls \in
	\reals^{\horizon\dimstate}} \obj(\auxctrls)) $, where we recall that
	$\min_{\auxstates \in \reals^{\horizon\dimstate}} \cost(\auxstates) =
	\min_{\auxctrls \in \reals^{\horizon\dimstate}} \obj(\auxctrls)$ as shown in
	Theorem~\ref{thm:global_conv}. Hence, we have that $\bar \reg(\ctrls) \|\nabla
	\cost(\augtraj(\ctrls)) \|_2 \geq \reg(\ctrls) $ for $\reg(\ctrls)$ defined in
	Lemma~\ref{lem:reg_strg_cvx}. Therefore, by Lemma~\ref{lem:reg_strg_cvx},  the
	line-search procedure of Algorithm~\ref{algo:RGGN_linesearch} at the
	$k$\textsuperscript{th} iteration necessarily terminates with a scaled
	regularization 
	$
	\bar \nu_k \leq 2 \bar \reg(\ctrls_k) 
	$
	since we chose $\bar \reg_{-1} \leq \bar \reg(\ctrls_0)$ and since $\bar
	\reg(\ctrls_k)$ necessarily increases over the iterations as $\obj(\ctrls_k)$
	decreases when  condition~\eqref{eq:suff_cond_descent} is satisfied.
	
	Moreover, since $\bar \reg(\ctrls)$ is upper bounded by $(1+
	\simp/2)\smooth_\augtraj $ the total number of calls to oracles made by the
	line-search inner loop to satisfy the decrease condition after $k$ iterations
	is at most 
	\[
	k + \left\lceil\log_2\left(\frac{(1+\simp/2) \smooth_\augtraj}{\bar \reg_{-1}}\right)\right\rceil.
	\]
	
	Since the line-search ensures the decrease
	condition~\eqref{eq:suff_cond_descent},  we have, as in
	Theorem~\ref{thm:global_local_conv} that for $\reg_k = \bar \reg_k \|\nabla
	\cost(\augtraj(\ctrls_k))\|_2$, 
	\begin{align*}
		\obj(\ctrls_{k+1}) - \obj(\ctrls_k) 
		& \leq - \frac{1}{2}\frac{\sigma_\augtraj^2}{\sigma_\augtraj^2 \smooth_\cost + \reg_k} \|\nabla \cost(\augtraj(\ctrls))\|_2^2 \\
		& \leq  - \frac{1}{2}\frac{\sigma_\augtraj^2}{\sigma_\augtraj^2 \smooth_\cost + 2\bar \reg(\ctrls_k) \|\nabla \cost(\augtraj(\ctrls))\|_2}\|\nabla \cost(\augtraj(\ctrls))\|_2^2 \\
		& \leq - \- \frac{1}{4}\frac{\sigma_\augtraj^2}{\sigma_\augtraj^2 \smooth_\cost + \bar \reg(\ctrls_k) \|\nabla \cost(\augtraj(\ctrls))\|_2}\|\nabla \cost(\augtraj(\ctrls))\|_2^2.
	\end{align*}
	The rest of the proof of Lemma~\ref{lem:detailed_cplxity_strgcvx} follows, and
	we get that the number of iterations of Algorithm~\ref{algo:RGGN_linesearch} to
	reach an accuracy $\varepsilon$ is at most $2 k(\delta_0, \varepsilon)$ for
	$k(\delta_0, \varepsilon)$ defined as in Theorem~\ref{thm:global_local_conv}.
	
	For the quadratic convergence rate, we have, with the notations of the proof
	of Theorem~\ref{thm:global_local_conv}, that 
	\[
	\reg_k /\lambda_\cost(\augtraj(\ctrls_k))
	\leq 2 \bar \reg(\ctrls_k) \|\nabla \cost(\augtraj(\ctrls_k))\|_2^2/\lambda_\cost(\augtraj(\ctrls_k)) 
	\leq 2 \sqrt{\smooth_\cost} (\smooth_\augtraj  + 2 \lip_{\augtraj}( \smoothess_\cost \lip_{\augtraj}^2/3 {+} \smooth_\augtraj\smooth_\cost )  /(\sigma_\augtraj  \strgcvx_\cost)).\] 
	The rest of the proof follows with a slightly modified quadratic convergence
	gap.
\end{proof}

\section{Convergence Analysis of IDDP}\label{app:conv_iddp}
The ILQR and IDDP algorithms differ only by the rolling-out phase. The former
uses the linearized dynamics, while the latter uses the original shifted
dynamics. We formalize the roll-out phase in Definition~\ref{def:roll_out}. 
\begin{definition}\label{def:roll_out} 
  We define the roll-out of $\horizon$ policies $\pi_t: \reals^\dimstate
  \rightarrow \reals^\dimctrl$ along $\horizon$ dynamics $\phi_t:
  \reals^{\dimstate} \times \reals^{\dimctrl} \rightarrow \reals^\dimstate$ from
  $\state_0$ as 
  \begin{align*}
    \rollout:\state_0, (\phi_t)_{t=0}^{\horizon-1}, (\pi_t)_{t=0}^{\horizon-1}
    \mapsto \ &  (\ctrl_0, \ldots, \ctrl_{\horizon-1}) \\
    \mbox{s.t.}  \ & 
      \ctrl_t =\pi_t(\state_t), 
      \quad \state_{t+1} = \phi_t(\state_t, \ctrl_t),
      \ \mbox{for} \
      t \in \{0, \ldots, \horizon-1\}. 
  \end{align*}
\end{definition}
With the notations of Definition~\ref{def:roll_out}, denoting 
\begin{align*}
  \Phi(\states, \ctrls) 
  & = (\phi_0(\state_0, \ctrl_0); \ldots; \phi_{\horizon-1}(\state_{\horizon-1}, \ctrl_{\horizon-1})) \\
  \pi(\states) 
  & = (\pi(\state_0); \ldots; \pi(\state_{\horizon-1}))
\end{align*}
for $\states = (\state_1; \ldots; \state_\horizon), \ctrls = (\ctrl_0; \ldots;
\ctrl_{\horizon-1})$, the trajectory $\states = \phi^{[\horizon]}(\state_0, \ctrls)$ associated to $\ctrls$ is the unique solution of $\states = \Phi(\states, \ctrls)$
and the roll-out is the unique solution of 
\begin{equation}\label{eq:rollout_implicit_eq}
  \ctrls = \pi(\states), \quad \ \states = \Phi(\states, \ctrls).
\end{equation}

Given a trajectory $(\state_1; \ldots; \state_\horizon) = \traj(\state_0,
\ctrls)$ computed from $\ctrls = (\ctrl_0; \ldots; \ctrl_{\horizon-1})$, and
$\horizon$ policies $(\pi_t)_{t=0}^{\horizon-1}$ computed in the backward pass
of Algorithm~\ref{algo:lqr_ddp}, the ILQR and IDDP algorithms can be expressed as 
\begin{align}
  \LQR_\reg(\obj)(\ctrls) 
  & = \rollout(0, (\lin_t)_{t=0}^{\horizon-1}, (\pi_t)_{t=0}^{\horizon-1}) \nonumber \\
  \mbox{for} \ 
  \lin_t(\auxstate_t, \auxctrl_t) 
  & = \lin_{\dyn_t}^{\state_t, \ctrl_t}(\auxstate_t, \auxctrl_t)
    = \nabla_{\state_t} \dyn_t(\state_t, \ctrl_t)^\top \auxstate_t
      + \nabla_{\ctrl_t} \dyn_t(\state_t, \ctrl_t)^\top \auxctrl_t \label{eq:lin_dyn_roll}\\
  \DDP_\reg(\obj)(\ctrls) 
  & = \rollout(0, (\delta_t)_{t=0}^{\horizon-1}, (\pi_t)_{t=0}^{\horizon-1}) \nonumber\\
  \mbox{for} \ 
  \delta_t(\auxstate_t, \auxctrl_t) 
  & = \delta_{\dyn_t}^{\state_t, \ctrl_t}(\auxstate_t, \auxctrl_t)
    = \dyn_t(\state_t + \auxstate_t, \ctrl_t + \auxctrl_t) - \dyn_t(\state_t, \ctrl_t) \label{eq:shift_dyn_roll}
\end{align}
To analyze the convergence of the IDDP algorithm, we consider how close it is
from the ILQR algorithm. This can be traced as measuring how the roll-out phase
differ between using $\lin_{\dyn_t}^{\state_t, \ctrl_t}(\auxstate_t,
\auxctrl_t)$ or $\delta_{\dyn_t}^{\state_t, \ctrl_t}(\auxstate_t, \auxctrl_t)$
as done in Lemma~\ref{lem:approx_ilqrs}.
\begin{lemma}\label{lem:approx_ilqrs} Given $\horizon$ discrete dynamics
  $\dyn_t: \reals^{\dimstate} \times \reals^{\dimctrl} \rightarrow
  \reals^\dimstate$ and $\horizon$ policies $\pi_t: \reals^\dimstate \rightarrow
  \reals^\dimctrl$ for $t=0, \ldots, \horizon-1$, denote 
  \begin{align*}
    \auxctrls = \rollout(0, (\lin_t)_{t=0}^{\horizon-1}, (\pi_t)_{t=0}^{\horizon-1}),
    \qquad 
    \auxxctrls = \rollout(0, (\delta_t)_{t=0}^{\horizon-1}, (\pi_t)_{t=0}^{\horizon-1}),
  \end{align*}
  for $\lin_t$, $\delta_t$ defined as in~\eqref{eq:lin_dyn_roll}
  and~\eqref{eq:shift_dyn_roll} from $(\state_1; \ldots; \state_\horizon) =
  \traj(\state_0, \ctrls)$ and $\ctrls = (\ctrl_0; \ldots; \ctrl_{\horizon-1})$.
  Suppose that the policies are affine of the form are $\pi_t(\state_t) =
  K_t\state_t + k_t$, and that all dynamics are Lipschitz continuous with
  Lipschitz-continuous Jacobians. Then the directions $\auxctrls$ and
  $\auxxctrls$ differ by
  \[
    \|\auxxctrls - \auxctrls\|_2  \leq \eta(K) \|\auxctrls\|_2^2
  \]
  for $\eta(K)$ an increasing function of $\|K\|_2$ detailed in the proof.
\end{lemma}
\begin{proof}
  In this proof, we ignore the dependency w.r.t. $\initstate$ and denote simply
  $\traj(\ctrls) = \traj(\initstate, \ctrls)$. Similarly, we denote
  $\ell^{[\horizon]}(\auxctrls) = \ell^{[\horizon]}(0, \auxctrls)$ and
  $\delta^{[\horizon]}(\auxxctrls) = \delta^{[\horizon]}(0, \auxxctrls)$ the
  trajectories associated to the linearized and shifted dynamics starting from
  $0$. For $\auxstates=(\auxstate_1;\ldots; \auxstate_\horizon)$, denote
  $\pi(\auxstates) = (\pi_0(0); \pi_1(\auxstate_1);
  \ldots;\pi_{\horizon-1}(\auxstate_{\horizon-1}))$. Denoting $K =
  \sum_{t=2}^{\horizon} e_te_{t-1}^\top \otimes K_{t-1} \in \reals^{\horizon
  \dimctrl \times \horizon \dimstate}$, $k = (k_0; \ldots; k_{\horizon-1}) \in
  \reals^{\horizon\dimstate}$ we have that $\pi(\auxctrls) = K\auxctrls + \polcst$ with 
  \[
    \begin{pmatrix}
			0 & \ldots &  &\ldots & 0 \\
			K_1 & \ddots &  &  & \vdots \\
			0 & \ddots& &  &  \\
			\vdots & \ddots & \ddots & \ddots & \vdots \\
			0 & \ldots & 0 & K_{\horizon-1} & 0 
		\end{pmatrix}.
  \]
  The roll-outs are defined as the solutions of
  \[
    \auxctrls = \pi(\lin^{[\horizon]}(\auxctrls)),
    \quad
    \auxxctrls = \pi(\delta^{[\horizon]}(\auxxctrls))
  \]
  From Lemma~\ref{lem:grad_hess_detailed}, the linearized trajectories can be
  expressed as 
  \[
    \lin^{[\horizon]}(\auxctrls) =  (\idm-A)^{-1}B\auxctrls
  \]
  for $B = \nabla_\ctrls F(\states, \ctrls)^\top$, $A = \nabla_\states F(\states,
  \ctrls)^\top$, with $F(\states, \ctrls) = (\dyn_0(\state_0, \ctrl_0); \ldots;
  \dyn_{\horizon-1}(\state_{\horizon-1}, \ctrl_{\horizon-1}))$ for $\states =
  (\state_1; \ldots; \state_\horizon)$, $\ctrls = (\ctrl_0; \ldots;
  \ctrl_{\horizon-1})$. 
  We have that $\auxctrls$ satisfy
  \[
    \auxctrls = K (\idm - A)^{-1}B\auxctrls + \polcst.
  \]
  Note that $A + BK = \sum_{t=2}^\horizon e_t e_{t-1}^\top \otimes (A_{t-1}+B_{t-1}
  K_{t-1})$ for $B_t = \nabla_{\ctrl_t} \dyn_t(\state_t,
  \ctrl_t)^\top, A_t = \nabla_{\state_t} \dyn_t(\state_t, \ctrl_t)^\top$, that is
  \[
    B K + A = \begin{pmatrix}
    0 & \ldots &  &\ldots & 0 \\
    U_1 & \ddots &  &  & \vdots \\
    0 & \ddots& &  &  \\
    \vdots & \ddots & \ddots & \ddots & \vdots \\
    0 & \ldots & 0 & U_{\horizon-1} & 0 
  \end{pmatrix},
  \]
  for $U_t = A_t + B_tK_t$.
  So $\idm - A - B K$ is invertible by solving an autoregressive problem. We
  then have that $(\idm - K(\idm -A)^{-1} B)(\idm + K(\idm -A -BK)^{-1}B) =
  \idm$ such that the solution of $\auxctrls =
  \pi(\lin^{[\horizon]}(\auxctrls))$ is 
  \[
    \auxctrls = \polcst  + K(\idm -A - BK)^{-1} B \polcst. 
  \]
  For $\auxxctrls=(\auxxctrl_0;\ldots;\auxxctrl_{\horizon-1})$, denote $\auxxstates=(\auxxstate_1; \ldots
  ;\auxxstate_\horizon) = \delta^{[\horizon]}(\auxxctrls)$ s.t.
  $\auxxstate_{t{+}1} {=} \dyn_t(\state_t {+} \auxxstate_t, \ctrl_t {+}
  \auxxctrl_t) {-} \dyn_t(\state_t, \ctrl_t)  $ for $t \in \{0, \ldots,
  \horizon{-}1\}$, with $\auxxstate_0 {=} 0$. By the mean value theorem, for all
  $t \in \{0, \ldots, \horizon{-}1\}$, there exists $\zeta_{t, 1}, \ldots,
  \zeta_{t, \dimstate} \in \reals^\dimstate$, $\eta_{t, 1}, \ldots, \eta_{t,
  \dimstate} \in \reals^\dimctrl$ s.t. for all $i \in \{1, \ldots, \dimstate\}$,
  denoting $f_i$ the i\textsuperscript{th} coordinate of $\dyn$, we have
	\begin{align*}
		\dyn_i(\state_t + \auxxstate_t, \ctrl_t + \auxxctrl_t)  - \dyn_i(\state_t, \ctrl_t + \auxxctrl_t) 
		& = \nabla_{\state_t + \zeta_{t, i}} \dyn_i(\state_t + \zeta_{t, i}, \ctrl_t{+}\auxxctrl_t)^\top \auxxstate_t\\
		\dyn_i(\state_t, \ctrl_t {+} \auxxctrl_t)  - \dyn_i(\state_t, \ctrl_t ) 
	& 	= \nabla_{\ctrl_t{+}\eta_{t, i}} \dyn_i(\state_t, \ctrl_t{+}\eta_{t, i})^\top \auxxstate_t,
	\end{align*}
	with $\|\zeta_{t, i}\|_2 \leq \|\auxxstate_t\|_2$ and $\|\eta_{t, i}\|_2 \leq
	\|\auxxctrl_t\|_2$. We can then write the dynamics of $\auxxstate_t$ as  
	\begin{gather*}
		\auxxstate_{t+1} = C_t \auxxstate_t + D_t \auxxctrl_t \quad \mbox{for}\  t \in \{0, \ldots, \horizon-1\}\\
		C_t = \sum_{i=1}^\dimstate e_i \otimes \nabla_{\state_t {+} \zeta_{t, i}} \dyn_i(\state_t {+} \zeta_{t, i}, \ctrl_t{+}\auxxctrl_t)^\top \qquad  	
		D_t = \sum_{i=1}^\dimstate e_i \otimes \nabla_{\ctrl_t{+}\eta_{t, i}} \dyn_i(\state_t, \ctrl_t{+}\eta_{t, i})^\top.
	\end{gather*}
	Denoting  
	$C = \sum_{t=2}^{\horizon} e_te_{t-1} \otimes C_{t-1}$,  
	$D= \sum_{t=1}^\horizon e_te_t^\top \otimes D_{t{-}1}$, we get that
	$\delta^{[\horizon]}(\ctrls) = (\idm -C)^{-1} D \auxctrls$. Since $\auxxctrls =
	\pi(\delta^{[\horizon]}(\ctrls))$, we get that $\auxxctrls$
	satisfies
  \[
    \auxxctrls = K (\idm -C)^{-1} D \auxxctrls + \polcst.
  \]
  The solution of this system can be found as before as 
  \[
    \auxxctrls = \polcst + K(\idm -C -DK)^{-1} D\polcst.
  \]
  We then have 
  \begin{align*}
    \|\auxxctrls - \auxctrls\|_2 
    & \leq \|K\|_2 \|(\idm -C -DK)^{-1} D - (\idm -A - BK)^{-1} B\|_2 \|\polcst\|_2
  \end{align*}
  Then, we decompose the middle term as
  \begin{align*}
    & (\idm -C -DK)^{-1} D - (\idm -A - BK)^{-1} B   \\
    & = ((\idm -C -DK)^{-1} - (\idm -A - BK)^{-1}) D 
    - (\idm -A - BK)^{-1}(B - D) \\
    & = (\idm -C -DK)^{-1}(C-A + (D-B)K)(\idm -A - BK)^{-1}D  \\
    & \quad - (\idm -A - BK)^{-1}(B - D).
  \end{align*}
  We have 
  $(\idm -A - BK)^{-1} 
  = \sum_{t=0}^{\horizon-1} (A+BK)^t$ 
  since $(A+BK)^\horizon =0$. 
  So we get
  \[
  \|(\idm - A - BK)^{-1}\|_2 
  \leq \sum_{t=0}^{\horizon-1}\|A+BK\|_2^t 
  \leq 
  \sum_{t=0}^{\horizon-1}(\lip_{\dyn}^\state + \lip_{\dyn}^\ctrl \|K\|_2)^t 
  \coloneqq S_1(K),
  \] 
  for $\lip_{\dyn}^\state , \lip_{\dyn}^\ctrl $ defined as in Lemma~\ref{lem:smooth_traj_from_dyn_time_varying}.
  Similarly, we have 
  \[
    \|(\idm -C -DK)^{-1}\|_2 \leq 
    \sum_{t=0}^{\horizon-1}
    (\dimstate \lip_{\dyn}^\state + \dimstate\lip_{\dyn}^\ctrl \|K\|_2)^t 
    \coloneqq S_2(K).
  \]
  Using the block structure of the matrices, we have, using that $\|\eta_{t, i}
  \| \leq \|\auxxctrl_t\|_2$, $\|\zeta_{t, i} \| \leq \|\auxxstate_t\|_2$,
  \begin{align*}
    \|B - D\|_2 & \leq \dimstate \smooth_\dyn^{\ctrl\ctrl} \|\auxxctrls\|_2 \\
    \|A-C\|_2 & \leq 
    \dimstate (
      \smooth_\dyn^{\state\state} \|\auxxstates\|_2
    + \smooth_\dyn^{\state\ctrl} \|\auxxctrls\|_2
    ), \\
    \|D\|_2 & \leq \dimstate \lip_{\dyn}^\ctrl,
  \end{align*}
  for $\smooth_\dyn^{\ctrl\ctrl}, \smooth_\dyn^{\state\state}, \smooth_\dyn^{\state\ctrl}$
  defined as in Lemma~\ref{lem:smooth_traj_from_dyn_time_varying}.
  In addition, we have that $\|\auxxstates\|_2\leq \lip_{\traj}\|\auxxctrls\|_2$,
  where $\lip_{\traj}$ is the Lipschitz-constant of $\traj$ computed in
  Lemma~\ref{lem:smooth_traj_from_dyn_time_varying}.

  So in total we get that 
  \begin{align*}
    \|\auxxctrls - \auxctrls\|_2 
    & \leq \|K\|_2 
    (\dimstate^2 S_1(K)S_2(K) \lip_\dyn^\ctrl(
      \smooth_\dyn^{\state\state} \lip_\traj 
      + \smooth_\dyn^{\state\ctrl} 
      + \smooth_\dyn^{\ctrl\ctrl}\|K\|_2
      ) 
    + \dimstate S_1(K)\smooth_{\dyn}^{\ctrl \ctrl})\|\auxxctrls\|_2 \|\polcst\|_2\\
    & \coloneqq \eta_1(K) \|\auxxctrls\|_2 \|\polcst\|_2
  \end{align*}
  Now since $\auxxctrls = \polcst + K(\idm -C -DK)^{-1} D\polcst$ and $\polcst = \auxctrls - K
  (\idm - A)^{-1}B\auxctrls$, we have
  \begin{align*}
    \|\auxxctrls\|_2 & \leq (1+\dimstate\|K\|_2S_2(K)\lip_\dyn^\ctrl)\|\polcst\|_2, \\
    \|\polcst\|_2 & \leq (1+\|K\|_2\lip_\traj)\|\auxctrls\|_2.
  \end{align*}
  Hence, we get 
  \begin{align*}
    \|\auxxctrls - \auxctrls\|_2 
    & \leq \eta_1(K) 
    (1+\dimstate\|K\|_2S_2(K)\lip_\dyn^\ctrl)(1+\|K\|_2\lip_\traj)^2\|\auxctrls\|_2^2 \\
    & \coloneqq \eta(K) \|\auxctrls\|_2^2.
  \end{align*}
\end{proof}

It remains to bound the Lipschitz continuity constant of the policies derived in
the backward pass of the ILQR and IDDP algorithms. In the general case, i.e.,
problem~\eqref{eq:discrete_pb_ctrl_cost},
Lemma~\ref{lem:bounded_policies_gen_cvg} shows that the policies are Lipschitz
continuous with a Lipschitz continuity parameter independent of $\reg$ provided
that $\reg$ is sufficiently large.
For the restricted problem~\eqref{eq:discrete_pb}, the policies are Lipschitz continuous with a Lipschitz continuous parameter independent of $\reg$ unconditionally, as shown in Lemma~\ref{lem:bounded_policies_restricted_case}.
\begin{lemma}\label{lem:bounded_policies_gen_cvg} Consider
  problem~\eqref{eq:discrete_pb_ctrl_cost} with dynamics and costs Lipschitz
  continuous with Lipschitz-continuous Jacobians as in
  Assumption~\ref{asm:gen_cvg}. For any $\nu \geq 2
  \lip_\fullaugtraj\smooth_\fullcost$, the policies $\pi_t:\auxstate_t \mapsto
  K_t\auxstate_t + k_t$ computed in Algorithm~\ref{algo:lqr_ddp} are
  well-defined and Lipschitz continuous with $\|K_t\|_2\leq c$ for $c$
  independent of $\reg$.
\end{lemma}
\begin{proof}
  For $t\in \{0, \ldots, \horizon-1\}$, denote $\dyn^{[t:\horizon]}(\state_t, \ctrls_{[t:\horizon-1]}) = (\state_{t+1}; \ldots, \state_\horizon)$ the control of the dynamics $\dyn_t, \ldots, \dyn_{\horizon-1}$ starting from $\state_t$ with control variables $\ctrls_{[t:\horizon-1]} = (\ctrl_t; \ldots;\ctrl_{\horizon-1})$. For $t=0$, denoting $[0:\horizon] = [\horizon]$, we retrieve Definition~\ref{def:traj_func_time_varying}. 
  Define similarly 
  $\fullaugtraj^{[t:\horizon]}(\ctrls_{[t:\horizon-1]}) 
  = (\dyn^{[t:\horizon]}(\state_t, \ctrls_{[t:\horizon-1]}), \ctrls_{[t:\horizon-1]})$ 
  and 
  $\fullcost^{[t:\horizon]}(\state_{[t+1;\horizon]}, \ctrls_{[t:\horizon-1]}) 
  = \sum_{s=t}^{\horizon-1} \cost_s(\state_s, \ctrl_s) 
  + \cost_\horizon(\state_\horizon)$. 
  The $t$\textsuperscript{th} policy $\pi_t(\auxstate_t)$ is formally equal to $\auxctrl_t^*(\auxstate_t)$ for 
  \begin{align*}
    \auxctrl_t^*(\auxstate_t), \ldots \auxctrl_{\horizon-1}^*(\auxstate_t) 
    & = \argmin_{\ctrl_t, \ldots, \ctrl_{\horizon-1}}
    \qua_{\fullcost^{[t:\horizon]}}^{\dyn^{[t:\horizon]}(\state_t, \ctrls_{[t:\tau-1]})} 
    (\lin_{\dyn^{[t:\horizon]}}^{\state_t, \ctrls_{t:\tau-1}}(\auxstate_t, \auxctrls_{[t:\tau-1]}), \auxctrls_{[t:\tau-1]}) 
    + \frac{\reg}{2} \|\auxctrls_{[t:\horizon-1]}\|_2^2 \\
    & = (\reg \idm  
    + \nabla \fullaugtraj^{[t:\horizon]}(\ctrls_{[t:\horizon]})
    \nabla^2 \fullcost^{[t:\horizon]}(\fullaugtraj^{[t:\horizon]}(\ctrls_{[t:\horizon]}))
    \nabla \fullaugtraj^{[t:\horizon]}(\ctrls_{[t:\horizon]})^\top
    )^{-1} 
    (A \auxstate_t + a)
  \end{align*}
  for some $A$, $a$ independent of $\reg$.
  For $\nu > \smooth_{\fullcost^{[t:\horizon]}} \lip_{\fullaugtraj^{[t:\horizon]}}^2$ the policies are well-defined. Since $ \smooth_{\fullcost^{[t:\horizon]}} = \smooth_{\fullcost^{[\horizon]}} =\smooth_\fullcost$ and $\lip_{\fullaugtraj^{[t:\horizon]}}^2 \leq \lip_{\fullaugtraj^{[0:\horizon]}}^2 = \lip_\fullaugtraj^2$, the policies are well-defined for any $\nu \geq 2
  \lip_\fullaugtraj\smooth_\fullcost$. Moreover, for any $\nu \geq 2
  \lip_\fullaugtraj\smooth_\fullcost$, 
  \[
  \|(\reg \idm  
  + \nabla \fullaugtraj^{[t:\horizon]}(\ctrls_{[t:\horizon]})
  \nabla^2 \fullcost^{[t:\horizon]}(\fullaugtraj^{[t:\horizon]}(\ctrls_{[t:\horizon]})))^{-1}\|_2
  \leq \frac{1}{
  \lip_\fullaugtraj\smooth_\fullcost}.
  \]
  Hence, the associated policy $\pi_t$ is at most $1/
  (\lip_\fullaugtraj\smooth_\fullcost)$ Lipschitz-continuous.
\end{proof}
\begin{corollary}\label{cor:approx_ilqrs_gen} Consider
  problem~\eqref{eq:discrete_pb_ctrl_cost} with dynamics and costs Lipschitz
  continuous with Lipschitz-continuous Jacobians as in
  Assumption~\ref{asm:gen_cvg}. For any $\nu \geq 2
  \lip_\fullaugtraj\smooth_\fullcost$, there exists a constant $\eta$
  independent of $\reg$, such that the ILQR or IDDP directions
  $\auxctrls=\LQR_\reg(\obj)(\ctrls)$ and $\auxxctrls=\DDP_\reg(\obj)(\ctrls)$
  on any control variables $\ctrls \in \reals^{\horizon \dimctrl}$ differ by
  \[
    \|\auxxctrls - \auxctrls\|_2  \leq \eta\|\auxctrls\|_2^2.
  \]
\end{corollary}
\begin{proof}
  The result follows from Lemma~\ref{lem:approx_ilqrs} and~\ref{lem:bounded_policies_gen_cvg}.
\end{proof}

\begin{lemma}\label{lem:bounded_policies_restricted_case}
  
  Consider Algorithm~\ref{algo:lqr_ddp} applied to problems of the
  form~\eqref{eq:discrete_pb}, that is, such that $R_t = 0, Q_t = 0, q_t = 0$.
  Assume in addition that the costs are strongly convex, the dynamics are
  surjective and both costs and dynamics are smooth as described in
  Assumption~\ref{asm:conv}. The policies $\pi_t: \auxstate_t \rightarrow K_t
  \auxstate_t + k_t$ computed in Algorithm~\ref{algo:lqr_ddp} are always
  well-defined and such that $\|K_t\|_2 \leq c$ for some $c$ independent of
  $\reg$.
\end{lemma}
\begin{proof}
  Consider $K_t, J_t$ defined in Algorithm~\ref{algo:lqr_ddp} for a
	command $\ctrls \in \reals^{\horizon\dimctrl}$, a regularization $\reg>0$ and no control costs ($R_t = 0, Q_t = 0, q_t = 0$).
	By recursion, we have that $J_t$ is positive definite, since 
  \[
    \J_t = \P_t +
	\A_t^\top\J_{t+1}^{1/2}(\idm + \reg^{-1}\J_{t+1}^{1/2}
	B_tB_t^\top\J_{t+1}^{1/2})^{-1}J_{t+1}^{1/2}A_t,
  \]
  and $J_\horizon = P_\horizon$, and $\P_t$ are positive definite. 
  
  In particular, $\J_t \succeq P_t \succeq \strgcvx_\cost\idm$ and for any $t\in \{1, \ldots, \horizon-1\}$,
  \begin{align*}
    \|J_t\|_2 & \leq \smooth_\cost + (\lip_\dyn^\state)^2\|J_{t+1}\|_2 \\
    & \leq \sum_{s=t}^{\horizon} (\prod_{j=t}^{s-1}(\lip_\dyn^\state)^2 )\smooth_\cost,
  \end{align*}
  where here and in the following we use $\lip_\dyn^\state,\lip_\dyn^\ctrl$
  defined in Lemma~\ref{lem:smooth_traj_from_dyn_time_varying}. Therefore, we
  have 
  \[
    \sup_{t\in \{1, \ldots,\horizon\}} \|J_t\|_2 
    \leq \smooth_\cost \sum_{t=1}^\horizon (\lip_\dyn^\state)^{2(s-t)}.
  \]
  On the other hand, we have
  \begin{align*}
    K_t & = - (\reg \idm + B_t^\top J_{t+1}B_t)^{-1}B_t^\top J_{t+1}A_t
  \end{align*}
  The spectral norm of the matrix $(\reg \idm + B_t^\top J_{t+1}B_t)^{-1}B_t^\top$ can be bounded just as in Lemma~\ref{lem:reg_strg_cvx} given the assumptions. Namely, we have, 
  \[
    \|(\reg \idm + B_t^\top J_{t+1}B_t)^{-1}B_t^\top\|_2 \leq \frac{2 \lip_\dyn^\ctrl}{\reg + \strgcvx_\cost \sigma_\dyn \lip_\dyn^\ctrl}.
  \]
  Hence, we have
  \begin{align*}
    \sup_{t\in \{0, \ldots, \horizon-1\}} \|K_t\|_2 
    & \leq  
    \frac{2 \lip_\dyn^\ctrl \lip_\dyn^\state}{\reg + \strgcvx_\cost \sigma_\dyn \lip_\dyn^\ctrl}
    \smooth_\cost \sum_{t=1}^\horizon (\lip_\dyn^\state)^{2(s-t)} \\
    & \leq \frac{2 \lip_\dyn^\state}{\strgcvx_\cost \sigma_\dyn }
    \smooth_\cost \sum_{t=1}^\horizon (\lip_\dyn^\state)^{2(s-t)}.
  \end{align*}

\end{proof}

\begin{corollary}\label{cor:approx_ilqrs} Consider Algorithm~\ref{algo:lqr_ddp}
  applied to problems of the form~\eqref{eq:discrete_pb}, that is, such that
  $R_t = 0, Q_t = 0, q_t = 0$. Assume in addition that the costs are strongly
  convex, the dynamics are surjective and both costs and dynamics are smooth as
  described in Assumption~\ref{asm:conv}. Then there exists a constant $\eta$
  independent of $\reg$, such that the ILQR or IDDP directions
  $\auxctrls=\LQR_\reg(\obj)(\ctrls)$ and $\auxxctrls=\DDP_\reg(\obj)(\ctrls)$
  on any control variables $\ctrls \in \reals^{\horizon \dimctrl}$ differ by
  \[
    \|\auxxctrls - \auxctrls\|_2  \leq \eta\|\auxctrls\|_2^2.
  \]
\end{corollary}
\begin{proof}
  Follows from Lemma~\ref{lem:approx_ilqrs} and~\ref{lem:bounded_policies_restricted_case}.
\end{proof}

\section{Detailed Computations}\label{app:comput}
In this Appendix, we detail some technical computations done in the
paper. 

\subsection{Details on Theorem~\ref{thm:suff_cond}}
\begin{proof}[Details on Eq.~\eqref{eq:decomp_grad_suff_cond}]
Note that $\nabla_{\auxctrls} \Microdyn(\auxstates, \auxctrls) =
\diag((\nabla_{\auxctrl_t} \microdyn(\auxstate_t, \auxctrl_t))_{t=0}^{\ksteps-1}
)$, such that, by definition of $\diag$, 
\begin{align*}
	\nabla_{\auxctrls} \Microdyn(\auxstates, \auxctrls)  &= \sum_{t=1}^{\ksteps} e_te_t^\top \otimes \nabla_{\auxctrl_{t-1}} \microdyn(\auxstate_{t-1}, \auxctrl_{t-1})\\
	&  = \sum_{t=1}^{\ksteps} e_te_t^\top \otimes \partial_{\auxctrl_{t-1}} \diffeoctrl(\auxstate_{t-1}, \auxctrl_{t-1}) e^\top \nabla \diffeostate(\auxstate_{t+1})^{-1}\\ 
	& = \left(\sum_{t=1}^{\ksteps} \partial_{\auxctrl_{t-1}} \diffeoctrl(\auxstate_{t-1}, \auxctrl_{t-1})  e_te_t^\top \otimes 1\right) 
	( \idm \otimes  e^\top ) \left(\sum_{t=1}^{\ksteps} e_te_t^\top \otimes \nabla \diffeostate(\auxstate_{t+1})^{-1}\right)\\
	& =  \diag((\partial_{\auxctrl_t} \diffeoctrl(\auxstate_t, \auxctrl_t) )_{t=0}^{\ksteps-1})
	(\idm \otimes e^\top )
	\diag((\nabla \diffeostate(\auxstate_{t+1})^{-1})_{t=0}^{\ksteps-1}),
\end{align*}
where $k=\dimstate$, and we use that $(A\otimes B ) (C\otimes D) = (AC \otimes
BD)$ for $A, B, C, D$ of appropriate sizes and $1$ is the identity in
$\reals^1$. Similarly, one has that, for $D= \sum_{t=1}^{\ksteps-1}
e_te_{t+1}^\top$ the upper-shift matrix in $\reals^\ksteps =
\reals^{\dimstate}$.
 \begin{align*}
 	\nabla_{\auxstates} \Microdyn(\auxstates, \auxctrls)  &
 	= \sum_{t=1}^{\ksteps-1} e_te_{t+1}^\top \otimes \nabla_{\auxstate_{t}} \microdyn(\auxstate_{t}, \auxctrl_{t})\\
 	& =  \left(\sum_{t=1}^{\ksteps-1} e_te_{t+1}^\top \otimes \idm\right) 
 	\left(\sum_{t=1}^{\ksteps} e_te_t^\top \otimes \nabla_{\auxstate_{t-1}} \microdyn(\auxstate_{t-1}, \auxctrl_{t-1})\right) = (D\otimes \idm)  \diag(\nabla_{\auxstate_t} \microdyn(\auxstate_t, \auxctrl_t)_{t=0}^{\ksteps-1}).
 \end{align*}
 On the other hand, we have
\begin{align*}
	\diag(\nabla_{\auxstate_t} \microdyn(\auxstate_t, \auxctrl_t)_{t=0}^{\ksteps-1})  &
	= \sum_{t=0}^{\ksteps-1} e_{t+1}e_{t+1}^\top \otimes 
	\left(\nabla\diffeostate(\auxstate_t) D^\top  + \nabla_{\auxstate_t} \diffeoctrl(\auxstate_t, \auxctrl_t) e^\top\right)
	\nabla \diffeostate(\auxstate_{t+1})^{-1} \\
	& = \underbrace{\sum_{t=0}^{\ksteps-1} e_{t+1}e_{t+1}^\top \otimes \nabla\diffeostate(\auxstate_t) D^\top 	\nabla \diffeostate(\auxstate_{t+1})^{-1}}_A \\
	& \quad
	+  \underbrace{\sum_{t=0}^{\ksteps-1} e_{t+1}e_{t+1}^\top \otimes \nabla_{\auxstate_t} \diffeoctrl(\auxstate_t, \auxctrl_t) e^\top 	\nabla \diffeostate(\auxstate_{t+1})^{-1}}_B 	\\
A & = \left(\sum_{t=0}^{\ksteps-1} e_{t+1}e_{t+1}^\top \otimes   \nabla\diffeostate(\auxstate_t) \right)
	(\idm \otimes D^\top) 
	\left(\sum_{t=0}^{\ksteps-1} e_{t+1}e_{t+1}^\top \otimes   \nabla\diffeostate(\auxstate_{t+1})^{-1} \right) \\
	& = \diag((\nabla \diffeostate(\auxstate_t))_{t=0}^{\ksteps-1})
	(\idm \otimes D^\top)
	\diag((\nabla \diffeostate(\auxstate_{t+1})^{-1})_{t=0}^{\ksteps-1})  \\
B & = \left(\sum_{t=0}^{\ksteps-1} e_{t+1}e_{t+1}^\top \otimes    \nabla_{\auxstate_t} \diffeoctrl(\auxstate_t, \auxctrl_t)  \right)
	(\idm \otimes e^\top) 
	\left(\sum_{t=0}^{\ksteps-1} e_{t+1}e_{t+1}^\top\otimes   \nabla\diffeostate(\auxstate_{t+1})^{-1} \right)\\
	& = 	\diag((\nabla_{\auxstate_t} \diffeoctrl(\auxstate_t, \auxctrl_t))_{t=0}^{\ksteps-1})
	(\idm \otimes e^\top)
	\diag((\nabla \diffeostate(\auxstate_{t+1})^{-1})_{t=0}^{\ksteps-1}).
\end{align*}
\end{proof}

\begin{proof}[Details on line~(ii) in Eq.~\eqref{eq:final_decomp_suff_cond}]
	 Denote $K_t = \nabla \diffeostate(\auxstate_t)$. Using that $D e_t {=}
	 e_{t-1}$, we have $FA {= } (D\otimes \idm) (\sum_{t=1}^{n} e_te_t^\top
	 \otimes K_{t-1}) = \sum_{t=2}^{n} e_{t-1}e_t^\top \otimes K_{t-1}) =
	 \sum_{t=1}^{n-1} e_te_{t+1}^\top K_t$ and, using that $e_t^\top D =
	 e_{t+1}^\top$,  we have $CF = (\sum_{t=1}^{n} e_te_t^\top \otimes K_{t})
	 (D\otimes \idm) = \sum_{t=1}^{n-1} e_te_{t+1}^\top \otimes K_{t}$. Therefore,
	 we have $FA= CF$, and similarly we can show that $FA^{-1} = C^{-1} F$. 
\end{proof}

\begin{proof}[Details on line~(iii) in Eq.~\eqref{eq:final_decomp_suff_cond}]
	Since $D = \sum_{t=1}^{\ksteps-1} e_te_{t+1}^\top$, we have $D^j =
	\sum_{t=1}^{\ksteps-j} e_te_{t+j}^\top$ hence $D^\ksteps = 0$.  Therefore,
	$F\otimes G = D\otimes D^\top$ is nilpotent of order $\ksteps$. Hence,
	$(\idm-F\otimes G)^{-1} = \sum_{t=0}^{\ksteps-1} D^t\otimes (D^\top)^t$ and
	so, for $F = D\otimes \idm$, we have   $(\idm-F\otimes G)^{-1}F = (
	\sum_{t=0}^{\ksteps-1} D^t\otimes (D^\top)^t) (D\otimes \idm) =
	\sum_{t=1}^{\ksteps-1} D^t\otimes (D^\top)^{t-1}$.
\end{proof}

\begin{proof}[Details on the extension to multi inputs]
	Consider the multi-input case as described in Def.~\ref{def:lin_feedback}. For
	any $\ksteps\geq \orderlin$, $j\in \{1, \ldots, \microdimctrl\}$, $i\in \{1,
	\ldots, \orderlin_j\}$, we have  $\auxxxstate_{\ksteps, j}^{(i)} =
	\auxxctrl_{\ksteps+i-\orderlin_j-1}^{(j)}$. Denote $T =
	\sum_{i=1}^{\ksteps}\sum_{j=1}^{\microdimctrl} e_je_i^\top \otimes
	e_ie_j^\top$ for $e_i, e_j$ canonical vectors of, respectively,
	$\reals^{\ksteps}$ and $\reals^{\microdimctrl}$. For $\auxxctrls =
	(\auxxctrl_0; \ldots;\auxxctrl_{\ksteps-1})$, we have that $\auxxxctrls =
	T\auxxctrls$ reorders the coordinates of $\auxxctrls$ such that $\auxxxctrls =
	(\auxxxctrl_1;\ldots;\auxxxctrl_{\microdimctrl})$ with $\auxxxctrl_j^{(i)} =
	\auxxctrl_{i-1}^{(j)}$ for $i \in \{1, \ldots, \ksteps\}$, $j\in \{1, \ldots,
	\microdimctrl\}$. Hence, we have for any $\ksteps \geq \orderlin$, $j\in \{1,
	\ldots, \microdimctrl\}$, denoting here $e_{i}$ the ith canonical vector in
	$\reals^{\orderlin_j}$, $D_{\orderlin_j}$ the upper-shift matrix in
	$\reals^{\orderlin_j}$, 
	\begin{align*}
	\auxxxstate_{\ksteps, j} & = (D_{\orderlin_j}^{k-1} e_{\orderlin_j}, \ldots, D_{\orderlin_j}e_{\orderlin_j}, e_{\orderlin_j}) \auxxxctrl_j \\
	& = (\underbrace{0_{\orderlin_j}, \ldots, 0_{\orderlin_j}}_{\ksteps-\orderlin_j}, e_1, \ldots, e_{\orderlin_j})\auxxxctrl_j  =  (0_{\orderlin_j}, \ldots, 0_{\orderlin_j}, \idm_{\orderlin_j})\auxxxctrl_j  := C_j \auxxxctrl_j,
	\end{align*}
where  $0_{\orderlin_j}$ is the null vector in $\reals^{\orderlin_j}$ and
$\idm_{\orderlin_j}$ is the identity matrix in $\reals^{\orderlin_j}$. So we get
that 
\begin{align*}
	\auxxstate_{\ksteps} = \left(\sum_{j=1}^{\microdimctrl} e_je_j^\top \otimes C_j\right) \auxxxctrls = \left(\sum_{j=1}^{\microdimctrl} e_je_j^\top \otimes C_j\right) T \auxxctrls,
\end{align*}
i.e., $\auxxstate_{\ksteps} = M \auxxctrls$ with $\sigma_{\min}(M^\top) =1$. 

Consider $\ksteps = \orderlin$ and the notations of the proof of
Theorem~\ref{thm:suff_cond}.  We can write   that $\auxstate_{t+1} =
\diffeostate^{-1}(M \Diffeoctrl(\microtraj(\auxstate_0, \auxctrls),
\auxctrls))$. Hence,
\begin{align*}
	\nabla_\auxctrls \trajin{\ksteps}(\auxstate_0, \auxctrls) & = 
	\left(\nabla_\auxctrls \Diffeoctrl(\auxstates, \auxctrls) 
	{+}  \nabla_\auxctrls \Microdyn(\auxstates, \auxctrls)
	(\idm - \nabla_\auxstates \Microdyn(\auxstates, \auxctrls))^{-1} 
	\nabla_\auxstates \Diffeoctrl(\auxstates, \auxctrls)\right) M^\top
	\nabla \diffeostate(\auxstate_k)^{-1}.
\end{align*}
The discrete time dynamic can be written
\[
\auxstate_{t+1} = \diffeostate^{-1}(J\diffeostate(\auxstate_t) + K\diffeoctrl(\auxstate_t, \auxctrl_t)),
\]
with, denoting $e_{j, i}$ the $i$\textsuperscript{th} canonical vector in
$\reals^{\orderlin_j}$ and $e_{\ell_j}$ the $\ell_j$\textsuperscript{th}
canonical vector in $\reals^{\dimstate}$.
\[
J = \left(\begin{matrix}
	D_{\orderlin_1} & 0 & \ldots & 0 \\
	0 & \ddots & \ddots & \vdots \\
	\vdots & \ddots & \ddots & 0 \\
	0 & \ldots & 0 & D_{\orderlin_\microdimctrl}
\end{matrix}\right), \quad D_{\orderlin_j}= \sum_{i=1}^{\orderlin_j} e_{j, i} e_{j, i+1}^\top, \quad K = \sum_{j=1}^{\microdimctrl} e_{\ell_j} e_j^\top, \quad \ell_j = \sum_{s=1}^{j}\orderlin_s.
\]
Hence, we have for $t \in \{0, \ldots, \horizon-1\}$, 
\begin{align*}
	\nabla_{\auxctrl_t} \microdyn(\auxstate_t, \auxctrl_t) 
	& = \nabla_{\auxctrl_t} \diffeoctrl(\auxstate_t, \auxctrl_t) K^\top \nabla \diffeostate(\auxstate_{t+1})^{-1} \\
	\nabla_{\auxstate_t} \microdyn(\auxstate_t, \auxctrl_t) & = \left(\nabla\diffeostate(\auxstate_t) J^\top  + \nabla_{\auxstate_t} \diffeoctrl(\auxstate_t, \auxctrl_t) K^\top\right)\nabla \diffeostate(\auxstate_{t+1})^{-1}.
\end{align*}
The rest of the proof follows as in the proof of Theorem~\ref{thm:suff_cond} by
redefining  $E = \idm\otimes K^\top$, $G = \idm\otimes J^\top$, $F=D_{\ksteps}
\otimes \idm$ for $D_{\ksteps}$ the upper-shift matrix in $\reals^{\ksteps} =
\reals^{\orderlin}$,  $V {=} \diag((\nabla_{\auxctrl_t} \diffeoctrl(\auxstate_t,
\auxctrl_t) )_{t{=}0}^{\ksteps-1})$. We then get 
\begin{align*}
		\nabla_\auxctrls \trajin{\ksteps}(\auxstate_0, \auxctrls) \nabla \diffeostate(\auxstate_k) & =V\left(\idm -\left(\sum_{i=1}^{\ksteps-1} D_{\ksteps}^i\otimes  K^\top (J^\top)^i\right)A^{-1}Y \right)^{-1}M^\top.
\end{align*}
The result follows for $\ksteps =\orderlin$ and for $\ksteps > \orderlin$ the
same reasoning as in the single input case applies. 
\end{proof}

\subsection{Details on Theorem~\ref{thm:ddp_conv}}
\begin{proof}[Details on Eq.~\eqref{eq:conv_ddp_iter}]
With the notations of Theorem~\ref{thm:global_conv}, we have that 
\[
\delta_{k+1} - \delta_k \leq -\frac{1}{2} \frac{\sigma_\augtraj^2 x^2}{\sigma_{\augtraj}^2 \smooth_\cost + \cst x + \condnb_\cost\condnbddp^2 x^2},
\]
with $x =\|\nabla \cost(\augtraj(\ctrls^{(k)}))\|_2$ and $\delta_k =
\obj(\ctrls^{(k)}) - \obj^*$. The function $f_1: x \rightarrow
\sigma_\augtraj^2 x^2/(2(\sigma_{\augtraj}^2 \smooth_\cost + \cst x +
\condnb_\cost\condnbddp^2 x^2))$ is strictly increasing, so we can follow the
steps of the proof Theorem~\ref{thm:global_conv} and obtain that $f_2(\delta_k)
(\delta_{k+1} - \delta_k) \leq -1$ with 
\[
f'_2(\delta) = \frac{1}{f_1(\sqrt{\strgcvx_\cost \delta})} 
= 2 \condnb_\cost \frac{1}{\delta} + 2 \scaling \cst \frac{1}{\sqrt{\delta}} + 2 \scaling^2 \condnb_\cost\condnbddp^2.
\]
The result follows by integrating $f'_2$ and, as in  the proof
Theorem~\ref{thm:global_conv}, we have that convergence to an accuracy
$\varepsilon$ is ensured after at most $k\leq f_2(\delta_0) - f_2(\varepsilon)$. 
\end{proof}

\section{Additional Numerical Evaluations}\label{app:exp_sup}
\paragraph{Realistic model of a car with tracking cost} 

\begin{figure}[t]
	\begin{center}
		\includegraphics[width=0.7\linewidth]{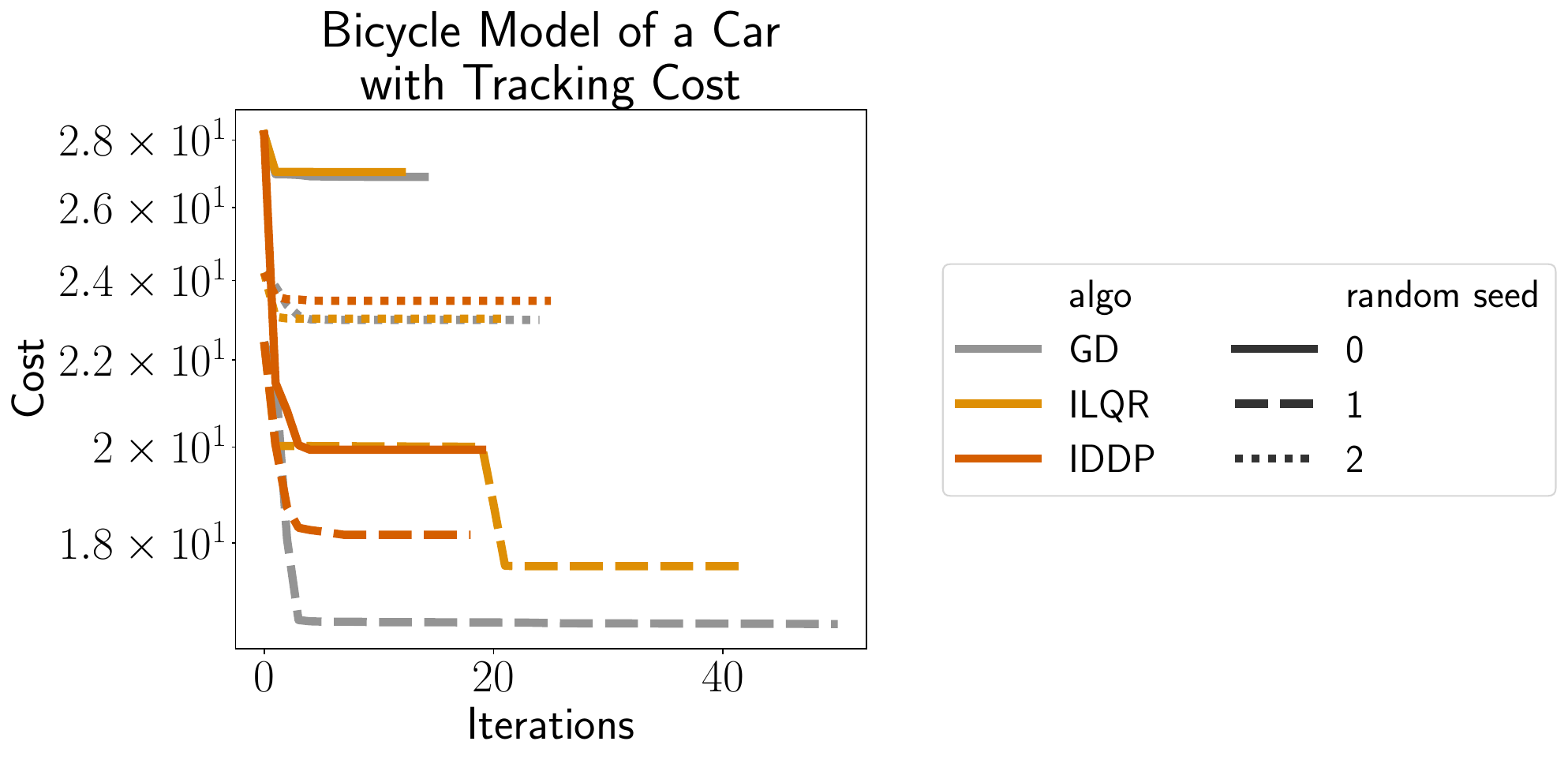}
	\end{center}
	\caption{Convergence of gradient descent (GD), ILQR and IDDP to control a
	bicycle model of a car for a tracking cost. \label{fig:add_conv}}
\end{figure}

On Figure~\ref{fig:add_conv}, we consider the same setting as for the simple
model of a car except that we replace the simple model of the dynamics of a car
by a bicycle model driven by tire forces taken
from~\citet{liniger2015optimization}, also detailed
by~\citet{roulet2021techreport}. We considered a fourth order Runge Kutta
discretization scheme of the continuous dynamics of the bicycle model of the
car. We keep a tracking cost as explained for the simple model of a car in
Section~\ref{sec:exp}. We use a discretization step $\Delta = T/\tau$ for a
total time $T=2$, and a number of discretization steps $\horizon=25$. We use
random initial control sequences $u_t^{(0)} \sim \mathcal{N}(0, \sigma)$ for
$\sigma = 1/\tau = 25$.

In this case, the ILQR and IDDP algorithms do not appear to converge to the same
value across random initial control sequences. This suggests no convergence to
global minima in this example.

Contouring costs and model predictive controllers can circumvent the difficulty
of this task as presented by~\citet{liniger2015optimization,
roulet2021techreport}. However, the bottleneck of, e.g., model predictive
controllers remain an algorithm such as ILQR or IDDP to compute the short term
policies. Understanding the behavior of these algorithms may then help the
design of model predictive controllers.

\paragraph{Convergence rates}
In Figure~\ref{fig:iter_rates_pendulum}, we plot convergence in iterates for the
pendulum example. We retrieve a similar superlinear rate of convergence after
some number of iterations.

In Figure~\ref{fig:function_value_rates}, we also
consider convergence rates in function values, that is, $\rho^{(k)} = (c^{(k)}-
c^*)/(c^{(k-1)}-c^*)$ for $c^{(k)}$ the cost at iterate $k$ and $c^*$ the
minimal cost. To plot this rate, we consider $c^*=0$ when subsampling the costs.

We observe generally a long phase where the convergence rate is close to one,
followed by a sudden phase of superlinear convergence where the rate drops to
$0$. The second phase of convergence outlined in the theory of
Section~\ref{sec:cvg} appears transient. The algorithms appear to mostly show a
phase of sublinear convergence followed by a phase of superlinear convergence.

\begin{figure}
  \begin{center}
    \includegraphics[width=0.7\linewidth]{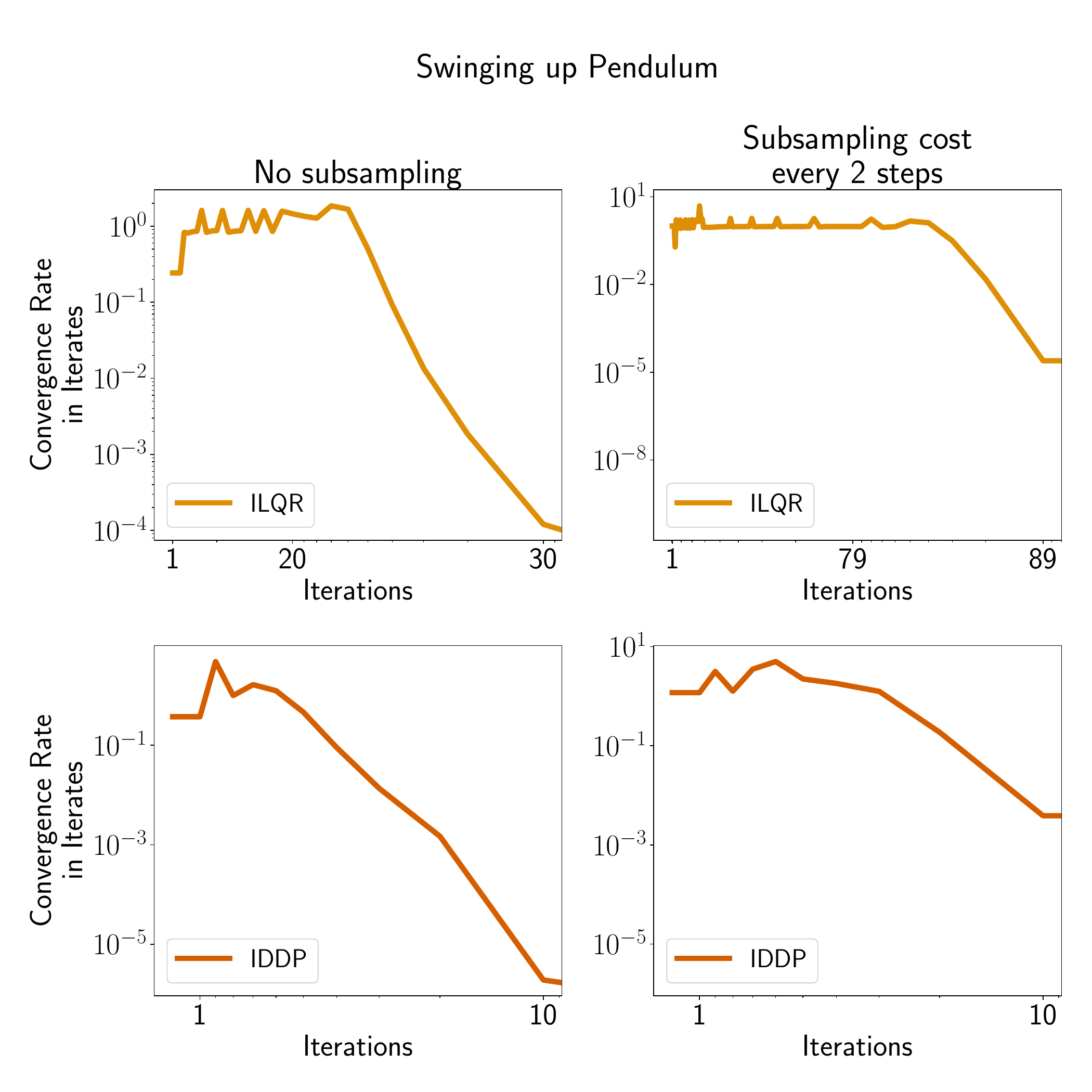}
    \caption{%
      Convergence rate in iterates, $\kappa^{(k)} =
      \|\ctrls^{(k+1)}-\ctrls^{(k)}\|_2/\|\ctrls^{(k)} - \ctrls^{(k-1)}\|_2$,
      along iterations of ILQR and IDDP algorithms for the pendulum example
      with or without subsampling the costs. For each algorithm and each setting
      we plot the convergence rate up to the final iterate before the algorithm
      stopped and use a log scale x-axis to zoom on the final iterates.
    }
    \label{fig:iter_rates_pendulum}
  \end{center}
\end{figure}

\begin{figure}
  \begin{center}
    \includegraphics[width=0.7\linewidth]{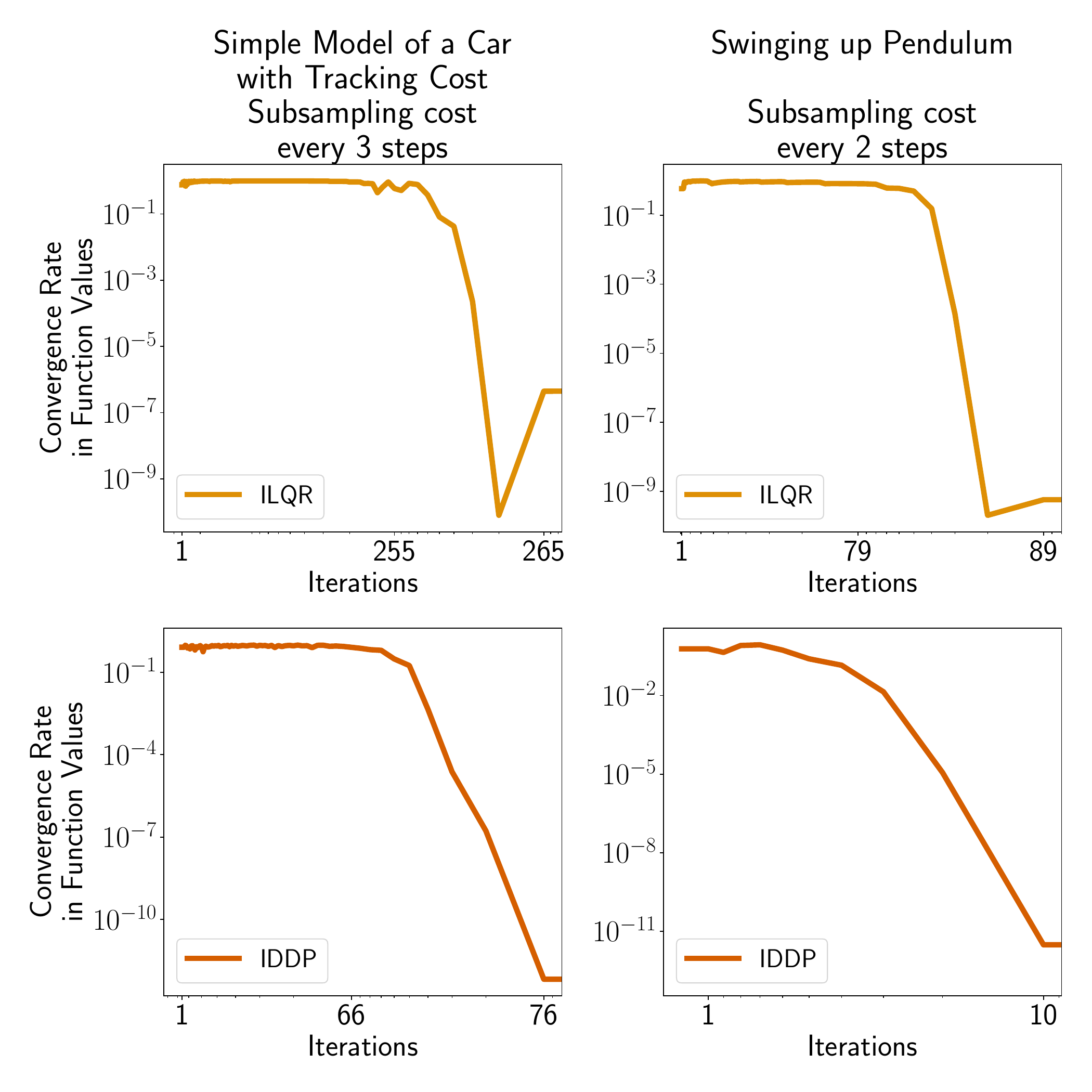}
    \caption{%
      Convergence rate in function values, $\rho^{(k)} = (c^{(k)}- c^*)/(c^{(k-1)}
      - c^*)$, along iterations of ILQR and IDDP algorithms for the pendulum
      example or the simple model of a car with subsampling the costs. The
      minimal cost is set to $c^*=0$. For each algorithm and each setting
      we plot the convergence rate up to the final iterate before the algorithm
      stopped and use a log scale x-axis to zoom on the final iterates.
    }
    \label{fig:function_value_rates}
  \end{center}
\end{figure}


\end{document}